\documentclass[reqno, 12pt, a4paper, twosided]{amsart}

\usepackage{tikz}
\usetikzlibrary{matrix,arrows,calc,snakes,patterns,decorations.markings}

\usepackage[mathscr]{eucal}
\usepackage{amscd}
\usepackage{amsfonts}
\usepackage{amsmath, amsthm, amssymb, bbm, bm}
\usepackage{stmaryrd}
\usepackage{lscape}

\usepackage{latexsym}
\usepackage{graphics}

\theoremstyle{plain}
\newtheorem{thm}{Theorem}[section]
\newtheorem{cor}[thm]{Corollary}
\newtheorem{lem}[thm]{Lemma}
\newtheorem{prop}[thm]{Proposition}
\theoremstyle{definition}
\newtheorem{obs}[thm]{Observation}
\newtheorem{defn}[thm]{Definition}
\newtheorem{ex}[thm]{Example}

\theoremstyle{remark}
\newtheorem{rem}[thm]{Remark}
\theoremstyle{setup}
\newtheorem{Setup}[thm]{Setup}
\newtheorem*{Key Statement}{Key Statement}

\numberwithin{equation}{section}

\newcommand{\ul}[1]{\underline{#1}}
\newcommand{\ol}[1]{\overline{#1}}
\newcommand{\brutegeq}[1]{\beta_{\geq #1}}

\DeclareMathAlphabet{\mathpzc}{OT1}{pzc}{m}{it}

\def\Rf{\mathop{\mathbf{R}f_*}\nolimits}
\def\Rfcs{\mathop{\mathbf{R}f_*^\cs}\nolimits}
\def\Lf{\mathop{\mathbf{L}f^*}\nolimits}

\DeclareMathOperator{\Perf}{\mathsf{Perf}}
\DeclareMathOperator{\Com}{\mathsf{Com}}

\DeclareMathOperator{\Band}{\mathsf{Band}}
\DeclareMathOperator{\Sing}{\mathsf{Sing}}

\def\GP{\mathop{\rm GP}\nolimits}

\DeclareMathOperator{\SCM}{\mathsf{SCM}}
\DeclareMathOperator{\SL}{\mathsf{SL}}
\DeclareMathOperator{\rk}{rk}
\DeclareMathOperator{\rad}{\mathsf{rad}}
\DeclareMathOperator{\tors}{\mathsf{tors}}

\DeclareMathOperator{\coker}{\mathsf{coker}}

\renewcommand{\ker}{\mathsf{ker}}
\newcommand{\im}{\mathsf{im}}

\renewcommand{\dim}{\mathsf{dim}}

\DeclareMathOperator{\Coh}{\mathsf{Coh}}

\DeclareMathOperator{\VB}{\mathsf{VB}}

\DeclareMathOperator{\MCM}{\mathsf{MCM}}

\DeclareMathOperator{\Pic}{Pic}

\newcommand{\del}{\partial}
\newcommand{\Cone}{\opname{cone}}

\DeclareMathOperator{\krdim}{\mathsf{kr.dim}}
\DeclareMathOperator{\prdim}{\mathsf{pr.dim}}
\DeclareMathOperator{\injdim}{\mathsf{inj.dim}}
\DeclareMathOperator{\gldim}{\mathsf{gl.dim}}

\DeclareMathOperator{\add}{\mathsf{add}}

\DeclareMathOperator{\depth}{\mathsf{depth}}
\DeclareMathOperator{\Mod}{\mathsf{Mod}}
\DeclareMathOperator{\Supp}{\mathsf{Supp}}

\newcommand{\Aus}{\opname{Aus}\nolimits}
\newcommand{\ind}{\opname{ind}}

\DeclareMathOperator{\Hom}{\mathsf{Hom}}
\DeclareMathOperator{\RHom}{\mathsf{RHom}}
\DeclareMathOperator{\Tor}{\mathsf{Tor}}

\DeclareMathOperator{\Ext}{\mathsf{Ext}}

\DeclareMathOperator{\GL}{\mathsf{GL}}

\DeclareMathOperator{\End}{\mathsf{End}}

\DeclareMathOperator{\Spec}{\mathsf{Spec}}

\newcommand{\opname}[1]{\operatorname{\mathsf{#1}}}

\newcommand{\ten}{\otimes}
\newcommand{\lten}{\overset{\opname{L}}{\ten}}
\newcommand{\cHom}{\mathcal{H}\it{om}}
\newcommand{\cEnd}{\mathcal{E}\it{nd}}

\renewcommand{\ker}{\opname{ker}\nolimits}

\newcommand{\thick}{\opname{thick}\nolimits}
\newcommand{\Tria}{\opname{Tria}\nolimits}
\newcommand{\per}{\opname{per}\nolimits}

\renewcommand{\mod}{\opname{mod}\nolimits}

\newcommand{\proj}{\opname{proj}\nolimits}

\renewcommand{\Mod}{\opname{Mod}\nolimits}

\renewcommand{\add}{\opname{add}\nolimits}
\newcommand{\op}{^{op}}

\newcommand{\ca}{{\mathcal A}}
\newcommand{\cb}{{\mathcal B}}
\newcommand{\cc}{{\mathcal C}}
\newcommand{\cd}{{\mathcal D}}
\newcommand{\ce}{{\mathcal E}}
\newcommand{\cf}{{\mathcal F}}

\newcommand{\ci}{{\mathcal I}}

\newcommand{\ck}{{\mathcal K}}
\newcommand{\cl}{{\mathcal L}}
\newcommand{\cm}{{\mathcal M}}

\newcommand{\co}{{\mathcal O}}
\newcommand{\cp}{{\mathcal P}}
\newcommand{\cq}{{\mathcal Q}}
\newcommand{\cs}{{\mathcal S}}
\newcommand{\ct}{{\mathcal T}}
\newcommand{\cu}{{\mathcal U}}
\newcommand{\cv}{{\mathcal V}}

\newcommand{\cx}{{\mathcal X}}
\newcommand{\cy}{{\mathcal Y}}
\newcommand{\cz}{{\mathcal Z}}

\newcommand{\Z}{\mathbb{Z}}
\newcommand{\N}{\mathbb{N}}

\newcommand{\C}{\mathbb{C}}

\renewcommand{\P}{\mathbb{P}}
\newcommand{\T}{\mathbb{T}}
\newcommand{\F}{\mathbb{F}}
\renewcommand{\H}{\mathbb{H}}
\newcommand{\G}{\mathbb{G}}
\newcommand{\I}{\mathbb{I}}
\newcommand{\ra}{\rightarrow}
\newcommand{\la}{\leftarrow}

\newcommand{\id}{\mathbf{1}}

\input xy
\xyoption{all}

\makeatletter
\def\l@section{\@tocline{1}{0pt}{1pc}{}{}}
\def\l@subsection{\@tocline{2}{0pt}{1pc}{4.6em}{}}
\renewcommand{\tocsection}[3]{%
  \indentlabel{\@ifnotempty{#2}{\makebox[2.3em][l]{%
    \ignorespaces\bf#1 #2.\hfill}}}\bf#3}
\renewcommand{\tocsubsection}[3]{%
  \indentlabel{\@ifnotempty{#2}{\hspace*{2.3em}\makebox[2.3em][l]{%
    \ignorespaces#1 #2.\hfill}}}#3}
\makeatother

\newcommand{\kk}{k}
\newcommand{\bsm}{\begin{smallmatrix}}
\newcommand{\esm}{\end{smallmatrix}}

\renewcommand{\mod}{\mathsf{mod}}
\newcommand{\fdmod}{\mathsf{fdmod}}

\newcommand{\kA}{\mathcal{A}}
\newcommand{\kB}{\mathcal{B}}
\newcommand{\kC}{\mathcal{C}}
\newcommand{\kD}{\mathcal{D}}
\newcommand{\kE}{\mathcal{E}}
\newcommand{\kF}{\mathcal{F}}
\newcommand{\kG}{\mathcal{G}}
\newcommand{\kH}{\mathcal{H}}
\newcommand{\kI}{\mathcal{I}}

\newcommand{\kO}{\mathcal{O}}

\newcommand{\kP}{\mathcal{P}}
\newcommand{\kQ}{\mathcal{Q}}
\newcommand{\kK}{\mathcal{K}}
\newcommand{\kM}{\mathcal{M}}
\newcommand{\kN}{\mathcal{N}}

\newcommand{\kT}{\mathcal{T}}
\newcommand{\kS}{\mathcal{S}}
\newcommand{\kR}{\mathcal{R}}

\newcommand{\kU}{\mathcal{U}}

\newcommand{\cP}{\mathsf{P}}

\newcommand{\lar}{\longrightarrow}

\newcommand{\gm}{\mathfrak{m}}

\newcommand{\idm}{\mathfrak{m}}

\renewcommand{\AA}{\mathbb{A}}

\newcommand{\FF}{\mathbb{F}}
\newcommand{\GG}{\mathbb{G}}
\newcommand{\HH}{\mathbb{H}}

\newcommand{\TT}{\mathbb{T}}

\newcommand{\PP}{\mathbb{P}}
\newcommand{\UU}{\mathbb{U}}
\newcommand{\ZZ}{\mathbb{Z}}
\newcommand{\XX}{\mathbb{X}}

\newcommand{\String}[3]{\mathcal{S}_{#1}(#2)[#3]}
\newcommand{\lo}{\gamma}
\newcommand{\ro}{\alpha}
\newcommand{\ru}{\delta}
\newcommand{\lu}{\beta}
\newcommand{\li}{-}
\newcommand{\re}{+}

\renewcommand{\t}[1]{\textnormal{#1}}

\begin{document}
\begin{center}
\thispagestyle{empty} $ $\,
\par\end{center}

\,

\,

\,

\,

\,

\,

\,

\,

\,

\,

\,

\,

\,

\,

\,

\,

\,

\,

\,

\,

\,

\,

\,

\,

\begin{center}
{\large Relative singularity categories}\emph{}\\
\emph{}\\

\par\end{center}

\,

\,

\,\,

\,

\,

\,

\,\,

\,

\,

\,

\,

\,

\,

\begin{center}
Dissertation 
\par\end{center}

\,

\,

\,

\,

\,

\,

\,

\,

\,

\,

\,

\,

\,

\,

\,

\,

\,

\,

\,

\,

\,

\,

\begin{center}
zur 
\par\end{center}

\begin{center}
Erlangung des Doktorgrades (Dr. rer. nat.) 
\par\end{center}

\begin{center}
der 
\par\end{center}

\begin{center}
Mathematisch-Naturwissenschaftlichen Fakult\"at
\par\end{center}

\begin{center}
der 
\par\end{center}

\begin{center}
Rheinischen Friedrich-Wilhelms-Universit\"at Bonn\\
\,
\par\end{center}

\,

\,

\,

\,

\,

\,

\,

\,

\,

\,

\,

\,

\,

\,

\,

\,

\,

\,

\,

\,

\,

\,

\,

\,

\,

\,

\,

\,\,

\,

\,

\,

\,

\begin{center}
vorgelegt von 
\par\end{center}

\,

\begin{center}
Martin Kalck
\par\end{center}

\,

\begin{center}
aus 
\par\end{center}

\,

\begin{center}
Hamburg\\
\,
\par\end{center}

\,

\,

\,

\,

\,

\,

\,

\,

\,

\,

\,

\,

\,

\,

\,

\,

\,

\,

\,

\,

\,

\,

\,

\,

\,

\,

\,

\,

\,

\,

\,

\,

\,

\,

\,

\,

\,

\,

\,

\,

\,

\,

\,

\begin{center}
Bonn 2013
\par\end{center}

\newpage\thispagestyle{empty}  \noindent Angefertigt mit Genehmigung
der Mathematisch-Naturwissenschaftlichen Fakult\"at der Rheinischen
Friedrich-Wilhelms-Universit\"at Bonn\\
\\
\\
\\
\\
\\
\\
\\
\\
\\
\\
\\
\\
\\
\\
\\
\\
\\
\\
\\
\\
\\
\\
\\
\\
\\
\\
\\

1. Gutachter: Prof.~Dr.~Igor Burban

2. Gutachter: Prof.~Dr.~Jan Schr\"oer\\

Tag der Promotion: 29.05.2013\\

Erscheinungsjahr: 2013

\newpage
\thispagestyle{empty}
\begin{center}
\large\bf Summary
\end{center}
\quad \\
In this thesis, we study a new class of triangulated categories associated with singularities of algebraic varieties. For a  Gorenstein ring $A$, Buchweitz introduced the triangulated category $\cd_{sg}(A)=\cd^b(\mod-A)/\Perf(A)$ nowadays called the \emph{triangulated category of singularities}. 
In 2006 Orlov introduced a graded version of these categories relating them with derived categories of coherent sheaves on projective varieties. This construction has already found various applications, for example in the Homological Mirror Symmetry.

The first result of this thesis is a description of $\cd_{sg}(A)$, for the class of Artinian Gorenstein algebras called \emph{gentle}. The main part of this thesis is devoted to the study of the following generalization of $\cd_{sg}(A)$.        
 Let  $X$ be a quasi-projective Gorenstein scheme with isolated singularities, $\cf$ a coherent sheaf on $X$ such that the sheaf of algebras $\ca={\mathcal End}_{X}(\co_{X}\oplus \cf)$ has finite global dimension. Then we have the following embeddings of triangulated categories
\begin{align*}
\cd^b\bigl(\Coh(X)\bigr) \supseteq \Perf(X) \subseteq \cd^b\bigl(\Coh(\XX)\bigr).
\end{align*}
Van den Bergh suggested to regard the ringed space $\XX=(X, \ca)$ as a non-commutative resolution of singularities of $X$. We introduce  the \emph{relative singularity category} \[\Delta_{X}(\XX)=\cd^b\bigl(\Coh(\XX)\bigr)/\Perf(X)\]
as a measure for the difference between $X$ and $\XX$. The main results of this thesis are the following

\noindent
(i) We prove the following localization property of $\Delta_{X}(\XX)$:
\begin{align*}\Delta_{X}(\XX) \cong \bigoplus_{x \in \Sing(X)} \Delta_{\widehat{\co}_{x}}\left(\widehat{\ca}_{x}\right):=  \bigoplus_{x \in \Sing(X)} \frac{\cd^b(\widehat{\ca}_{x}-\mod)}{\Perf(\widehat{\co}_{x})}.
\end{align*}
Thus the study of $\Delta_{X}(\XX)$ reduces to the affine case $X=\Spec(O)$.

\noindent
(ii) We prove Hom-finiteness and idempotent completeness of $\Delta_{O}(A)$ and determine its Grothendieck group.

\noindent
(iii) For the nodal singularity $O=k\llbracket x, y\rrbracket/(xy)$ and its Auslander resolution $A=\End_{O}(O \oplus \mathfrak{m})$, we classify all indecomposable objects of $\Delta_{O}(A)$.

\noindent
(iv) We study relations between $\Delta_{O}(A)$ and $\cd_{sg}(A)$. For a simple hypersurface singularity $O$ and its \emph{Auslander resolution} $A$, we show that these categories determine each other.

\noindent
(v)  The developed technique leads to the following `purely commutative' statement:\[\underline{\underline{\mathsf{SCM}}}(R) \xrightarrow{\cong} \mathcal{D}_{sg}(X) \cong  \bigoplus_{x \in \mathsf{Sing}(X)} \underline{\mathsf{MCM}}(\widehat{\mathcal{O}}_{x}),\]
where $R$ is a  rational surface singularity, $\mathsf{SCM}(R)$ is the Frobenius category of \emph{special} Cohen--Macaulay modules, and $X$ is the \emph{rational double point resolution} of $\mathsf{Spec}(R)$.
\newpage

\tableofcontents
\newpage

\section{Introduction}
This thesis studies triangulated and exact categories arising in singularity theory and representation theory.  
Starting with Mukai's seminal work \cite{Mukai} on Fourier-type equivalences between the derived categories of an abelian variety and its dual variety, derived categories of coherent sheaves have been recognised as interesting invariants of the underlying projective scheme. In particular, the relations to the Minimal Model Program and Kontsevich's Homological Mirror Symmetry Conjecture \cite{Kontsevich} are active fields of current research.

On the contrary, two affine schemes with equivalent bounded derived categories are already isomorphic as schemes, by Rickard's `derived Morita theory' for rings \cite{RickardMorita}. This statement remains true if the bounded derived category is replaced by the subcategory of perfect complexes, i.e.~complexes which are quasi-isomorphic to bounded complexes of finitely generated projective modules. So neither the bounded derived category $\cd^b(\mod-R)$ nor the subcategory of perfect complexes $\Perf(R) \subseteq \cd^b(\mod-R)$ are interesting categorical invariants of a commutative Noetherian ring $R$. In 1987, Buchweitz \cite{Buchweitz87} suggested to study the quotient category  
\begin{align}
\cd_{sg}(R):=\frac{\cd^b(\mod-R)}{\Perf(R)}.
\end{align}
As indicated below, this category is a useful invariant of $R$. It is known as the \emph{singularity category} of $R$. 
If $R$ is a \emph{regular} ring, then every bounded complex admits a finite projective resolution and therefore we have the equality $\Perf(R)=\cd^b(\mod-R)$. In particular, the singularity category of a regular ring vanishes. Moreover, the category of perfect complexes $\Perf(R)$ can be considered as the smooth part of $\cd^b(\mod-R)$ and $\cd_{sg}(R)$ as a measure for the complexity of the singularities of $\Spec(R)$. 

If $R$ is Gorenstein, i.e.~of finite injective dimension as a module over itself, then Buchweitz proved the following  equivalence of triangulated categories
\begin{align}\label{E:Buchweitz}
\ul{\MCM}(R) \longrightarrow \cd_{sg}(R),
\end{align} 
where the left hand side denotes the stable category of maximal Cohen--Macaulay modules. We list some properties of a singularity, which are `detected' by the singularity category.
\begin{itemize}
\item[(a)] Orlov \cite{Orlov11} has shown that two analytically isomorphic singularities have equivalent singularity categories (up to taking direct summands). 
\item[(b)] Let $k$ be an algebraically closed field. A commutative complete Gorenstein $k$-algebra $(R, \mathfrak{m})$ satisfying $k \cong R/\mathfrak{m}$ has an \emph{isolated singularity} if and only if $\cd_{sg}(R)$ is a Hom-finite category, by work of Auslander \cite{Auslander84}.
\item[(c)] Let $S=\C\llbracket z_{0}, \ldots, z_{d}\rrbracket$ and $f \in (z_{0}, \ldots, z_{d}) \setminus \{0\}$. Then Kn\"orrer \cite{Knoerrer} proved the following equivalence of triangulated categories
\begin{align}\label{E:Knoerrer}
\cd_{sg}\bigl(S/(f)\bigr) \longrightarrow \cd_{sg}\bigl(S\llbracket x, y \rrbracket/(f+x^2+y^2)\bigr).
\end{align}
\item[(d)] Kn\"orrer \cite{Knoerrer} and Buchweitz, Greuel \& Schreyer \cite{BuchweitzGreuelSchreyer} showed that a complete hypersurface ring $R=S/(f)$ defines a \emph{simple singularity}, i.e.~a singularity which deforms only into finitely many other singularities, if and only if the singularity category has only finitely many indecomposable objects. 
\end{itemize}
Recently, Orlov \cite{Orlov09} introduced a global version of singularity categories for any Noetherian scheme $X$ and related it to Kontsevich's Homological Mirror Symmetry Conjecture. Moreover, if $X$ has isolated singularities, then the global singularity category is determined by the local singularity categories from above (up to taking direct summands) \cite{Orlov11}.
\medskip

This work consists of four parts, which deal with (relative) singularity categories in singularity theory and representation theory. 
\begin{itemize}
\item Firstly, we study the \emph{relative singularity category} associated with a non-commutative resolution of a Noetherian scheme $X$. This construction is inspired by the constructions of Buchweitz and Orlov and we show that its description reduces to the local situation if $X$ has \emph{isolated} singularities. 
\item In the second part, we investigate these local relative singularity categories. In particular, we give an explicit description of these categories for Auslander resolutions of $\mathbb{A}_{1}$-hypersurface singularities $x_{0}^2+\ldots+x_{d}^2$ in all Krull dimensions. Moreover, we study the relations between the relative and the classical singularity categories using general techniques: for example, dg-algebras and recollements. Our main result shows that the classical and relative singularity categories mutually determine each other in the case of ADE-singularities. 
\item The third part is inspired by the techniques developed in the second part. However, the result  is `purely commutative'. Namely, we give an explicit description of Iyama \& Wemyss' stable category of special Cohen--Macaulay modules over rational surface singularities in terms of the well-known singularity categories of rational double points. 
\item Finally, we describe the singularity categories of finite dimensional gentle algebras. They are equivalent to finite products of $m$-cluster categories of type $\mathbb{A}_{1}$.
\end{itemize}

\subsection{Global relative singularity categories}
Let $X$ be a Noetherian scheme. The \emph{singularity category} of $X$ is the triangulated quotient category $\cd_{sg}(X):=\cd^b(\Coh(X))/\Perf(X)$, where $\Perf(X)$ denotes the full subcategory of complexes which are locally quasi-isomorphic to bounded complexes of locally free sheaves of finite rank. If $X$ has isolated singularities and every coherent sheaf is a quotient of a locally free sheaf, then Orlov \cite{Orlov11} proves the following block decomposition of the idempotent completion\footnote{For example, the singularity category of an irreducible nodal cubic curve is \emph{not} idempotent complete, see \cite[Appendix]{KMV} for a detailed explanation.} $(-)^\omega$ \cite{BalmerSchlichting} of the singularity category
\begin{align}\label{E:OrlovLoc}
\bigl(\cd_{sg}(X)\bigr)^\omega \cong \bigoplus_{x \in \Sing(X)} \cd_{sg}(\widehat{\co}_{x}).
\end{align} 
In other words, it suffices to understand the `local' singularity categories $\cd_{sg}(\widehat{\co}_{x})$ in this case.

Starting with Van den Bergh's works \cite{VandenBergh04, NCCR}, non-commutative analogues of (crepant) resolutions of singularities have been studied intensively in recent years. Non-commutative resolutions are useful even if the primary interest lies in commutative questions: for example, the Bondal--Orlov Conjecture concerning derived equivalences between (commutative) crepant resolutions and the derived McKay-Correspondence \cite{BKR, KapranovVasserot00} led Van den Bergh to the notion of a non-commutative crepant resolution (NCCR). Moreover, moduli spaces of quiver representations provide a very useful technique to obtain commutative resolutions from non-commutative resolutions, see for example~\cite{NCCR, WemyssReconstructionTypeA}.

Globally, we consider the following construction. Let $\cf=\co_{X} \oplus \cf'$ be a coherent sheaf on $X$. This yields a coherent sheaf of $\co_{X}$-algebras $\ca={\mathcal End}_{X}(\cf)$ and a locally ringed space $\XX=(X, \ca)$. For example, Burban \& Drozd \cite{Tilting}Ê  studied \emph{Auslander sheaves} on rational curves with only nodal and cuspidal singularities. The corresponding derived categories $\cd^b(\Coh \XX)$ admit tilting complexes which, in the nodal case, have \emph{gentle} endomorphism algebras.

It is well-known that the triangle functor
\begin{align}
\cf \lten_{X} - \colon \Perf(X) \longrightarrow \cd^b(\Coh \XX)
\end{align}
is fully faithful.
If $\gldim \Coh(\XX)<\infty$, then we consider $\XX$ as a non-commutative resolution of $X$. 
In analogy with the classical construction of Buchweitz and Orlov, it is natural to study the (idempotent completion\footnote{An explicit example of a non-split idempotent may be found in Lemma \ref{L:Non-Split}.} of the) following Verdier
quotient category, which we call \emph{relative singularity category}
\begin{align}
\Delta_{X}(\XX):=\left(\frac{\cd^b(\Coh \XX)}{\Perf(X)}\right)^{\omega}.
\end{align}
In a joint work with Igor Burban \cite{BurbanKalck11}, we obtained the following analogue of Orlov's Localization Theorem \eqref{E:OrlovLoc}. It is the main result of this section, see Theorem \ref{t:main-global}.

\begin{thm} \label{t:main-global-intro}
Let $k$ be an algebraically closed field and let $X$ be a seperated excellent Noetherian scheme with \emph{isolated} singularities $\{x_{1}, \ldots, x_{n}\}$ over $k$, such that every coherent sheaf is a quotient of a locally free sheaf. Let $\cf=\co_{X} \oplus \cf' \in \Coh(X)$ such that $\cf$ is locally free on $X \setminus \Sing(X)$. We set $\ca={\mathcal End}_{X}(\cf)$ and $\XX=(X, \ca)$. 

Then there is an equivalence of triangulated categories
\begin{align}\label{E:BlockDecompIntro}
\Delta_{X}(\XX) \cong \bigoplus_{i=1}^n \Delta_{\widehat{\co}_{x_{i}}}\left(\widehat{\ca}_{x_{i}}\right):=  \bigoplus_{i=1}^n \left(\frac{\cd^b(\widehat{\ca}_{x_{i}}-\mod)}{\Perf(\widehat{\co}_{x_{i}})}\right)^\omega.
\end{align}
\end{thm}
\noindent This motivates our study of the \emph{local} relative singularity categories 
\begin{align}
\Delta_{\widehat{\co}_{x_{i}}} (\widehat{\ca}_{x_{i}}):=\bigl(\cd^b(\widehat{\ca}_{x_{i}}-\mod)/\Perf(\widehat{\co}_{x_{i}})\bigr)^{\omega}
\end{align}
which we explain in the next subsection.

\newpage

\subsection{Local relative singularity categories}
\subsubsection{Setup}\label{ss:SetupIntro}
Let $k$ be an algebraically closed field and $(R, \mathfrak{m})$ be a commutative local complete Gorenstein $k$-algebra such that $k\cong R/\mathfrak{m}$. Recall that the full subcategory of \emph{maximal Cohen--Macaulay} $R$-modules may be written as  $\MCM(R)=\left\{\left.M \in \mod-R \, \right| \Ext^i_{R}(M, R)=0 \text{ for all } i>0 \right\}$ since $R$ is Gorenstein. 
Let $M_{0}=R, M_{1}, \ldots, M_{t}$ be pairwise non-isomorphic indecomposable $\MCM$ $R$-modules, $M:=\bigoplus_{i=0}^t M_{i}$ and $A=\End_{R}(M)$. If $\gldim(A)< \infty$ then $A$ is called \emph{non-commutative resolution} (NCR) of $R$ (this notion was recently studied in~\cite{DaoIyamaTakahashiVial12}). For example, if $R$ has only finitely many indecomposables $\MCM$s and $M$ denotes their direct sum, then the \emph{Auslander algebra} $\Aus(\MCM(R)):=\End_{R}(M)$ is a NCR (\cite[Theorem A.1]{Auslander84}).  In analogy with the global situation, there is a fully faithful triangle functor
$
K^b(\proj-R) \rightarrow \cd^b(\mod-A),
$
whose essential image equals $\thick(eA) \subseteq \cd^b(\mod-A)$ for a certain idempotent $e \in A$.

\begin{defn}\label{D:RelSingCat}
The \emph{relative singularity category} is the Verdier quotient category \begin{align}\label{E:DefRelSingCat} \Delta_{R}(A):=\frac{\cd^b(\mod-A)}{K^b(\proj-R)} \cong \frac{\cd^b(\mod-A)}{\thick(eA)}.\end{align} 
\end{defn}

\begin{rem}
(a) \, Our notion of relative singularity categories is a special case of Chen's definition \cite{Chen11}. Moreover, these categories have been studied by Thanhoffer de V\"olcsey \& Van den Bergh \cite{ThanhofferdeVolcseyMichelVandenBergh10} for certain Gorenstein quotient singularities, see Remark \ref{R:Overlap} for more details. Different notions of relative singularity categories were introduced and studied by Positselski \cite{Positselski11} and also by Burke \& Walker \cite{BurkeWalker12}.

\noindent
(b) \, We show (Proposition \ref{P:IdempCompl}) that the quotient category $\cd^b(\mod-A)/K^b(\proj-R)$ is idempotent complete, if $R$ is a complete Gorenstein ring as in the setup above. In particular, Definition \ref{D:RelSingCat} is compatible with the definition given in \eqref{E:BlockDecompIntro}.
\end{rem}

\subsubsection{Main result} It is natural to ask how the notions of classical and relative singularity category  are related. In joint work with Dong Yang \cite{KalckYang12}, we obtained a first answer to this question for Auslander resolutions of $\MCM$--representation finite singularities, see Theorem \ref{t:Classical-versus-generalized-singularity-categories}.

\begin{thm}\label{t:Classical-versus-generalized-singularity-categories-intro}
 Let $R$ and $R'$ be $\MCM$--representation finite complete Gorenstein $k$-algebras with Auslander algebras $A=\Aus(\MCM(R))$ and $A'=\Aus(\MCM(R'))$, respectively. Then the following statements are equivalent.
\begin{itemize}
\item[$(i)$] There is an equivalence $\cd_{sg}(R) \cong \cd_{sg}(R')$ of triangulated categories.
\item[$(ii)$] There is an equivalence $\Delta_{R}\!\left(A\right) \cong \Delta_{R'}\!\left(A'\right)$ of triangulated categories.
\end{itemize}
The implication $(ii) \Rightarrow (i)$ holds more generally for non-commutative resolutions $A$ and $A'$ of arbitrary isolated Gorenstein singularities $R$ and $R'$, respectively.
\end{thm} 
\begin{rem}
(a) \, Kn\"orrer's periodicity theorem \eqref{E:Knoerrer} yields a wealth of non-trivial examples for triangle equivalences $\cd_{sg}(R) \cong \cd_{sg}(R')$. 

\noindent
(b) \, The definition of the relative Auslander singularity category $\Delta_{R}(\Aus(\MCM(R))$ does not involve any choices. Using Theorem \ref{t:Classical-versus-generalized-singularity-categories-intro} and Kn\"orrer's periodicity, there are two (new) canonical triangulated categories associated with any Dynkin diagram of ADE-type. Namely, the relative Auslander singularity category of the even and odd dimensional ADE hypersurface singularities.

\noindent
(c) \, The implication $(ii) \Rightarrow (i)$ may also be deduced from work of Thanhoffer de V\"olcsey \& Van den Bergh \cite{ThanhofferdeVolcseyMichelVandenBergh10}.
\end{rem}

\begin{ex} Let $R=\C\llbracket x \rrbracket/(x^2)$ and $R'=\C\llbracket x, y, z \rrbracket/(x^2+y^2+ z^2)$. Kn\"orrer's equivalence (\ref{E:Knoerrer}) in conjunction with Theorem \ref{t:Classical-versus-generalized-singularity-categories-intro}, yields a triangle equivalence $\Delta_{R}(\Aus(\MCM(R))) \cong \Delta_{R'}(\Aus(\MCM(R')))$, which may be written explicitly as  
\begin{align}
\frac{\cd^b\left(\left.
\begin{xy}
\SelectTips{cm}{10}
\xymatrix{1 \ar@/^10pt/[rr]|{\,\, p\,\,}  && 2 \ar@/^10pt/[ll]|{\,\, i\,\,} }\end{xy}\right/ (pi)
\right)}{K^b(\add P_{1})}
\stackrel{\sim}\longrightarrow
\frac{\cd^b\left(\left.
\begin{xy}
\SelectTips{cm}{10}
\xymatrix{1 \ar@<-1.5pt>@/^10pt/[rr]|{\,\, y\,\,} \ar@<5pt>@/^10pt/[rr]|{\,\, x\,\,} && 2 \ar@<-1.5pt>@/^10pt/[ll]|{\,\, y\,\,} \ar@<5pt>@/^10pt/[ll]|{\,\, x\,\,}}
\end{xy}\right/ (xy-yx)
\right)}{K^b(\add P_{1})}.
\end{align}
\noindent The quiver algebra on the right is the completion of the preprojective algebra of the Kronecker quiver $\begin{xy}\SelectTips{cm}{10}\xymatrix{\circ \ar@/^/[r] \ar@/_/[r] &  \circ}\end{xy}$. Moreover, the derived McKay--Correspondence \cite{KapranovVasserot00, BKR} shows that this algebra is derived equivalent to the derived category of coherent sheaves on the minimal resolution of the completion of the Kleinian singularity $\C^2/\Z_{2}$.
\end{ex}

\subsubsection{Idea of the proof}
We prove Theorem \ref{t:Classical-versus-generalized-singularity-categories-intro} by developing a 
general dg algebra framework. 
To be more precise, let $\ct$ be a $k$-linear Hom-finite idempotent complete algebraic triangulated category with finitely many indecomposable objects. If $\ct$ satisfies a certain (weak) extra assumption\footnote{For example, all `standard' categories, i.e. categories such that the Auslander algebra is given as the quiver algebra of the Auslander--Reiten quiver modulo the ideal generated by the mesh relations, satisfy this assumption.}, then $\ct$ determines a 
dg algebra $\Lambda_{dg}(\ct)$, which we call the \emph{dg Auslander algebra of $\ct$}, see Definition~\ref{d:dg-aus-alg}. In particular, this applies to the stable category $\ct=\ul{\MCM}(R)$. Now, using recollements generated by idempotents, Koszul duality and the fractional Calabi--Yau property \eqref{E:IntroFractCY}, we prove the following key statement, see Theorem \ref{t:main-thm-2}.
\begin{Key Statement} There is an equivalence of triangulated categories
\begin{align}\label{E:KeyStatement}
\Delta_{R}\Bigl(\Aus\bigl(\MCM(R)\bigr)\Bigr) \cong \per\Bigl(\Lambda_{dg}\bigl(\ul{\MCM}(R)\bigr)\Bigr). 
\end{align}
\end{Key Statement}
In particular, this shows that $(i)$ implies $(ii)$. Conversely, written in this language, the quotient functor (\ref{E:class-as-quotient}), induces an equivalence of triangulated categories
\begin{align}\label{E:Cluster-Equiv}
\frac{\per\Big(\Lambda_{dg}\big(\ul{\MCM}(R)\big)\Big)}{\cd_{fd}\Big(\Lambda_{dg}\big(\ul{\MCM}(R)\big)\Big)} \longrightarrow \ul{\MCM}(R).
\end{align} Since the category $\cd_{fd}(\Lambda_{dg}(\ul{\MCM}(R)))$ of dg modules with finite dimensional total cohomology admits an intrinsic characterization inside $\per(\Lambda_{dg}(\ul{\MCM}(R)))$, this proves that $\ul{\MCM}(R)$ is determined by $\Delta_{R}(\Aus(\MCM(R)))$. Hence, (ii) implies (i).

\begin{rem}\label{R:Overlap}
Thanhoffer de V\"olcsey \& Van den Bergh \cite{ThanhofferdeVolcseyMichelVandenBergh10} prove an analogue of the Key Statement \eqref{E:KeyStatement} for `cluster resolutions' of certain Gorenstein quotient singularities $R$. By using an analogue of \eqref{E:Cluster-Equiv}, they show that the stable category of maximal Cohen--Macaulay $R$-modules is a  generalized cluster category in the sense of Amiot and Guo \cite{Amiot09, Guolingyan11a}. This was first proved by Amiot, Iyama \& Reiten \cite{AIR} by different means.

In Krull dimension two, it is well-known that the ADE-singularities are quotient singularities. Moreover, in this case, the Auslander and cluster resolutions coincide (indeed, the sum of all indecomposable MCMs is a $1$-cluster tilting object in a $1$-Calabi--Yau category) and the dg Auslander algebra is the deformed dg preprojective algebra $\Pi(Q, 2, 0)$  as defined by Ginzburg \cite{Ginzburg06} (see also Van den Bergh \cite{VanDenBergh10}). Here $Q$ is a Dynkin quiver of the same type as the singularity $R$.

However, this is the only overlap between the setup of Thanhoffer de V\"olcsey \& Van den Bergh and ours. Indeed, in Krull dimensions different from two the ADE-singularities are never quotient singularities: one dimensional quotient singularities are regular and Schlessinger \cite{Schlessinger} has shown that isolated quotient singularities in dimensions greater than three are \emph{rigid}, i.e.~do not admit any non-trivial (infinitesimal) deformations. On the other hand, it is well-known that ADE-singularities always admit non-trivial deformations.
\end{rem}

\begin{ex}
Let $R=\C\llbracket z_{0}, \ldots, z_{d}\rrbracket/(z_{0}^{n+1} + z_{1}^2 + \ldots + z_{d}^2)$ be an $A_{n}$-singularity of \emph{even} Krull dimension. Then the graded quiver $Q$ of the dg Auslander algebra $\Lambda_{dg}(\ul{\MCM}(R))$ is given as

\begin{equation*}\begin{tikzpicture}[description/.style={fill=white,inner sep=2pt}]
    \matrix (n) [matrix of math nodes, row sep=3em,
                 column sep=2.5em, text height=1.5ex, text depth=0.25ex,
                 inner sep=0pt, nodes={inner xsep=0.3333em, inner
ysep=0.3333em}]
    {  
       1 & 2 & 3 &\cdots& n-1 &n \\
    };
    
    \draw[->] ($(n-1-1.east) + (0mm,1mm)$) .. controls +(2.5mm,1mm) and
+(-2.5mm,+1mm) .. ($(n-1-2.west) + (0mm,1mm)$);
  \node[scale=0.75] at ($(n-1-1.east) + (5.5mm, 3.8mm)$) {$\alpha_{1}$};
  
    \draw[->] ($(n-1-2.west) + (0mm,-1mm)$) .. controls +(-2.5mm,-1mm)
and +(+2.5mm,-1mm) .. ($(n-1-1.east) + (0mm,-1mm)$);
\node[scale=0.75] at ($(n-1-1.east) + (5.5mm, -4.1mm)$) {$\alpha^*_{1}$};

    \draw[->] ($(n-1-2.east) + (0mm,1mm)$) .. controls +(2.5mm,1mm) and
+(-2.5mm,+1mm) .. ($(n-1-3.west) + (0mm,1mm)$);
  \node[scale=0.75] at ($(n-1-2.east) + (5.5mm, 3.8mm)$) {$\alpha_{2}$};
  
    \draw[->] ($(n-1-3.west) + (0mm,-1mm)$) .. controls +(-2.5mm,-1mm)
and +(+2.5mm,-1mm) .. ($(n-1-2.east) + (0mm,-1mm)$);
  \node[scale=0.75] at ($(n-1-2.east) + (5.5mm, -4.1mm)$) {$\alpha^*_{2}$};

    \draw[->] ($(n-1-3.east) + (0mm,1mm)$) .. controls +(2.5mm,1mm) and
+(-2.5mm,+1mm) .. ($(n-1-4.west) + (0mm,1mm)$);
  \node[scale=0.75] at ($(n-1-3.east) + (5.5mm, 3.8mm)$) {$\alpha_{3}$};

    \draw[->] ($(n-1-4.west) + (0mm,-1mm)$) .. controls +(-2.5mm,-1mm)
and +(+2.5mm,-1mm) .. ($(n-1-3.east) + (0mm,-1mm)$);
  \node[scale=0.75] at ($(n-1-3.east) + (5.5mm, -4.1mm)$) {$\alpha^*_{3}$};
   
    \draw[->] ($(n-1-4.east) + (0mm,1mm)$) .. controls +(2.5mm,1mm) and
+(-2.5mm,+1mm) .. ($(n-1-5.west) + (0mm,1mm)$);
  \node[scale=0.75] at ($(n-1-4.east) + (5.5mm, 3.8mm)$) {$\alpha_{n-2}$};

    \draw[->] ($(n-1-5.west) + (0mm,-1mm)$) .. controls +(-2.5mm,-1mm)
and +(+2.5mm,-1mm) .. ($(n-1-4.east) + (0mm,-1mm)$);
  \node[scale=0.75] at ($(n-1-4.east) + (5.5mm, -4.1mm)$) {$\alpha^*_{n-2}$};

    \draw[->] ($(n-1-5.east) + (0mm,1mm)$) .. controls +(2.5mm,1mm) and
+(-2.5mm,+1mm) .. ($(n-1-6.west) + (0mm,1mm)$);
  \node[scale=0.75] at ($(n-1-5.east) + (5.5mm, 3.8mm)$) {$\alpha_{n-1}$};

    \draw[->] ($(n-1-6.west) + (0mm,-1mm)$) .. controls +(-2.5mm,-1mm)
and +(+2.5mm,-1mm) .. ($(n-1-5.east) + (0mm,-1mm)$);
  \node[scale=0.75] at ($(n-1-5.east) + (5.5mm, -4.1mm)$) {$\alpha^*_{n-1}$};

 \draw[dash pattern = on 0.5mm off 0.3mm,->] ($(n-1-1.south) +
    (1.2mm,0.5mm)$) arc (65:-245:2.5mm);
    \node[scale=0.75] at ($(n-1-1.south) + (0mm,-6.5mm)$) {$\rho_{1}$};
 
  \draw[dash pattern = on 0.5mm off 0.3mm,->] ($(n-1-2.south) +
    (1.2mm,0.5mm)$) arc (65:-245:2.5mm);
   \node[scale=0.75] at ($(n-1-2.south) + (0mm,-6.5mm)$) {$\rho_{2}$};
 
   \draw[dash pattern = on 0.5mm off 0.3mm,->] ($(n-1-3.south) +
    (1.2mm,0.5mm)$) arc (65:-245:2.5mm);
   \node[scale=0.75] at ($(n-1-3.south) + (0mm,-6.5mm)$) {$\rho_{3}$};
   
     \draw[dash pattern = on 0.5mm off 0.3mm,->] ($(n-1-5.south) +
    (1.2mm,0.5mm)$) arc (65:-245:2.5mm);
   \node[scale=0.75] at ($(n-1-5.south) + (0mm,-6.5mm)$) {$\rho_{n-1}$};
   
     \draw[dash pattern = on 0.5mm off 0.3mm,->] ($(n-1-6.south) +
    (1.2mm,0.5mm)$) arc (65:-245:2.5mm);
   \node[scale=0.75] at ($(n-1-6.south) + (0mm,-6.5mm)$) {$\rho_{n}$};
\end{tikzpicture}\end{equation*}
where the broken arrows are concentrated in degree $-1$ and the remaining generators, i.e.~ solid arrows and idempotents, are in degree $0$. The continuous $k$-linear differential $d\colon \widehat{kQ} \ra \widehat{kQ}$ is completely specified by sending $\rho_{i}$ to the mesh relation (or preprojective relation) starting at the vertex $i$, e.g.~ $d(\rho_{2})=\alpha_{1}\alpha_{1}^*+\alpha_{2}^*\alpha_{2}$.
\end{ex}
 We include a complete list of the graded quivers of the dg Auslander algebras for ADE--singularities in all Krull dimensions in Subsection \ref{ss:DGAuslander}. In our case, these quivers completely determine the corresponding dg algebra.

\begin{rem} 
 Bridgeland determined a connected component of the stability manifold of $\cd_{fd}\bigl(\Lambda_{dg}(\ul{\MCM}(R))\bigr)$ for ADE-surfaces $R$ \cite{Bridgeland09}. We refer to Subsection \ref{ss:Bridgeland} for more details on this remark. 
\end{rem}

\subsubsection{General properties of relative singularity categories}
In the notations of the setup given in Paragraph \ref{ss:SetupIntro}, we assume that $R$ has an \emph{isolated} singularity and that  $A$ is a NCR of $R$. Let $\ul{A}:=A/AeA \cong\ul{\End}_{R}(M)$ be the corresponding stable endomorphism algebra. Since $R$ is an isolated singularity, $\ul{A}$ is a finite dimensional $k$-algebra. We denote the simple $\ul{A}$-modules by $S_{1}, \ldots, S_{t}$.  Then the relative singularity category $\Delta_{R}(A)=\cd^b(\mod-A)/K^b(\proj-R)$ has the following properties:
\begin{itemize}
\item[(a)] All morphism spaces are finite dimensional over $k$, see \cite
{ThanhofferdeVolcseyMichelVandenBergh10} or Prop. \ref{p:hom-finiteness-of-delta}.
\item[(b)] $\Delta_{R}(A)$ is idempotent complete, see Prop. \ref{P:IdempCompl}. 
\item[(c)] $K_{0}\bigl(\Delta_{R}(A)\bigr)\cong \mathbb{Z}^t$, see Prop. \ref{P:GrothendieckGroup}.
\item[(d)] There is an exact sequence of triangulated categories, see \cite{ThanhofferdeVolcseyMichelVandenBergh10} or Prop.~\ref{C:From-gen-to-class} 
\begin{equation}\label{E:class-as-quotient} \thick(S_{1}, \ldots, S_{t}) = \cd^b_{\ul{A}}(\mod-A) \longrightarrow \Delta_{R}(A) \longrightarrow \cd_{sg}(R), \end{equation}
where $\cd^b_{\ul{A}}(\mod-A)\subseteq \cd^b(\mod-A)$ denotes the full subcategory consisting of complexes with cohomologies in $\mod-\ul{A}$. Moreover, this subcategory admits an intrinsic description inside $\cd^b(\mod-A)$, see \cite{ThanhofferdeVolcseyMichelVandenBergh10}  or Cor. \ref{C:Intrinsic}. 
\item[(e)] If $\add M$ has $d$--almost split sequences \cite{Iyama07a}, then $\cd^b_{\ul{A}}(\mod-A)$ has a Serre functor $\nu$, whose action on the generators $S_{i}$ is given by \begin{align}\label{E:IntroFractCY}\nu^n(S_{i}) \cong S_{i}[n(d+1)],\end{align} where $n=n(S_{i})$ is given by the length of the $\tau_{d}$--orbit of $M_{i}$, see Thm.~\ref{t:fractionally-cy-property}.
\item[(f)] Let $\left(\cd_{\ul{A}}(\Mod-A)\right)^c \subseteq \cd^b_{\ul{A}}(\Mod-A)$ be the full subcategory of compact objects. There is an equivalence of triangulated categories, see Remark \ref{r:CompInterprOfRelSingCat}.
\begin{align}Ê\Delta_{R}(A) \cong \bigl(\cd_{\ul{A}}(\Mod-A)\bigr)^c. \end{align} 
\item[(g)] Let $M_{t+1}, \ldots, M_{s}$ be further indecomposable $\MCM$ $R$-modules and  let $A'=\End_{R}(\bigoplus_{i=0}^s M_{i})$. There exists a fully faithful triangle functor, see Prop. \ref{P:Embedding-of-relative}
\begin{align}
\Delta_{R}(A) \longrightarrow \Delta_{R}(A').
\end{align}
\item[(h)] If $\krdim R=3$ and $\MCM(R)$ has a cluster-tilting object $M$, then $C=\End_R(M)$ is a \emph{non-commutative crepant resolution} of $R$, see~\cite[Section 5]{Iyama07}. If $M'$ is another cluster-tilting object in $\MCM(R)$ and $C'=\End_R(M')$, then \begin{align}-\lten_{C}\Hom_R(M',M)\colon\cd^b(\mod-C)\rightarrow\cd^b(\mod-C')\end{align} is a triangle equivalence (see \emph{loc. cit.} and~\cite[Prop. 4]{Palu09}), which is compatible with the embeddings from $K^b(\proj R)$~\cite[Cor. 5]{Palu09}. Hence one obtains a triangle equivalence
\begin{align}
\Delta_R(C)\longrightarrow\Delta_R(C').
\end{align}
\end{itemize}

\begin{rem}
 The Hom-finiteness in (a) is surprising since (triangulated) quotient categories tend to behave quite poorly in this respect, see Example \ref{Ex:Hom-infinite}.
\end{rem}

\subsubsection{Explicit description of the nodal block} This is joint work with Igor Burban \cite{BurbanKalck11}. 
Let $R=k\llbracket x, y \rrbracket/(xy)$ be the nodal curve singularity, let $A=\End_{R}(R \oplus k\llbracket x \rrbracket \oplus k\llbracket y \rrbracket)$ be the Auslander algebra of $\MCM(R)$ and $C=\End_{R}(R \oplus k\llbracket x \rrbracket) \cong eAe$ be the 'relative cluster-tilted' algebra. Here $e \in A$ denotes the idempotent endomorphism corresponding to the identity of $R \oplus k\llbracket x \rrbracket$.  We give an explicit description of the relative singularity categories
$\Delta_{R}(A)$ and $\Delta_{R}(C)$, respectively. Let us fix some notations. $A$ may be written as the completion, with respect to the arrow ideal, of the path algebra of the following quiver with relations:
\begin{align}
\begin{array}{c}
\begin{xy}
\SelectTips{cm}{}
\xymatrix{
\li \ar@/^/[rr]^{ \ro }  & &  \ast \ar@/^/[ll]^{ \lu }
 \ar@/_/[rr]_{ \ru }
 & &
\ar@/_/[ll]_{ \lo } \re}\end{xy}  \qquad  \ru   \ro  = 0, \quad   \lu   \lo  = 0.
\end{array}
\end{align} 
Similarly, $C$ is given by the completion of the following quiver with relations:
\begin{equation}
\begin{xy}\SelectTips{cm}{}
\xymatrix
{- \ar@/^/[rr]^{ \ro }  & &  \ast \ar@/^/[ll]^{ \lu } \ar@(rd, ru)[]_{[\lo\ru]} 
}\end{xy}  \qquad  [\lo\ru]   \ro  = 0, \quad   \lu   [\lo\ru]  = 0.
\end{equation}

Let $\sigma, \tau \in \{\li, \re\}$ and $l \in \mathbb{N}$. A \emph{minimal string} $\mathcal{S}_{\tau}(l)$ is a complex of indecomposable projective $A$-modules
\begin{align}
\begin{xy}\SelectTips{cm}{}
\xymatrix{
\cdots\ar[r] & 0 \ar[r] & P_{\sigma} \ar[r] & P_{\ast} \ar[r] & \cdots \ar[r] &P_{\ast} \ar[r] &P_{\tau} \ar[r]& 0 \ar[r] &\cdots}
\end{xy}
\end{align}
of length $l+2$ with differentials given by non-trivial paths of minimal possible
length and $P_{\tau}$ located in degree $0$. The main results of this section (see Theorem \ref{T:MainT} and Proposition \ref{P:ClusterAone}) are summarized in the following theorem.
\begin{thm} Let $R=k\llbracket x, y \rrbracket/(xy)$ with Auslander algebra $A$ and relative cluster-tilted algebra $C$. Then the following statements hold:
\begin{itemize}
\item[(a)] The indecomposable objects in $\Delta_{R}(A)$ are precisely the shifts of the indecomposable projective $A$-modules $P_{\pm}$ and the minimal strings $\mathcal{S}_{\pm}(l)$, with $l \in \mathbb{N}$. In particular, $\Delta_{R}(A)$ is of discrete representation type.
\item[(b)] All morphism spaces in $\Delta_{R}(A)$ may be computed explicitly. Moreover, the dimension of $\Hom(X,Y)$ for $X$ and $Y$ indecomposable is at most one. 
\item[(c)] The quiver of irreducible maps has two $\ZZ \AA_{\infty}$-components and two equioriented $\mathbb{A}_{\infty}^\infty$-components.
\item[(d)] There is a full embedding $\Delta_{R}(C) \subseteq \Delta_{R}(A)$. Moreover, the indecomposable objects in the image are the shifts of the indecomposable projective $A$-module $P_{-}$ and the shifts of the minimal strings $\mathcal{S}_{-}(2l)$, with $l \in \mathbb{N}$. 
\end{itemize}
\end{thm}
\begin{rem}
In combination with Theorem \ref{t:Classical-versus-generalized-singularity-categories-intro}, we get descriptions of the relative singularity categories $\Delta_{S}(A)$, where $S=k\llbracket z_{0}, \ldots, z_{d} \rrbracket/(z_{0}^2+z_{1}z_{2}+z_{2}z_{3}+\ldots+z_{d-1}z_{d}) $ is an odd-dimensional $\mathbb{A}_{1}$-singularity and $A$ is the Auslander algebra of $\MCM(S)$. The $0$-dimensional case $k[x]/(x^2)$ may be treated with the same techniques (see Proposition \ref{P:ZeroAone}). In particular, we obtain descriptions of the relative Auslander singularity categories of all even dimensional $\mathbb{A}_{1}$-singularities as well.
\end{rem}
Using a tilting result of Burban \& Drozd \cite{Tilting}, we apply this result to describe triangulated quotient categories arising from certain gentle algebras. 
More precisely, we assume that $E=E_{n}$ is a cycle of $n$ projective lines intersecting transversally. This is also known as \emph{Kodaira cycle} of projective lines. Let $\mathcal{I}$ be the ideal sheaf of the singular locus of $E$ and $\ca={\mathcal End}_{E}(\co_{E}\oplus \ci)$ be the Auslander sheaf. The bounded derived category $\cd^b(\Coh(\ca))$ of coherent $\ca$-modules has a tilting complex, with gentle endomorphism algebra $\Lambda_{n}$ \cite{Tilting}. For example, if $n=1$, then $\Lambda_{1}$ is given by the following quiver with relations
\begin{align}
\begin{array}{c}
\begin{xy}
\SelectTips{cm}{}
\xymatrix{
\circ \ar@/^/[rr]^{a} \ar@/_/[rr]_{c} & & \circ \ar@/^/[rr]^{b} \ar@/_/[rr]_{d} & & \circ
}
\end{xy} \qquad \quad ba=0=dc.
\end{array}
\end{align}  
Let $\tau=\nu [-1]$ be the Auslander--Reiten translation of $\cd^b(\Lambda_{n}-\mod)$, then the full subcategory of \emph{band} objects
\begin{align}
\Band(\Lambda_{n}):=\left\{X \in \cd^b(\Lambda_{n}-\mod) \big| \tau(X) \cong X\right\} \subseteq \cd^b\left(\Lambda_{n}-\mod\right)
\end{align}
is triangulated, see \cite{Tilting}. Now tilting and the localization result (Theorem \ref{t:main-global-intro}) yield the following description of the corresponding Verdier quotient category 
\begin{align}
\left(\frac{\cd^b(\Lambda_{n}-\mod)}{\Band(\Lambda_{n})}\right)^\omega \cong \bigoplus_{i=1}^n \Delta_{R}(A)
\end{align}
where $R=k\llbracket x, y\rrbracket/(xy)$ and $A$ is the Auslander algebra of $\MCM(R)$.

\subsection{Frobenius categories and Gorenstein rings for rational surface singularities}
The techniques developed in the second part inspired the following `purely commutative' application.

Let $(R, \mathfrak{m})$ be a complete local \emph{rational surface singularity} over an algebraically closed field of characteristic zero. By definition, a singularity is rational if \[H^1(X, \co_{X})=0\] for a resolution of singularities $X \ra \mathsf{Spec}(R)$. Rational surface singularities are Cohen--Macaulay and the class of maximal Cohen--Macaulay modules coincides with the class of \emph{reflexive} modules. Note, that we do \emph{not} assume that $R$ is a Gorenstein ring. Building on work of Cartan \cite{Cartan}, Brieskorn \cite{Brieskorn} has shown that quotient singularities $\C^2/G$ are rational. Here, we can assume without restriction that $G \subseteq \GL_{2}(\C)$ is a \emph{small} subgroup, i.e. $G$ does not contain \emph{pseudo-reflections}. 

Let $\pi\colon Y \ra \mathsf{Spec}(R)$ be the minimal resolution of $\mathsf{Spec}(R)$ and $\{E_{i}\}_{i \in I}$ be the \emph{finite} set of irreducible components of the exceptional curve $E=\pi^{-1}(\mathfrak{m})$. In order to generalize the classical McKay--Correspondence to finite subgroups $G \subseteq \GL(2, \C)$, Wunram \cite{Wunram88} introduced the notion of \emph{special} maximal Cohen--Macaulay modules (SCM)\footnote{A maximal Cohen--Macaulay $R$-module $M$ is \emph{special} if $\Ext^1_{R}(M, R)=0$.}. He proved that the indecomposable non-free SCMs are in natural bijection with the set of irreducible exceptional curves $\{E_{i}\}_{i \in I}$. If $G \subseteq \SL(2, \C)$, then all maximal Cohen--Macaulay modules are special and he recovers the classical McKay--Correspondence.

If $R$ is not Gorenstein, then the exact category of maximal Cohen--Macaulay $R$-modules is \emph{not} Frobenius. In particular, the corresponding stable category is not triangulated and therefore there cannot be a triangle equivalence as in \eqref{E:Buchweitz}. Moreover, Iyama \& Wemyss \cite{IWnewtria} point out that the singularity category $\cd_{sg}(R)$ does not have the Krull--Remak--Schmidt property. 

These problems in the non-Gorenstein situation motivated their study \cite{IWnewtria} of the exact category of special Cohen--Macaulay $R$-modules $\SCM(R)$. As it turns out, $\SCM(R)$ is a Frobenius category and they describe the projective objects of this category in terms of the geometry of the exceptional divisor. In this way, they associate a Krull--Remak--Schmidt  triangulated category, namely the stable category $\ul{\ul{\SCM}}(R)$\footnote{This notation is used to distinguish this quotient from the subcategory $\ul{\SCM}(R) \subseteq \ul{\MCM}(R)$ obtained by only factoring out the projective $R$-modules. }  of this Frobenius category, to any complete rational surface singularity.    

In joint work with Osamu Iyama, Michael Wemyss and Dong Yang \cite{IKWY12}, we gave a description of this triangulated category in terms of finite products of stable categories of ADE-singularities, which explains an observation in \cite{IWnewtria}. Let $X$ be the space obtained from $Y$ by contracting all the exceptional $(-2)$-curves. $X$ has only isolated singularities and these singularities of $X$ are known to be of ADE-type, by work of Artin \cite{Artin66}, see also Proposition \ref{P:RationalDoublePoint}. Accordingly, $X$ is called the \emph{rational double point resolution} of $\Spec(R)$. The minimal resolution $\pi$ factors over $X$:
\begin{align}
Y \xrightarrow{f} X \xrightarrow{g} \Spec(R). 
\end{align}
The following result (see Theorem \ref{C:StandardStableSCM}) is a consequence of our general Frobenius category (Theorem \ref{T:MainIKWY}) and tilting (Theorem \ref{T:Wemyss}) results below.

\begin{thm}\label{C:StandardStableSCM-intro}
Let $R$ be a complete rational surface singularity with rational double point resolution $X$. Denote the singularities of $X$ by $x_{1}, \ldots, x_{n}$. 
Then there are equivalences of triangulated categories
\begin{align}\label{E:SCMdecomposed}
\ul{\ul{\SCM}}(R) \cong \cd_{sg}(X) \cong \bigoplus_{i=1}^n \ul{\MCM}\bigl(\widehat{\co}_{x_{i}}\bigr).
\end{align}  
\end{thm}
The second equivalence is a consequence of Orlov's localization result \eqref{E:OrlovLoc} and Buchweitz' equivalence  \eqref{E:Buchweitz} - note that $\ul{\ul{\SCM}}(R)$ is idempotent complete since $R$ is a complete local ring. The first equivalence follows from a combination of geometric and algebraic results, which are of independent interest. Before stating them, we note the following consequence of the main result.

\begin{cor}
$\ul{\ul{\SCM}}(R)$ is a $1$-Calabi--Yau category and $[2] \cong \id$.
\end{cor}

\subsubsection{Ingredients of the proof of Theorem \ref{C:StandardStableSCM-intro}}
Every \emph{algebraic} triangulated category, i.e.~every triangulated category arising in algebraic geometry or representation theory, may be expressed as the stable category of some Frobenius category. We study a special class of Frobenius categories, which contains the category of special Cohen--Macaulay modules over rational surface singularities and many other  categories coming from representation theory. 

A two-sided Noetherian ring $\Lambda$ satisfying $\injdim _{\Lambda}\!\Lambda < \infty$ and 
$\injdim \Lambda_{\Lambda} < \infty$ is called Iwanaga--Gorenstein. It is well-known that the category 
\begin{align}\label{E:GPIntro}
\GP(\Lambda):=\{X\in\mod-\Lambda\mid \Ext^i_\Lambda(X,\Lambda)=0\mbox{ for any }i>0\}.
\end{align}
of \emph{Gorenstein projective} $\Lambda$-modules\footnote{If $\Lambda$ is also commutative, then the notions of Gorenstein projective and maximal Cohen--Macaulay $\Lambda$-modules coincide.} is a Frobenius category, see e.g. Proposition \ref{P:IwanagaFrobenius}. The following theorem (see Theorem \ref{t:main-thm} and also the alternative approach to \eqref{E:IntroStableEq} and part (3) in Theorem \ref{t:alternative-main}) gives a sufficient criterion for a Frobenius category to be of this form.

\begin{thm}\label{T:MainIKWY}
Let $\ce$ be a Frobenius category and assume that there exists $P \in \proj \ce$ such that $\proj \ce \cong \add P$.  Let $E=\End_{\ce}(P)$. If there is an object $M \in \ce$, such that $\End_{\ce}(P \oplus M)$ is Noetherian and has global dimension $n<\infty$, then the following statements hold. 
\begin{itemize}
\item[(1)] $E=\End_\ce(P)$ is an Iwanaga--Gorenstein ring of dimension at most $n$.\\
\item[(2)] We have an equivalence 
\begin{align}\label{E:IntroExEq}
\Hom_{\ce}(P,-)\colon\ce\to\GP(E)
\end{align}
 up to direct summands. It is an equivalence if $\ce$ is idempotent complete. This induces a triangle equivalence 
\begin{align}\label{E:IntroStableEq}
\underline{\ce}\xrightarrow{\sim}\underline{\GP}(E) \, \, \bigl(\cong \cd_{sg}(E)\bigr).
\end{align}
up to direct summands. It is an equivalence if $\ce$ or $\underline{\ce}$ is idempotent complete.\\
\item[(3)] $\ul{\ce}=\thick_{\ul{\ce}}(M)$.
\end{itemize}
\end{thm}
Let $R$ be a rational surface singularity as above and let $M=R \oplus \bigoplus_{i \in I} M_{i}$ be the sum of all indecomposable SCM $R$-modules. Then $\Lambda=\End_{R}(M)$ is called the \emph{reconstruction algebra} of $R$, see \cite{Wemyss10}.
The relationship between Theorem \ref{T:MainIKWY} and the geometry is given by a tilting result of Wemyss (see Theorem \ref{main Db}), which is based on work of Van den Bergh \cite{VandenBergh04}. 
\begin{thm}\label{T:Wemyss}
Let $P$ be an additive generator of $\proj \SCM(R)$ and let $e \in \Lambda$ be the corresponding idempotent endomorphism. In particular, $e\Lambda e=\End_R(P)$. Then there are tilting bundles $\cv_{Y}$ on the minimal resolution $Y$ and $\cv_{X}$ on the rational double point resolution $X$ such that the following diagram commutes
\[
\begin{array}{c}
{\SelectTips{cm}{10}
\xy0;/r.4pc/:
(-10,20)*+{\cd^b(\mod-\Lambda)}="A2",(20,20)*+{\cd^b(\Coh Y)}="A3",
(-10,10)*+{\cd^b(\mod-e\Lambda e)}="a2",(20,10)*+{\cd^b(\Coh X)}="a3",
\ar"A3";"A2"_{\RHom_Y(\cv_Y,-)}^{\sim}
\ar"A2";"a2"_{(-)e}
\ar"a3";"a2"_{\RHom_{X}(\cv_X,-)}^{\sim}
\ar"A3";"a3"^{\Rf}
\endxy}
\end{array}
\]
\end{thm}
In particular, the reconstruction algebra has finite global dimension, since $Y$ is a \emph{smooth} scheme. Now, we can explain the first triangle equivalence in Theorem \ref{C:StandardStableSCM-intro}. Since $\SCM(R)$ is idempotent complete and $\Lambda$ is Noetherian and of finite global dimension, Theorem \ref{T:MainIKWY} and Buchweitz' \eqref{E:Buchweitz}\footnote{Buchweitz proved the equivalence \eqref{E:Buchweitz} for Iwanaga--Gorenstein rings.} yield triangle equivalences
\begin{align}
\ul{\ul{\SCM}}(R) \cong \ul{\GP}(e\Lambda e) \cong \cd_{sg}(e \Lambda e).
\end{align}
The tilting equivalence $\RHom_{X}(\cv_X,-)\colon \cd^b(\Coh X) \ra \cd^b(\mod-e\Lambda e)$ from Theorem \ref{T:Wemyss} induces a triangle equivalence
\begin{align}
\cd_{sg}(X) \ra \cd_{sg}(e\Lambda e),
\end{align}
which completes the explanation of the first equivalence in \eqref{E:SCMdecomposed}.

\begin{rem}
In general, the category $\SCM(R)$ has many other Frobenius exact structures. More precisely, if we take any subset of the indecomposable SCMs corresponding to exceptional $(-2)$-curves, then there is a Frobenius structure such that these modules become projective-injective as well. All techniques explained in this paragraph apply to this more general setup. In particular, the corresponding stable categories decompose into a direct sum of stable categories of maximal Cohen--Macaulay modules indexed by the singularities of a certain Gorenstein scheme $X'$, which lies  between the minimal resolution and the rational double point resolution, see Corollary \ref{C:ModifiedStableSCM}. 
\end{rem}

\subsection{Singularity categories of gentle algebras} The results in this part are contained in the preprint \cite{Kalck12}.
Gentle algebras are certain finite dimensional algebras, whose module and derived category are well understood: for example, there is a complete classification of indecomposable objects in both categories. Moreover, the class of gentle algebras is closed under derived equivalence \cite{SchroerZimmermann}. 
The following examples provide relations to other fields: Burban \cite{Burban}  obtained a family of gentle algebras which are derived equivalent to an $\mathbb{A}_{n}$-configuration of projective lines. With each triangulation of an unpunctured marked Riemann surface, Assem, Br\"ustle,  Charbonneau-Jodoin \& Plamondon \cite{ABCP} associated a gentle algebra. Moreover, they show that every cluster-tilted algebra of types $\mathbb{A}$ and $\widetilde{\mathbb{A}}$ arises in this way.
 
Geiss \& Reiten \cite{GeissReiten} have shown that gentle algebras are \emph{Iwanaga--Gorenstein} rings. Hence Buchweitz' equivalence \eqref{E:Buchweitz} reduces the computation of the singularity category to the determination of the stable category of Gorenstein projective modules, see \eqref{E:GPIntro} for a definition. This leads to the main result of this part, for which we need some notation.

Let $\Lambda=kQ/I$ be a finite dimensional gentle algebra and denote by $\cc(\Lambda)$ the set of equivalence classes of repetition free cyclic paths $\alpha_{1}\ldots \alpha_{n}$ in $Q$ (with respect to cyclic permutation) such that $\alpha_{i}\alpha_{i+1} \in I$ for all $i$, where we set $n+1=1$. The following proposition is contained as Proposition \ref{P:Main} in the main body of this work.
\begin{prop}\label{P:MainGentle}
Let $\Lambda$ be a gentle algebra. Then 
there is a triangle equivalence  
\begin{align}
\cd_{sg}(\Lambda) \cong \prod_{c \in \cc(\Lambda)} \frac{\displaystyle \cd^b(k-\mod)}{\displaystyle [l(c)]},
\end{align}
where $l(\alpha_{1}\ldots\alpha_{n})=n$ and $\cd^b(k)/[l(c)]$ denotes the triangulated orbit category, \cite{Orbit}. This category is also known as the $(l(c)-1)$-\emph{cluster category} of type $\mathbb{A}_{1}$, \cite{Thomas}.
\end{prop} 

In particular, we obtain the following derived invariant of gentle algebras, which is a special case of an invariant introduced by Avella-Alaminos \& Gei{\ss} \cite{GeissAvellaAlaminos}.

\begin{cor}\label{C:Invariant} Let $\Lambda$ and $\Lambda'$ be gentle algebras. If there is a triangle equivalence $\cd^b(\Lambda-\mod) \cong \cd^b(\Lambda'-\mod)$, then there is a bijection of sets
\begin{align}
f \colon \cc(\Lambda) \stackrel{\sim}\longrightarrow \cc(\Lambda'),
\end{align}
such that $l(c)=l(f(c))$ for all $c \in \cc(\Lambda)$.
\end{cor}

\begin{rem}
Buan and Vatne \cite{BuanVatne} showed that for two cluster-tilted algebras $\Lambda$ and $\Lambda'$ of type $\mathbb{A}_{n}$, for some fixed $n \in \mathbb{N}$, the converse of Corollary \ref{C:Invariant} holds. In other words, two such algebras are derived equivalent if and only if their singularity categories are triangle equivalent.
\end{rem}

Proposition \ref{P:MainGentle} has further consequences. Let $A(S, \Gamma)$ be the Jacobian algebra associated  with a triangulation $\Gamma$ of a surface $S$. Then $A(S, \Gamma)$ is a gentle algebra and the cycles in $\cc(A(S, \Gamma))$ are in bijection with the inner triangles of $\Gamma$, see \cite{ABCP}.
\begin{cor}
The number of direct factors of $\cd_{sg}\bigl(A(S, \Gamma)\bigr)$ equals the number of inner triangles of $\Gamma$.
\end{cor}

Using a tilting equivalence of Burban \cite{Burban}, we obtain another explanation for the following well-known result,
see Orlov's localization result \eqref{E:OrlovLoc}.
\begin{cor}
Let $\mathbb{X}_{n}$ be an $\mathbb{A}_{n}$-configuration of projective lines  
\begin{align*}
\begin{array}{c}
\begin{tikzpicture} 
\draw (0,0,0) to [bend left=25]  (2,0,0);
\draw (1.5,0,0) to [bend left=25]  (3.5,0,0);
\node at (4, 0, 0) {$\cdots$}; 
\draw (4.5,0,0) to [bend left=25]  (6.5,0,0);
\draw (6,0,0) to [bend left=25]  (8,0,0);
\draw (7.5,0,0) to [bend left=25]  (9.5,0,0);
\end{tikzpicture}
\end{array} 
\end{align*}
There is an equivalence of triangulated categories
\begin{align}
\cd_{sg}(\XX_{n}) \cong \bigoplus_{i=1}^{n-1} \frac{\cd^b(k-\mod)}{[2]}
\end{align}
\end{cor}

\subsection{Contents and Structure} 
The following picture shows the dependencies between the different sections of this work.  

\bigskip

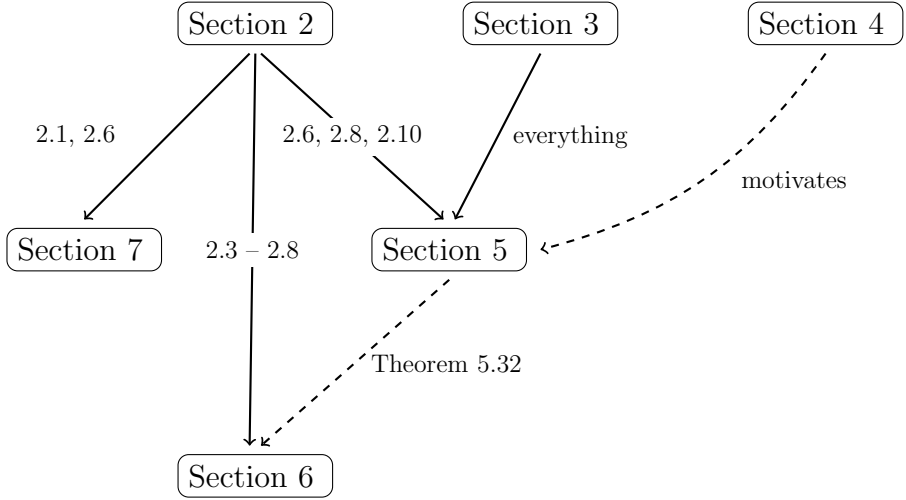
\begin{tikzpicture}
 \node[draw,rectangle,rounded corners,top color=white, bottom color=white] at (0.25\columnwidth,0) (x) {\begin{minipage}{50pt}
Section 2
\end{minipage} };

 \node[draw,rectangle,rounded corners,top color=white, bottom color=white] at (0.50\columnwidth,0) (x) {\begin{minipage}{50pt}
Section 3
\end{minipage} };

 \node[draw,rectangle,rounded corners,top color=white, bottom color=white] at (0.75\columnwidth,0) (x) {\begin{minipage}{50pt}
Section 4
\end{minipage} };

\path[->, dashed, thick=0.5cm] ($(0.75\columnwidth, -0.4)$) edge [bend left=20] node[scale=0.8, xshift=1.5cm] {motivates}
($(0.5\columnwidth, -3)$);

\path[->, thick=0.5cm] ($(0.5\columnwidth, -0.4)$) edge  node[scale=0.8, xshift=1.2cm] {everything}
($(0.425\columnwidth, -2.6)$);

\path[->, thick=0.5cm] ($(0.255\columnwidth, -0.4)$) edge  node[fill=white, scale=0.8] {\ref{ss:SingCatIwanaga}, \ref{ss:Alternative}, \ref{ss:IdempCompl}}
($(0.415\columnwidth, -2.6)$);

\path[->, thick=0.5cm] ($(0.245\columnwidth, -0.4)$) edge  node[scale=0.8, xshift=-1.5cm] {\ref{ss:FrobeniusBasic}, \ref{ss:SingCatIwanaga}}
($(0.1\columnwidth, -2.6)$);

\path[->, thick=0.5cm] ($(0.25\columnwidth, -0.4)$) edge  node[fill=white, scale=0.8] {\ref{ss:AuslanderSolberg} -- \ref{ss:Alternative}}
($(0.245\columnwidth, -5.6)$);

\path[->, dashed, thick=0.5cm] ($(0.42\columnwidth, -3.4)$) edge  node[scale=0.8, xshift=1.5cm] {Theorem \ref{t:main-thm-2}}
($(0.255\columnwidth, -5.6)$);

 \node[draw,rectangle,rounded corners,top color=white, bottom color=white] at (0.1\columnwidth,-3) (x) {\begin{minipage}{50pt}
Section 7
\end{minipage} };

 \node[draw,rectangle,rounded corners,top color=white, bottom color=white] at (0.42\columnwidth,-3) (x) {\begin{minipage}{50pt}
Section 5
\end{minipage} };

 \node[draw,rectangle,rounded corners,top color=white, bottom color=white] at (0.25\columnwidth,-6) (x) {\begin{minipage}{50pt}
Section 6
\end{minipage} };

\end{tikzpicture}

\medskip
We give a brief outline of the contents of this work. More detailed descriptions can be found at the beginning of each section.
Sections \ref{s:Frobenius} and \ref{s:dg} provide methods from the theory of Frobenius categories and dg algebras, respectively. Although these parts are quite technical in nature, some of the results might be of interest in their own right. The localization result for the global relative singularity categories given in Section \ref{S:Global}, serves as a motivation for the study of the local relative singularity categories in Section \ref{S:Local} but does not depend on results from other parts of the text. In Section \ref{S:Local} the techniques from Section \ref{s:Frobenius} and Section \ref{s:dg} are combined to describe the relative singularity categories for Auslander algebras of MCM-representation finite singularities as categories of perfect complexes over some explicit dg algebra: the dg Auslander algebra. This is the key ingredient in the reconstruction of the relative singularity category from the classical singularity category in this setup. Section \ref{s:Specials} uses the abstract results on Frobenius categories from Section \ref{s:Frobenius} in conjunction with geometric methods to obtain a description of the Iyama \& Wemyss' stable category of special Cohen--Macaulay modules over rational surface singularities. We include a relation to the relative singularity categories from Section \ref{S:Local}  as an aside. This uses Theorem \ref{t:main-thm-2}. Section \ref{S:Gentle} uses Buchweitz' Theorem \ref{T:Buchweitz}  Êand basic properties of Gorenstein projective modules from Section \ref{s:Frobenius} to describe the singularity categories of gentle algebras.

\subsection*{Acknowledgement}
I would like to thank my advisor Igor Burban for many inspiring discussions, helpful remarks, comments and advices, his patience and for his contributions to our joint work.

I am very grateful to my coauthors Osamu Iyama, Michael Wemyss and Dong Yang for sharing their insights and for helpful explanations during our fruitful and pleasant collaboration.
Moreover, I thank Hanno Becker, Lennart Galinat, Wassilij Gnedin, Nicolas Haupt, Jens Hornbostel, Bernhard Keller, Henning Krause, Julian K\"ulshammer, Daniel Labardini-Fragoso, Helmut Lenzing, Jan Schr\"oer, Greg Stevenson and Michel Van den Bergh for helpful remarks and discussions on parts of this work and Stefan Steinerberger for improving the language of this text.

Finally, I would like to thank my family, friends, colleagues and everyone who supported and encouraged me during 
the last years.

This work was supported by the DFG grant Bu--1866/2--1 and the Bonn International Graduate School of Mathematics (BIGS).

\newpage

\section{Frobenius categories}\label{s:Frobenius}

In this section, we study Frobenius categories (a special class of exact categories) from various perspectives.
The Subsections \ref{ss:AuslanderSolberg}, \ref{ss:Iyama} and \ref{ss:Alternative} are based on a joint work with Osamu Iyama, Michael Wemyss and Dong Yang \cite{IKWY12}. Proposition \ref{new Frobenius structure} in Subsection \ref{ss:AuslanderSolberg} is implicitely contained in an article of Auslander \& Solberg \cite{AS93-1}.  Subsections \ref{ss:BuchweitzHappel} to \ref{S:Tale} grew out of a joint work with Dong Yang \cite{KalckYang12}. The results in Subsections \ref{ss:BuchweitzHappel} and \ref{ss:SingCatIwanaga} are well-known (see Keller \& Vossieck \cite{KellerVossieck87} and Buchweitz \cite{Buchweitz87}, respectively), however, our proofs are quite different. Finally, Subsection \ref{ss:IdempCompl} is an extended version of a section of a joint article with Igor Burban \cite{BurbanKalck11}. The remaining parts \ref{ss:FrobeniusBasic}, \ref{Sub:Idemp} and \ref{ss:Schlichting} follow Keller \cite{Keller}, Balmer \& Schlichting \cite{BalmerSchlichting} and Schlichting \cite{Schlichting06}, respectively. 

Let us briefly describe the content of this section. Subsection \ref{ss:FrobeniusBasic} starts with the definition of Frobenius categories. We illustrate this definition with several examples including the category of Gorenstein projective modules\footnote{In the commutative case, these are exactly the maximal Cohen--Macaulay modules.} over (Iwanaga--) Gorenstein rings, which is important throughout this work. Next, we recall the definitions of the stable category of a Frobenius category and its triangulated structure, discovered by Heller \cite{Heller} and Happel \cite{Happel}. Following Balmer \& Schlichting \cite{BalmerSchlichting}, we recall the notion of the idempotent completion (or Karoubian hull) of an additive category in Subsection \ref{Sub:Idemp}. In particular, they equip the idempotent completion of a triangulated category with a natural triangulated structure. This will be important throughout this work. In Subsection \ref{ss:AuslanderSolberg} we explain a method to construct \emph{Frobenius} exact structures on certain exact categories with an Auslander--Reiten type duality. This is applied to obtain new Frobnenius structures on the category of special Cohen--Macaulay modules over rational surface singularities in Section \ref{s:Specials}. Subsection \ref{ss:Iyama} contains Iyama's Morita-type Theorem for Frobenius categories. More precisely, he provides conditions guaranteeing that a given Frobenius category is equivalent to the category of Gorenstein projective modules over an Iwanaga--Gorenstein ring. In combination with Buchweitz' Theorem \ref{T:Buchweitz}, this result yields a description of certain stable Frobenius categories as singularity categories of Iwanaga--Gorenstein rings (Corollary \ref{C:BuchweitzIyama}), which is an essential ingredient in the proof of our main result of Section \ref{s:Specials}. In Subsection \ref{ss:BuchweitzHappel}, we prove a result of Keller \& Vossieck in an alternative and `elementary' way. This is used to establish Buchweitz' triangle equivalence between the stable category of Gorenstein projective modules and the singularity category of an Iwanaga--Gorenstein ring (see Theorem \ref{T:Buchweitz}), which plays a central role in this text. Subsection \ref{S:Tale} serves as a preparation for our alternative proof of the stable version of Iyama's result (Corollary \ref{C:BuchweitzIyama}) in Subsection \ref{ss:Alternative}. Subsection \ref{ss:Schlichting} gives a brief introduction to Schlichting's negative K-theory for triangulated categories admitting a (Frobenius) model. This is applied to prove the idempotent completeness of certain triangulated quotient categories in Subsection \ref{ss:IdempCompl}. We use this result to express the relative singularity categories of complete Gorenstein rings as perfect derived categories of certain dg algebras in Subsection \ref{S:MCM-over-Gor}.

\subsection{Definitions, basic properties and examples} \label{ss:FrobeniusBasic}

We mainly follow Keller's expositions in \cite{Keller} and \cite[Appendix A]{Keller2}.
\begin{defn}\label{D:ExactKeller}
Let $\kE$ be an additive category. A pair of morphisms $A \xrightarrow{i} B \xrightarrow{p} C$ is called \emph{exact} if $i$ is a kernel of $p$ and $p$ is a cokernel of $i$. 

An \emph{exact structure} on $\kE$ consists of a collection $\cs$ of exact pairs, which is closed under isomorphisms and satisfies the axioms below. If $(i, p) \in \cs$, then $(i, p)$ is called \emph{conflation}, $i$ is called \emph{inflation} and $p$ is called \emph{deflation}.  
\begin{itemize}
\item[(Ex0)] The identity $0 \xrightarrow{0} 0$ of the zero object is a deflation.
\item[(Ex1)] Deflations are closed under composition.
\item[(Ex2)] Every diagram $B \xrightarrow{p} C \xleftarrow{c} C'$ in $\ce$, where $p$ is a deflation has a pullback
\begin{align}
\begin{array}{cc}
\begin{xy}
\SelectTips{cm}{}
\xymatrix{
B'  \ar[d]^{b} \ar[r]^{p'} & C \ar[d]^{c} \\
B \ar[r]^{p} & C'
}
\end{xy}
\end{array}
\end{align}
such that $p'$ is a deflation.
\item[(Ex2$^{\rm op}$)] Dually, every diagram $B \xleftarrow{i} C \xrightarrow{c} C'$ in $\ce$, where $i$ is an inflation has a pushout
\begin{align}
\begin{array}{cc}
\begin{xy}
\SelectTips{cm}{}
\xymatrix{
C  \ar[d]^{i} \ar[r]^{c} & C' \ar[d]^{i'} \\
B \ar[r]^{b} & B'
}
\end{xy}
\end{array}
\end{align}
such that $i'$ is an inflation.
\end{itemize} 
Sometimes we consider several different exact structures on a given additive category $\ce$. In these situations, it is convenient to write $(\ce, \cs)$ for a Frobenius category $\ce$ with exact structure $\cs$.
\end{defn}

The following notions will be important throughout this work.
\begin{defn}\label{D:IwanagaGorensteinD}
A two-sided Noetherian ring $\Lambda$ is called \emph{Iwanaga--Gorenstein of dimension $n$} if $\injdim_\Lambda\!\Lambda=n=\injdim \Lambda_\Lambda$\footnote{By a result of Zaks \cite{Zaks}, finiteness of $\injdim_\Lambda\!\Lambda$ and $\injdim \Lambda_\Lambda$ implies that both numbers coincide.}.
The category of \emph{Gorenstein projective} $\Lambda$-modules $\GP(\Lambda)$ is defined as follows
\begin{align}
\GP(\Lambda):=\{X\in\mod-\Lambda\mid \Ext^i_\Lambda(X,\Lambda)=0\mbox{ for any }i>0\}.
\end{align}
\end{defn}

\begin{rem}
If $\Lambda$ is a commutative local Noetherian Gorenstein ring, then there is an equivalence of categories $\MCM(\Lambda) \cong \GP(\Lambda)$, where 
\begin{align} \label{E:MCMvsGP}
\MCM(\Lambda)=\{ M \in \mod-\Lambda \mid \depth_{R}(M)=\dim(M) \}
\end{align}
denotes the category of maximal Cohen--Macaulay $\Lambda$-modules, see e.g. \cite{BrunsHerzog}.
If $\Lambda$ is a selfinjective algebra, then $\GP(\Lambda)=\mod-\Lambda$.
\end{rem}

\begin{ex}\label{E:Exact}
\begin{itemize}
\item[a)] Let $\kB \subseteq \kE$ be a full additive subcategory of an exact category $\kE$. If $\kB$ is closed under extensions (i.e.~for every exact pair $A \xrightarrow{i} B \xrightarrow{p} C$ in $\kE$, with $A$ and $C$ in $\kB$ we have $B \in \kB$), then the collection of those conflations in $\kE$ with all terms in $\kB$ defines an exact structure on $\kB$.
\item[b)] Let $\kA$ be an abelian category and consider the collection of pairs given by \emph{all} short exact sequences in $\kA$. It is well known that this defines an exact structure on $\kA$.
\item[c)] Let $(R, \gm)$ be a local Noetherian ring. By the Depth Lemma (see \cite[Proposition 1.2.9]{BrunsHerzog}), the full subcategory of maximal Cohen--Macaulay modules $\MCM(R) \subseteq \mod-R$ is closed under extensions. Thus a) and b) above imply that $\MCM(R)$ is an exact category.
\item[d)] Let $\Lambda$ be an Iwanaga--Gorenstein ring, then $\GP(\Lambda) \subseteq \mod \Lambda$ is exact by a) and b).
\item[e)] Let $\kC$ be an additive category and $\Com(\kC)$ be the category of complexes over $\kC$. Take the collection of all pairs $(i, p)$ of morphisms in $\Com(\kC)$ such that the components $(i^n, p^n)$ are split exact sequences. This defines an exact structure on $\Com(\kC)$.
\end{itemize}
\end{ex}

\begin{rem}
Call a category \emph{svelte} if it is equivalent to a small category. Let $\kE$ be a svelte exact category, then $\kE$ is equivalent to a full extension closed subcategory of an abelian category $\kA$ (see for example \cite[Proposition A.2]{Keller2}). This shows that the situation encountered in Example \ref{E:Exact}  c) and d) above is actually quite general.
\end{rem}

The definitions of projective and injective objects generalize to exact categories.

\begin{defn}
Let $\kE$ be an exact category. An object $I$ in $\kE$ is called \emph{injective} if the sequence of abelian groups
$\kE(B, I) \xrightarrow{i^*} \kE(A, I) \rightarrow 0$ is exact for every deflation $A \xrightarrow{i} B$ in $\kE$. The category $\kE$ \emph{has enough injective objects} if every object $X$ in $\kE$ admits a conflation
$X \xrightarrow{i_{X}} I(X) \xrightarrow{p_{X}} \Omega^{-1}(X)$
in $\kE$, with $I(X)$ injective.

The notions of a \emph{projective} object and of a category having \emph{enough projective objects} are defined dually. 
\end{defn}

\begin{defn}\label{D:Frobenius}
An exact category $\kE$ is called \emph{Frobenius} category, if it has enough injective and enough projective objects and these two classes of objects \emph{coincide}. The full subcategory of projective-injective objects is denoted by $\proj \kE$

An additive functor $F\colon \kE \rightarrow \kE'$ between Frobenius categories, is called \emph{map of Frobenius categories} if it maps conflations to conflations (in other words $F$ is \emph{exact}) and $F(\proj \kE) \subseteq \proj \kE'$ holds. 
\end{defn}

The following proposition is well-known, see \cite{Buchweitz87}.

\begin{prop}\label{P:IwanagaFrobenius}
Let $\Lambda$ be an Iwanaga--Gorenstein ring, then $\GP(\Lambda)$ satisfies the following properties:
\begin{itemize}
\item[(GP1)] A GP $\Lambda$-module is either projective or of infinite projective dimension.
\item[(GP2)] $M$ is GP if and only if $M \cong \Omega^d(N)$ for some $N\in \mod-\Lambda$, where $d=\injdim \Lambda_{\Lambda}$. In particular, every Gorenstein projective module is a submodule of a projective module and the syzygies $\Omega^n(N)$ are GP for all $n \geq d$ and all $N\in \mod-\Lambda$.
\item[(GP3)] $\GP(\Lambda)$ is a Frobenius category with $\proj \GP(\Lambda)=\proj-\Lambda$.
\end{itemize}
\end{prop}
\begin{proof}
Throughout, we need the fact that syzygies $\Omega^n(M)$, $n\geq 0$ of Gorenstein projectives are again GP. This follows from the definition and the long exact Ext-sequence.

 For (GP1), we take $M \in \GP(\Lambda)$ with finite projective dimension. So we have an  exact sequence
 \begin{align}\label{E:ProjResGP}
 0 \ra P_{t} \ra P_{t-1} \ra \cdots \ra P_{1} \ra M \ra 0,
 \end{align}
 which yields an exact sequence $0 \ra P_{t} \ra P_{t-1} \ra \Omega^{t-2}(M) \ra 0$. Since syzygies are GP, this sequence splits and $\Omega^{t-2}(M)$ is projective. Continuing in this way to reduce the length of \eqref{E:ProjResGP}, we see that $M$ is projective.
 
 Next, we show that $\Omega^n(N)$ is GP, for all $N \in \mod-\Lambda$ and all $n \geq d$. Since syzygies of Gorenstein-projectives are Gorenstein-projective, it suffices to treat the case $n=d$. The projective resolution of $N$ may be spliced into short exact sequences
\begin{align}
0 \ra \Omega^n(N) \ra P^n \ra \Omega^{n-1}(N) \ra 0.
\end{align}
In particular, we obtain isomorphisms $\Ext^i_{\Lambda}(\Omega^{n}(N), \Lambda) \cong  \Ext^{i+1}_{\Lambda}(\Omega^{n-1}(N), R)$ for all $i>0$. By definition, $\injdim \Lambda = d$ implies $\Ext^{d+1}(-, \Lambda)=0$. This yields a chain of ismorphisms, for all $i>0$
\begin{align}
\Ext^i_{\Lambda}(\Omega^d(N), \Lambda) \cong \Ext^{i+1}_{\Lambda}(\Omega^{d-1}(N), \Lambda) \cong \ldots \cong
\Ext^{i+d}_{\Lambda}(N, \Lambda) = 0,
\end{align}
which proves the statement. We need some preparation to prove the converse direction.
 We claim that the the following two functors are well-defined and mutually quasi-inverse:
\begin{align}
\begin{array}{c}
\begin{xy}
\SelectTips{cm}{}
\xymatrix{
\GP(\Lambda) \ar@/^/[rr]^{\mathbb{D}=\Hom_{\Lambda}(-, \Lambda)} && \GP(\Lambda\op{}) \ar@/^/[ll]^{\mathbb{D}\op{}=\Hom_{\Lambda\op{}}(-, \Lambda)}
}
\end{xy}
\end{array}
\end{align} 
To prove this, let $M$ be a Gorenstein projective $\Lambda$-module with projective resolution $P^\bullet$
\begin{align}
\cdots \xrightarrow{f_{k+1}} P_{k} \xrightarrow{f_{k}} P_{k-1} \xrightarrow{f_{k-1}} \cdots \xrightarrow{f_{2}} P_{1} \ra M \ra 0
\end{align}
Since all syzygies of $M$ are GP, $\mathbb{D}(P^\bullet)$ is a coresolution of $\mathbb{D}(M)$. Using the direction of (GP2) which we have already proved, we conclude that $\mathbb{D}(M)$ and $\ker \mathbb{D}(f_{k})$ are GP $\Lambda\op{}$-modules. In particular, this shows that $\mathbb{D}$ is well-defined and so is $\mathbb{D}\op{}$ by dual arguments. Moreover, since all the kernels $\ker \mathbb{D}(f_{k})$ are GP, $\mathbb{D}\op{}\mathbb{D}(P^\bullet)$ is a resolution of $\mathbb{D}\op{}\mathbb{D}(M)$. Now, since $\mathbb{D}\op{}\mathbb{D}(P^\bullet)$ is naturally isomorphic to $P^\bullet$, we obtain a natural isomorphism $\mathbb{D}\op{}\mathbb{D}(M) \cong M$. This proves the claim.

The converse direction in (GP2) can be seen in the following way. Given a Gorenstein projective $\Lambda$-module $M$, we may take a projective resolution $P^\bullet$ of the GP $\Lambda\op{}$-module $\mathbb{D}(M)$. Applying $\mathbb{D}\op{}$ yields a projective coresolution $\mathbb{D}\op{}(P^\bullet)$ of $M$. In particular, for any $n\geq 0$ there is a $\Lambda$-module $N$ such that $M \cong \Omega^n(N)$.

We already know that $\GP(\Lambda)$ is an exact category, by Example \ref{E:Exact}. Since syzygies of GP modules are again GP, $\GP(\Lambda)$ has enough projective objects. By definition, the dualities $\mathbb{D}$ and $\mathbb{D}\op{}$ are exact. Hence they send projectives to injectives and vice versa. In particular, $\GP(\Lambda)$ has enough injectives and $\proj \GP(\Lambda)=\proj-\Lambda$.
\end{proof}

\begin{ex}\label{E:Frobenius}
 In the situation of Example \ref{E:Exact} e), define
\[I(X)^n=X^n \oplus X^{n+1}, \, \, d_{I(X)}^n=\begin{pmatrix} 0 & 1 \\ 0 & 0 \end{pmatrix}, \, \, (i_{X})^n=\begin{pmatrix} 1 \\ d_{X}^n \end{pmatrix} \]
\[\bigl(\Omega^{-1}(X)\bigr)^n=X^{n+1}, \, \, d_{\Omega^{-1}(X)}^n=-d_{X}^{n+1}, \, \, p_{X}^n=\begin{pmatrix} -d_{X}^n & 1 \end{pmatrix}.\] One can check that $I(X)$ is injective and we have a conflation in $\Com(\kC)$
$X \xrightarrow{i_{X}} I(X) \xrightarrow{p_{X}} \Omega^{-1}(X)$. Moreover, $I(X)$ is also projective and since $\Omega^{-1}(-)$ is invertible the conflation $(i_{X}, p_{X})$ also shows that $\Com(\kC)$ has enough projectives. In order to see that projective and injective objects coincide, one can proceed as follows: $i_{X}$ splits if and only if $X$ is homotopic to zero if and only if $p_{X}$ splits. By definition, direct summands of projective (respectively injective) objects are projective (respectively injective) and any deflation onto an projective (respectively any inflation from an injective) object splits. Thus we obtain: a complex $X$ is injective if and only if it is homotopic to zero if and only if it is projective. In particular, $\Com(\kC)$ is a Frobenius category.
\end{ex}

The following technical lemma will be used in the proof of Theorem \ref{t:main-thm}.

\begin{lem}\label{technical lemma}
Let $\ce$ be a Frobenius category.  If $f\colon X\to Y$ is a morphism in $\ce$ such that $\Hom_{\ce}(f,P)$ is surjective for all $P \in \proj \ce$, then there exists a conflation
\begin{align}
X\xrightarrow{(f\ 0)^{\rm tr}}Y\oplus P'\to Z
\end{align}
in $\ce$ with a projective-injective object $P'$.
\end{lem}
\begin{proof}
Since $\ce$ has enough injective objects, there is an inflation $ X\xrightarrow{i_{X}} I(X)$  in $\ce$ with a projective-injective object $I(X)$. 
By the surjectivity of $\Hom_{\ce}(f,I(X))$, we obtain a morphism $e\colon Y \ra I(X)$ such that $i_{X}=ef$.
By axiom (Ex$2^{op}$) of an exact category, we have a pushout diagram
\begin{align}
\begin{array}{cc}
\begin{xy}
\SelectTips{cm}{}
\xymatrix{
X  \ar[d]_{i_{X}=ef} \ar[r]^{f} & Y \ar[d] \\
I(X) \ar[r] & Z,
}
\end{xy}
\end{array}
\end{align}
which yields a conflation
$X\xrightarrow{(f\ ef)^{\rm tr}}Y\oplus I(X)\xrightarrow{g} Z$ in $\ce$ (see \cite[Proof of Proposition A.1]{Keller2}). There is an isomorphism of exact pairs
\begin{align}
\begin{array}{cc}
\begin{xy}
\SelectTips{cm}{}
\xymatrix{
X \ar[d]^{1} \ar[rr]^{(f\ ef)^{\rm tr}} &&Y \oplus I(X) \ar[rr]^g \ar[d]^{\left( \begin{smallmatrix}1 & 0 \\ -e & 1 \end{smallmatrix} \right)}  && Z \ar[d]^1 \\
X \ar[rr]_{(f\ 0)^{\rm tr}} && Y\oplus I(X) \ar[rr]_{g \circ \left( \begin{smallmatrix}1 & 0 \\ e & 1 \end{smallmatrix} \right)}  && Z.
}
\end{xy}
\end{array}
\end{align}
Hence the lower row is a conflation as well and $P':=I(X)$ completes the proof. 
\end{proof}

\subsubsection{The stable category}\label{SubSub:Stable}
Let throughout $\kE$ be a Frobenius category. 
\begin{defn}
For every two objects $X, Y$ in $\kE$ define \[\kP(X, Y)=\{f \in \Hom_{\kE}(X, Y) \big| f \text{  factors over some projective object} \}.\] The \emph{stable category} $\underline{\kE}$ of $\kE$ has the same objects as $\kE$ and morphisms are given by $\underline{\Hom}_{\kE}(X, Y):=\Hom_{\kE}(X, Y)/\kP(X, Y)$. For a morphism $f \in \Hom_{\kE}(X, Y)$, let $\overline{f}$ be its residue class in $\underline{\Hom}_{\kE}(X, Y)$. 
\end{defn}

\begin{rem}
 $\underline{\kE}$ inherits the structure of an additive category from $\kE$. Moreover, the projective-injective objects in $\kE$ become isomorphic to zero in the stable category.
\end{rem}

Let $X\xrightarrow{f} Y$ be a morphism of $\kE$. Then there are commutative diagrams such that all the rows are conflations
\[
\begin{xy}
\SelectTips{cm}{}
\xymatrix{
X \ar[d]^{f} \ar[rr]^{i_{X}} && I(X) \ar[rr]^{p_{X}} \ar[d]^g && \Omega^{-1}(X) \ar[d]^h \\
Y \ar[rr]^{i_{Y}} && I(Y) \ar[rr]^{p_{Y}} && \Omega^{-1}(Y)
}
\end{xy}
\]
and
\[
\begin{xy}
\SelectTips{cm}{}
\xymatrix{
\Omega(X) \ar[d]^{h'} \ar[rr]^{i_{X}} && P(X) \ar[rr]^{p_{X}} \ar[d]^{g'} && X \ar[d]^f \\
\Omega(Y) \ar[rr]^{i_{Y}} && P(Y) \ar[rr]^{p_{Y}} && Y.
}
\end{xy}
\]
This follows from the definition of injective (respectively projective) objects and the universal property of the cokernel (respectively kernel). One checks that this defines mutually inverse autoequivalences $\Omega$ and $\Omega^{-1}$ of the stable category $\ul{\ce}$.  More precisely, $\Omega(\overline{f})=\overline{h'}$ and $\Omega^{-1}(\overline{f})=\overline{h}$. Moreover, these assignments do \emph{not} depend on the choice of $g'$ and $g$.

We are ready to introduce a triangulated structure with shift functor $S=\Omega^{-1}$ on $\underline{\kE}$ (see Happel \cite{Happel}, who was inspired by Heller \cite{Heller}).

\begin{thm}[Happel--Heller] \label{T:HappelFrob}
Let $X \xrightarrow{u} Y \xrightarrow{v} Z$ be a conflation in $\kE$. Consider the following commutative diagram in $\kE$
\[
\begin{xy}
\SelectTips{cm}{}
\xymatrix{
X \ar[d]^{1_{X}} \ar[rr]^{u} && Y \ar[rr]^{v} \ar[d]^{g} && Z \ar[d]^w \\
X \ar[rr]^{i_{X}} && I(X) \ar[rr]^{p_{X}} && S(X).
}
\end{xy}
\] A sequence of morphisms $A \rightarrow B \rightarrow C \rightarrow S(A)$ in $\ul{\ce}$ is called  \emph{distinguished triangle}, if it is isomorphic to a sequence of the following form
\begin{align}\label{A:StandardTria}
X \xrightarrow{\overline{u}} Y \xrightarrow{\overline{v}} Z \xrightarrow{\overline{w}} S(X).
\end{align}
This class of distinguished triangles defines the structure of a triangulated category on  $\underline{\kE}$. Moreover, a map of Frobenius categories (see Definition \ref{D:Frobenius}) induces a triangulated functor $\underline{F}\colon \underline{\kE} \rightarrow \underline{\kE'}$ between the stable categories.
\end{thm}

\begin{ex} The algebra of dual numbers $R=k[x]/(x^2)$ is selfinjective. Hence $\mod-R$ is a Frobenius category. Its stable category $\ul{\mod}-R$ is equivalent to the category of finite-dimensional vector spaces $\mod-k$.
\end{ex}

\begin{ex}
Let $\Com(\kC)$ be the Frobenius category of complexes over an additive category $\kC$ (see Example \ref{E:Frobenius} ). Since in this case the projective-injective objects are precisely those complexes which are homotopic to zero, the stable category is just the homotopy category $K(\kC)$ of complexes over $\kC$. 
\end{ex}

\subsection{Idempotent completion}\label{Sub:Idemp}
In this subsection, we collect some well-known results on idempotent completions (or Karoubian hulls) of 
additive categories. We mainly follow Balmer \& Schlichting \cite{BalmerSchlichting}.

\begin{defn}
An additive category $\kC$ is called \emph{idempotent complete}, if every idempotent endomorphism $e\colon X \rightarrow X$ in $\kC$ splits. In other words, there exists a direct sum decomposition $X \cong Y \oplus Z$ in $\kC$ and a factorisation $e=i_{Y}p_{Y}$, where $i_{Y}$ and $p_{Y}$ denote the canonical inclusion and projection and we identify $X$ with $Y \oplus Z$.
\end{defn}

\begin{ex}\label{E:AbIdempC}
Abelian categories are idempotent complete. Indeed, every idempotent $e \in \End(X)$ yields a direct sum decomposition $X = \im(e) \oplus \im(1-e)$. We obtain $e=i_{\im(e)} p_{\im(e)}.$ 
\end{ex}

\begin{ex}
The category $\kE$ of even dimensional vector spaces over a field $k$ is \emph{not} idempotent complete. For example, the endomorphism \[k^2 \xrightarrow{\begin{pmatrix}Ê1 &0 \\ 0 & 0\end{pmatrix}} k^2\] does \emph{not} split. 
\end{ex}

There is a natural way to embed every additive category into some idempotent complete category.

\begin{defn}
Let $\kC$ be an additive category. The \emph{idempotent completion} $\kC^{\omega}$ of $\kC$ consists of objects $(X, e)$, with $X$ in $\kC$ and $e \in \End_{\kC}(X)$ is an idempotent. A morphism $(X, e) \xrightarrow{\alpha} (Y, f)$ in $\kC^{\omega}$ is a morphism $X \xrightarrow{\alpha} Y$ in $\kC$ satisfying $\alpha e = \alpha = f \alpha$.
\end{defn}

The map $X \mapsto (X, 1)$ may be extended to a functor $\kC \xrightarrow{\iota} \kC^{\omega}$.
The following proposition is well-known, see e.g.~\cite[Proposition 1.3.]{BalmerSchlichting}.

\begin{prop}
In the notations above $\kC^{\omega}$ is an additive, idempotent complete category and $\iota$ is additive and fully faithful.

In particular we may and will consider $\kC$ as a full subcategory of $\kC^{\omega}$.
\end{prop}

\begin{rem}\label{R:IdempCompl}
(a) \, The direct sum of $(X, e)$ and $(Y, f)$ is given by \[\left(X \oplus Y, \begin{pmatrix} e & 0 \\ 0 & f \end{pmatrix}\right).\]
\noindent
(b)\, Every additive functor $\FF\colon \kC \rightarrow \kD$ induces an additive functor $\FF^{\omega}\colon \kC^{\omega} \rightarrow \kD^{\omega}$ given by $\FF^{\omega}(X, e)=(\FF(X), \FF(e))$ on objects and $\FF^{\omega}(\alpha)=\FF(\alpha)$ on morphisms.

In particular, $\FF^{\omega}$ is fully faithful (an equivalence), if $\FF$ is fully faithful (an equivalence).
\end{rem}

\begin{ex}
The idempotent completion $\kE^{\omega}$ of the category $\kE$ of even dimensional vector spaces over $k$ is equivalent to the category $k-\mod$ of finite dimensional vector spaces over $k$.
\end{ex}

\begin{defn}\label{D:IdempTria}
Let $\kT$ be a triangulated category. The shift functor $[1]$ may be lifted to a functor $[1]^{\omega} \colon \kT^{\omega} \rightarrow \kT^{\omega}$ (see Remark \ref{R:IdempCompl}). Define a triangle $\Delta$ in $\kT^{\omega}$ to be exact, if there exists another triangle $\Delta'$ such that $\Delta \oplus \Delta'$ defines an exact triangle in $\kT$.
\end{defn}

The following proposition is due to Balmer and Schlichting \cite[Theorem 1.5]{BalmerSchlichting}.

\begin{prop}
Let $\kT$ be a triangulated category. Equipped with the structure of Definition \ref{D:IdempTria}, the idempotent completion $\kT^{\omega}$ is a triangulated category and the inclusion $\iota$ is a triangle functor. This is the unique triangulated structure on $\kT^{\omega}$ such that $\iota$ is triangulated. 

Moreover, triangulated functors $\FF$ yield triangulated functors $\FF^{\omega}$.
\end{prop}

\begin{rem}
We give an explicit example showing that Verdier quotient categories of idempotent complete triangulated categories need \emph{not} be idempotent complete, see Lemma \ref{L:Non-Split}. However, Balmer and Schlichting show \cite[Theorem 2.8.]{BalmerSchlichting} that the bounded derived category of an idempotent complete exact category is always idempotent complete. In particular, the bounded derived category of an abelian category is idempotent complete. 
\end{rem}

\subsection{Auslander \& Solberg's modification of exact structures} \label{ss:AuslanderSolberg}

In this subsection, we explain a technique, which (under certain technical assumptions) produces Frobenius categories from exact categories, by a modification of the exact structure. This is implicitely contained in  a work of Auslander \& Solberg \cite{AS93-1}. In particular, the technique may be used to obtain new Frobenius structures on a given Frobenius category. We need some preparations.

\begin{defn}
Let $\ce$ be an exact category. A full additive subcategory $\cm \subseteq \ce$ is called \emph{contravariantly finite}, if every object $X$ in $\ce$ admits a \emph{right $\cm$-approximation}, i.e.~there exists an object $M \in \cm$ and a morphism $f\colon M \ra X$, such that the induced map $\Hom_{\ce}(N, M) \to \Hom_{\ce}(N, X)$ is surjective for all $N \in \cm$. Dually, we define the notion of a \emph{covariantly finite} subcategory. We say that $\cm$ is \emph{functorially finite} if it is both covariantly and contravariantly finite.
\end{defn}

The following example is well-known (see e.g.~\cite{AuslanderSmalo80}) and will be important later.

\begin{ex}\label{Ex:FunctFinite}
Let $(R, \mathfrak{m})$ be a commutative local complete Noetherian ring and let $\ce \subseteq \mod-R$ be a full exact subcategory, i.e.~$\ce$ is closed under extensions in $\mod-R$. Let $\cm \subseteq \ce$ be a full additive subcategory with only \emph{finitely} many isomorphism classes of indecomposable objects $M_{1}, \ldots, M_{t}$, then $\cm$ is functorially finite. Indeed, let $M=M_{1} \oplus \ldots \oplus M_{t}$ and let $X \in \ce$, then $\Hom_{R}(M, X)$ is a finitely generated $R$-module. In particular, it is finitely generated as a right $\End_{R}(M)$-module. Let $f_{1}, \ldots, f_{s}$ be a set of generators. Then one can check that 
$
M^s \xrightarrow{\left(f_{1}, \ldots, f_{s}\right)} X
$
is a right $\cm$-approximation. Hence $\cm$ is contravariantly finite. A dual argument shows that it is also covariantly finite. Therefore, $\cm$ is functorially finite.
\end{ex}

 The following result is implicitly included in Auslander \& Solberg's relative homological algebra \cite{AS93-1}. As usual, we denote by $\ol{\ce}$ the factor category of an exact category $\ce$ by the ideal of morphisms, which factor through an injective object. Accordingly, the morphism spaces in $\ol{\ce}$ are denoted by $\ol{\Hom}$.

\begin{prop}\label{new Frobenius structure}
Let $(\ce, \cs)$ be a $k$-linear exact category with enough projectives $\cp$ and enough injectives $\ci$. Assume that there exist an equivalence $\tau\colon\underline{\ce}\to\overline{\ce}$ and
a functorial isomorphism $\Ext^1_{\ce}(X,Y)\cong D\overline{\Hom}_{\ce}(Y,\tau X)$ for any $X,Y\in\ce$.
Let $\cm$ be a functorially finite subcategory of $\ce$ containing $\cp$ and $\ci$ and satisfying $\tau\underline{\cm}=\overline{\cm}$. Then the following statements hold.

\noindent \t{(1)} Let $X\xrightarrow{f} Y\xrightarrow{g} Z$ be a conflation in $(\ce, \cs)$. Then $\Hom_{\ce}(\cm,g)$ is surjective if and only if $\Hom_{\ce}(f,\cm)$ is surjective. We denote the subclass of the exact pairs in $\cs$  satisfying one of these equivalent conditions by $\cs'$.\\
\t{(2)} $(\ce, \cs')$ is a Frobenius category whose projective objects are precisely $\add\cm$. 
\end{prop}

\begin{proof}
(1) Applying $\Hom_{\ce}(\cm,-)$ to $X\xrightarrow{f} Y\xrightarrow{g} Z$, we have an exact sequence
\begin{eqnarray}
\Hom_{\ce}(\cm,Y)\rightarrow\Hom_{\ce}(\cm,Z)\to\Ext^1_{\ce}(\cm,X)\rightarrow\Ext^1_{\ce}(\cm,Y).
\end{eqnarray}
Thus we know that $\Hom_{\ce}(\cm,g)$ is surjective if and only if
$\Ext^1_{\ce}(\cm,f)$ is injective, which is equivalent to $\Ext^1_{\ce}(\ul{\cm},f)$ being injective.
Using the equivalence $\tau$, this holds if and only if $\overline{\Hom}_{\ce}(f,\tau\ul{\cm})$ is surjective, which holds if and only if $\overline{\Hom}_{\ce}(f,\ol{\cm})$ is surjective. Since $f$ is an inflation, this holds if and only if $\Hom_{\ce}(f,\cm)$ is surjective.

(2) Let us check that $\ce$ together with the  subclass of exact pairs $\cs'\subseteq \cs$ satisfying the conditions in (1) also satisfies the axioms of exact categories (see Definition \ref{D:ExactKeller}). The axioms (Ex0)  and (Ex1) for $(\ce, \cs')$ are satisfied, since they hold for the original structure  $(\ce, \cs)$ and compositions of surjections are surjective. Let us check axiom (Ex2). Let $Y \stackrel{p}\ra Z \stackrel{f}\la Z'$ be a diagram of morphisms in $\ce$ and assume that $p$ is a deflation with respect to $\cs'$. By axiom (Ex2) for the pair $(\ce, \cs)$, there exists a pull back diagram
\begin{align*}
\begin{xy}
\SelectTips{cm}{}
\xymatrix{
Y' \ar[d]^{f'} \ar[rr]^{p'} && Z'  \ar[d]^{f}\\
Y \ar[rr]^p && Z
}
\end{xy}
 \end{align*} 
such that $p'$ is a deflation with respect to $\cs$. It remains to show that $\Hom_{\ce}(\cm, p')$ is surjective. Let $T \in \cm$ and let $g\colon  T \ra Z'$ be a morphism in $\ce$. By the surjectivity of $\Hom_{\ce}(\cm, p)$, we obtain a morphism $g'\colon T \ra Y$ such that $pg'=fg$. Now, the universal property of the pull back yields a morphism $h\colon T \ra Y'$ such that $p'h=g$. This shows that $p'$ is a deflation with respect to $\cs'$. Hence, $(\ce, \cs')$ satisfies (Ex2). The axiom ($\text{Ex2}^{\rm \op}$) is satisfied by a dual argument.

By definition of $\cs'$, every object in $\add\cm$ is projective and injective in $(\ce, \cs')$.
We will show that $\ce$ has enough projectives with respect to the exact structure $(\ce, \cs')$.
For any $X\in\ce$, we take a right $\cm$-approximation $f\colon M \to X$ of $X$.
Since $\cm$ contains $\cp$, any morphism from $\cp$ to $X$ factors through $f$.
By the dual of Lemma \ref{technical lemma}, we have a conflation in $(\ce, \cs)$
\begin{align} \label{E:NewEnoughProj}
 Y\to M\oplus P\xrightarrow{(f \ 0)} X.
\end{align}
 with $P\in\cp$. This conflation is actually contained in $\cs'$, since $f$ is a right $\cm$-approximation. In particular, this sequence shows that $(\ce, \cs')$ has enough projective objects.
Dually, we have that $(\ce, \cs')$ has enough injective objects. Take any projective object $Q$ in $(\ce, \cs')$. Then \eqref{E:NewEnoughProj} yields a deflation from $\add \cm$ onto $Q$. Now, \eqref{E:NewEnoughProj} splits and therefore $Q \in \add \cm$. The injective objects can be treated similarly. Summing up, the classes  of projective objects and injective objects in $(\ce, \cs')$ both equal
$\add\cm$. This completes the proof.
\end{proof}

\begin{ex}\label{Ex:Kong}
Consider the finite dimensional $k$-algebra \[T=k[x]/(x^2) \otimes_{k} k(A \xrightarrow{y} B).\] It is isomorphic to the quiver algebra $kQ/R$, where $Q$ is the following quiver
\begin{align}
\begin{xy}
\SelectTips{cm}{}
\xymatrix{
1  \ar@(lu,ld)[]_{\alpha} \ar[rr]^{\beta} && 2 \ar@(ru,rd)[]^{\gamma}
}
\end{xy}
\end{align}
and $R$ is generated by $\alpha^2$, $\gamma^2$ and $\beta\alpha-\gamma\beta$. We claim that the pair \begin{align}\cm:=\add_{T}(P_{1} \oplus P_{2} \oplus I_{2}) \subseteq \mod-T=:\ce \end{align} satisfies the conditions in Proposition \ref{new Frobenius structure}. Indeed, since $T$ is finite dimensional, we may take $\tau\colon \ul{\mod}-T \ra \ol{\mod}-T$ to be the Auslander--Reiten translation. Then the Auslander--Reiten formula (see e.g. \cite{ARS}), yields a functorial isomorphism $\Ext^1_{\ce}(X,Y)\cong D\overline{\Hom}_{\ce}(Y,\tau X)$. Since $P_{2}=I_{1}$ is a projective and injective object, $\cm$ contains all projective objects $\cp$ and all injectives $\ci$. It remains to check that $\tau(I_{2}) \cong P_{1}$ holds. The minimal projective resolution of $I_{2}$ is given by
\begin{align}
0 \ra P_{1} \xrightarrow{\beta \cdot} P_{2} \ra I_{2} \ra 0.
\end{align}
Applying the Nakayama functor $\nu$, yields a map
$
 I_{1} \xrightarrow{\nu(\beta \cdot)} I_{2} 
$
with kernel $P_{1}$. This shows that  $\tau(I_{2}) \cong P_{1}$. Therefore,  $\mod-T$ admits a Frobenius structure with projective-injective objects $\cm$, by Proposition \ref{new Frobenius structure}. Moreover, one can show that the stable category $\ul{\mod}-T$ is triangle equivalent to the triangulated orbit category $\cd^b(kA_{3})/[1]$.
\end{ex}

\subsection{A Morita type Theorem for Frobenius categories }\label{ss:Iyama}

\begin{defn}\label{defNCR}
Let $\ce$ be a Frobenius category with $\proj \ce= \add P$ for some $P \in \proj \ce$.  If there exists  $M\in\ce$ with $P\in\add M$, such that $A:=\End_{\ce}(M)$  is Noetherian and $\gldim(A)<\infty$, then $A$ is a  \emph{noncommutative resolution} of $\ce$.
\end{defn}

The purpose of this section is to show that the existence of a noncommutative resolution puts strong restrictions on $\ce$. The following theorem is based on joint work with Osamu Iyama, Michael Wemyss and Dong Yang \cite{IKWY12}. The strategy of the proof is inspired by \cite[Theorem 2.2.~(a)]{AIR}.

\begin{thm}\label{t:main-thm}
Let $\ce$ be a Frobenius category with $\proj \ce= \add P$ for some $P \in \proj \ce$.  Assume that there exists a noncommutative resolution $\End_{\ce}(M)$ of $\ce$ with $\gldim\End_{\ce}(M)=n$.  Then the following statements hold
\begin{itemize} 
\item[(1)] $E=\End_\ce(P)$ is an Iwanaga--Gorenstein ring of dimension at most $n$.\\
\item[(2)] We have an equivalence $\Hom_{\ce}(P,-)\colon\ce\to\GP(E)$ up to direct summands. It is an equivalence if $\ce$ is idempotent complete. This induces a triangle equivalence 
\[
\underline{\ce}\xrightarrow{\sim}\underline{\GP}(E)
\]
up to direct summands. It is an equivalence if $\ce$ or $\underline{\ce}$ is idempotent complete.\\
\item[(3)] $\ul{\ce}=\thick_{\ul{\ce}}(M)$.
\end{itemize}
\end{thm}

\begin{rem}
Not every Frobenius category with a projective generator admits a noncommutative resolution.  Indeed, let $R$ be a local complete normal Gorenstein surface singularity over $\mathbb{C}$ and consider the Frobenius category $\ce=\MCM(R)$. For example, isolated hypersurface singularities $R=\mathbb{C}\llbracket x, y, z\rrbracket/(f)$ satisfy these conditions. Then every noncommutative resolution of $\ce$ in the sense of Definition \ref{defNCR} is an NCR of $R$ in the sense of \cite{DaoIyamaTakahashiVial12}. In particular, if $\MCM(R)$ has a noncommutative resolution, then $\Spec(R)$ has rational singularities, by \cite[Corollary 3.3.]{DaoIyamaTakahashiVial12}. The complete rational Gorenstein surface singularities are precisely the ADE-surface singularities, see e.g. \cite{Durfee} or Section \ref{s:Specials}. In other words, if $R=\mathbb{C}\llbracket x, y, z\rrbracket/(f)$ is an isolated hypersurface singularity $R=\mathbb{C}\llbracket x, y, z\rrbracket/(f)$ which is not of ADE-type, then $\MCM(R)$ does not admit a noncommutative resolution.

This example also shows that the existence of noncommutative resolutions  is not necessary for Theorem \ref{t:main-thm} (2), since $\ce=\MCM(R) =\GP(R)$ for commutative Gorenstein rings $R$.
\end{rem}

\begin{proof}
It is well-known that the functor $\Hom_{\ce}(P,-)\colon \ce\to\mod-E$ is fully faithful. Moreover, it restricts to an equivalence $\Hom_{\ce}(P,-)\colon\add P\to\proj-E$ up to summands.  We can drop the supplementary `up to summands'  if $\ce$ is idempotent complete. We establish (1) in three steps:

(i) We first show that $\Ext^i_E(\Hom_{\ce}(P,X),E)=0$ for any $X\in\ce$ and $i>0$.  Since $\ce$ has enough projective objects, there exists a conflation
\begin{equation}\label{P resolution}
0\to Y\to P'\to X\to 0
\end{equation}
 in $\ce$ with $P'$ projective.
Applying $\Hom_{\ce}(P,-)$, we obtain an exact sequence
\begin{equation}\label{4th syzygy}
0\to\Hom_{\ce}(P,Y)\to\Hom_{\ce}(P,P')\to\Hom_{\ce}(P,X)\to0
\end{equation}
with a projective $E$-module $\Hom_{\ce}(P,P')$.
Applying $\Hom_{\ce}(-,P)$ to \eqref{P resolution} and $\Hom_E(-,E)$ to \eqref{4th syzygy}
respectively and comparing them, 
we have a commutative diagram of exact sequences
\[
\begin{array}{c}
{\SelectTips{cm}{10}
\xy0;/r.35pc/:
(-15,20)*+{\Hom_\ce(P^\prime,P)}="A0",(18,20)*+{\Hom_\ce(Y,P)}="A1",(50,20)*+{0}="A2",
(-15,10)*+{\Hom_E(\Hom_{\ce}(P,P'),E)}="a0",(18,10)*+{\Hom_E(\Hom_{\ce}(P,Y),E)}="a1",(50,10)*+{\Ext^1_E(\Hom_{\ce}(P,X),E)}="a2",(68,10)*+{0,}="a3",
\ar"A0";"A1"
\ar"A1";"A2"
\ar"a0";"a1"
\ar"a1";"a2"
\ar"a2";"a3"
\ar"A0";"a0"^\wr_{\Hom_{\ce}(P, -)}
\ar"A1";"a1"^\wr_{\Hom_{\ce}(P, -)}
\ar"A2";"a2" 
\endxy}
\end{array}
\]
where the vertical arrows are isomorphisms since $\Hom_{\ce}(P, -)$ is fully faithful.

In particular, the diagram shows that $\Ext^1_E(\Hom_{\ce}(P,X),E)=0$.
Since the syzygy of $\Hom_{\ce}(P,X)$ has the same form $\Hom_{\ce}(P,Y)$, we may repeat this argument to obtain $\Ext^i_E(\Hom_{\ce}(P,X),E)=0$ for any $i>0$.

(ii) We show that for any $X\in\mod-E$, there exists an exact sequence
\begin{equation}\label{approximation}
0\to Q_n\stackrel{g_{n}}\to\cdots\stackrel{g_{1}}\to Q_0\stackrel{g_{0}}\to X\to0
\end{equation}
of $E$-modules with $Q_i\in\add\Hom_{\ce}(P,P\oplus M)$.

Define an $A$-module by $\widetilde{X}:=X \otimes_E \Hom_{\ce}(P\oplus M,P)$.
Let $e$ be the idempotent of $A=\End_{\ce}(P\oplus M)$ corresponding to the summand $P$ of $P\oplus M$.
Then we have $eAe=E$ and $\widetilde{X}e=X$.
Since $\gldim(A) \leq n$, there exists a projective resolution
\[0\to P_n\to\cdots\to P_0\to\widetilde{X}\to0.\]
Multiplying by $e$ and using $Ae=\Hom_{\ce}(P,P\oplus M)$, we have the assertion.

(iii) Consider the short exact sequence of $E$-modules $0 \ra \ker(g_{0}) \ra Q_{0} \stackrel{g_{0}}\ra X \ra 0 $ obtained from \eqref{approximation} above. Applying $\Hom_{E}(-, E)$ to this sequence, we obtain a long exact Ext--sequence. By (i), $\Ext^i_{E}(Q_{0}, E)$ for all $i>0$. Hence, $\Ext^{n+1}_E(X,E)\cong \Ext^n_{E}(\ker(g_{0}), E)$ and proceeding inductively, we obtain 
\[\Ext^{n+1}_E(X,E) \cong \Ext^1_{E}(\ker(g_{n-1}), E) =\Ext^1_{E}(Q_{n}, E) \stackrel{(i)} =0\]  for any $X\in\mod-E$.
This shows that the injective dimension of the $E$-module $E$ is bounded above by $n$.
The dual argument shows that the injective dimension of the $E^{\rm op}$-module $E$ is at most $n$.
Thus $E$ is an Iwanaga--Gorenstein ring of dimension at most $n$, which shows (1).\\
(2) By (i) again, we have a functor $\Hom_{\ce}(P,-)\colon \ce\to\GP(E)$, which is fully faithful.  We will now show that it is dense up to taking direct summands.

Let $X\in\GP(E)$.
We claim that there exists a complex 
\begin{equation}\label{add P+M}
M_n\xrightarrow{f_n}\cdots\xrightarrow{f_1} M_0
\end{equation}
in $\ce$ with $M_i\in\add(P\oplus M)$ such that the induced sequence
\begin{equation}\label{(P,P+M)}
0\to\Hom_{\ce}(P,M_n)\xrightarrow{f_n \circ}\hdots\xrightarrow{f_1 \circ}\Hom_{\ce}(P,M_0)\to X\oplus Y\to0
\end{equation}
is exact for some $Y\in\GP(E)$. Indeed, consider the $\add \Hom_{\ce}(P, P \oplus M)$ resolution of $X$ constructed in \eqref{approximation}. Since $\Hom_{\ce}(P, -)$ is dense up to direct summands, there exists $M_{n} \in \add(P \oplus M)$ and $Q_{n}' \in \GP(E)$ such that $\Hom_{\ce}(P, M_{n}) \cong Q_{n} \oplus Q_{n}'$. Similarly, we can find $M_{n-1} \in \add(P \oplus M)$  and $Q_{n-1}' \in \GP(E)$ such that $\Hom_{\ce}(P, M_{n-1}) \cong (Q_{n-1} \oplus Q_{n}') \oplus Q_{n-1}'$. Since $\Hom_{\ce}(P, -)$ is full, the monomorphism
\[
\widetilde{g}_{n}\colon Q_{n} \oplus Q_{n}' \xrightarrow{\left(\begin{smallmatrix} g_{n} & 0 &0 \\ 0 &1 &0 \end{smallmatrix}\right)^{\rm tr}} (Q_{n-1} \oplus Q_{n}') \oplus Q_{n-1}'
\]
has a preimage $f_{n}\colon M_{n} \ra M_{n-1}$. Now, let $M_{n-2} \in\add(P \oplus M)$  and $Q_{n-2}' \in \GP(E)$ such that $\Hom_{\ce}(P, M_{n-2}) \cong (Q_{n-2} \oplus Q_{n-1}') \oplus Q_{n-2}'$. Since \eqref{approximation} is exact, 
\[
0 \ra Q_{n} \oplus Q_{n}' \xrightarrow{\widetilde{g}_{n}} (Q_{n-1} \oplus Q_{n}') \oplus Q_{n-1}' \xrightarrow{\widetilde{g}_{n-1}=\left(\begin{smallmatrix} g_{n-1} & 0 &0 \\ 0 & 0&1 \\0&0&0  \end{smallmatrix}\right)}
(Q_{n-2} \oplus Q_{n-1}') \oplus Q_{n-2}'
\]
is exact as well. Let $f_{n-1}$ be a preimage of $\widetilde{g}_{n-1}$. Then $f_{n-1}f_{n}=0$ holds, since $\Hom_{\ce}(P, -)$ is faithful. One can proceed in this way to construct preimages $M_{n-3},$ $ \ldots, M_{0}$ and $f_{n-2}, \ldots, f_{1}$ for $Q_{n-3}\oplus Q_{n-2}' \oplus Q_{n-3}', \ldots, Q_{0} \oplus Q_{1}' \oplus Q_{0}'$  and  $\widetilde{g}_{n-2}, \ldots,\widetilde{g}_{1}$, respectively. By faithfulness of $\Hom_{\ce}(P, -)$, we get a complex \eqref{add P+M}.  Moreover, by construction, the cokernel of $\widetilde{g}_{1}=\Hom_{\ce}(P, f_{1})$ is isomorphic to $X \oplus Q_{0}'$. Hence setting $Y=Q_{0}'$ proves the claim.

Applying $\Hom_{\ce}(-,P)$ to \eqref{add P+M} and $\Hom_E(-,E)$ to \eqref{(P,P+M)} and comparing them, we have a commutative diagram
\begin{align} \label{ExUpperLine}
\begin{array}{c}
\begin{xy}
\SelectTips{cm}{}
\xymatrix{
\Hom_{\ce}(M_0,P)\ar[r]\ar[d]^\wr_{\Hom_{\ce}(P, -)}&\cdots\ar[r]&\Hom_{\ce}(M_n,P)\ar[d]^\wr_{\Hom_{\ce}(P, -)}\ar[r]&0 \\
\Hom_E(\Hom_{\ce}(P,M_0),E)\ar[r]&\cdots\ar[r]&\Hom_E(\Hom_{\ce}(P,M_n),E)\ar[r]&0 }
\end{xy}
\end{array}
\end{align}
where the lower sequence is exact since $X\oplus Y\in\GP(E)$. Indeed, this implies that all the cohomologies in \eqref{(P,P+M)} are in $\GP(E)$. But $E$ is injective in $\GP(E)$. Hence, $\Hom_{E}(-, E)$ preserves exactness of the sequence \eqref{(P,P+M)}.

Thus the upper sequence is also exact. Applying Lemma \ref{technical lemma} repeatedly to \eqref{add P+M}, we have
a complex
\begin{align}\label{E:ComplexIyama} 
\begin{array}{c}
 M_n\xrightarrow{(f_n\ 0)^{\rm tr}} M_{n-1}\oplus P_{n-1}\xrightarrow{{f_{n-1}\ 0\ 0\choose \ \ 0\ \ \ 1\ 0}^{\rm tr}}
M_{n-2}\oplus P_{n-1}\oplus P_{n-2}\xrightarrow{} \cdots \\ \\
\cdots\to M_0\oplus P_1\oplus P_0\to N
\end{array}
\end{align}
with projective objects $P_i$, which is a glueing of conflations in $\ce$. Indeed, starting on the left end of \eqref{add P+M}, we obtain a short exact sequence $0 \ra M_{n} \xrightarrow{\left({f_{n} \ 0}\right)^{\rm tr}} M_{n-1} \oplus P_{n-1} \stackrel{\pi_{n-1}}\ra M_{n-1}' \ra 0$ in $\ce$, by Lemma \ref{technical lemma}, which may be applied since \eqref{ExUpperLine} is exact. Now, the universal property of the cokernel of $({f_{n} \ 0})^{\rm tr}$ yields a map $\iota_{n-1}$
\begin{align}\label{E:CommDiaIyama}
\begin{xy}
\SelectTips{cm}{}
\xymatrix{
 M_{n} \ar[r]^(0.35){({f_{n} \ 0})^{\rm tr}} & M_{n-1} \oplus P_{n-1} \ar[rr]^{\left( \begin{smallmatrix} f_{n-1} & 0 \\ 0 &1 \end{smallmatrix} \right)} \ar[rd]^{\pi_{n-1}} && M_{n-2} \oplus P_{n-1} \\
&& M_{n-1}' \ar[ru]^{\iota_{n-1}}
}
\end{xy}
\end{align}
Since $P$ is injective in $\ce$, $\Hom_{\ce}(-, P)$ is exact. Hence, \eqref{E:CommDiaIyama} and the exactness of \eqref{ExUpperLine} shows that $\Hom_{\ce}(\iota_{n-1}, P)$ is surjective. We apply Lemma \ref{technical lemma} and proceed in a similar way to construct the complex \eqref{E:ComplexIyama} above.

Applying $\Hom_{\ce}(P, -)$ to \eqref{E:ComplexIyama} and using the exactness of \eqref{(P,P+M)}, yields an isomorphism $X\oplus Y\oplus \Hom_{\ce}(P,P_0)\simeq \Hom_{\ce}(P,N)$. Hence, $\Hom_{\ce}(P, -)$ is dense up to summands.\\
(3) The existence of (\ref{approximation}) implies that the fully faithful functor $\Hom_{\ce}(P, -)$ induces an equivalence between the idempotent completions of $\thick_{\underline{\ce}}(M)$ and $\underline{\GP}(E)$. Hence the idempotent completions of $\thick_{\underline{\ce}}(M)$ and $\underline{\ce}$ are equivalent, by (2). But the former category is a full subcategory closed under taking direct summands in $\underline{\ce}$. Therefore, we already have $\thick_{\underline{\ce}}(M) \cong \underline{\ce}$ before passing to the idempotent completions.
\end{proof}

\begin{rem}
We remark that  \cite[4.2]{Chen12} gives an embedding $\ce\to\GP(\proj \ce)$ in the general setting, where $\proj \ce$ does  not necessarily admit an additive generator $P$. 
\end{rem}

\subsection{The Buchweitz--Happel--Keller--Rickard--Vossieck equivalence}\label{ss:BuchweitzHappel}
\quad 

\noindent
Let throughout $\ce$ be a Frobenius category.
\begin{defn}\label{D:AcyclicFrob} Let $N \in \mathbb{Z}$. A complex $P^\bullet$ of $\ce$-projective objects is called \emph{acyclic in degrees $\leq N$} if there exists conflations $
\begin{xy}\SelectTips{cm}{}
\xymatrix{
Z^n(P^{\bullet}) \ar[r]^(.65){i^{n}} & P^n \ar[r]^(.35){p^{n}} & Z^{n+1}(P^\bullet) }
\end{xy}
$
in $\ce$ such that $d^n_{P^\bullet}=i^{n+1}p^n$ holds for all $n\leq N$. The full subcategory of $K^-(\proj \ce)$ consisting of those complexes which are acyclic in degrees $\leq d$ for some $d \in \mathbb{Z}$ is denoted by $K^{-, b}(\proj \ce)$. This defines a triangulated subcategory of $K^-(\proj \ce)$ (c.f. \cite{KellerVossieck87}).
\end{defn}
\begin{rem}
The definition of the triangulated category $K^{-, b}(\proj \ce)$ depends on the exact structure of the ambient Frobenius category $\ce$ of $\proj \ce$. We refer to Remark \ref{R:Right-bounded-with-bounded-coho} for a discussion of this subtle issue.
\end{rem}

Taking projective resolutions yields a functor $\P\colon \ce \ra K^{-,b}(\proj \ce) $. We need the following dual version of a result of Keller \& Vossieck \cite[Exemple 2.3.]{KellerVossieck87}. Special cases of this result were obtained 
by  Rickard \cite{Rickardstable} for module categories over selfinjective algebras and more generally by
Buchweitz \cite{Buchweitz87} for Gorenstein projective modules over Iwanaga--Gorenstein rings, see Theorem \ref{T:Buchweitz}. A dual result, of Theorem \ref{T:Buchweitz} for finite dimensional Gorenstein algebras was obtained by Happel \cite{Happel91}.
\begin{prop}\label{P:StFrob}
The functor $\P$ induces an equivalence of triangulated categories 
\begin{align}\ul{\P}\colon \ul{\ce} \longrightarrow K^{-, b}(\proj \ce)/K^b(\proj \ce). \end{align}
\end{prop}
\begin{proof}
We show that $\ul{\P}\colon \ul{\ce} \ra K^-(\proj \ce)/K^b(\proj \ce)$ is a fully faithful triangle functor. This is sufficient since any complex in $K^{-, b}(\proj \ce)$ becomes isomorphic to a shifted projective resolution in $K^{-, b}(\proj \ce)/K^b(\proj \ce)$. Indeed, this may be achieved by brutally truncating the bounded part of the complex, which is not acyclic. Hence the functor $\ul{\P}$ is dense.

In order to show that $\ul{\P}$ commutes with shifts, let $X \in \ce$ and let
\begin{align} \label{E:Confl}
\begin{xy}\SelectTips{cm}{}
\xymatrix{
X \ar[r]^(.4){i} & I(X) \ar[r]^(.4){\pi} & \Omega^{-1}(X)
}
\end{xy}
\end{align}
be a conflation with $I(X)$ projective-injective. We have to construct a natural isomorphism
$\ol{\eta}_{X}\colon \ul{\P}(\Omega^{-1}(X)) \ra \ul{\P}(X)[1]$
in $K^-(\proj \ce)/K^b(\proj \ce)$. Using the conflation (\ref{E:Confl}), $\Omega^{-1}(X)$ has a projective resolution
\begin{align}\label{E:Res}
\begin{array}{cc}
\begin{xy}\SelectTips{cm}{}
\xymatrix{
\cdots \ar[r]^(.4){d^3} & P^2_{\Omega^{-1}(X)} \ar[r]^{d^2} & P^1_{\Omega^{-1}(X)} \ar[rr]^{d^1}  \ar[rd] && I(X) \ar[r]^(.4){\pi} & \Omega^{-1}(X) \ar[r] & 0 \\
&&&X \ar[ru]^(.4){i}
}
\end{xy}
\end{array}
\end{align}
Hence, $X$ has a projective resolution
$
\begin{xy}\SelectTips{cm}{}
\xymatrix{
\cdots \ar[r]^(.4){-d^3} & P^2_{\Omega^{-1}(X)} \ar[r]^{-d^2} & P^1_{\Omega^{-1}(X)} \ar[r] & X \ar[r] & 0
}
\end{xy}
$
and we can define a morphism $\eta_{X}\colon\P(\Omega^{-1}(X)) \rightarrow \P(X)[1]$ in $K^-(\proj \ce)$ as follows
\[
\begin{xy}\SelectTips{cm}{}
\xymatrix{
\cdots \ar[r]^(.4){d^3}  & P^2_{\Omega^{-1}(X)} \ar[r]^{d^2} \ar[d]^1 & P^1_{\Omega^{-1}(X)} \ar[r]^{d^1} \ar[d]^1   & I(X)  \ar[d]\ar[r] & 0 \ar[d]\ar[r] & \cdots \\
\cdots \ar[r]^(.4){d^3} & P^2_{\Omega^{-1}(X)} \ar[r]^{d^2} & P^1_{\Omega^{-1}(X)}  \ar[r] & 0 \ar[r] & 0 \ar[r] & \cdots
}
\end{xy}
\]
One can check that  $\eta_{X}$ is natural in $X$. Since $\Cone(\eta_{X}) \cong I(X)[1] \in K^b(\proj \ce)$ the image $\ol{\eta}_{X}\colon\ul{\P}(\Omega^{-1}(X)) \ra \ul{\P}(X)[1]$ of $\eta_{X}$ in $K^-(\proj \ce)/K^b(\proj \ce)$ is an isomorphism. It is natural since $\eta_{X}$ is natural.

To show that $\P$ respects triangles, recall that a standard triangle $A \xrightarrow{\overline{i}} B \xrightarrow{\overline{p}} C \xrightarrow{\ol{h}} \Omega^{-1}(A)$ in $\ul{\ce}$ is defined via a commutative diagram in $\ce$
\begin{align}\label{E:Comm}
\begin{array}{cc}
\begin{xy}\SelectTips{cm}{}
\xymatrix{
A \ar[d]^1 \ar[rr]^i && B \ar[d]^g \ar[rr]^p && C \ar[d]^h\\
A  \ar[rr]^{i_{A}}  && I(A) \ar[rr]^(.5){p_{A}} && \Omega^{-1}(A)
}
\end{xy}
\end{array}
\end{align}
where both rows are conflations. We have to show that \[\ul{\P}(A) \xrightarrow{\ol{\P(i)}} \ul{\P}(B) \xrightarrow{\ol{\P(p)}} \ul{\P}(C) \xrightarrow{\ol{\eta_{A}\P(h)}} \ul{\P}(A)[1]\] defines a triangle in $K^-(\proj \ce)/K^b(\proj \ce)$. By the definition of the triangulated structure of the Verdier quotient category $K^-(\proj \ce)/K^b(\proj \ce)$ it suffices to show that ${\P}(A) \xrightarrow{{\P(i)}} {\P}(B) \xrightarrow{{\P(p)}} {\P}(C) \xrightarrow{{\eta_{A}\P(h)}} {\P}(A)[1]$
is a triangle in $K^-(\proj \ce)$. To do this, we construct a commutative diagram in $\ce$ in several steps.
\begin{align}\label{E:BigCommDiag}
\begin{array}{cc}
\begin{xy}\SelectTips{cm}{}
\xymatrix@C=35pt{
\vdots \ar[d]^{d_{A}^3} & \vdots \ar[d]^(.4){\left( \begin{smallmatrix} d_{A}^{3} & \phi^2 \\ 0 & d_{C}^3 \end{smallmatrix} \right)} & \vdots \ar[d]^{d_{C}^3} & \vdots \ar[d]^{-d_{A}^2}  & \vdots \ar[d]^{-d_{A}^2} \\
P_{A}^2 \ar[d]^{d_{A}^2} \ar[r]^(.4){\left( \begin{smallmatrix} 1 & 0 \end{smallmatrix} \right)^{\rm tr}}  & P_{A}^2 \oplus P_{C}^2 \ar[d]^(.45){\left( \begin{smallmatrix} d_{A}^{2} & \phi^1 \\ 0 & d_{C}^2 \end{smallmatrix} \right)} \ar[r]^(.6){\left( \begin{smallmatrix} 0 & 1 \end{smallmatrix} \right)}&P_{C}^2 \ar[d]^{d_{C}^2} \ar@{-->}[r]^{-\phi^1} & P_{A}^1 \ar[d]^{-d_{A}^1} \ar[r]^1 & P_{A}^1 \ar[d]^{-d_{A}^1}\\
P_{A}^1 \ar[d]^{d_{A}^1} \ar[r]^(.4){\left( \begin{smallmatrix} 1 & 0 \end{smallmatrix} \right)^{\rm tr}}  & P_{A}^1 \oplus P_{C}^1 \ar[d]^(.45){\left( \begin{smallmatrix} d_{A}^{1} & \phi^0 \\ 0 & d_{C}^1 \end{smallmatrix} \right)} \ar[r]^(.6){\left( \begin{smallmatrix} 0 & 1 \end{smallmatrix} \right)}&P_{C}^1 \ar[d]^{d_{C}^1} \ar@{-->}[r]^{-\phi^0} & P_{A}^0 \ar[d]^{i_{A}\pi_{A}} \ar[r]^1 & P_{A}^0 \ar[d]^{0}\\
P_{A}^0 \ar[d]^{\pi_{A}} \ar[r]^(.4){\left( \begin{smallmatrix} 1 & 0 \end{smallmatrix} \right)^{\rm tr}}  & P_{A}^0 \oplus P_{C}^0 \ar[d]^{\left( \begin{smallmatrix} i\pi_{A} &\pi'_{C} \end{smallmatrix} \right)} \ar[r]^(.6){\left( \begin{smallmatrix} 0 & 1 \end{smallmatrix} \right)}&P_{C}^0 \ar[d]^{\pi_{C}} \ar@{-->}[r]^{g\pi'_{C}} & I(A) \ar[d]^{p_{A}} \ar[r]^0 & 0 \ar[d]\\
A \ar[r]^(.45)i & B \ar[r]^p & C \ar[r]^(.42)h & \Omega^{-1}(A) \ar[r] & 0
}
\end{xy}
\end{array}
\end{align}
Firstly, we take projective resolutions of $A$ and $C$, respectively. Then we can inductively construct a projective resolution of $B$ (Horseshoe Lemma). Next, we resolve $\Omega^{-1}(A)$ as in (\ref{E:Res}). Using the commutativity of (\ref{E:Comm}) and of the  square 
\begin{align}
\begin{array}{cc}
\begin{xy}\SelectTips{cm}{}
\xymatrix{
 P_{A}^0 \oplus P_{C}^0 \ar[d]^{\left( \begin{smallmatrix} i\pi_{A} &\pi'_{C} \end{smallmatrix} \right)} \ar[r]^(.6){\left( \begin{smallmatrix} 0 & 1 \end{smallmatrix} \right)}&P_{C}^0 \ar[d]^{\pi_{C}}  \\
 B \ar[r]^p & C 
}
\end{xy}
\end{array}
\end{align}
on the bottom of diagram \eqref{E:BigCommDiag}, we see that $p_{A}(g\pi'_{C})=hp\pi'_{C}=h\pi_{C}$. This gives us the first broken arrow lifting $h$. The other commutativity relations for the lift of $h$, follow from the fact that the resolution of $B$ is a complex. Finally, the chain map on the right is just $\eta_{A}$. This shows that ${\P}(A) \xrightarrow{{\P(i)}} {\P}(B) \xrightarrow{{\P(p)}} {\P}(C) \xrightarrow{{\eta_{A}\P(h)}} {\P}(A)[1]$ is isomorphic to the standard triangle \[\P(A) \ra \Cone(\phi^{\bullet}) \ra (\P(C)[-1])[1] \xrightarrow{-\phi^{\bullet}[1]} \P(A)[1].\] Hence, $\ul{\P}$ is a triangle functor.

Let us prove that $\ul{\P}$ is faithful. Let $f\colon M \ra N$ be a morphims in $\ce$ such that the lift to the projective resolutions $\P(f)\colon P^\bullet_{M} \ra P^\bullet_{N}$ admits a factorisation $P^\bullet_{M} \xrightarrow{\phi} Q^\bullet \xrightarrow{\psi} P^\bullet_{N},$ where $Q^\bullet \in K^b(\proj \ce)$. We have to show that $f$ factors through $\proj \ce$.
As a first step, we show that if $\phi^{-i}=0$ holds for some $i>1$, then there is a homotopy equivalent chain map $\widetilde{\phi}$ with $\widetilde{\phi}^{-i+1}=0$.

The exactness of $P^\bullet_{M}$ and the injectivity of the $Q^i$ yield a diagram
\[
\begin{xy}\SelectTips{cm}{}
\xymatrix
{
\cdots\ar[r] & P_{M}^{-i} \ar[r]^{d_{M}^{-i}} \ar[dd]^{0} & P_{M}^{-i+1} \ar[dd]^{\phi^{-i+1}} \ar@{-->}[rd] \ar[rrr]^{d_{M}^{-i+1}} &&& P_{M}^{-i+2} \ar[dd]^{\phi^{-i+2}}\ar@/^15.0pt/[llldd]^h \ar[r]^{d_{M}^{-i+2}} & P_{M}^{-i+3} \ar[dd]^{\phi^{-i+3}} \ar[r] & \cdots \\
&&& \im \, d_{M}^{-i+1} \ar@{-->}[rru] \ar@{-->}[ld] \\
\cdots\ar[r] & Q^{-i} \ar[r]^{\del^{-i}} & Q^{-i+1} \ar[rrr]^{\del^{-i+1}} &&& Q^{-i+2} \ar[r]^{\del^{-i+2}} & Q^{-i+3} \ar[r] & \cdots
}
\end{xy}
\]
where $\phi^{-i+1}=h \circ d_{M}^{-i+1}$. Let $\widetilde{\phi}$ be the following chain map
\[
\begin{xy}\SelectTips{cm}{}
\xymatrix
{
\cdots\ar[r] & P_{M}^{-i} \ar[r]^{d_{M}^{-i}} \ar[d]^{0} & P_{M}^{-i+1} \ar[d]^{0} \ar[rrr]^{d_{M}^{-i+1}} &&& P_{M}^{-i+2} \ar[d]_{\phi^{-i+2}-\del^{-i+1}h} \ar[r]^{d_{M}^{-i+2}} & P_{M}^{-i+3} \ar[d]^{\phi^{-i+3}} \ar[r] & \cdots  \\
\cdots\ar[r] & Q^{-i} \ar[r]^{\del^{-i}} & Q^{-i+1} \ar[rrr]^{\del^{-i+1}} &&& Q^{-i+2} \ar[r]^{\del^{-i+2}} & Q^{-i+3} \ar[r] & \cdots
}
\end{xy}
\]
The degree $-1$ map $(\ldots, 0, h, 0, \ldots)$ defines a homotopy between $\phi$ and $\widetilde{\phi}$. Since $Q^{\bullet}$ is left bounded, induction yields a chain map, which is homotopic to $\phi$:
\[
\begin{xy}\SelectTips{cm}{}
\xymatrix
{
\cdots\ar[r] & P_{M}^{-1} \ar[r]^{d_{M}^{-1}} \ar[d]^{0} & P_{M}^{0} \ar[d]^{\widetilde{\phi}^{0}} \ar@{-->}[r]^{\epsilon_{M}} & M \\
\cdots\ar[r] & Q^{-1} \ar[r]^{\del^{-1}} & Q^{0} \ar[r]^{\del^{0}} & Q^{1} \ar[r]^{\del^{1}} & Q^{2} \ar[r] & \cdots
}
\end{xy}
\]
 Now the universal property of the cokernel gives a map $f_{1}\colon M \ra Q^{0}$ such that $\widetilde{\phi}^0=f_{1}\epsilon_{M}$. Let $f_{2}=\epsilon_{N}\psi^{0}$. Then $\psi\widetilde{\phi}$ is a lift of $f_{2}f_{1}$. Thus $f$ and $f_{1}f_{2}$ have homotopic lifts $\P(f)$ and $\psi\widetilde{\phi}$. Hence, $f=f_{2}f_{1}$ and $f$ factors through $Q^0$.

To show that $\ul{\P}$ is full, let  $P_{M}^\bullet \xrightarrow{f} E^\bullet \xleftarrow{s} P_{N}^{\bullet}$ be a morphism in $K^-(\proj \ce)/K^b(\proj \ce)$, i.~e.~$f$ and $s$ are morphisms in $K^-(\proj \ce)$ and $Q^\bullet[-1] \xrightarrow{\phi} P^\bullet_{N} \xrightarrow{s} E^\bullet \ra Q^\bullet$ is a triangle with $Q^\bullet \in K^b(\proj \ce)$. Hence, $E^\bullet$ has the following form
\[\cdots \xrightarrow{\begin{pmatrix} d_{N}^{-2} & \phi^{-1} \\ 0 & d_{Q}^{-2} \end{pmatrix}} P_{N}^{-1} \oplus Q^{-1} \xrightarrow{\begin{pmatrix} d_{N}^{-1} & \phi^{0} \\ 0 & d_{Q}^{-1} \end{pmatrix}} P_{N}^{0} \oplus Q^{0} \xrightarrow{\begin{pmatrix} 0 & d_{Q}^{0} \end{pmatrix}} Q^1 \xrightarrow{\displaystyle{d_{Q}^1}} \cdots
\]
The brutally truncated complex $\brutegeq{0}Q^\bullet$ is a subcomplex of $E^\bullet$. Hence we get a triangle
$\brutegeq{0}Q^\bullet \ra E^\bullet \xrightarrow{q} \widetilde{E}^{\bullet} \xrightarrow{+},$ where $\widetilde{E}^\bullet$ is homotopy equivalent to
\[\cdots \xrightarrow{\begin{pmatrix} d_{N}^{-2} & \phi^{-1} \\ 0 & d_{Q}^{-2} \end{pmatrix}} P_{N}^{-1} \oplus Q^{-1} \xrightarrow{\begin{pmatrix} d_{N}^{-1} & \phi^{0} \end{pmatrix}} P_{N}^{0} \ra 0 \ra \cdots\]
Since $\Cone(q) \in K^b(\proj \ce)$ our original roof is equivalent to $P^{\bullet}_{M} \xrightarrow{qf}\widetilde{E}^{\bullet} \xleftarrow{qs} P_{N}^{\bullet}$. Let $g=qf$. Similar to the proof of faithfulness, we will show that if $g^{-i}=\begin{pmatrix} g_{1}^{-i} & 0 \end{pmatrix}^{\rm tr}$ for some $i>1$, we can find a homotopic chain map $\widetilde{g}$ with $\widetilde{g}^{-i+1}=\begin{pmatrix} \widetilde{g}_{1}^{-i+1} & 0 \end{pmatrix}^{\rm tr}$.

Consider the following diagram
\[
\begin{xy}\SelectTips{cm}{}
\xymatrix
{
\cdots P_{M}^{-i} \ar[r]^{d_{M}^{-i}} \ar[dd]_{\left(\begin{smallmatrix} g_{1}^{-i} \\ 0 \end{smallmatrix}\right)} & P_{M}^{-i+1} \ar[dd]_{\left(\begin{smallmatrix} g_{1}^{-i+1} \\ g_{2}^{-i+1} \end{smallmatrix}\right)} \ar@{-->}[rd]^{\delta} \ar[rr]^{d_{M}^{-i+1}} && P_{M}^{-i+2} \ar[dd]^{\left(\begin{smallmatrix} g_{1}^{-i+2} \\ g_{2}^{-i+2} \end{smallmatrix}\right)}\ar@/^20.0pt/[lldd]^(0.3){\left(\begin{smallmatrix} 0 \\ h \end{smallmatrix}\right)} \ar[r]^{d_{M}^{-i+2}} & P_{M}^{-i+3} \ar[dd]_{\left(\begin{smallmatrix} g_{1}^{-i+3} \\ g_{2}^{-i+3} \end{smallmatrix}\right)}  \cdots \\
&& \im \, d_{M}^{-i+1} \ar@{-->}[ru]^{\iota} \ar@{-->}[ld]_{{\left(\begin{smallmatrix} 0 \\ \beta \end{smallmatrix}\right)}} \\
\cdots  P_{N}^{-i} \oplus Q^{-i} \ar[r]_(.44){\left(\begin{smallmatrix} d_{N}^{-i} & \phi^{-i+1} \\ 0 & d_{Q}^{-i} \end{smallmatrix}\right)} & P_{N}^{-i+1} \oplus Q^{-i+1} \ar[rr]_{\left(\begin{smallmatrix} d_{N}^{-i+1} & \phi^{-i+2} \\ 0 & d_{Q}^{-i+1} \end{smallmatrix}\right)} && P_{N}^{-i+2} \oplus Q^{-i+2} \ar[r]_(.38){\left(\begin{smallmatrix} d_{N}^{-i+2} & \phi^{-i+3} \\ 0 & d_{Q}^{-i+2} \end{smallmatrix}\right)} & P_{N}^{-i+3} \oplus Q^{-i+3} \cdots
}
\end{xy}
\]
Since $g_{2}^{-i+1}d_{M}^{-i}=0$ holds, the universal property of the cokernel yields a morphism $\beta\colon \im \, d_{M}^{-i+1} \ra Q^{-i+1}$ such that $\beta\delta=g_{2}^{-i+1}$. By construction of $P_{M}^{\bullet}$ the differential $d_{M}^{-i+1}$ factors into an deflation $\delta$ followed by an inflation $\iota$. Since $Q^{-i+1}$ is injective, we get a map $h\colon P_{M}^{-i+2} \ra Q^{-i+1}$ satisfying $h\iota=\beta$. Summing up we have $g_{2}^{-i+1}=\beta \delta=h \iota \delta=h d_{M}^{-i+1}.$

Define a degree $0$ map $\widetilde{g}\colon P_{M}^\bullet \ra \widetilde{E}^{\bullet}$ by $\widetilde{g}^j=g^j$ for all $j \notin \{-i+1, -i+2\}$, $\widetilde{g}^{-i+1}=\begin{pmatrix} g_{1}^{-i+1} , & 0\end{pmatrix}^{\rm tr}$ and $\widetilde{g}^{-i+2}=\begin{pmatrix} g_{1}^{-i+2} - \phi^{-i+2} h ,& g_{2}^{-i+2} -d_{Q}^{-i+1}h\end{pmatrix}^{\rm tr}$. One can check that the difference $g-\widetilde{g}$ defines a chain map which is homotopic to zero. Indeed, the homotopy is given by the degree $-1$ map $(\ldots 0, \begin{pmatrix} 0 & h \end{pmatrix}^{\rm tr}, 0 \ldots)$. Hence, $\widetilde{g}$ is a chain map which is homotopic to $g$. Since $Q^\bullet$ is left bounded, induction shows that $g$ is homotopic to the following chain map
\[
\begin{xy}\SelectTips{cm}{}
\xymatrix
{
\cdots \ar[r] &P_{M}^{-3} \ar[r]^{d_{M}^{-3}} \ar[d]_{\left(\begin{smallmatrix} \widetilde{g}_{1}^{-3} \\ 0 \end{smallmatrix}\right)} & P_{M}^{-2} \ar[d]_{\left(\begin{smallmatrix} \widetilde{g}_{1}^{-2} \\ 0 \end{smallmatrix}\right)} \ar[rr]^{d_{M}^{-2}} && P_{M}^{-1} \ar[d]^{\left(\begin{smallmatrix} \widetilde{g}_{1}^{-1} \\ 0 \end{smallmatrix}\right)} \ar[r]^{d_{M}^{-1}} & P_{M}^{0} \ar[d]_{\widetilde{g}^0}   \\
\cdots  \ar[r] & P_{N}^{-3} \oplus Q^{-3} \ar[r]_(.46){\left(\begin{smallmatrix} d_{N}^{-3} & \phi^{-2} \\ 0 & d_{Q}^{-3} \end{smallmatrix}\right)} & P_{N}^{-2} \oplus Q^{-2} \ar[rr]_{\left(\begin{smallmatrix} d_{N}^{-2} & \phi^{-1} \\ 0 & d_{Q}^{-2} \end{smallmatrix}\right)} && P_{N}^{-1} \oplus Q^{-1} \ar[r]_(.55){\left(\begin{smallmatrix} d_{N}^{-1} & \phi^{0} \end{smallmatrix}\right)} & P_{N}^{0}
}
\end{xy}
\]
Thus we obtain a chain map $\overline{g}\colon P_{M}^\bullet \ra P_{N}^\bullet$ whose image $P_{M}^\bullet \xrightarrow{\ol{g}} P_{N}^{\bullet} \xleftarrow{1} P_{N}^\bullet$ in $K^-(\proj \ce)/K^b(\proj \ce)$ is equivalent to
our roof $P^{\bullet}_{M} \xrightarrow{g=qf}\widetilde{E}^{\bullet} \xleftarrow{qs} P_{N}^{\bullet}$. Hence the map $G\colon M \rightarrow N$ induced by $\overline{g}$ is the desired preimage.
\end{proof}

\subsection{Singularity Categories of Iwanaga--Gorenstein rings}\label{ss:SingCatIwanaga}
The following notion was introduced by Buchweitz \cite{Buchweitz87}.

\begin{defn}
Let $R$ be a right Noetherian ring. The Verdier quotient category \[\cd_{sg}(R):=\cd^b(\mod-R)/K^b(\proj-R)\] is called the
\emph{singularity category} of $R$.
\end{defn}

\begin{rem}
 If $R$ is a regular ring (i.e.~$\gldim R < \infty$), every object in $\cd^b(\mod-R)$ admits a bounded free resolution. This yields a triangle equivalence $\cd^b(\mod-R) \cong K^b(\proj-R)$. In other words, the singularity category $\cd_{sg}(R)$ is trivial if $R$ has no singularities. Moreover, this is the reason to consider $K^b(\proj-R)$ as the `smooth part' of $\cd^b(\mod-R)$.

If $R$ is a commutative Noetherian ring, $\cd_{sg}(R)$ can be considered as a measure for the complexity of the singularities of $\Spec(R)$, see e.g.~\cite{Orlov09}.
\end{rem}
Proposition \ref{P:StFrob} has the following consequence.

\begin{cor}\label{C:Buchweitz}
Let $\ce$ be a Frobenius category such that $\proj \ce= \add P$ for some $P \in \proj \ce$.  If $R=\End_{\ce}(P)$ is a right Noetherian ring, then there is a fully faithful triangle functor
\begin{align}\label{E:PreBuchweitz}
\ul{\widetilde{\P}}\colon \ul{\ce} \longrightarrow  \cd_{sg}(R).
\end{align}
\end{cor}
\begin{proof}
The fully faithful functor $\Hom_{\ce}(P, -)\colon \proj \ce \rightarrow \proj-R$ induces an embedding $K^-(\proj \ce) \ra K^-(\proj-R)$. The restriction $K^{-, b}(\proj \ce) \ra K^{-, b}(\proj-R)$ is well defined since $P$ is projective. Now, we can define $\ul{\widetilde{\P}}$ as composition of fully faithful functors
\begin{align}
\begin{xy}\SelectTips{cm}{}
\xymatrix
{
\ul{\ce} \ar[r]^(.3){\displaystyle{\ul{\P}}} & \displaystyle{\frac{K^{-, b}(\proj \ce)}{K^b(\proj \ce)}} \ar[r] & \displaystyle{\frac{K^{-, b}(\proj-R)}{K^b(\proj-R)}} \ar[r]^{\sim} & \displaystyle{\frac{\cd^b(\mod-R)}{K^b(\proj-R)}},
}
\end{xy}
\end{align}
where $\ul{\P}$ is the equivalence from Proposition \ref{P:StFrob} and the last functor is induced by the well-known triangle equivalence $K^{-, b}(\proj-R) \stackrel{\sim}\longrightarrow \cd^b(\mod-R)$.
\end{proof}

Let $R$ be an Iwanaga--Gorenstein ring. Recall from Proposition \ref{P:IwanagaFrobenius} that the category of Gorenstein projective $R$-modules \[\GP(R)=\left\{\left. M \in \mod-R \right| \Ext_{R}^i(M, R)=0 \text{  for all Ê}Êi>0 \right\}\] is a Frobenius category, with $\proj \GP(R)= \proj-R$. If $R$ is commutative, then the notions of Gorenstein projective $R$-modules and maximal Cohen--Macaulay $R$-modules $\MCM(R)$ (see \eqref{E:MCMvsGP}) coincide.  Buchweitz has shown that for these Frobenius categories \eqref{E:PreBuchweitz} is an equivalence, see \cite[Theorem 4.4.1.]{Buchweitz87}.

\begin{thm}\label{T:Buchweitz}
Let $R$ be Iwanaga--Gorenstein. There is a triangle equivalence
\begin{align}
\underline{\GP}(R) \xrightarrow{\ul{\widetilde{\P}}} \cd_{sg}(R)
\end{align}
induced by the inclusion $\GP(R) \subseteq K^{-, b}(\proj-R) \cong \cd^b(\mod-R)$.
\end{thm}
\begin{proof}
The main point is that in this situation, $\proj \GP(R)= \add R=\proj-R$. Hence the additive functor $\Hom_{R}(R, -)\colon \proj \GP(R) \ra \proj-R$ considered in the proof of Corollary \ref{C:Buchweitz} is an equivalence. This induces a triangle equivalence
$K^{-, b}(\proj \GP(R))/K^b(\proj \GP(R)) \cong \cd_{sg}(R)$\footnote{This equivalence is actually more subtle than it might seem on first sight. Indeed, the definition of $K^{-, b}(\proj-\GP(R))$ (respectively $K^{-, b}(\proj-R)$) depends on the ambient exact category $\GP(R)$ (respectively $\mod-R$). Since $\GP(R) \subseteq \mod-R$ is an exact subcategory and $\proj \GP(R)=\proj-R$, $K^{-, b}(\proj-\GP(R)) \subseteq K^{-, b}(\proj-R)$. The other inclusion follows from the fact that over an Iwanaga--Gorenstein ring, the syzygies $\Omega^n(X)$ are Gorenstein projective for every $X\in \mod-R$ and all $n >>0$ (see Proposition \ref{P:IwanagaFrobenius} (GP2)). See also Remark \ref{R:Right-bounded-with-bounded-coho}.}.  Therefore, in this situation, the functor $\ul{\widetilde{\P}}$ from Corollary \ref{C:Buchweitz} is a composition of triangle equivalences.
\end{proof}

\begin{rem}\label{R:Right-bounded-with-bounded-coho}
In the situation of Corollary \ref{C:Buchweitz}, assume that there is an additive equivalence $\proj \ce \ra \proj-R$. This induces a triangle equivalence $K^-(\proj \ce) \ra K^-(\proj-R)$. However, the restriction to $K^{-, b}(\proj \ce) \ra K^{-, b}(\proj-R)$ need \emph{not} be essentially surjective. The reason is that the two ambient exact categories $\proj \ce \subseteq \ce$ and $\proj-R \subseteq \mod-R$ may differ. This leads to different notions of acyclicity.

In order to be more concrete, let $R=k[x]/x^2$ and take $\ce=\proj-R$ with the split Frobenius structure, i.e.~$\proj \ce=\ce$. Then the complex of projective $R$-modules
\[C^\bullet \colon \ldots \stackrel{x\cdot}\longrightarrow R \stackrel{x\cdot}\longrightarrow R \stackrel{x\cdot}\longrightarrow \ldots \stackrel{x\cdot}\longrightarrow R \longrightarrow 0 \longrightarrow \ldots\]
is contained in $K^{-, b}(\proj-R) \subseteq K^{-}(\mod-R)$, since the truncated complex
\[\ldots \stackrel{x\cdot}\longrightarrow R \stackrel{x\cdot}\longrightarrow \ldots \stackrel{x\cdot}\longrightarrow R\twoheadrightarrow Z^0(C^\bullet)\cong R/(x) \longrightarrow 0 \longrightarrow \ldots,\]
is composed from exact sequences 
\[0 \ra Z^{n-1}(C^{\bullet}) \cong R/(x) \ra R \ra R/(x) \cong Z^{n}(C^{\bullet}) \ra 0\] in $\mod-R$, where $n \leq 0$. But $R/(x) \notin \proj-R$. Hence, $C^\bullet \notin K^{-,b}(\proj \ce) \subseteq K^-(\proj- R)$.

This explains why the fully faithful functor from Corollary \ref{C:Buchweitz} \[\ul{\widetilde{\P}}\colon 0=\ul{\ce} \longrightarrow  \cd_{sg}(R)\cong \mod-k\] is not dense in this situation.
\end{rem}

A combination of  Buchweitz' Theorem \ref{T:Buchweitz} and Iyama's Theorem \ref{t:main-thm} yields the following corollary.

\begin{cor}\label{C:BuchweitzIyama}
Let $\ce$ be a Frobenius category satisfying the conditions in Theorem \ref{t:main-thm} and let $E=\End_{\ce}(P)$, where $\add P=\proj \ce$.  Then there is a triangle equivalence  (up to direct summands) 
\begin{align}\ul{\ce} \to \cd_{sg}(E).\end{align}
Moreover, if $\ce$ or $\underline{\ce}$ are idempotent complete, then the statement holds without the supplement `up to direct summands'. In particular, idempotent completeness of $\underline{\ce}$ implies idempotent completeness of the singularity category $\cd_{sg}(E).$
\end{cor}

\begin{rem}
We recover Corollary \ref{C:BuchweitzIyama} in Subsection \ref{ss:Alternative}
using a different approach, which is inspired by our study of the (local) relative singularity categories in Section \ref{S:Local}.
Actually, our `alternative approach' came first and inspired Iyama's more general Theorem \ref{t:main-thm}. The alternative approach does \emph{not} show that $E$ is an Iwanaga--Gorenstein ring. However, this property is not used in our description of the stable categories of special Cohen--Macaulay modules over complete rational surface singularities (Theorem \ref{C:StandardStableSCM}), which was the main motivation for the whole story.
 \end{rem}

\subsection{A tale of two idempotents}\label{S:Tale}

\begin{defn}\label{D:quotientfunctor}
A triangulated functor $\F\colon \cc \ra \cd$ is called \emph{triangulated quotient functor}, if the induced functor $\ul{\F}\colon \cc/\ker \F \ra \cd$ is an equivalence of categories.
\end{defn}

\begin{lem}\label{L:Verdier}
Let $\F \colon \cc \ra \cd$ be a triangulated quotient functor with kernel $\ck$. Let $\cu \subseteq \cc$ be a full triangulated subcategory, let $q\colon\cc \ra \cc/\cu$ be the quotient functor and $\cv=\thick(\F(\cu))$. Then $\F$ induces an equivalence of triangulated categories. 
\[\frac{(\cc/\cu)}{\thick\bigl(q(\ck)\bigr)} \longrightarrow \frac{\cd}{\cv}.\]
\end{lem}
\begin{proof}
$\F$ induces a triangle functor $\ol{\F}\colon \cc/\cu \ra \cd/\cv$. We have $\thick(q(\ck)) \subseteq \ker(\ol{\F})$. We show that $\ol{\F}$ is universal with this property, i.e.~ given a triangle functor $\G\colon \cc/\cu \ra \ct$ satisfying 
$\thick(q(\ck)) \subseteq \ker(\G)$ there exists a unique (up to unique isomorphism) triangle functor 
$\H \colon \cd/\cv \ra \ct$ such that $\G=\H \circ \ol{\F}$.  We explain the following commutative diagram of triangulated categories and functors
\[
\begin{xy}
\SelectTips{cm}{}
\xymatrix@R=13pt@C=40pt
{
\ck \ar[dd]_{\displaystyle q} \ar[rr] && \cc \ar[rr]^{\displaystyle \F} \ar[dd]_{\displaystyle q} && \cd \ar[dd]^{\displaystyle q'} \ar@{-->}[dl]_(.64){\displaystyle \I_{1}} \\
&&& \ct \\
\thick\left(q(\ck)\right) \ar[rr] && \cc/\cu \ar[ru]^{ \displaystyle \G} \ar[rr]^{\displaystyle \ol{\F}} && \cd/\cv. \ar@{-->}[lu]_{\displaystyle \I_{2}}
}
\end{xy}
\]
The functor $\I_{1}$, satisfying $\I_{1}\circ \F= \G \circ q$, exists by the universal property of $\F$ and $\I_{2}$ exists by the universal property of $q'$ and satisfies $\I_{2} \circ q'= \I_{1}$. Since  $\I_{2}\circ \ol{\F}\circ q=\I_{2} \circ q' \circ \F  =\I_{1}\circ\F=\G \circ q$ the (uniqueness part of the) universal property of $q$ implies that $\I_{2}\circ\ol{\F}=\G$ holds.

In order to show uniqueness of $\I_{2}$, let $\mathbb{H}\colon \cd/\cv \ra \ct$ be a triangle functor such that $\mathbb{H}\circ\ol{\F}=\G$ holds. Then $\mathbb{H} \circ q' \circ \F = \G \circ q= \I_{1}\circ\F$ holds and the universal property of $\F$ implies $\mathbb{H} \circ q'=\I_{1}$. Since $\mathbb{H} \circ q' =\I_{1}= \I_{2} \circ q'$ the universal property of $q'$ yields $\I_{2}=\mathbb{H}$.
\end{proof} 

\begin{rem}
The universal property of the triangulated quotient category may be used to show that there is a triangle equivalence
\begin{align}
\frac{(\cc/\cu)}{\thick(q(\ck))} \xrightarrow{\sim} \frac{\cc}{\thick(\cu, \ck)}.
\end{align}
However, for our purposes it is more convenient to work with the quotient on the left hand side.
\end{rem}

The following proposition is a dual version of \cite[Proposition III.5]{Gabriel62}.

\begin{prop} \label{P:Gabriel}
Let $\ca$ and $\cb$ be abelian categories. Let $\F\colon \ca \ra \cb$ be an exact functor admitting a left adjoint $\G\colon \cb \ra \ca$. 

If the unit of the adjunction $\eta\colon \mathsf{Id}_{\cb} \ra \F\G$ is a natural isomorphism, then the induced functor $\ul{\F}\colon \ca/\ker \F \ra \cb$ is an equivalence of categories.
\end{prop}

In conjunction with classical Morita theory (see e.g.~\cite[Theorem 8.4.4.]{DrozdKirichenko}), Proposition \ref{P:Gabriel} yields the following well-known corollary (see e.g.~\cite[Proposition 5.9]{Miyachi}).

\begin{cor}\label{C:GabrielMorita}
Let $A$ be a right Noetherian ring and $e \in A$ be an idempotent. Then $\F=\Hom_{A}(eA, -)\colon \mod-A \ra \mod-eAe$ induces an equivalence of abelian categories\footnote{The ring $eAe$ is automatically right Noetherian. Indeed, it is enough to show that sending a right ideal $I \subseteq eAe$ to $IA$, defines an injection from right ideals in $eAe$ to right ideals in $A$. Let $I \subseteq eAe$ be a right ideal and set $J=IA \cap eRe$. Then $J = Je \subseteq IAe = IeAe=I$. Since $I \subseteq J$, the claim follows.}
\[
\ul{\F}\colon \frac{\displaystyle \mod-A}{\displaystyle \mod-A/AeA} \longrightarrow \mod-eAe.
\]
\end{cor}

\begin{prop}\label{P:TwoIdempotents}
Let $A$ be a right Noetherian ring and let $e, f \in A$ be idempotents. The exact functor $\F=\Hom_{A}(eA, -)$ induces a triangle equivalence 
\begin{align}\label{E:twoidempotents2}
\frac{\cd^b(\mod-A)/\thick(fA)}{\thick(q(\mod-A/AeA))} \longrightarrow \frac{\cd^b(\mod-eAe)}{\thick(fAe)},
\end{align}
where $q\colon \cd^b(\mod-A) \ra \cd^b(\mod-A)/\thick(fA)$ is the canonical quotient functor.
\end{prop}
\begin{proof}
 Using Miyachi's compatibility result (\cite[Theorem 3.2]{Miyachi}), the equivalence of abelian categories in Corollary \ref{C:GabrielMorita} shows that $\F$ induces a triangulated \emph{quotient} functor
$\F\colon \cd^b(\mod-A) \ra \cd^b(\mod-eAe)$. An application of Lemma \ref{L:Verdier} to $\F$ and $\thick(fA)$ completes the proof.
\end{proof}

\begin{rem}
Proposition \ref{P:TwoIdempotents} contains X.-W.~Chen's \cite[Theorem 3.1]{Chen10} as a special case. Namely, if we set $f=1$ and assume that $\prdim_{eAe}(Ae)< \infty$ holds, then $\thick_{eAe}(Ae) \cong K^b(\proj-eAe)$ since $eAe$ is a $eAe$-direct summand of $Ae$. Hence, (\ref{E:twoidempotents2}) yields a triangle equivalence
$\cd_{sg}(A)/\thick(q(\mod-A/AeA)) \ra \cd_{sg}(eAe)$. Moreover, if every finitely generated $A/AeA$-module has finite projective dimension over $A$ (i.e.~the idempotent $e$ is \emph{singularly-complete} in the terminology of loc.~cit.), we get an equivalence of singularity categories \cite[Corollary 3.3]{Chen10}. \begin{align}\label{E:XiaoWu}Ê\cd_{sg}(A) \ra \cd_{sg}(eAe) \end{align} 
\end{rem}

\begin{ex}[{\cite[Example 4.3.]{XWChenSchurFunctTria}}]
For $n>0$, let $\ck_{n}$ be the path algebra of the following quiver with relation $\alpha^2=0$.
\begin{equation*}
\begin{tikzpicture}[description/.style={fill=white,inner sep=2pt}]
    \matrix (n) [matrix of math nodes, row sep=0.1em,
                 column sep=2.5em, text height=1.5ex, text depth=0.25ex,
                 inner sep=0pt, nodes={inner xsep=0.3333em, inner
ysep=0.3333em}]
    {  & \,  \\
       1 && 2  , \\
       & \, \\
    };
    
    \draw[<-] ($(n-2-1.east) + (0mm,1mm)$) .. controls +(5mm,3mm) and
+(-5mm,+3mm) .. ($(n-2-3.west) + (0mm,1mm)$);
  \node[scale=0.75] at ($(n-2-1.east) + (12.5mm, 5.8mm)$) {$\beta_{1}$};
  
    \draw[->] ($(n-2-3.west) + (0mm,-1mm)$) .. controls +(-5mm,-3mm)
and +(+5mm,-3mm) .. ($(n-2-1.east) + (0mm,-1mm)$);
\node[scale=0.75] at ($(n-2-1.east) + (12.5mm, -6.1mm)$) {$\beta_{n}$};

\path[dotted] ($(n-1-2.south)$) edge ($(n-3-2.north)$);

 \draw[->] ($(n-2-1.west) +
    (.4mm,1.8mm)$) arc (45:315:2.5mm);
    \node[scale=0.75] at ($(n-2-1.west) + (-6.5mm,0mm)$) {$\alpha$};
\end{tikzpicture}
\end{equation*}
One can check that $\ck=\ck_{n}$ is a finite dimensional Iwanaga--Gorenstein $k$-algebra, e.g.~use \cite[Theorem 3.3.]{XWChenSchurFunctTria}. Consider the idempotent $e=e_{1}$. We have $\ck e \cong e\ck e$ as right $e\ck e$-modules. Moreover, the idempotent $e_{1}$ is singularly complete. Indeed, $S_{2}$ is a projective $\ck$-module. Hence we get a triangle equivalence $\cd_{sg}(\ck) \cong \cd_{sg}(k[x]/(x^2))$, by \eqref{E:XiaoWu}. Using Buchweitz' Theorem \ref{T:Buchweitz}, this yields a triangle equivalence $\ul{\GP}(\ck) \cong \ul{\mod}-k[x]/(x^2)$. In particular, $\ck$ has precisely three indecomposable Gorenstein projective modules. Namely, two projectives and one non-projective corresponding to $k \in \cd_{sg}(k[x]/(x^2))$.\end{ex}


The special case $e=f$ in Proposition \ref{P:TwoIdempotents} yields the following corollary.

\begin{cor}\label{P:OneIdempotent}
Let $A$ be a right Noetherian ring and let $e\in A$ be an idempotent. The exact functor $\F=\Hom_{A}(eA, -)$ induces a triangle equivalence 
\begin{align}\label{E:twoidempotents}
\frac{\cd^b(\mod-A)/\thick(eA)}{\thick_{A}(\mod-A/AeA)} \longrightarrow \frac{\cd^b(\mod-eAe)}{\thick(eAe)}=\cd_{sg}(eAe).
\end{align}
\end{cor}
\begin{proof}
Since $\thick(eA) \subseteq ^\perp\!\!\thick(\mod-A/AeA)$ holds,  $q\bigl(\thick_{A}(\mod-A/AeA)\bigr)$ is equivalent to $\thick_{A}(\mod-A/AeA)$, by Lemma \ref{L:Verdier2}. Hence we can omit $q$ in \eqref{E:twoidempotents2}. This completes the proof.
\end{proof}

\subsection{Alternative approach to the `stable' Morita type Theorem}\label{ss:Alternative}
Let $\ce$ be a Frobenius category as in Corollary \ref{C:Buchweitz}. Let $M \in \ce$  and let $A=\End_{\ce}(P \oplus M)$. Let $e \in A$ be the idempotent corresponding to the identity endomorphism $1_{P}$ of $P$. 

\begin{thm} \label{t:alternative-main}
 If  $A$ is a right Noetherian ring such that $\prdim_{A}(N)< \infty$ holds for every $N \in \mod-A/AeA$, then there is a commutative diagram
 \begin{align*}
\begin{xy}\SelectTips{cm}{}
\xymatrix{
\frac{\left( \displaystyle\cd^b(\mod-A)/ \thick(eA) \right)}{\displaystyle \thick(\mod-A/AeA)  } \ar[rr]^(.58){\displaystyle \ul{\G}}_(.58){\sim} && \frac{ \displaystyle \cd^b(\mod-eAe) }{\displaystyle K^b(\proj- eAe) } && \,  \ul{\ce} \ar@{_{(}->}[ll]_(.39){\displaystyle \ul{\widetilde{\P}}} \\
\frac{\left( \displaystyle K^b(\proj-A)/ \thick(eA) \right)}{\displaystyle  \thick(\mod-A/AeA) } \ar[rr]^(.58){\displaystyle \ul{\G}^{\rm restr.}}_(.52){\sim} \ar@{^{(}->}[u]^{\displaystyle \I_{1}}&& \frac{ \displaystyle \thick(Ae) }{ \displaystyle K^b(\proj- eAe) } \ar@{^{(}->}[u]^{\displaystyle \I_{2}} && \thick_{\underline{\ce}}(M) \ar[ll]_(.39){\displaystyle \ul{\widetilde{\P}}^{\rm restr.}}^(.47){\begin{smallmatrix} \sim \\ \text{(up to summands)} \end{smallmatrix} } \ar@{^{(}->}[u]^{\displaystyle \I_{3}}
}
\end{xy}
\end{align*}
Moreover, if $A$ has finite global dimension, then the inclusions are equivalences. In particular, this has the following consequences 
\begin{itemize}
\item[(a)] $\ul{\widetilde{\P}}\colon \ul{\ce} \longrightarrow \cd_{sg}(eAe)$ is a triangle \emph{equivalence} (up to summands). If $\underline{\ce}$ is idempotent complete, then the supplement `up to summands' can be omitted. In particular, $\cd_{sg}(eAe)$ is idempotent complete in this case.
\item[(b)]  $\ul{\ce}=\thick_{\ul{\ce}}(M)$. 
\end{itemize}

\end{thm}
\begin{proof}
Let us explain the  commutative diagram above. The vertical arrows denote the canonical inclusion functors. The functors in the lower row being the restrictions of the functors in the upper row, both squares commute by definition. 
$\ul{\G}$ is the equivalence from Corollary \ref{P:OneIdempotent} and its restriction $\ul{\G}^{\rm restr.}$ is an equivalence since $\G$ maps the generator $A$ to $Ae$. $\ul{\widetilde{\P}}$ denotes the fully faithful functor from Corollary \ref{C:Buchweitz}. Since $\ul{\widetilde{\P}}$ maps $M$ to $\Hom_{\ce}(P, M)$, which is isomorphic to $Ae$ as right $eAe$-modules, up to taking direct summands the restriction $\ul{\widetilde{\P}}^{\rm restr.}$ is a triangle equivalence (note that $\thick_{\ul{\ce}}(M)$ is not necessarily idempotent complete since $\ul{\ce}$ is not assumed to be idempotent complete. But $\thick(eA)$ is always idempotent complete!)

If $A$ has finite global dimension, then the inclusion $K^b(\proj-A) \ra \cd^b(\mod-A)$ is an equivalence. Hence, $\I_{1}$ is an equivalence. 
Thus by commutativity of the left square $\I_{2}$ is an equivalence.  Now, $\I_{3}$ and $\ul{\widetilde{\P}}$ are fully faithful and the right square commutes. Hence up to taking direct summands these two functors are equivalences. In particular, $\I_{3}$ is an equivalence since we have already taken all possible direct summands of direct sums of $M$ in $\ul{\ce}$. This completes the proof. 
\end{proof}

\subsection{Frobenius pairs and Schlichting's negative K-theory}\label{ss:Schlichting}

\begin{defn}\label{D:ExSeqTria}
We call a sequence $\kS \xrightarrow{F} \kT \xrightarrow{G} \kR$ of triangulated categories \emph{exact} if the following properties hold. 

$F$ is fully faithful and the composition $GF$ is zero. Moreover, the induced functor $\kT/\kS \rightarrow \kR$ is fully faithful and every object of $\kR$ occurs as a direct summand of an object in $\kT/\kS$.  
\end{defn}

\begin{ex}
If $\kU \subseteq \kT$ is a full triangulated subcategory then the canonical sequences
$\kU \rightarrow \kT \rightarrow \kT/\kU$ and $\kU \rightarrow \kT \rightarrow (\kT/\kU)^{\omega}$ are exact.
\end{ex}

We need the following statement from \cite{Schlichting06}.

\begin{prop} \label{P:Schlichting}
If $\kS \xrightarrow{F} \kT \xrightarrow{G} \kR$ is exact then the induced sequence of the idempotent completions
$\kS^{\omega} \xrightarrow{F^{\omega}} \kT^{\omega} \xrightarrow{G^{\omega}} \kR^{\omega}$ is exact.
\end{prop}

\begin{defn}
Let $\kU \subseteq \kT$ be a full triangulated subcategory. $\kU$ is called \emph{dense} if every object in $\kT$ is a direct summand of an object in $\kU$ and $\kU$ is closed under isomorphisms. 
\end{defn}

\begin{ex}
$\kT \subseteq \kT^{\omega}$ is the most important example of a dense subcategory in our situation. 
\end{ex}

Thomason \cite[Theorem 2.1]{Thomason} gave the following description of the set of dense subcategories. 

\begin{thm}\label{T:Thomason}
Let $\kT$ be a triangulated category. Then there is a bijection of sets
\[K_{0}(-)\colon \{ \text{Dense subcategories of  } \kT  \} \rightarrow \{\text{Subgroups of  } K_{0}(\kT) \},\]
sending a dense subcategory $\kU$ to its Grothendieck group $K_{0}(\kU)$.

In particular, the inclusion of a dense subcategory $\kU \xrightarrow{i} \kT$ yields an inclusion $K_{0}(\kU) \xrightarrow{K_{0}(i)} K_{0}(\kT)$.
\end{thm}

This has the following consequences, which will be important in the sequel.

\begin{cor}\label{C:Thomason}
Let $\kT$ be a triangulated category. If $K_{0}(\kT)=K_{0}(\kT^{\omega})$, then $\kT=\kT^{\omega}$. In other words, $\kT$ is idempotent complete.
\end{cor}

\begin{cor}\label{C:ExactGro}
If $\kS \xrightarrow{F} \kT \xrightarrow{G} \kR$ is an exact sequence of triangulated categories, then  $K_{0}(\kS) \xrightarrow{K_{0}(F)} K_{0}(\kT) \xrightarrow{K_{0}(G)} K_{0}(\kR)$ is an exact sequence of abelian groups.
\end{cor}
\begin{proof}
By definition of an exact sequence, we have a commutative diagram
\begin{align}
\begin{array}{cc}
\begin{xy}
\SelectTips{cm}{}
\xymatrix{
\kS \ar[rr]^{F} && \kT \ar[rr]^{G} \ar[rd]^{\pi} && \kR \\
&&& \kT/\kS \ar[ru]^{\iota}
}
\end{xy}
\end{array}
\end{align}

such that $\iota$ is fully faithful and $\kT/\kS$ is dense in $\kR$. It is well-known  that the sequence of abelian groups
$K_{0}(\kS) \xrightarrow{K_{0}(F)} K_{0}(\kT) \xrightarrow{K_{0}(\pi)} K_{0}(\kT/\kS) \rightarrow 0$ is exact (see for example \cite[Proposition VIII 3.1]{SGA5}). Theorem \ref{T:Thomason} shows  that $K_{0}(\iota)$ is injective, which proves the claim. 
\end{proof}

Given an exact sequence $\kS \xrightarrow{F} \kT \xrightarrow{G} \kR$ of triangulated categories, one may wonder how to extend the induced exact sequence (Proposition \ref{P:Schlichting} and Corollary \ref{C:ExactGro})
\begin{align}\label{S:ExGro}
K_{0}(\kS^{\omega}) \xrightarrow{K_{0}(F^{\omega})} K_{0}(\kT^{\omega}) \xrightarrow{K_{0}(G^{\omega})} K_{0}(\kR^{\omega})
\end{align}
to the right. This was the motivation for Schlichting to introduce negative K-theory for triangulated categories in \cite{Schlichting06}. We will use this theory to show that certain Verdier quotient categories are idempotent complete, see Proposition \ref{P:IdempCompl}.

To be more precise, Schlichting does \emph{not} define K-groups for triangulated categories directly. Given a triangulated category $\kT$ one has to construct a model $\kM$ to which one can associate a sequence of groups $\{\mathbb{K}_{-n}(\kM)\}_{n \geq 0}.$ One possible class of models considered by Schlichting are Frobenius pairs.
\begin{defn}\label{D:FrobPairs}
A \emph{Frobenius pair} is a pair of Frobenius categories $\bm{\kA}=(\kA, \kA_{0})$ such that $\kA_{0} \subseteq \kA$ is a full subcategory and moreover $\mathsf{proj} \kA_{0} \subseteq \mathsf{proj} \kA$. One can check that the stable category $\underline{\kA}_{0}$ is a full subcategory of the stable category $\underline{\kA}$. In particular, we may define the \emph{derived category} $\cd(\bm{\kA})$ of a Frobenius pair as the Verdier quotient category $\underline{\kA}/\underline{\kA}_{0}$.

A \emph{map of Frobenius pairs} $F\colon (\kA, \kA_{0}) \rightarrow (\kB, \kB_{0})$ is a map of Frobenius categories $F\colon \kA \rightarrow \kB$ such that $F(\kA_{0}) \subseteq \kB_{0}$ holds. We get an induced triangulated functor on the derived categories $\underline{F}\colon \cd(\bm{\kA}) \rightarrow \cd(\bm{\kB})$.

We say that two Frobenius pairs $(\kA, \kA_{0})$, $(\kB, \kB_{0})$ are \emph{equivalent} if there is a map $F\colon (\kA, \kA_{0}) \rightarrow (\kB, \kB_{0})$ of Frobenius pairs such that the induced functor $\underline{F}$ is an equivalence.
\end{defn}

\begin{ex}\label{E:FrobPairs}
\rm{(a)} Let $\bm{\kA}=(\kA, \kA_{0})$ be a Frobenius pair and $\kU \subseteq \cd(\bm{\kA})$ be a full triangulated subcategory of its derived category. Following Schlichting \cite[Section 6.2.]{Schlichting06}, we can construct Frobenius pairs for $\kU$ and $\cd(\bm{\kA})/\kU$ as follows.

Let $\kB \subseteq \kA$ be the full subcategory of objects whose images are contained in $ \cu \subseteq \cd(\bm{\kA})$. Then $\kB$ is closed under extensions since $\kU$ is triangulated. Therefore, $\kB$ inherits an exact structure from $\kA$. Since the projective-injective objects of $\kA$ vanish in $\cd(\bm{\kA})$ they are contained in $\kB$. By definition of the exact structure on $\kB$, these objects are still projective-injective. Since $\kU$ is triangulated $\kB$ is closed under kernels of deflations and cokernels of inflations of $\kA$. Thus using the deflations and inflations from $\kA$ we see that $\kB$ has enough projective and enough injective objects. In particular, every injective (respectively projective) object in $\kB$ is a direct summand of some projective-injective object of $\kA$. Hence, $\mathsf{proj}(\kB)=\mathsf{proj}(\kA)$ and $\kB$ is a Frobenius category. We get Frobenius pairs $\bm{\kB}=(\kB, \kA_{0})$ and $\bm{\kC}=(\kA, \kB)$. The inclusions \begin{align}\label{S:Standard}\bm{\kB} \rightarrow \bm{\kA}  \rightarrow \bm{\kC}\end{align} yield an exact sequence of triangulated categories \begin{align}\kU \rightarrow \cd(\bm{\kA}) \rightarrow \cd(\bm{\kA})/\kU. \end{align}
\rm{(b)} Let $\kE$ be an exact category. Then the category of bounded chain complexes $\Com^b(\kE)$ over $\kE$ admits the structure of a Frobenius category as in Example \ref{E:Frobenius}. Its stable category is the bounded homotopy category $K^b(\kE)$. Let $\Com^b_{\rm ac}(\kE) \subseteq \Com^b(\kE)$ be the full subcategory of complexes which are isomorphic to acyclic complexes in $K^b(\kE)$, see Definition \ref{D:AcyclicFrob}.

One can show that $\bm{\mathsf{Com}^b(\kE)}=(\Com^b(\kE), \Com^b_{\rm ac}(\kE))$ is a Frobenius pair. Its derived category is called the \emph{derived category} of $\kE$, see e.g. \cite[Section 11]{Keller}.

If $\kA=\kE$ is abelian this construction yields the usual bounded derived category $\cd^b(\kA)$. If every conflation in $\kE$ splits (e.g.~if $\kE=\mathsf{proj}-R$ for some ring $R$), then the acyclic complexes become isomorphic to zero in $K^b(\kE)$. In particular, the derived category of $\kE$ equals $K^b(\kE)$ in this case.
\end{ex}

\begin{defn}
We call a sequence $\bm{\kA} \rightarrow \bm{\kB} \rightarrow \bm{\kC}$ of Frobenius pairs \emph{exact} if the induced sequence of derived categories $\cd(\bm{\kA}) \xrightarrow{i} \cd(\bm{\kB}) \xrightarrow{p} \cd(\bm{\kC})$ is exact in the sense of Definition \ref{D:ExSeqTria}.

In particular, the sequence (\ref{S:Standard}) constructed in Example \ref{E:FrobPairs} is exact.
\end{defn}

We end this subsection by stating two results of Schlichting \cite{Schlichting06}, which will be essential in the next subsection.

\begin{thm} \label{T:Schlichting}
Let $\bm{\kA} \rightarrow \bm{\kB} \rightarrow \bm{\kC}$ be an exact sequence of Frobenius pairs. Then there is a long exact sequence of abelian groups
\begin{align}
\begin{array}{cc}
K_{0}(\cd(\bm{\kA})^{\omega}) &\longrightarrow K_{0}(\cd(\bm{\kB})^{\omega}) \longrightarrow K_{0}(\cd(\bm{\kC})^{\omega}) 
\longrightarrow \mathbb{K}_{-1}(\bm{\kA}) \longrightarrow \cdots \\ \\
&\cdots  \longrightarrow \mathbb{K}_{-i}(\bm{\kA}) \longrightarrow \mathbb{K}_{-i}(\bm{\kB}) \longrightarrow \mathbb{K}_{-i}(\bm{\kC}) \longrightarrow \cdots
\end{array}
\end{align}
extending the sequence (\ref{S:ExGro}) to the right.

In particular, in the situation of Example \ref{E:FrobPairs}: if $\kU$ and $\cd(\bm{\kA})$ are idempotent complete, then the long exact sequence starts as follows
\begin{align}
\begin{array}{cc}
\begin{xy}
\SelectTips{cm}{}
\xymatrix@C=4pt{
K_{0}(\kU) \ar[rr] && K_{0}(\cd(\bm{\kA})) \ar[rr] \ar@{->>}[rd] && K_{0}((\cd(\bm{\kA})/\kU)^{\omega}) \ar[rr] && \cdots \\
&&& K_{0}(\cd(\bm{\kA})/\kU) \ar@{>->}[ru]
}
\end{xy}
\end{array}
\end{align}
\end{thm}

\begin{rem}
The theorem shows that equivalent Frobenius pairs (Definition \ref{D:FrobPairs}) have isomorphic $\mathbb{K}$-groups.
\end{rem}

\begin{thm} \label{T:Schlichting2}
Let $\kA$ be an abelian category. Then \[\mathbb{K}_{-1}(\bm{\mathsf{Com}^b(\kA)})=0,\] where $\bm{\mathsf{Com}^b(\kA)}$ is the Frobenius model of $\cd^b(\kA)$ introduced in Example \ref{E:FrobPairs} (b).
\end{thm}

\subsection{Idempotent completeness of quotient categories} \label{ss:IdempCompl}
Let $\ce$ be an idempotent complete Frobenius category, such that  $\proj \ce$ admits an additive generator $P$. Let $F' \in \ce$. We assume that $A=\End_{\ce}(P \oplus F' )$ is a right Noetherian ring and let $e=\id_{P} \in A$. 

We give two criteria, which imply idempotent completeness of the Verdier quotient $K^b(\proj-A)/\thick(eA)$. The proof uses Schlichting's negative $\mathbb{K}$-theory, as introduced in the previous Subsection \ref{ss:Schlichting}. We need some preparation. Let $S$ be a commutative local complete Noetherian ring. The next lemma is well-known.

\begin{lem}\label{L:stable-complete}
Let $R$ be an $S$-algebra which is finitely generated as an $S$-module. Then the stable category $\ul{\MCM}(R)$ is idempotent complete.
\end{lem}
\begin{proof}
Let $M \in \MCM(R)$ without projective direct summands and let $e \in \ul{\End}_{R}(M)$ be an idempotent endomorphism. We claim that there exists an idempotent $\epsilon \in \End_{R}(M)$, which is mapped to $e$ under the canonical projection $\End_{R}(M) \ra \ul{\End}_{R}(M)$. In other words, we need some lifting property for idempotents. This is known to hold for any $S$-algebra $B$, which is finitely generated as an $S$-module and any two-sided ideal $I \subseteq \rad B$ \cite[Proposition 6.5 and Theorem 6.7]{CurtisReiner}. In particular, to prove our claim it thus suffices to show $\cp(M) \subseteq \rad \End_{R}(M)$, where $\cp(M)$ is the two-sided ideal of endomorphisms factoring through a projective $R$-module. Assume that there exists $f \in \cp(M) \setminus \rad \End_{R}(M)$. This means that there is a $g \in \End_{R}(M)$ such that $\id_{M} - gf$ is not invertible. Hence, $\id_{M} - gf$ is not surjective by \cite[Proposition 5.8]{CurtisReiner}. Let $M=\bigoplus_{i=1}^t M_{i}$ be a decomposition of $M$ into indecomposable modules and denote by $\iota_{i}$ and $\pi_{j}$ the canonical inclusions and projections, respectively. Without loss of generality, we can assume that there exists some $i$ such that $\pi_{i} (\id_{M} - gf)\colon M \ra M_{i}$ is not surjective. Hence, $\pi_{i} (\id_{M} - gf)\iota_{i}=\id_{M_{i}}-\pi_{i}gf\iota_{i}=\id_{M_{i}}-(g_{1i}f_{i1}+ \cdots + g_{ii}f_{ii}+ \cdots + g_{ti}f_{ti})$ is not surjective. Here, $f_{ij}=\pi_{j}f\iota_{i}$ and $g_{ij}=\pi_{j}g\iota_{i}$. Since $f \in \cp(M)$ we have $g_{ji}f_{ij} \in \cp(M_{i})$ for $j=1, \cdots, t.$ Since $M_{i}$ is indecomposable $\End_{R}(M_{i})$ is local \cite[Proposition 6.10]{CurtisReiner}. It follows that $\cp(M_{i}) \subseteq \rad \End_{R}(M_{i})$, for otherwise $M_{i}$ is a direct summand of a projective module, which contradicts our assumptions on $M$. Hence, $\pi_{i}gf\iota_{i}=\sum_{j=1}^t g_{ji}f_{ij} \in \rad \End_{R}(M_{i})$ and therefore $\id_{M_{i}}-\pi_{i}gf\iota_{i}$ is an isomorphism. Contradiction. Thus $f \in \rad \End_{R}(M)$.

To show that $\ul{\MCM}(R)$ is idempotent complete, let $M \in \ul{\MCM}(R)$. We can assume that $M$ has no projective direct summands. Let $e \in \ul{\End}_{R}(M)$ be an idempotent endomorphism. By the considerations above, $e$ lifts to an idempotent $\epsilon \in \End_{R}(M)$. $\MCM(R)$ is idempotent complete since it is closed under direct summands in the abelian and hence idempotent complete category $\mod-R$. Thus we have a direct sum decomposition $M \cong N_{1} \oplus N_{2}$ such that $\epsilon=\iota_{1}\pi_{1}$, where $\iota_{1}$ and $\pi_{1}$ denote the canonical inclusion and projection of $N_{1}$, respectively. Passing to the stable category yields the desired factorization of $e$.  \end{proof}

\begin{prop}\label{P:IdempCompl}
In the notations above assume that one of the following holds
\begin{itemize}
\item[(i)] $\gldim(A)<\infty$ and $\ul{\ce}$ is idempotent complete.
\item[(ii)] $eAe$ is an Iwanaga--Gorenstein $S$-algebra, which is finitely generated as an $S$-module.
\end{itemize}
Then the triangulated quotient category $\cd^b(\mod-A)/\thick(eA)$ is idempotent complete. In particular, this holds for $K^b(\proj-A)/\thick(eA)$.\end{prop}

\begin{proof}
We claim that it is sufficient to show that $\cd^b(\mod-A)/\thick(eA)$ is idempotent complete. Indeed, $K^b(\proj-A) \subseteq \cd^b(\mod-A)$ is closed under direct summands. Hence the same holds for $K^b(\proj-A)/\thick(eA) \subseteq \cd^b(\mod-A)/\thick(eA)$, since the additive structure of the quotient is inherited from the additive structure of the `nominator'.

In both cases (i) and (ii), the triangulated quotient category $\cd_{sg}(eAe)=\cd^b(\mod-eAe)/K^b(\proj-eAe)$ is idempotent complete. In the first case, this follows from Theorem \ref{t:main-thm} or Theorem \ref{t:alternative-main} and the idempotent completeness of $\ul{\ce}$. In the second case, this follows from Buchweitz' Theorem \ref{T:Buchweitz} Êand Lemma \ref{L:stable-complete}.

We  deduce the idempotent completeness of $\cd^b(\mod-A)/\thick(eA)$ from the idempotent completeness of $\cd^b(\mod-eAe)/K^b(\proj-eAe)$.

Let $\bm{\mathsf{Com}^b(eAe)}=(\Com^b(\mod-eAe), \Com_{\mathsf{ac}}^b(\mod-eAe))$ be the Frobenius pair associated to $\cd^b(\mod-eAe)$ and $\bm{\cb}=(\cb, \Com_{\mathsf{ac}}^b(\mod-eAe))$ be the Frobenius pair associated to the full triangulated subcategory $\thick(eAe)\subseteq \cd^b(\mod-eAe)$, see Example \ref{E:FrobPairs} (a).

Another possibility to realize $\thick(eAe) \cong K^b(\proj-eAe)$ as the derived category of a Frobenius pair is $\bm{\mathsf{Com}^b(\mathsf{add} \, \, eAe)}=(\Com^b(\add eAe), \Com_{\mathsf{ac}}^b(\add eAe))$.

 Let us show that these two realizations are equivalent. The inclusion of Frobenius pairs $\bm{\mathsf{Com}^b(\mathsf{add} \, \,  eAe)} \rightarrow \bm{\cb}$ induces a triangle functor between the corresponding derived categories $K^b(\add eAe) \xrightarrow{F} \thick(eAe)$. It follows from the definition of $\thick(eAe)$ that $F$ is dense. 
 
We note that $K^b(\add eAe) \rightarrow \ul{\cb}$ is fully faithful, since $\ul{\cb} \subseteq K^b(\mod-eAe)$.
 Using  $K^b(\add eAe) \subseteq ^{\perp}\!\!(K^b_{\mathsf{ac}}(\mod-eAe))$ and Lemma \ref{L:Verdier2}, we obtain that $F$ is fully faithful.

 Let us look at some part of the long exact sequence associated to the exact sequence of Frobenius pairs $\bm{\cb} \rightarrow \bm{\mathsf{Com}^b(eAe)} \rightarrow (\Com^b(\mod-eAe), \cb)$.
\[
\begin{xy}\SelectTips{cm}{}
\xymatrix{
K_0\bigl(\cd^b(\mod-eAe)\bigr) \ar@{->>}[r]  \ar@{->>}[d] & K_0\bigl(\cd_{sg}(eAe)^{\omega})  \ar[r]^(0.6)0 & \mathbb{K}_{-1}(\bm{\cb}) \ar[r] &
\mathbb{K}_{-1}(\bm{\mathsf{Com}^b(eAe)}) \\
 K_0\bigl(\cd_{sg}(eAe)\bigr) \ar[ru]^{\cong}
}
\end{xy}
\]
Since $\mathbb{K}_{-1}(\bm{\mathsf{Com}^b(eAe)})$ vanishes by Schlichting's Theorem \ref{T:Schlichting2} and the sequence is exact, we obtain $\mathbb{K}_{-1}(\bm{\cb})=0$ and therefore $\mathbb{K}_{-1}(\bm{\mathsf{Com}^b(\mathsf{add} \, \,  eAe)})=0$, since these Frobenius pairs are equivalent.

The situation $\thick(eA) \subseteq \cd^b(\mod-A)$ may be treated similarly. The only difference is that we do not know yet whether the quotient $\cd^b(\mod-A)/\thick(eA)$ is idempotent complete. Let $\bm{\cc}=(\cc, \Com^b_{\mathsf{ac}}(\mod-A)) \rightarrow \bm{\mathsf{Com}^b(A)} \rightarrow (\Com^b(\mod-A), \cc)$ be the corresponding exact sequence of Frobenius pairs. The associated long exact sequence is given as follows

\begin{align} \label{S:Second}
\begin{array}{cc}
\begin{xy}\SelectTips{cm}{}
\xymatrix@C=0.8pc{
K_0\bigl(\cd^b(\mod-A)\bigr) \ar[r]  \ar@{->>}[d] & K_0\!\left(\left(\frac{ \cd^b(\mod-A)}{ \thick(eA)}\right)^{\!\! \omega \, }\right)  \ar[r] & \mathbb{K}_{-1}(\bm{\cc}) \ar[r] &
\mathbb{K}_{-1}(\bm{\mathsf{Com}^b(A)}) \\
 K_0\!\left(\frac{ \cd^b(\mod-A)}{ \thick(eA)}\right) \ar@{>->}[ru]^{\displaystyle \iota}
}
\end{xy}
\end{array}
\end{align}
As above, one can show that the Frobenius pair $\bm{\cc}$ is equivalent to the Frobenius pair $\bigl(\Com^b(\add eA ), \Com^b_{\mathsf{ac}}(\add eA)\bigr)$. Moreover, there is an equivalence of Frobenius pairs:
\[
-\, \otimes_{eAe} eA \colon \bigl(\Com^b(\add eAe), \Com^b_{\mathsf{ac}}(\add eAe)\bigr)
\longrightarrow
\bigl(\Com^b(\add eA), \Com^b_{\mathsf{ac}}(\add eA)\bigr)
\]
In particular, we get $\mathbb{K}_{-1}(\bm{\cc})=\mathbb{K}_{-1}(\bm{\mathsf{Com}^b(\mathsf{add} \, \,  eAe)})=0$, which implies idempotent completeness of $\cd^b(\mod-A)/\thick(eA)$ as explained in \cite[Remark 1]{Schlichting06}.

For the convenience of the reader, we briefly recall the argument: $\iota$
is an isomorphism by the exactness of the sequence (\ref{S:Second}). Hence Corollary \ref{C:Thomason} shows that
the canonical inclusion $\cd^b(\mod-A)/\thick(eA) \rightarrow \left(\cd^b(\mod-A)/\thick(eA)\right)^\omega$ is an equivalence
of triangulated categories, i.e.~the triangulated category $\cd^b(\mod-A)/\thick(eA)$ is idempotent complete.
\end{proof}

\newpage

\section{DG algebras and their derived categories}\label{s:dg}
 
This section is based on joint work with Dong Yang \cite{KalckYang12}. Except for the Hom-finiteness result (Proposition \ref{p:hom-finiteness-of-per}), which generalizes work of Amiot \cite{Amiot09}
and Guo \cite{Guolingyan11a} and the parts of Subsection \ref{ss:Recollements} on recollements generated by idempotents (in particular Corollary \ref{c:restriction-and-induction}), most of these results are known to the experts.

Subsection \ref{ss:dg-alg} collects well-known notions about dg algebras and their derived categories. Our main reference is Keller's article \cite{Keller94}. In Subsection \ref{ss:nakayama-functor}, we study the Nakayama--functor on the derived category of a dg algebra in some detail. In particular, we prove a form of Serre duality in this setup, which is used in Subsection \ref{ss:fract-cy}. We recall the notions of $t$-strutures and co-$t$-structures in Subsection \ref{ss:nonpositive-dg-alg-1}. The latter are used to prove that the perfect derived category $\per(B)$ of a non-positive dg $k$-algebra is Hom-finite, provided its $0$-th cohomology is finite dimensional and $\per(B)$ contains the derived category of dg $B$-modules with finite dimensional total cohomology $\cd_{fd}(B)$ (see Proposition \ref{p:hom-finiteness-of-per}). Consequently, the relative singularity categories of Gorenstein isolated singularities are Hom-finite (see Proposition \ref{p:hom-finiteness-of-delta}). This was independently proved by Thanhoffer de V\"olcsey \& Van den Bergh \cite{ThanhofferdeVolcseyMichelVandenBergh10} in a quite different way. Moreover, Hom-finiteness makes Koszul duality work smoothly (Corollary \ref{c:koszul-double-dual}), which we need in the proof of Theorem \ref{t:main-thm-2}. The results on $t$-structures in combination with recollements generated by idempotents are applied to obtain dg descriptions of relative singularity categories in Corollary \ref{c:restriction-and-induction}. In Subsection \ref{ss:minimal-relation}, we recall the notion of a set of minimal relations for complete quiver algebras. This is necessary to define the dg Auslander algebra in Subsection \ref{ss:Independance}. Subsection \ref{ss:dual-bar-construction} collects definitions and results on Koszul duality, which are needed to prove Theorem \ref{t:main-thm-2}. Although we are mainly interested in dg algebras, it turns out to be necessary to introduce Koszul duality in the more general framework of $A_{\infty}$-algebras. We mainly follow Lu, Palmieri, Wu \& Zhang  \cite{LuPalmieriWuZhang04,LuPalmieriWuZhang08}  Êand Lef{\`e}vre-Hasegawa \cite{Lefevre03}. In Subsection \ref{ss:Recollements}, we first recall some general statements on recollements of triangulated categories \cite{BeilinsonBernsteinDeligne82}. In particular, we briefly discuss the equivalent notion of triangulated TTF triples appearing in the work of Nicol{\'a}s \cite{Nicolas} and Nicol{\'a}s \& Saorin \cite{NicolasSaorin09}. Moreover, the connection to the theory of Bousfield (co-)localisation functors is explained, see e.g.~Neeman \cite{Neeman99}. These techniques are applied to construct a recollement for every pair $(A, e)$, where $A$ is a $k$-algebra and $e \in A$ is an idempotent. More precisely, the recollement involves the unbounded derived categories of $eAe$, $A$ and of a certain dg algebra $B$, which exists by Nicol{\'a}s \& Saorin's work. Using Neeman's generalization \cite{Neeman92a} of the localization theorems of Thomason \& Trobaugh and Yao, passing to the subcategories of compact objects yields a description of the relative singularity category $\Delta_{eAe}(A)$ in terms of the perfect derived category of $B$, see Corollary \ref{c:restriction-and-induction}. For an ADE-singularity $R$, we show that this dg algebra may be constructed from the stable category of maximal Cohen--Macaulay $R$-modules, see Subsections \ref{ss:Independance} -- \ref{ss:DGAuslander}.

\subsection{Notations}\label{ss:notation}
Let $k$ be an algebraically closed field. Let $D=\Hom_k( - , k)$ denote the
$k$-dual. When the input is a graded $k$-module, $D$ means the
graded dual. Namely, for $M=\bigoplus_{i\in\mathbb{Z}}M^i$, the graded dual $DM$
has components $(DM)^i=\Hom_k(M^{-i},k)$.
\newcommand{\Add}{\opname{Add}}
\subsubsection*{Generating subcategories/subsets}

Let $\ca$ be an additive $k$-category. Let $\cs$ be a subcategory or a subset of objects of $\ca$. We denote by $\add_{\ca}(\cs)$
(respectively, $\Add_{\ca}(\cs)$) the smallest full subcategory of
$\ca$ which contains $\cs$ and which is closed under taking finite
direct sums (respectively, all existing direct sums) and taking
direct summands.

If $\ca$ is a triangulated category, then 
$\thick_{\ca}(\cs)$ (respectively, $\Tria_{\ca}(\cs)$) denotes the smallest
triangulated subcategory of $\ca$ which contains $\cs$ and which is
closed under taking direct summands (respectively, all existing
direct sums).

When it does not cause confusion, we omit the subscripts and write
the above notations as $\add(\cs)$, $\Add(\cs)$, $\thick(\cs)$ and
$\Tria(\cs)$.

\subsubsection*{Derived categories of abelian categories}

Let $\ca$ be an additive $k$-category. 
Let $*\in\{\emptyset,-,+,b\}$ be a boundedness condition. Denote by $K^*(\ca)$ the homotopy
category of complexes of objects in $\ca$ satisfying the boundedness
condition $*$.

Let $\ca$ be an abelian $k$-category. 
Denote by $\cd^*(\ca)$ the derived category of complexes of objects
in $\ca$ satisfying the boundedness condition $*$.

Let $R$ be a $k$-algebra. Without further remark, by an $R$-module
we mean a right $R$-module. Denote by $\Mod-R$ the category of
$R$-modules, and denote by $\mod-R$ (respectively, $\proj-R$) its
full subcategory of finitely generated $R$-modules (respectively,
finitely generated projective $R$-modules). We will also
consider the category $\fdmod-R$ of those $R$-modules which are
finite-dimensional over $k$. We often view
$K^b(\proj-R)$ as a triangulated subcategory of $\cd^*(\Mod-R)$.

\subsection{Definitions}\label{ss:dg-alg}

A \emph{dg $k$-algebra} $A$ is a $\mathbb{Z}$-graded $k$-algebra with a $k$-linear differential $d=d_{A}$ of degree $1$ satisfying the graded Leibniz rule:
\begin{align}
d(ab)=d(a)b + (-1)^{\deg(a)}ad(b),
\end{align}
where $a, b \in A$ are homogeneous elements. In particular, every associative $k$-algebra may be viewed as a dg $k$-algebra concentrated in degree $0$ with trivial differential $d=0$. A \emph{right dg $A$-module} $M$ is a $\ZZ$-graded right $A$-module with a  $k$-linear differential $d_{M}$ of degree $1$, satisfying
\begin{align}
d_{M}(ma)=d_{M}(m)a + (-1)^{\deg(m)}md_{A}(a),
\end{align}
where $m \in M$ is a homogeneous element of degree $\deg(m)$.  Every dg $A$-module may be viewed as a complex of $k$-modules. In particular, there are associated cohomology groups $H^i(M)$. If $A$ is concentrated in degree zero, then a dg $A$-module is just a complex of $A$-modules. A \emph{morphism} $f\colon M \ra N$ between dg $A$-modules is an $A$-linear map of degree zero, such that
\begin{align}
f\circ d_{M}(m)=d_{N}\circ f(m),
\end{align}  
for all $m \in M$. In other words, $f$ is a chain map. In particular, $f$ induces morphisms  
$
H^i(f)\colon H^i(M) \ra H^i(N)
$
between the corresponding cohomology groups. $f$ is a \emph{quasi-isomorphism} if $H^i(f)$ is an isomorphism, for all $i \in \ZZ$. The category of dg $A$-modules is denoted by $\cc(A)$.

Now, the \emph{derived category}
$\cd(A)$ of dg $A$-modules is obtained from $\cc(A)$ 
by formally inverting all quasi-isomorphisms. It is a
triangulated category with shift functor being the shift of
complexes $[1]$, see Keller \cite{Keller94}. If $A$ is concentrated in degree zero, then $\cd(A)$ is the usual unbounded derived category $\cd(\Mod-A)$ of right $A$-modules.

An object $X \in \cd(A)$ is called $\emph{compact}$ (or \emph{small} in Keller's terminology) if the functor $\Hom_{\cd(A)}(X, -)$ commutes with all (set-indexed) direct sums in $\cd(A)$.  Since $\Hom_{\cd(A)}(A, -)\cong \Hom_{\ck(A)}(A, -) \cong H^0(-)$ (see the next paragraph) and taking cohomologies commutes with direct sums, $A$ is compact when considered as a right dg $A$-module. The smallest triangulated subcategory $\thick(A) \subseteq \cd(A)$ which contains $A$ and is closed under taking direct summands consists of compact objects, by the five lemma. It follows from ideas of Ravenel \cite{Ravenel} that these are \emph{all} the compact objects of $\cd(A)$, see also \cite[Section 5]{Keller94}. In the sequel, we denote $\thick(A)$ by $\per(A)$ and call it the category of perfect dg $A$-modules. If $A$ is concentrated in degree $0$, then $\per(A)$ may be identified with the bounded homotopy category of finitely generated projective $A$-modules $K^b(\proj-A)$.

Moreover, one can consider the full subcategory $\cd_{fd}(A) \subseteq \cd(A)$ consisting of those dg $A$-modules $M$ whose total cohomology $\bigoplus_{i \in \ZZ} H^i(M)$
is finite-dimensional. In case $A$ is a finite dimensional algebra concentrated in degree zero, the category $\cd_{fd}(A)$ can be identified with the bounded derived category of finite dimensional $A$-modules $\cd^b(\mod-A)$.

\begin{ex}
Let $A=k[x]$ be the polynomial ring over a field $k$. Then $\cd_{fd}(A)$ is strictly contained in the perfect category $\per(A)=K^b(\proj-A)$. Conversely, if $A$ is a finite dimensional algebra, then $\per(A) \subseteq \cd_{fd}(A)$ always holds. Moreover, this inclusion is strict if and only if $A$ has infinite global dimension. 
\end{ex}

\subsubsection{Homotopy categories and derived functors}\label{sss:DerFunct}
Let $M$, $N$ be dg $A$-modules. Define the complex $\cHom_A(M,N)$ componentwise as
\begin{eqnarray*}\cHom_A^i(M,N)=\left.\left\{f\in\prod_{j\in\mathbb{Z}}\Hom_k(M^j,N^{i+j}) \, \right|
\, f(ma)=f(m)a\right\},\end{eqnarray*} with
differential given by $d(f)=d_N\circ f-(-1)^i f\circ d_M$ for
$f\in\cHom_A^i(M,N)$. The complex $\cEnd_A(M)=\cHom_A(M,M)$ with product given by the
composition of maps  is a dg $k$-algebra. Moreover, the $0$-cocycles $Z^0(\cHom_A(M,N))$ of $\cHom_A(M,N)$ are just the morphisms in the category of dg $A$-modules $\cc(A)$. On the other hand, keeping the objects and taking the $0$-th cohomology group $H^0(\cHom_A(M,N))$ as morphism space, we obtain the homotopy category $\ck(A)$ of dg $A$-modules. If $A$ is concentrated in degree $0$, then $\ck(A)$ is the homotopy category of complexes of $A$-modules $K(\Mod-A)$.

There is another way to define the homotopy category $\ck(A)$, which will be needed later. Namely, $\cc(A)$ admits an exact Frobenius category structure and $\ck(A)$ is the corresponding stable category. More precisely, define a sequence $X \ra Y \ra Z$ in $\cc(A)$ to be exact, if it is \emph{split}-exact as a sequence of graded $A$-modules. Then the projective-injective objects are precisely the null-homotopic dg $A$-modules, i.e.~direct summands of dg $A$-modules of the form $\Cone(X \xrightarrow{\id_{X}} X)$ for a dg $A$-module $X$, see Keller \cite[Section 2.2.]{Keller94}.

A dg $A$-module $N$ is called \emph{acyclic} if $H^i(N)=0$ for all $i \in \mathbb{Z}$. Acyclic dg $A$-modules form a triangulated subcategory ${\mathcal Ac}(A)$ of $\ck(A)$. Taking the triangulated quotient category $\ck(A)/{\mathcal Ac}(A)$ yields another description of the derived category $\cd(A)$. We denote the corresponding quotient functor by $Q\colon \ck(A) \ra \cd(A)$.

Objects $X \in \ck(A)$ such that $\Hom_{\ck(A)}(X, {\mathcal Ac}(A))=0$ are called \emph{$K$-projective} and the corresponding triangulated subcategory is denoted by $\ck_{p}(A)$. Since we have a natural isomorphism $\Hom_{\ck(A)}(A, -)\cong H^0(-)$ by definition of $\ck(A)$, $A$ is a $K$-projective object. Hence, $\thick_{\ck(A)}(A) \subseteq \ck_{p}(A)$. Keller \cite[Theorem 3.1.]{Keller94} shows that every object $X \in \ck(A)$ fits into a triangle
\begin{align}
pX \ra X \ra aX \ra pX[1],
\end{align} 
with $pX \in \ck_{p}(A)$ and $aX \in {\mathcal Ac}(A)$. In particular, $pX \ra X$ is a quasi-isomorphism, which is often called \emph{$K$-projective resolution} of $X$. If $A$ is concentrated in degree $0$ and $X$ is a \emph{right bounded} complex of $A$-modules, then $pX$ is given as a projective resolution of $X$. However, for unbounded $X \in K(\Mod-A)$, the cofibrant resolution is given by Spaltenstein's $K$-projective resolution, \cite{Spaltenstein}. 
Dually, the objects which are right orthogonal to ${\mathcal Ac}(A)$ are called \emph{$K$-injective} and the above statements may be dualized.

We will use the following results from~\cite{Keller94}. Let $A$ and $B$ be dg $k$-algebras.
\begin{itemize}
 \item[--] Every dg $A$-module $M$ has a natural structure of dg $\cEnd_A(M)$-$A$-bimodule.
 \item[--] If $M$ is a dg $A$-$B$-bimodule, then there is an adjoint pair of functors between Frobenius categories (in the sense of Definition \ref{D:Frobenius})
\begin{align}
\begin{xy}
\SelectTips{cm}{}
\xymatrix{\cc(A)\ar@<.7ex>[rr]^{-\otimes_A M}&&\cc(B)\ar@<.7ex>[ll]^{\cHom_B(M,-)}.}
\end{xy}
\end{align}
For a dg $A$-modules $X$ one defines $X \otimes_A M$ in several steps. Firstly, one can form the usual tensor product of complexes of $k$-modules $X \otimes_k M$. Next, the $k$-submodule $U$ generated by elements of the form $xa \otimes m - x \otimes am$ is invariant under the differential of $X \otimes_k M$ and under the right action of $B$. Thus the factor module $(X \otimes_k M)/U$ is a right dg $B$-module, which is denoted by $X \otimes_A M$, see also \cite[Subsection 2.6]{KellerPap}. As in the case of ordinary algebras, the left derived functor of $-\otimes_A M$ is defined by passing to a resolution first:
\begin{align}
-\lten_A M \colon \cd(A) \xrightarrow{p} \ck_{p}(A) \xrightarrow{-\otimes_A M} \ck(B) \xrightarrow{Q} \cd(B).
\end{align}
Similarly, the right derived functor of $\cHom_B(M,-)$ is defined by
\begin{align}
\RHom_B(M,-) \colon \cd(B) \xrightarrow{i} \ck_{i}(B) \xrightarrow{\cHom_B(M,-)} \ck(A) \xrightarrow{Q} \cd(A).
\end{align}
The derived functors form again an adjoint pair:
\[
\begin{xy}
\SelectTips{cm}{}
\xymatrix{\cd(A)\ar@<.7ex>[rr]^{-\lten_A M}&&\cd(B)\ar@<.7ex>[ll]^{\RHom_B(M,-)}.}
\end{xy}
\]

 \item[--] Let $f:A\rightarrow B$ be a quasi-isomorphism of dg algebras. Then the induced triangle functor $-\lten_A B:\cd(A)\rightarrow
\cd(B)$ is an equivalence. A quasi-inverse is given by the restriction $\cd(B)\rightarrow\cd(A)$ along $f$. It can be 
written as $-\lten_B B=\RHom_B(B,-)$ where $B$ is considered as a dg $B$-$A$-bimodue respectively dg $A$-$B$-bimodule via $f$.
These equivalences restrict to equivalences between $\per(A)$ and $\per(B)$ and between $\cd_{fd}(A)$ and $\cd_{fd}(B)$.
By abuse of language, by a quasi-isomorphism we will also mean a zigzag of quasi-isomorphisms.
\end{itemize}

\subsection{The Nakayama functor}\label{ss:nakayama-functor}

Let $A$ be a dg $k$-algebra. We
consider the functor of Frobenius categories $\nu=D\cHom_A(-,A)\colon \cc(A)\rightarrow\cc(A)$. It is clear that
$\nu(A)=D(A)$ holds. Moreover, for dg $A$-modules $M$ and $N$ there is a bifunctorial map
\begin{align}\label{E:Naka}
\begin{array}{c}
\begin{xy}
\SelectTips{cm}{}
\xymatrix@R=0.5pc{D\cHom_A(M,N)\ar[r] & \cHom_A(N,\nu(M))\\
\varphi\ar@{|->}[r]&\bigl(n\mapsto (f\mapsto \varphi(g))\bigr)}
\end{xy}
\end{array}
\end{align} where
$f\in\cHom_A(M,A)$ and $g\colon m\mapsto nf(m)$.  If we let $M=A$, then $(\ref{E:Naka})$ is an isomorphism. Taking the zeroth cohomology of \eqref{E:Naka}, yields a bifunctorial isomorphism in the homotopy category $\ck(A)$
\begin{align}\label{E:Naka2}
D\Hom_{\ck(A)}(A,N) \cong \Hom_{\ck(A)}(N,\nu(A)),
\end{align}
which may be extended to any $M \in \thick_{\ck(A)}(A)$. These objects $M$ are $K$-projective (since $A$ is) and $\nu$ maps $M$ into $\ck_{i}(A)$, the subcategory of $K$-injective objects (since this holds for $A$, see \cite[Subsection 10.4]{Keller94}). Using the orthogonality properties of $\ck_{p}(A)$ and $\ck_{i}(A)$ (see the discussion above) together with Lemma \ref{L:Verdier2}, we obtain bifunctorial isomorphisms
\begin{align}\label{E:Naka3}
\begin{array}{cc}
\Hom_{\ck(A)}(M,N) \cong \Hom_{\cd(A)}(M,N) \\ \\ \quad  \Hom_{\ck(A)}(N,\nu(M)) \cong \Hom_{\cd(A)}(N,\nu(M))
\end{array}
\end{align}
induced by the quotient functor $Q\colon \ck(A) \ra \cd(A)$, where $M \in \thick_{\ck(A)}(A)$. Putting \eqref{E:Naka2} and \eqref{E:Naka3} together, we obtain a chain of bifunctorial isomorphisms
\begin{align*}
D\Hom_{\cd(A)}(M,N) \cong D\Hom_{\ck(A)}(M,N) \cong  \Hom_{\ck(A)}(N,\nu(M)) \cong \Hom_{\cd(A)}(N,\nu(M))
\end{align*}
 The image of $\thick_{\ck(A)}(A)$ under the quotient functor $Q$ is $\per(A)$. Summing up, there is a 
 binatural isomorphism for $M\in\per(A)$ and $N\in\cd(A)$:
\begin{align}\label{E:AR-formula}
\begin{xy}
\SelectTips{cm}{}
\xymatrix{D\Hom_{\cd(A)}(M,N)\cong\Hom_{\cd(A)}(N,\nu(M))}
\end{xy}\! \! ,
\end{align}
where by an abuse of notation $\nu=\mathsf{L}\nu$ denotes the left derived functor. This is the \emph{Nakayama functor} and \eqref{E:AR-formula} is a form of Serre duality.

\newpage

\subsection[Non-positive dg algebras]{Non-positive dg algebras: $t$-structures, 
co-$t$-structures \\ and Hom-finiteness}\label{ss:nonpositive-dg-alg-1}
\subsubsection*{Truncations}
Let $\ca$ be an abelian $k$-category. For $i\in\mathbb{Z}$ and for a
complex $M$ of objects in $\ca$, we define the \emph{standard
truncations} $\sigma^{\leq i}$ and $\sigma^{>i}$ by 

\begin{align*}(\sigma^{\leq i}M)^j=\begin{cases} M^j & \text{ if
} j<i,\\ \ker d_M^i & \text{ if } j=i,\\ 0 & \text{ if }
j>i,\end{cases} &&&& (\sigma^{>
i}M)^j=\begin{cases} 0 & \text{ if } j<i,\\
{\displaystyle \frac{M^i}{\ker
d_M^i} }& \text{ if } j=i,\\ M^j & \text{ if } j>i,\end{cases}
\end{align*}
and the
\emph{brutal truncations} $\beta_{\leq i}$ and $\beta_{\geq i}$ by
\begin{align*}
(\beta_{\leq i}M)^j=\begin{cases} M^j & \text{ if } j\leq i,\\
0 & \text{ if } j>i,\end{cases} &&&~~
(\beta_{\geq i}M)^j=\begin{cases} 0 & \text{ if } j<i,\\
M^j & \text{ if } j\geq i.\end{cases}
\end{align*}
Their respective differentials are inherited from $M$.
Notice that $\sigma^{\leq i}(M)$ and $\beta_{\geq i}(M)$ are
subcomplexes of $M$ and $\sigma^{>i}(M)$ and $\beta_{\leq i-1}(M)$
are the corresponding quotient complexes. Thus we have two sequences,
which are componentwise short exact,
\begin{align*}
0 \ra \sigma^{\leq i}(M)\ra M\ra \sigma^{>i}(M)\ra 0 & \quad  \text{  and } & 0\ra \beta_{\geq i}(M)\ra  M\ra \beta_{\leq i-1}(M)\ra  0.
\end{align*}
 Moreover, taking standard truncations behaves well with respect to cohomology.
\begin{align*}
H^j(\sigma^{\leq i}M)=\begin{cases} H^j(M) & \text{ if } j\leq
i,\\ 0 & \text{ if } j>i,\end{cases} &&&&
H^j(\sigma^{>i}M)=\begin{cases} 0 & \text{ if } j\leq i,\\
H^j(M) & \text{ if } j>i.\end{cases}
\end{align*}

\subsubsection{$t$-structures}
Let $\cc$ be a triangulated $k$-category with shift functor $[1]$. A
\emph{$t$-structure} on $\cc$ (\cite{BeilinsonBernsteinDeligne82})
is a pair $(\cc^{\leq 0},\cc^{\geq 0})$ of strictly (i.e.~ closed under isomorphisms) full
subcategories such that
\begin{itemize}
\item[(T1)] $\cc^{\leq 0}[1]\subseteq\cc^{\leq 0}$ and
$\cc^{\geq 0}[-1]\subseteq\cc^{\geq 0}$,
\item[(T2)] $\Hom(M,N[-1])=0$ for $M\in\cc^{\leq 0}$
and $N\in\cc^{\geq 0}$,
\item[(T3)] for each $M\in\cc$, there is a triangle $M'\rightarrow
M\rightarrow M''\rightarrow M'[1]$ in $\cc$ with $M'\in\cc^{\leq 0}$
and $M''\in\cc^{\geq 0}[-1]$.
\end{itemize}
The \emph{heart} $\cc^{\leq 0}\cap\cc^{\geq 0}$ of the $t$-structure $(\cc^{\leq 0},\cc^{\geq 0})$ is an abelian category, see \cite{BeilinsonBernsteinDeligne82}. Examples of $t$-structures may be obtained from Proposition \ref{p:standard-t-str} below (see also Example \ref{Ex:TandCoT}).

Let $A$ be a dg $k$-algebra such that $A^i=0$ for $i>0$. Such a
dg algebra is called a \emph{non-positive dg algebra}. The canonical projection $A\rightarrow H^0(A)$ is a homomorphism of
dg algebras. We view a module over $H^0(A)$ as a dg module over $A$ via this homomorphism.
This defines a natural functor $\Phi\colon \Mod-H^0(A)\rightarrow \cd(A)$.

\begin{prop}\label{p:standard-t-str}
Let $A$ be a non-positive dg $k$-algebra.
\begin{itemize}
\item[(a)] \emph{(\cite[Theorem 1.3]{HoshinoKatoMiyachi02}, \cite[Section
2.1]{Amiot09} and \cite[Section 5.1]{KellerYang11})} Let $\cd^{\leq
0}$ respectively $\cd^{\geq 0}$ denote the full subcategory of
$\cd(A)$ which consists of objects $M$ such that $H^i(M)=0$ for
$i>0$ respectively for $i<0$. Then $(\cd^{\leq 0},\cd^{\geq 0})$ is
a $t$-structure on $\cd(A)$. Moreover,  $H^0$ defines an
equivalence from the heart to $\Mod-H^0(A)$, and the natural functor $\Phi\colon\Mod-H^0(A)\rightarrow \cd(A)$
induces a quasi-inverse to this equivalence. We will identify $\Mod-H^0(A)$ with the heart via these equivalences.
\item[(b)]  The $t$-structure in (a) restricts to a $t$-structure on
$\cd_{fd}(A)$ whose heart is $\fdmod-H^0(A)$.
Moreover, as a triangulated category $\cd_{fd}(A)$ is generated by the heart.
\end{itemize}
\end{prop}

\begin{proof} (a)  To show that $(\cd^{\leq 0},\cd^{\geq 0})$ is a $t$-structure on $\cd(A)$, it suffices to show condition (T3). Let $M$ be a dg
$A$-module. Thanks to the assumption that $A$ is non-positive, the
standard truncations $\sigma^{\leq 0}M$ and $\sigma^{>0}M$ are again
dg $A$-modules. Hence we have a distinguished triangle
\begin{eqnarray}\label{e:triangle-standard-truncation} \sigma^{\leq
0}M\longrightarrow M\longrightarrow \sigma^{>0}M\longrightarrow \sigma^{\leq
0}M[1]\end{eqnarray} in $\mathcal{D}(A)$. This proves (T3).

For the second statement, we refer to  \cite[Proposition 2.3.]{Amiot09}.

(b) For the first statement, it suffices to show that, the standard truncations are endo-functors of
$\cd_{fd}(A)$. This is true because $H^*(\sigma^{\leq 0}M)$ and
$H^*(\sigma^{>0}M)$ are $k$-subspaces of $H^*(M)$.

To show the second statement, let $M\in\cd_{fd}(M)$. Suppose that
for $m\geq n$ we have $H^{n}(M)\neq 0$, $H^{m}(M)\neq 0$ but
$H^i(M)=0$ for $i \notin [n, m]$. We prove that $M$ is generated by the
heart by induction on $m-n$. If $m-n=0$, then a shift of $M$ is in the heart.
Now suppose $m-n>0$. The standard truncations yield a triangle
\[
\sigma^{\leq n}M \longrightarrow M \longrightarrow \sigma^{>n}M \longrightarrow \sigma^{\leq n}M[1].
\]
Now the cohomologies of $\sigma^{\leq n}M$ are concentrated in degree
$n$, and hence $\sigma^{\leq n}M$ belongs to a shifted copy of the heart. 
By induction hypothesis, $\sigma^{>n}(M)$ is
generated by the heart. Therefore, $M$ is generated by the heart.
\end{proof}

\subsubsection{Co-$t$-structures}
Let $\cc$ be as above. A
\emph{co-$t$-structure} on $\cc$~\cite{Pauksztello08}  (or
\emph{weight structure}~\cite{Bondarko10}) is a pair $(\cc_{\geq
0},\cc_{\leq 0})$ of strictly full subcategories of $\cc$ satisfying the following conditions
\begin{itemize}
\item[(C1)] both $\cc_{\geq 0}$ and $\cc_{\leq
0}$ are closed under finite direct sums and direct summands,
\item[(C2)] $\cc_{\geq 0}[-1]\subseteq\cc_{\geq 0}$ and
$\cc_{\leq 0}[1]\subseteq\cc_{\leq 0}$,
\item[(C3)] $\Hom(M,N[1])=0$ for $M\in\cc_{\geq 0}$
and $N\in\cc_{\leq 0}$,
\item[(C4)] for each $M\in\cc$ there is a triangle $M'\rightarrow
M\rightarrow M''\rightarrow M'[1]$ in $\cc$ with $M'\in\cc_{\geq 0}$
and $M''\in\cc_{\leq 0}[1]$.
\end{itemize}
Examples may be obtained from Proposition \ref{p:standard-co-t-str} below (see also Example \ref{Ex:TandCoT}).

It follows from the definition that $\cc_{\leq 0}=\cc_{\geq 0}^\perp[-1]$. Indeed, by property (C3) $\cc_{\leq 0} \subseteq \cc_{\geq 0}^\perp[-1]$. Conversely, if $M$  is in  $\cc_{\geq 0}^\perp[-1]$, then we may consider the triangle in (C4) for $M[1]$. The triangle splits and thus (C1) implies that $M[1] \in \cc_{\leq 0}[1]$.
The
\emph{co-heart} is defined as the intersection $\cc_{\geq
0}~\cap~\cc_{\leq 0}$.

\begin{lem}\label{l:stupid-truncation}
 \emph{(\cite[Proposition 1.3.3.6]{Bondarko10})} For $M\in\cc_{\leq 0}$, there exists a distinguished triangle
$M' \ra M\ra M''\ra M'[1]$ with $M'\in\cc_{\geq 0}\cap\cc_{\leq 0}$ and $M''\in\cc_{\leq
0}[1]$.
\end{lem}
\begin{proof}
Consider the triangle in (C4). It remains to show that $M' \in \cc_{\leq 0}$. Using the equality $\cc_{\leq 0}=\cc_{\geq 0}^\perp[-1]$ from above together with the shifted triangle $M''[-1] \ra M' \ra M \ra M''$, where $M''[-1]$, $M \in \cc_{\leq 0}$ proves the claim. 
\end{proof}

\noindent Let $A$ be a non-positive dg $k$-algebra. Let $\cp_{\geq 0}$
respectively $\cp_{\leq 0}$ denote the smallest full subcategory of
$\per(A)$ which contains $A[i]$ for $i\leq 0$ respectively
$i\geq 0$ and is closed under taking extensions and direct
summands. 

\begin{prop}\label{p:standard-co-t-str}
$(\cp_{\geq 0},\cp_{\leq 0})$ is a co-$t$-structure of $\per(A)$, with co-heart $\add(A)$.
\end{prop}
\begin{proof}
 This follows from~\cite[Proposition 5.2.2, Proposition 6.2.1]{Bondarko10}, see also \cite{KellerNicolas11}.
\end{proof}

\noindent Objects in $\cp_{\leq 0}$ are characterised by
the vanishing of the positive cohomologies:
\begin{cor}\label{c:description-of-coaisle-by-cohomologies}
 $\cp_{\leq 0}=\{M\in\per(A)\, | \, H^i(M)=0\text{ for any } i>0\}$.
\end{cor}
\begin{proof} Let $\cs$ be the category on the right.
 By the preceding proposition, $\cp_{\leq 0}=\cp_{\geq 0}^{\perp}[-1]=(\cp_{\geq 0}[-1])^\perp$.
In particular, for $M\in\cp_{\leq 0}$ and $i<0$ this implies that 
$\Hom(A[i],M)=0$. Hence, $H^i(M)=0$ holds for any
$i>0$ and $M$ is in $\cs$. Conversely, if $H^i(M)=\Hom(A[-i],M)=0$
for any $i>0$, then it follows by d\'evissage that $\Hom(N,M)=0$ for any $N\in\cp_{\geq
0}[-1]$. This shows that $M$ is contained in $\cp_{\leq 0}$.
\end{proof}

\begin{ex}\label{Ex:TandCoT}  We consider the path algebra $A=kQ$ of the graded quiver 
\begin{equation*}
\begin{tikzpicture}[description/.style={fill=white,inner sep=2pt}]
    \matrix (n) [matrix of math nodes, row sep=3em,
                 column sep=2.5em, text height=1.5ex, text depth=0.25ex,
                 inner sep=0pt, nodes={inner xsep=0.3333em, inner
ysep=0.3333em}]
    {  
       Q\colon& + && -, \\
    };
    
    \path[dash pattern = on 0.5mm off 0.3mm, ->] ($(n-1-2.east) + (0mm,1mm)$) edge [bend left=15] node[fill=white, scale=0.75] [midway] {$a_{+}$} ($(n-1-4.west) + (0mm,1mm)$);

    \path[dash pattern = on 0.5mm off 0.3mm, ->] ($(n-1-4.west) + (0mm,-1mm)$) edge [bend left=15] node[fill=white, scale=0.75] [midway] {$a_{-}$}($(n-1-2.east) + (0mm,-1mm)$);
\end{tikzpicture}   
\end{equation*}
where $a_{-}$ and $a_{+}$ are both in degree $-1$. $A$ may be viewed as a dg algebra with trivial differential\footnote{$A$ is the dg Auslander algebra of an odd dimensional simple singularity $R$ of type $A_{1}$, see Subsection \ref{ss:Independance} and Paragraph \ref{ss:OddTypeA}. In particular, equation \eqref{E:rel-sing-cat-as-per} shows that $\per(A)$ is triangle equivalent to the relative Auslander singularity category $\Delta_{R}(\Aus(R))$, which we describe in Subsection \ref{ss:NodalBlock} by elementary means.}. We depict the Auslander--Reiten quivers of $\cd_{fd}(A)$ and $\per(A)$ below. We want to write down the standard $t$- and co-$t$-structures ($(\cd^{\leq 0}, \cd^{\geq 0})$ respectively $(\cp_{\geq 0}, \cp_{\leq 0})$) on $\per(A)$. In particular, one can see that $(\cd^{\leq 0}, \cd^{\geq 0})$ restricts to a $t$-structure on $\cd_{fd}(A)$. 

We start with the Auslander--Reiten quiver of this subcategory. For $l \geq0$ define the dg $A$-module $T_{\mp}^{2l+1}$ as the cone of the morphism $P_{\pm}(2l+1) \xrightarrow{(a_{\pm}a_{\mp})^{2l}a_{\pm}\cdot} P_{\mp}$ and similarly $T_{\pm}^{2l}:=\mathsf{cone}(P_{\pm}(2l) \xrightarrow{(a_{\mp}a_{\pm})^{2l}\cdot} P_{\pm})$. The indecomposable objects in $\cd_{fd}(A)$ are just the degree shifts of the $T^l_{\pm}(s)$, $l \geq 1$ and $s \in \Z$. Now, the Auslander--Reiten quiver of $\cd_{fd}(A)$ consists of two $\Z A_{\infty}$-components, where the Auslander--Reiten translation $\tau$ acts as degree shift $(1)$.
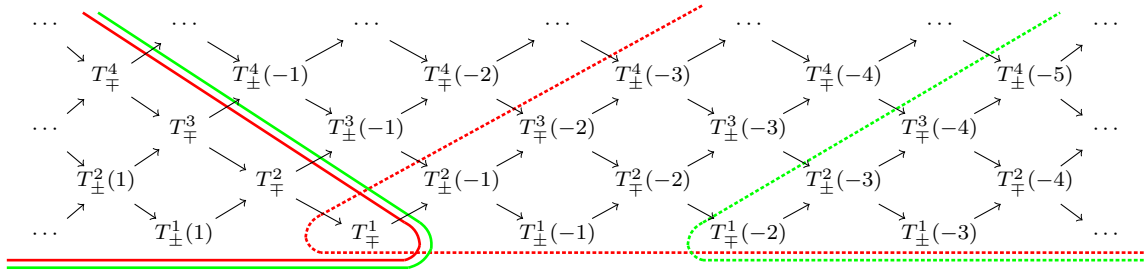
\begin{equation*}
\begin{tikzpicture}[description/.style={fill=white,inner sep=2pt}, font=\tiny]
    \matrix (n) [matrix of math nodes, row sep=0.5em,
                 column sep=0.0em, text height=3.5pt, text depth=0.5ex, 
                 inner sep=0pt, nodes={inner xsep=0.3em, inner
ysep=0.3333em}]
    {  \cdots && \cdots  && \cdots && \cdots  && \cdots && \cdots  && \cdots\\
       & T^4_{\mp} && T^4_{\pm}(-1) && T^4_{\mp}(-2) && T^4_{\pm}(-3) && T^4_{\mp}(-4) && T^4_{\pm}(-5) \\
       \cdots && T^3_{\mp} && T^3_{\pm}(-1) && T^3_{\mp}(-2) && T^3_{\pm}(-3) && T^3_{\mp}(-4) && \cdots \\
       &T^2_{\pm}(1) && T^2_{\mp} && T^2_{\pm}(-1) && T^2_{\mp}(-2) && T^2_{\pm}(-3) && T^2_{\mp}(-4)  \\
       \cdots && T^1_{\pm}(1) && T^1_{\mp} && T^1_{\pm}(-1) && T^1_{\mp}(-2) && T^1_{\pm}(-3) && \cdots   \\   
    };



\path[-, color=red,line width=1pt] ($(n-5-1)+(-5mm, -3mm)$) edge ($(n-5-5)+(5mm, -3mm)$);

\path[-, color=red,line width=1pt]  ($(n-5-5)+(5mm, -3mm)$) edge [bend right=30]  ($(n-5-5)+(7mm, -1mm)$);

\path[-, color=red,line width=1pt]   ($(n-5-5)+(7mm, -1mm)$) edge [bend right=30]  ($(n-5-5)+(5mm, 2mm)$);

\path[-, color=red,line width=1pt]   ($(n-5-5)+(5mm, 2mm)$) edge ($(n-1-1)+(5mm, 2mm)$);


\path[dash pattern = on 0.5mm off 0.3mm,  -, color=red,line width=1pt] ($(n-5-13)+(5mm, -2mm)$) edge ($(n-5-5)+(-6mm, -2mm)$);

\path[dash pattern = on 0.5mm off 0.3mm, -, color=red,line width=1pt]  ($(n-5-5)+(-6mm, -2mm)$) edge [bend left=30]  ($(n-5-5)+(-8mm, 0mm)$);

\path[dash pattern = on 0.5mm off 0.3mm, -, color=red,line width=1pt]   ($(n-5-5)+(-8mm, 0mm)$) edge [bend left=30]  ($(n-5-5)+(-6mm, 3mm)$);

\path[dash pattern = on 0.5mm off 0.3mm, -, color=red,line width=1pt]   ($(n-5-5)+(-6mm, 3mm)$) edge ($(n-1-9)+(-6mm, 3mm)$);

\path[-, color=green,line width=1pt] ($(n-5-1)+(-5mm, -4mm)$) edge ($(n-5-5)+(5.5mm, -4mm)$);

\path[-, color=green,line width=1pt]  ($(n-5-5)+(5.5mm, -4mm)$) edge [bend right=25]  ($(n-5-5)+(8.5mm, -1.5mm)$);

\path[-, color=green,line width=1pt]   ($(n-5-5)+(8.5mm, -1.5mm)$) edge [bend right=30]  ($(n-5-5)+(7mm, 2mm)$);

\path[-, color=green,line width=1pt]   ($(n-5-5)+(7mm, 2mm)$) edge ($(n-1-1)+(7mm, 2mm)$);

 
\path[dash pattern = on 0.5mm off 0.3mm, -, color=green,line width=1pt] ($(n-5-13)+(5mm, -3mm)$) edge ($(n-5-9)+(-6mm, -3mm)$);

\path[dash pattern = on 0.5mm off 0.3mm, -, color=green,line width=1pt]  ($(n-5-9)+(-6mm, -3mm)$) edge [bend left=30]  ($(n-5-9)+(-8mm, -1mm)$);

\path[dash pattern = on 0.5mm off 0.3mm, -, color=green,line width=1pt]   ($(n-5-9)+(-8mm, -1mm)$) edge [bend left=30]  ($(n-5-9)+(-6mm, 2mm)$);

\path[dash pattern = on 0.5mm off 0.3mm, -, color=green,line width=1pt]   ($(n-5-9)+(-6mm, 2mm)$) edge ($(n-1-13)+(-6mm, 2mm)$);


\path[->] (n-3-1) edge (n-4-2);
\path[->] (n-4-2) edge (n-5-3);

\path[->] (n-1-1) edge (n-2-2);
\path[->] (n-2-2) edge (n-3-3);
\path[->] (n-3-3) edge (n-4-4);
\path[->] (n-4-4) edge (n-5-5);

\path[->] (n-1-3) edge (n-2-4);
\path[->] (n-2-4) edge (n-3-5);
\path[->] (n-3-5) edge (n-4-6);
\path[->] (n-4-6) edge (n-5-7);

\path[->] (n-1-5) edge (n-2-6);
\path[->] (n-2-6) edge (n-3-7);
\path[->] (n-3-7) edge (n-4-8);
\path[->] (n-4-8) edge (n-5-9);

\path[->] (n-1-7) edge (n-2-8);
\path[->] (n-2-8) edge (n-3-9);
\path[->] (n-3-9) edge (n-4-10);
\path[->] (n-4-10) edge (n-5-11);

\path[->] (n-1-9) edge (n-2-10);
\path[->] (n-2-10) edge (n-3-11);
\path[->] (n-3-11) edge (n-4-12);
\path[->] (n-4-12) edge (n-5-13);

\path[->] (n-1-11) edge (n-2-12);
\path[->] (n-2-12) edge (n-3-13);

\path[->] (n-5-1) edge (n-4-2);
\path[->] (n-4-2) edge (n-3-3);
\path[->] (n-3-3) edge (n-2-4);
\path[->] (n-2-4) edge (n-1-5);

\path[->] (n-5-3) edge (n-4-4);
\path[->] (n-4-4) edge (n-3-5);
\path[->] (n-3-5) edge (n-2-6);
\path[->] (n-2-6) edge (n-1-7);

\path[->] (n-5-5) edge (n-4-6);
\path[->] (n-4-6) edge (n-3-7);
\path[->] (n-3-7) edge (n-2-8);
\path[->] (n-2-8) edge (n-1-9);

\path[->] (n-5-7) edge (n-4-8);
\path[->] (n-4-8) edge (n-3-9);
\path[->] (n-3-9) edge (n-2-10);
\path[->] (n-2-10) edge (n-1-11);

\path[->] (n-5-9) edge (n-4-10);
\path[->] (n-4-10) edge (n-3-11);
\path[->] (n-3-11) edge (n-2-12);
\path[<-] (n-1-13) edge (n-2-12);

\path[->] (n-3-1) edge (n-2-2);
\path[->] (n-2-2) edge (n-1-3);

\path[->] (n-5-11) edge (n-4-12);
\path[<-] (n-3-13) edge (n-4-12);



\end{tikzpicture}   
\end{equation*}

\newpage
\noindent
The AR-quiver of the category $\per(A)$ has two additional $A_{\infty}^\infty$-components.

\begin{equation*}
\begin{tikzpicture}[description/.style={fill=white,inner sep=2pt}, font=\tiny]
    \matrix (n) [matrix of math nodes, row sep=3em,
                 column sep=2.2em, text height=1.5ex, text depth=0.25ex,
                 inner sep=0pt, nodes={inner xsep=0.3333em, inner
ysep=0.3333em}]
    {  
       \cdots & P_{\pm}(3) & P_{\mp}(2) & P_{\pm}(1) & P_{\mp}(0) & P_{\pm}(-1) & P_{\mp}(-2) & P_{\pm}(-3) & \cdots  \\
    };
    
\path[dash pattern = on 0.5mm off 0.3mm, -, color=green,line width=1pt] ($(n-1-9)+(2mm, 2mm)$) edge ($(n-1-5)+(-3mm, 2mm)$);
\path[dash pattern = on 0.5mm off 0.3mm, -, color=green,line width=1pt] ($(n-1-9)+(2mm, -2.5mm)$) edge ($(n-1-5)+(-3mm, -2.5mm)$);
\draw[dash pattern = on 0.5mm off 0.3mm, -, color=green,line width=1pt] ($(n-1-5)+(-3mm, 2mm)$) arc (90: 270: 2.25mm);

\path[-, color=green,line width=1pt] ($(n-1-1)+(-5mm, 4mm)$) edge ($(n-1-5)+(3mm, 4mm)$);
\path[-, color=green,line width=1pt] ($(n-1-1)+(-5mm, -4mm)$) edge ($(n-1-5)+(3mm, -4mm)$);
\draw[-, color=green,line width=1pt] ($(n-1-5)+(3mm, -4mm)$) arc (-90: 90: 4mm);

\path[-, color=red,line width=1pt] ($(n-1-1)+(-5mm, 3mm)$) edge ($(n-1-5)+(3mm, 3mm)$);
\path[-, color=red,line width=1pt] ($(n-1-1)+(-5mm, -3mm)$) edge ($(n-1-5)+(3mm, -3mm)$);
\draw[-, color=red,line width=1pt] ($(n-1-5)+(3mm, -3mm)$) arc (-90: 90: 3mm);

\draw[->] (n-1-1) -- (n-1-2);
\draw[->] (n-1-2) -- (n-1-3);
\draw[->] (n-1-3) -- (n-1-4);
\draw[->] (n-1-4) -- (n-1-5);
\draw[->] (n-1-5) -- (n-1-6);
\draw[->] (n-1-6) -- (n-1-7);
\draw[->] (n-1-7) -- (n-1-8);
\draw[->] (n-1-8) -- (n-1-9);

\end{tikzpicture}   
\end{equation*}
Here, everything to the left of the straight red line (respectively green line) is contained in $\cd^{\leq 0}$ (respectively $\cp_{\leq 0}$). Moreover, everything to the right of the dashed red line (respectively green line) is contained in $\cd^{\geq 0}$ (respectively $\cp_{\geq 0}$). 

Note, that $\cd^{\geq 0}$ is contained in $\cd_{fd}(A)$. In particular, the $t$-structure $(\cd^{\leq 0}, \cd^{\geq 0})$ and its restriction to $\cd_{fd}(A)$ have the same heart $\cd^{\leq 0} \cap\cd^{\geq 0} = \add T^1_{-} \oplus T^1_{+}$. Since there are no non-zero morphisms between the two simples $T^1_{-}$ and $T^1_{+}$, we see that the heart is equivalent to $\mod-k \times k$ as predicted by Proposition \ref{p:standard-t-str} (note that $H^0(A) \cong k \times k$).

The coheart $\cp_{\leq 0} \cap\cp_{\geq 0}$ of $(\cp_{\geq 0}, \cp_{\leq 0})$ is $\add P_{+} \oplus P_{-} = \add A$ in agreement with Proposition \ref{p:standard-co-t-str}. Moreover, the equality stated in Corollary \ref{c:description-of-coaisle-by-cohomologies} may be observed in this example.

\end{ex}

\subsubsection{Hom-finiteness}\label{ss:nonpositive-dg-alg-2}

Let $A$ be a dg algebra. The subcomplex $\sigma^{\leq 0}A$ inherits a dg algebra structure from $A$.
If $H^i(A)=0$ for any $i>0$, then the embedding $\sigma^{\leq 0}A\hookrightarrow A$ is a quasi-isomorphism of dg algebras and it suffices to study $\sigma^{\leq 0}A$.

We generalise~\cite[Lemma 2.5 \& Prop. 2.4]{Amiot09}
and~\cite[Lemma 2.4 \& Prop. 2.5]{Guolingyan11a}.

\begin{prop}\label{p:hom-finiteness-of-per}
 Let  $A$ be a dg $k$-algebra such that
\begin{itemize}
 \item[--] $A^i=0$ for any $i>0$,
 \item[--] $H^0(A)$ is finite-dimensional,
 \item[--] $\cd_{fd}(A)\subseteq\per(A)$.
\end{itemize}
Then $H^i(A)$ is finite-dimensional for any $i$. Consequently,
$\per(A)$ is Hom-finite.
\end{prop}
\begin{proof}
 It suffices to prove the following induction step: 
 if $H^i(A)$ is finite-dimensional for $-n\leq i\leq 0$, then
$H^{-n-1}(A)$ is finite-dimensional.

To prove this claim, we consider the triangle induced by the standard truncations 
\[
\sigma^{\leq -n-1}A \longrightarrow A\longrightarrow \sigma^{> -n-1}A\longrightarrow (\sigma^{\leq -n-1}A)[1].
\]
Since $H^i(\sigma^{> -n-1}A)=H^i(A)$ for $i\geq -n$, it follows by
the induction hypothesis that $\sigma^{> -n-1}A$ belongs to
$\cd_{fd}(A)$, and hence to $\per(A)$ by the third assumption on
$A$. Therefore, $\sigma^{\leq -n-1}A\in\per(A)$. By
Corollary~\ref{c:description-of-coaisle-by-cohomologies},
$(\sigma^{\leq -n-1}A)[-n-1]\in \cp_{\leq 0}$. Moreover,
Lemma~\ref{l:stupid-truncation} and
Proposition~\ref{p:standard-co-t-str} imply that there is a triangle
\[
M' \longrightarrow (\sigma^{\leq -n-1}A)[-n-1] \longrightarrow M'' \longrightarrow M'[1]
\]
with $M'\in\add(A)$ and $M''\in\cp_{\leq 0}[1]$. It follows from
Corollary~\ref{c:description-of-coaisle-by-cohomologies} that
$H^0(M'')=0$. Thus applying $H^0$ to the triangle above, we
obtain an exact sequence
\[
H^0(M')\longrightarrow  H^0((\sigma^{\leq -n-1}A)[-n-1])=H^{-n-1}(A) \longrightarrow 0.
\]
Now $H^0(M')$ is finite-dimensional because $M'\in\add(A)$ and
$H^0(A)$ has finite dimension by assumption. Thus $H^{-n-1}(A)$ is finite-dimensional.
\end{proof}

\subsection{Minimal relations}\label{ss:minimal-relation}

Let $Q$ be a finite quiver. Denote by $\widehat{kQ}$ the \emph{complete path algebra} of $Q$, i.e.~
the completion of the path algebra $kQ$ with respect to the $\mathfrak{m}$-adic topology, where $\mathfrak{m}$ is the ideal
of $kQ$ generated by all arrows. Namely, $\widehat{kQ}$ is the inverse limit in the category of algebras of the inverse system $\{kQ/\mathfrak{m}^n,
\pi_n:kQ/\mathfrak{m}^{n+1}\rightarrow kQ/\mathfrak{m}^n\}_{n\in\mathbb{N}}$, where $\pi_n$ is the canonical projection.
Later we will also work with complete path algebras of graded quivers: they are defined as above with the inverse limit taken in
the category of graded algebras. We refer to this as the \emph{graded completion}.

The complete path algebra $\widehat{kQ}$ has a natural topology, the $J$-adic topology for $J$ the ideal
generated by all arrows. Let $I$ be a closed ideal of $\widehat{kQ}$ contained in $J^2$ and let $A=\widehat{kQ}/I$. For a vertex $i$ of $Q$, let $e_i$ denote the trivial path at $i$. A \emph{set of minimal relations} of $A$ (or of $I$) is a finite subset $R$ of $\bigcup_{i,j\in Q_0}e_iIe_j$ such that  $I$ coincides with the closure $\overline{(R)}$ of the ideal of $\widehat{kQ}$ generated by $R$  but not with $\overline{(R')}$ for any proper subset $R'$ of $R$. For completeness, we include the following result, which is known to the experts and generalises~\cite[Proposition 1.2]{Bongartz83} (cf. also \cite[Section 6.9]{Keller11}). 

\begin{prop}\label{p:minimal-relation} Let $i$ and $j$ be vertices of $Q$.
If $e_{i}Re_{j}=\{r_{1}, \ldots, r_{s}\}$, then the equivalence classes $\ol{r}_{1}, \ldots, \ol{r}_s$ 
form a basis of $e_i(I/(IJ+JI))e_j$. In particular, the cardinality of $e_iRe_j$ does not depend on the choice of $R$.
\end{prop}

\begin{proof}
We only have to show that $\ol{r}_{1}, \ldots, \ol{r}_s$ are linearly independent. Otherwise, there would be elements $\lambda_1,\ldots,\lambda_s$ in $k$ and a relation
$
\sum_{a=1}^s \lambda_a r_a = 0 ~(\mod IJ+JI),
$
where without loss of generality $\lambda_{1} \neq 0$.
In other words, there exists an index set $\Gamma$ and elements $c_{\gamma} \in e_{i}\widehat{kQ}$ and $c^{\gamma} \in \widehat{kQ}e_{j}$ such that 
for any  fixed $\gamma \in \Gamma$ at least one of $c_{\gamma}$ and $c^{\gamma}$ belongs to $J$ and such that
$
\sum_{a=1}^s\lambda_a r_a = \sum_{r \in R} \sum_{\gamma\in \Gamma} c_{\gamma}r c^{\gamma}
$
holds.  Then we have
\begin{eqnarray*}
r_1 & = & -\lambda_1^{-1}\sum_{a=2}^s\lambda_a r_a + \lambda_1^{-1}\sum_{r \in R} \sum_{\gamma\in \Gamma} c_{\gamma}r c^{\gamma}.
\end{eqnarray*}
 The right hand side is a function $f(r_1)$ in $r_{1}$. We define a sequence $f_{1}=f(0), f_{2}=f(f(0)),\ldots$ of elements in the ideal $(R \setminus \{r_1\})$. One can check that $f_{n}$ defines a Cauchy sequence in the $J$-adic topology and $\lim_{n} f_{n}=r_{1}$. Hence, $r_{1} \in \overline{(R \setminus \{r_1\})}$,  contradicting the minimality of $R$.
\end{proof}

For non-complete presentations of algebras, this result fails in general, see for example~\cite[Example 4.3]{Plamondon11b}.

\subsection{Koszul duality}\label{ss:dual-bar-construction} 
Inspired by work of Beilinson, Ginzburg \& Schechtman \cite{BeilinsonGinzburgSchechtman}, Keller \cite{Keller94} defined the Koszul dual $A'$ of a dg algebra $A$. Under certain conditions on $A$, there is a quasi-isomorphism between $A$ and the double Koszul dual $A''$, see for example \cite{Keller94}. Sometimes the Koszul dual is easier to describe than the original algebra and using the quasi-isomorphism $A \ra A''$, this might give some insight into the structure of the original algebra.

Although our primary interest lies in the study of dg algebras, it will be necessary to work with a more general definition of Koszul duality for $A_{\infty}$-algebras. This gives us the freedom to pass to the ($A_{\infty}$ quasi-isomorphic) \emph{minimal model} of the Koszul dual. It is an $A_\infty$-algebra, whose underlying graded vector space is quite accessible in our applications. Applying Koszul duality to this $A_{\infty}$-algebra allows us to describe our original dg algebra up to a dg quasi-isomorphism, which is sufficient since we are only interested in derived categories. Our main references are~\cite{LuPalmieriWuZhang04,LuPalmieriWuZhang08,Lefevre03}.



 An $A_\infty$-algebra $A$ is a graded $k$-vector
space endowed with a family of homogenous $k$-linear maps  $\{b_n\colon(A[1])^{\ten
n}\rightarrow A[1]|n\geq 1\}$ of degree
$1$  satisfying the following identities
\begin{eqnarray}\label{e:multiplications}
\sum_{j+k+l=n}b_{j+1+l}(id^{\ten j}\ten b_k\ten id^{\ten l})=0,
\qquad n\geq 1.
\end{eqnarray}
The maps $b_{i}$ are called \emph{(higher) multiplications}.
It follows from this definition that $b_{1}^2=0$. In other words, $b_{1}$ is a differential, which in addition satisfies the graded Leibniz rule with respect to the multiplication $b_{2}$. Moreover, $b_{2}$ is associative up to a homotopy.
For example, after an appropriate shift of the multiplication, a dg
algebra can be viewed as an $A_\infty$-algebra with vanishing $b_n$
for $n\geq 3$. In particular, an associative algebra corresponds to an $A_{\infty}$-algebra with $b_{n}=0$ for all $n \neq 2$. $A$ is said to be \emph{minimal} if $b_1=0$. Now, suppose that either $A$ satisfies
\begin{itemize}
\item[--] $A^i=0$ for all $i<0$,
\item[--] $A^0$ is the product of $r$ copies of the base field $k$ for some positive integer $r$,
\item[--]
$b_n(a_1\ten\cdots\ten a_n)=0$ if one of $a_1,\ldots,a_n$ belongs to
$A^0$ and $n\neq 2$.
\end{itemize}
or $A$ satisfies
\begin{itemize}
\item[--]$A^i=0$ for all $i>0$,
\item[--] $H^0(A)\cong \widehat{kQ}/\overline{(R)}$, for a finite quiver $Q$ and a  set $R$ of minimal relations,
\item[--] $b_n(a_1\ten\cdots\ten a_n)=0$ if one of $a_1,\ldots,a_n$ is the trival path at some vertex and $n\neq 2$.
\end{itemize}
In particular, $A$ is \emph{strictly unital} by the third condition in both setups. Let $K=A^0$ in the former case and $K=H^0(A)/\rad H^0(A)$ in the latter case. In both cases, there is an injective
homomorphism $\eta \colon K\rightarrow A$ and surjective homomorphism $\varepsilon \colon A\rightarrow K$ of $A_\infty$-algebras.
Denote by $\bar{A}=\ker\varepsilon$. Note that $\bar{A}$ inherits the structure of an $A_\infty$-algebra.
The \emph{bar construction} of $A$, denoted by $BA$, is the graded vector space
\[T_{K}(\bar{A}[1])=K\oplus \bar{A}[1]\oplus\bar{A}[1]\ten_{K}\bar{A}[1]\oplus\ldots.\]
It is naturally a coalgebra with comultiplication defined by
splitting the tensors, i.e. $\Delta\colon BA \ra BA \otimes BA$ is determined by
\begin{align}
\Delta(a_{1}, \ldots, a_{n})= (a_{1}, \ldots, a_{n}) \otimes 1 +  \sum_{i=1}^{n-1} (a_{1}, \ldots, a_{i}) \otimes (a_{i+1}, \ldots, a_{n}) + 1 \otimes (a_{1}, \ldots, a_{n})
\end{align}

Moreover, $\{b_n|n\geq 1\}$ uniquely extends
to a differential on $BA$ which makes it a dg coalgebra. The
\emph{Koszul dual} of $A$ is the graded $k$-dual of $BA$:
\[E(A)=B^{\#}A:=D(BA).\]
Then $E(A)$ is a dg algebra and 
as a graded algebra $E(A)=\hat{T}_{K}(D(\bar{A}[1]))$ is the graded completion (with respect to the ideal generated by $D(\bar{A}[1])$) of the tensor algebra of $D(\bar{A}[1])=\Hom_k(\bar{A}[1],k)$
over $K$. Its differential $d$ is the unique continuous $k$-linear map satisfying the graded Leibniz rule and taking 
$f \!\in D(\bar{A}[1])$ to $d(f)\! \in B^{\#}\!A$,~defined by
\[d(f)(a_1\ten\cdots \ten a_n)=f(b_n(a_1\ten \cdots \ten a_n)),~~a_1,\ldots,a_n\in \bar{A}[1].\]
Let $\mathfrak{m}$ be the ideal of $E(A)$ generated by $D(\bar{A}[1])$. Then $A$ being minimal amounts to saying that
$d(\mathfrak{m})\subseteq \mathfrak{m}^2$ holds true.

\begin{lem}\label{L:QuisAndKoszul}
If $C$ and $C'$ are $A_\infty$ quasi-isomorphic, then $E(C)$ and $E(C')$ are quasi-isomorphic as dg algebras.
\end{lem}
\begin{proof}
By \cite[1.3.3.6 and 2.3.4.3]{Lefevre03} (see also \cite[Proposition 1.14]{LuPalmieriWuZhang08}), for any $A_{\infty}$-algebra $A$ there is a \emph{natural} $A_{\infty}$ quasi-isomorphism $A \ra \Omega B(A)$, where $\Omega$ denotes the so called \emph{cobar construction}. In particular, this yields an $A_{\infty}$ quasi-ismorphism $\Omega B(C) \ra \Omega B(C')$. Hence by definition (see the paragraph before \cite[Theorem 1.3.1.2]{Lefevre03}), there is a weak equivalence between dg coalgebras $w\colon B(C) \ra B(C')$. Dualizing $w$ yields the claim.
\end{proof}
It is well-known that the Koszul dual $E(A)$ admits another interpretation, which is sometimes more accessible.
We need the following statement, see  \cite[Lemma 11.1]{LuPalmieriWuZhang04}.

\begin{prop}\label{Lemma 11.1}
 Let $A$ be an $A_{\infty}$-algebra, which satisfies the properties above. If the multiplications $b_{i}$ of $A$ vanish for all $i \geq 3$ (i.e. $A$ comes from a dg algebra), then there is a quasi-isomorphism of dg algebras
 \begin{align}
 E(A)\cong \RHom_{\cd(A)}(K,K), 
 \end{align}
 where $K$ is viewed as an dg $A$-module
via the homomorphism $\varepsilon$ and $\RHom_{\cd(A)}(K,K)$ denotes the dg endomorphism algebra, i.e. $\RHom_{\cd(A)}(K,K) \cong \cHom_{A}(P_{K}, P_{K})$, where $P_{K}$ is a $K$-projective resolution of $K$. 

In particular, $H^*(E(A))$ is isomorphic to
$\bigoplus_{i\in\mathbb{Z}}\Hom_{\cd(A)}(K,K[i])$.
\end{prop}
\begin{proof}
We note that the setup in \cite{LuPalmieriWuZhang04} is slightly different. However,
the proof of \cite[Lemma 11.1]{LuPalmieriWuZhang04} only relies on statements from \cite{Lefevre03} and \cite{FelixHalperinThomas}, which use setups compatible with ours.
\end{proof}

 The minimal model of $E(A)$ (in the sense of \cite{Kadeishvili80}) is called
the \emph{$A_\infty$-Koszul dual}
of $A$ and is denoted by $A^*$. Proposition \ref{Lemma 11.1} has the following consequence (again the proof in 
\cite{LuPalmieriWuZhang04} applies to our setup).

\begin{thm}\label{t:koszul-double-dual}
 \emph{(\cite[Theorem 11.2]{LuPalmieriWuZhang04})} Let $A$ be an $A_\infty$-algebra as above.
If the space $A^i$ is finite-dimensional for each $i\in\mathbb{Z}$,
then $E(E(A))$ is $A_\infty$ quasi-isomorphic to $A$. In particular, $A$ is $A_\infty$ quasi-isomorphic to $E(A^*)$.
\end{thm}

If $A$ is a dg algebra, then the $A_\infty$ quasi-isomorphism in the theorem
can be replaced by a quasi-isomorphism of dg algebras, see~\cite[Proposition 2.8]{Lunts10}.

\begin{cor} \label{c:koszul-double-dual}
Let $A$ be an $A_\infty$-algebra as above.
If the space $H^i(A)$ is finite-dimensional for each $i\in\mathbb{Z}$,
then $E(E(A))$ is $A_\infty$ quasi-isomorphic to $A$. In particular, $A$ is $A_\infty$ quasi-isomorphic to $E(A^*)$.
\end{cor}
\begin{proof}
After passing to the minimal model $H^*(A)$ of $A$, we may apply Theorem \ref{t:koszul-double-dual}. Lemma \ref{L:QuisAndKoszul} yields the claim, since $H^*(A)$ and $A$ are $A_{\infty}$ quasi-isomorphic.
\end{proof}

Moreover, we can describe the graded algebra underlying $E(A)$ in terms of quivers. Recall that $K$ is a product of $r$ copies of $k$. Let $e_1,\ldots,e_r$ be the standard basis of $K$. Define $Q$ as the following graded quiver:
\begin{itemize}
\item[--] the set of vertices is $\{1,\ldots,r\}$, 
\item[--] the set of arrows degree $m$ from $i$ to $j$ is given by a $k$-basis of the degree $m$ component of $e_jD(\bar{A}[1])e_i$. 
\end{itemize}
Then
as a graded algebra $E(A)$ is the graded completion $\widehat{kQ}$ (with respect to the ideal generated by the arrows) of the path algebra $kQ$ of the graded quiver $Q$.


\subsection{Recollements}\label{ss:Recollements}

In this section our object of study is the triangle quotient $K^b(\proj-A)/\thick(eA)$, where $A$ is an algebra and $e\in A$
is an idempotent. By Keller's Morita theorem for triangulated categories~\cite[Theorem 3.8 b)]{Keller06d}, the idempotent completion
of this category is equivalent to the perfect derived category $\per(B)$ of some dg algebra $B$. However, to determine $B$ explicitly is a difficult task in general. It is well-known that the unbounded derived category of \emph{all} $A$-modules is much better suited to obtain (abstract) existence theorems (for example of adjoint functors), due to the existence of set-indexed direct sums. Then Neeman's theorem allows us to deduce results on the bounded and finitely generated level by passing to compact objects. 

\medskip
Using the general theory on
recollements we can choose $B$ such that there is a homomorphism of dg algebras $A\rightarrow B$, the restriction
$\cd(B) \rightarrow \cd(A)$ along which is fully faithful. This helps us to study properties of
$K^b(\proj-A)/\thick(eA)$ and under some conditions to construct $B$ explicitely. The results on recollements obtained in this section are of independent interest.

\smallskip

\noindent Following~\cite{BeilinsonBernsteinDeligne82}, a \emph{recollement} of triangulated categories is a diagram
\begin{align}
\begin{xy}
\SelectTips{cm}{}
\xymatrix{\ct''\ar[rr]|{i_*=i_!}&&\ct\ar[rr]|{j^!=j^*}\ar@/^15pt/[ll]^{i^!}\ar@/_15pt/[ll]_{i^*}&&\ct'\ar@/^15pt/[ll]^{j_*}\ar@/_15pt/[ll]_{j_!}}
\end{xy}
\end{align}
of
triangulated categories and triangle functors
such that
\begin{itemize}
\item[1)] $(i^*,i_*=i_!,i^!)$ and
$(j_!,j^!=j^*,j_*)$ are adjoint triples;
\item[2)] $j_!,i_*=i_!,j_*$ are fully faithful;
\item[3)] $j^* i_*=0$;
\item[4)] for every object $X$ of $\ct$ there exist two distinguished triangles
\[i_!i^!X \rightarrow X \rightarrow j_*j^*X \rightarrow i_!i^!X[1] \, \,
\text{ and } \, \,  j_!j^!X \rightarrow X \rightarrow
i_*i^*X \rightarrow j_!j^!X[1],\] where the morphisms starting from and
ending at $X$ are the units and counits.
\end{itemize}

We need the following well-known lemma, see e.g. \cite{Miyachi}.

\begin{lem}\label{L:Miyachi}
Let $F\colon \ct \ra \ct'$ be a triangulated functor, such that the induced functor $\ol{F}\colon \ct/\ker(F) \ra \ct'$ is full. Then $\ol{F}$ is faithful.

In particular, if $F$ is full and dense, then $\ol{F} \colon \ct/\ker(F) \ra \ct'$ is an equivalence. In other words, $F$ is a quotient functor, see Definition \ref{D:quotientfunctor}.
\end{lem}
\begin{proof}
Take a morphism $\varphi=(X \stackrel{s}\leftarrow T \stackrel{f}\rightarrow Y)$ from $X$ to $Y$ in $\ct/\ker(F)$ such that $\ol{F}(\varphi)=0$. Thus $0=\ol{F}(\varphi)= F(f)F(s)^{-1}$ and therefore $F(f)=0$. 

We have to show that $\varphi=0$ in $\ct/\ker(F)$. In other words, we need a morphism $q$ in $\ct$ with $\mathsf{cone}(q) \in \ker(F)$ and $fq=0$.

$f\colon T \ra Y$ is contained in a triangle $Z \stackrel{p}\ra T \stackrel{f}\ra Y \ra Z[1]$ in $\ct$. Applying $F$ to this triangle, we see that $F(p)$ admits a right inverse $r \colon F(T) \ra F(Z)$. Since $\ol{F}$ is full, there exists a morphism
$\rho=(T \stackrel{t}\leftarrow R \stackrel{r'}\ra Z)$ in the quotient $\ct/\ker(F)$ such that $\ol{F}(\rho)=r$. Now, 
$1_{F(T)}=F(p) r =F(p) F(r') F(t)^{-1}=F(p r') F(t)^{-1}$ shows that $F(pr')$ is invertible in $\ct'$. Thus $\mathsf{cone}(pr') \in \ker(F)$. Hence we may take $q=pr'$. Indeed, we have $fq=(fp) r' = 0 \cdot r' =0$, since $p$ and $f$ are consecutive arrows in a triangle. This shows that $\ol{F}$ is faithful. The second statement follows directly from this.
\end{proof}

Let us deduce some well-known consequences from the definition of a recollement.

\begin{lem} \label{L:PropRecoll} Using the notations and assumptions in the definition above the following statements hold:
\begin{itemize}
\item[(a)] There are functorial isomorphisms $j^! j_{!} \cong j^* j_{*} \cong 1_{\ct'}$ and $i^* i_{*} \cong i^! i_{!} \cong 1_{\ct''}$.
\item[(b)] $i^*j_{!}=i^!j_{*}=0$.
\item[(c)] $j^*$, $i^*$ and $i^!$ are quotient functors. They induce triangle equivalences 
\begin{align*}
j^*\colon \ct/i_{*}(\ct'') \ra \ct',\, i^*\colon \ct/j_{!}(\ct') \ra \ct'' \, \text{ and } \,\, i^!\colon \ct/j_{*}(\ct') \ra \ct'',
\end{align*}
 respectively. 
\item[(d)] $(\im \, j_!,\im \, i_*)$ and $(\im \, i_*,\im \, j_*)$ are two \emph{stable}  Ê$t$-structures of $\ct$, i.e.~$t$-structures with triangulated aisles.
\end{itemize}
\end{lem}
\begin{proof}
(a) Since $j_{!}$ is fully faithful and right adjoint to $j^!$, we get a chain of natural isomorphisms for all $X$ and $Y$ in $\ct'$
\begin{align*}
\Hom_{\ct'}(X, Y) \cong \Hom_{\ct}(j_{!}(X), j_{!}(Y)) \cong \Hom_{\ct'}(j^!j_{!}(X), Y).
\end{align*}
Hence we obtain a natural isomorphism $j^! j_{!} \cong 1_{\ct'}$ by Yoneda's Lemma. The other isomorphisms can be shown analogously.

\noindent (b) Since $j^*i_{*}=0$ and using the adjunctions, we obtain a chain of natural isomorphisms for $X$ in $\ct'$ and $Y$ in $\ct''$
\begin{align*}
\Hom_{\ct''}(i^*j_{!}(X), Y) \cong \Hom_{\ct}(j_{!}(X), i_{*}(Y)) \cong \Hom_{\ct'}(X, j^* i_{*}(Y)) =0. 
\end{align*}
Hence taking $Y=i^*j_{!}(X)$ proves the first statement. The second follows similarly.

\noindent (c) By the well-known Lemma \ref{L:Miyachi}, to see that the functors are indeed quotient functors, it suffices to show that they are dense and full. But this follows immediately from (a). Thus it suffices to identify their respective kernels. We only show that $\ker j^*= \ct''$ the other cases are treated analogously. By conditions 2) and 3) above, we know that $\ct'' \cong i_{*}(\ct'') \subseteq \ker j^*$. The other inclusion follows directly from the first triangle in 4). 

\noindent (d) Since $\im \, j_!$ and $\im \, i_*$ are full triangulated subcategories, it suffices to show that $\Hom_{\ct}(\im \, j_{!}, \im \, i_{*})=0$ and the existence of the distinguished `glueing' triangles. The first statement follows from 1) and 3). The second statement follows from the second triangle in 4). The other pair is treated in a similar way.  
\end{proof}

\subsubsection{Recollements and TTF triples}

\begin{def}\label{D:TTF}
Let $\ct$ be a triangulated category. A \emph{ triangulated torsion torsionfree triple (TTF triple)}  is a triple $(\cx, \cy, \cz)$ of full subcategories of $\ct$, such that $(\cx, \cy)$ and $(\cy, \cz)$ are two $t$-structures on $\ct$.
\end{def}

\begin{rem}
One can check that the components of a TTF triple $(\cx, \cy, \cz)$ are triangulated subcategories, see e.g. \cite{Nicolas}.
\end{rem}

 A recollement yields a TTF triple $(\im \, j_{!}, \im \, i_{*}, \im \, j_{*})$, by Lemma \ref{L:PropRecoll} (d). Conversely any TTF triple $(\cx, \cy, \cz)$ on $\ct$ yields a recollement in the following way, see~\cite[Section 2.1]{NicolasSaorin09}. We use the notation from loc.~cit.: for a $t$-structure $(\cu, \cv)$ on $\ct$ we denote by $u$ and $v$ the canonical inclusions of $\cu$ and $\cv$, respectively. Moreover, $\tau_{\cu}$ is the right adjoint to $u$ and $\tau^\cv$ is the left adjoint to $v$. Then
\begin{align}
\begin{xy}
\SelectTips{cm}{}
\xymatrix{\cy \ar[rr]|{y}&&\ct\ar[rr]|{\tau_{\cx}}\ar@/^15pt/[ll]^{\tau_{\cy}}\ar@/_15pt/[ll]_{\tau^\cy}&&\cx\ar@/^15pt/[ll]^{z \tau^\cz x}\ar@/_15pt/[ll]_{x}}
\end{xy}
\end{align} 
 is a recollement of triangulated categories \cite[Proposition 4.2.4]{Nicolas} and one can check that these two construction are mutually inverse to each other (up to a natural notion of equivalence).
 
\subsubsection{Bousfield (co-)localisation and recollements}
We follow Neeman's treatment in \cite[Chapter 9]{Neeman99}.
\begin{defn}
Let $\ct$ be a triangulated category and $\cu \subseteq \ct$ be a thick subcategory. We say that a \emph{Bousfield localisation functor exists} for $\cu \subseteq \ct$, if there is a right  adjoint functor $Q_{\rho}$ of the natural quotient functor $Q\colon \ct \ra \ct/\cu$. $Q_{\rho}$ is called the \emph{Bousfield localisation functor}. Dually, a left adjoint $Q_{\lambda}$ of $Q$ is called \emph{Bousfield colocalisation functor}.
\end{defn}

The following two propositions give a relation to recollements, see Beilinson, Bernstein \& Deligne \cite[Section 1.4.4]{BeilinsonBernsteinDeligne82} and also \cite[Proposition 9.1.18]{Neeman99}.

\begin{prop}
In the notations above:
\begin{itemize}
\item[(a)] There exists a Bousfield localisation functor for $\cu \subseteq \ct$ if and only if the inclusion $I \colon \cu \ra \ct$ has a right adjoint $I_{\rho}$.
\item[(b)] Dually, there exists a Bousfield colocalisation functor for $\cu \subseteq \ct$ if and only if the inclusion $I\colon \cu \ra \ct$ has a left adjoint $I_{\lambda}$. 
\end{itemize}
\end{prop}

\begin{prop}\label{P:AbstractRecoll}
In the notations above, assume that there exists a Bousfield localisation and colocalisation functor for $\cu \subseteq \ct$. Then there exists a recollement of triangulated categories
\begin{align}
\begin{xy}
\SelectTips{cm}{}
\xymatrix{\cu \ar[rr]|{I}&&\ct\ar[rr]|{Q}\ar@/^15pt/[ll]^{I_{\rho}} \ar@/_15pt/[ll]_{I_{\lambda}}&&\ct/\cu \ar@/^15pt/[ll]^{Q_{\rho}} \ar@/_15pt/[ll]_{Q_{\lambda}}}
\end{xy}
\end{align}

\end{prop}
\begin{proof}
$Q_{\rho}$ is fully faithful by the adjointness property together with \cite[Corollary 9.1.9]{Neeman99}. The fully faithfulness of $Q_{\lambda}$ follows by a dual argument. Since by definition $I$ is fully faithful and $Q \circ I=0$, only property 4) remains to be shown. This follows from \cite[Theorem 9.1.13]{Neeman99} and its dual version.
\end{proof}
\begin{rem}
A similar result was proved by Cline, Parshall \& Scott in \cite[Theorem 1.1.]{Cline-Parshall-Scott88}.
\end{rem}

\subsubsection{Recollements generated by idempotents}
In the algebraic setting, recollements of derived module categories $\cd(\Mod-A)$ generated by idempotents $e \in A$ are of particular interest,
see for example~\cite{Cline-Parshall-Scott88,ClineParshallScott96,ClineParshallScott97,KoenigNagase09}.
The idempotents considered in these papers
are special idempotents.  

\medskip



The following is well-known, see e.g.~~\cite[Example 4.16]{Krause}.

\begin{prop}\label{P:AdjTriple}
Let $A$ be a $k$-algebra and $e \in A$ be an idempotent. Then there is a triple of adjoint triangle functors
\begin{align}\label{E:AdjTriple}
\begin{xy}
\SelectTips{cm}{}
\xymatrix{\cd(A)\ar[rrrr]|{\begin{smallmatrix} j^!= \RHom_{A}(eA, -) \\ = \\ j^*=-\lten_{A}Ae \end{smallmatrix} }&&&&\cd(eAe)\ar@/_25pt/[llll]_{j_!=-\lten_{eAe}eA}\ar@/^25pt/[llll]^{ j_* =\RHom_{eAe}(Ae, - )}}
\end{xy},
\end{align}
i.e.~$(j_{!}, j^!)$ and $(j^*, j_{*})$ are adjoint pairs. Moreover, $j_{!}$ is fully faithful and $j^*=j^!$ is a triangulated quotient functor.
\end{prop}
\begin{proof}
Already on the abelian level, we have a triple of adjoint functors
\begin{align*}
\begin{xy}
\SelectTips{cm}{}
\xymatrix{\Mod-A \ar[rrrr]|{\begin{smallmatrix} \Hom_{A}(eA, -) \\ = \\ -\ten_{A}Ae \end{smallmatrix} }&&&& \Mod-eAe \ar@/_25pt/[llll]_{-\ten_{eAe}eA}\ar@/^25pt/[llll]^{ \Hom_{eAe}(Ae, - )}}
\end{xy},
\end{align*}
by the adjunction formula. Deriving this, yields the adjoint triple \eqref{E:AdjTriple} above. It follows from \cite[Theorem 1.6.]{JorgensenRecoll} that $j_{!}$ is fully faithful. Now, using the same arguments as in the proof of Lemma \ref{L:PropRecoll} (a) and (c), we deduce that $j^!$ is a quotient functor. 
\end{proof}

\begin{cor}
The adjoint triple in Proposition \ref{P:AdjTriple} extends to a recollement
\begin{align}\label{E:RecollFirstStep}
\begin{xy}
\SelectTips{cm}{}
\xymatrix{\cd_{A/AeA}(A)\ar[rr]&&\cd(A)\ar[rr]|{\, \,  j^!=j^* \, \, }\ar@/^15pt/[ll] \ar@/_15pt/[ll] &&\cd(eAe)\ar@/^15pt/[ll]|{\, \, j_* \, \,}\ar@/_15pt/[ll]|{\, \, j_! \, \,}}
\end{xy},
\end{align}
\end{cor}
\begin{proof}
Proposition \ref{P:AdjTriple} shows that $j^*$ has a left and a right adjoint. One checks directly that it's kernel is $\cd_{A/AeA}(A)$. Now, Proposition \ref{P:AbstractRecoll} completes the proof.
\end{proof}
So far we have a good understanding of the right hand side of the recollement above. Before we continue to make the left hand side more concrete, let us pause for a moment to see which problems one might encounter.
\begin{rem}\label{R:CPS}
Let $A$ be a $k$-algebra and let $e \in A$ be an idempotent. Cline, Parshall \& Scott observed that the following conditions are equivalent, see \cite[Remark 1.2 and Example 1.3.]{Cline-Parshall-Scott88}. 

\noindent (i) The forgetful functor $\Mod-A/AeA \ra \Mod-A$ induces a fully faithful triangle functor $i_{*}\colon \cd(A/AeA) \ra \cd(A)$.

\noindent (ii) There is a triangle equivalence $\cd(A/AeA) \cong \cd_{A/AeA}(A)$. In particular, there is a recollement of triangulated categories
\begin{align*}
\begin{xy}
\SelectTips{cm}{}
\xymatrix{\cd(A/AeA)\ar[rr]|{\, \, i_*=i_! \, \,}&&\cd(A)\ar[rr]|{\, \,  j^!=j^* \, \, }\ar@/^15pt/[ll]|{\, \, i^! \, \,}\ar@/_15pt/[ll]|{\, \, i^*\, \, }&&\cd(eAe)\ar@/^15pt/[ll]|{\, \, j_* \, \,}\ar@/_15pt/[ll]|{\, \, j_! \, \,}}
\end{xy},
\end{align*}
In this situation, the functors on the left are explicitly given by
\[\begin{array}{ll}
i^*=-\lten_A  A/AeA, &
i_*=\RHom_{A/AeA}(A/AeA,-), \\
i_!=-\lten_{A/AeA}A/AeA, & 
i^!=\RHom_A(A/AeA,-).
\end{array}\]
\end{rem}

\begin{ex} \label{Ex:RecollByIdemp1} 
We consider the Auslander algebra $A=\Aus(k[x]/x^3)$ of $\mod-k[x]/(x^3)$. It may be written as the following quiver with relations
\begin{align}
\begin{xy}
\SelectTips{cm}{}
\xymatrix{
1 \ar@/^15pt/[rr]^a && 2 \ar@/^15pt/[rr]^b \ar@/^15pt/[ll]^c &&  \ar@/^15pt/[ll]^d 3
} 
\end{xy} \qquad  \qquad  \text{ca=0=ac-db}.
\end{align}
Here, the indecomposable $k[x]/(x^3)$-modules $k$, $k[x]/(x^2)$ and $k[x]/(x^3)$ correspond to the vertices $1$, $2$ and $3$, respectively. 

\noindent (i) \, \, Let $e=e_{1}+e_{2}$ be the idempotent corresponding to the vertices $1$ and $2$. As an $A$-module $A/AeA$ is isomorphic to $S_{3}$. Since $\Ext^i_{A}(S_{3}, S_{3})=0$ holds for all $i>0$, the functor $i_{*}\colon \cd(A/AeA) \ra \cd(A)$ is fully faithful. In particular, we obtain a recollement of triangulated categories $\bigl( \cd(k), \cd(A), \cd(eAe)\bigl)$, by Remark \ref{R:CPS} above. In this case, $eAe$ is isomorphic to the Auslander algebra of $k[x]/(x^2)$. One can check that more generally there are recollements 
\begin{align}
\begin{xy}
\SelectTips{cm}{}
\xymatrix{\cd(k)\ar[rr]|(.43){\, \, i_*=i_! \, \,}&&\cd\bigl(\Aus(k[x]/x^n)\bigr)\ar[rr]|{\, \,  j^!=j^* \, \, }\ar@/^15pt/[ll]|{\, \, i^! \, \,}\ar@/_15pt/[ll]|{\, \, i^*\, \, }&&\cd\bigl(\Aus(k[x]/x^{n-1})\bigr)\ar@/^15pt/[ll]|{\, \, j_* \, \,}\ar@/_15pt/[ll]|{\, \, j_! \, \,}}
\end{xy},
\end{align}
which are obtained by an analogous construction. 

\noindent (ii) \, \, Let $e=e_{2} \in A$. Then the simple $A$-module $S_{1}$ is in the image of $i_{*}$. One checks that $\Ext^2_{A}(S_{1}, S_{1}) \cong k$. Hence the functor $i_{*}\colon \cd(k \times k) \cong \cd(A/AeA) \ra \cd(A)$ is not full and there is no recollement of derived module categories as in (i).

\noindent (iii) \, \, Let $e=e_{3} \in A$. The simple $A$-module $S_{1}$ is in the image of $i_{*} \colon \cd(A/AeA) \ra \cd(A)$. Using $A/AeA \cong k(
\begin{xy}
\SelectTips{cm}{10}
\xymatrix{
1 \ar@/^5pt/[r]|{\, \, \alpha \, \,} & 2 \ar@/^5pt/[l]|{\, \, \beta \, \, }}
\end{xy}
)/(\alpha \beta, \beta \alpha)
$, 
we see that $\Ext^4_{A/AeA}(S_{1}, S_{1}) \cong k$ holds. Since $A$ has global dimension $2$, the functor $i_{*}$ is not faithful and there is no recollement of derived module categories as in (i).

\noindent (iv) \, \, Example (iii) may be seen as a special case of the following situation. Let $\ce$ be a $k$-linear Frobenius category with finitely many indecomposable objects $E_{0}, \ldots, E_{t}$ and Hom-finite stable category $\underline{\ce}$. Let $A=\End_{\ce}(\bigoplus_{i=0}^t E_{i})$ be the Auslander algebra of $\ce$ and $e \in A$ be the idempotent corresponding to the identity endomorphism of the projective generator of $\ce$. Then the stable Auslander algebra $\underline{A}:= A/AeA \cong \ul{\End}_{\ce}(\bigoplus_{i=0}^t E_{i})$ is a finite-dimensional selfinjective $k$-algebra, see e.g. \cite[Proposition 4.10]{IWnewtria}. In particular, if $\ul{A}$ is not semi-simple, then it has infinite global dimension. Let us assume that this is true and that $A$ has finite global dimension (e.g.~$\ce=\MCM(R)$, for some local complete Gorenstein $k$-algebra $(R, \mathfrak{m})$, with $k \cong R/\mathfrak{m}$ and an isolated singularity in $\mathfrak{m}$, see \cite[Theorem A.1]{Auslander84}). Then the forgetful functor $i_{*} \colon \cd(A/AeA) \ra \cd(A)$ is not faithful. Indeed, the infinite global dimension of $A/AeA$ implies that for all $j \in \N$ there exists $i>j$ such that 
$\Ext_{A/AeA}^i(\bigoplus_{k=1}^s S_{k}, \bigoplus_{k=1}^s S_{k}) \neq 0$, where the $S_{i}$ form a complete set of simple $A/AeA$-modules.   
\end{ex}

 The next result uses the technique of dg algebras to make the recollement \eqref{E:RecollFirstStep} more explicit. 
Recall that a $k$-algebra $A$ can be viewed as a dg $k$-algebra concentrated in degree $0$ and in this case $\cd(A)=\cd(\Mod-A)$.

\begin{prop}\label{p:recollement-from-projective-general-case}
Let $A$ be a $k$-algebra and $e\in A$ an idempotent. There is a
dg $k$-algebra $B$ with a homomorphism of dg $k$-algebras
$f \colon A\rightarrow B$ and a recollement of derived categories
\begin{align}
\begin{xy}
\SelectTips{cm}{}
\xymatrix{\cd(B)\ar[rr]|{\, \, i_*=i_! \, \,}&&\cd(A)\ar[rr]|{\, \,  j^!=j^* \, \, }\ar@/^15pt/[ll]|{\, \, i^! \, \,}\ar@/_15pt/[ll]|{\, \, i^*\, \, }&&\cd(eAe)\ar@/^15pt/[ll]|{\, \, j_* \, \,}\ar@/_15pt/[ll]|{\, \, j_! \, \,}}
\end{xy},
\end{align}
such that the following conditions are satisfied
\begin{itemize}
 \item[(a)] the
adjoint triples $(i^*,i_*=i_!,i^!)$ and $(j_!,j^!=j^*,j_*)$ are given by
\[\begin{array}{ll}
i^*=-\lten_A B, & j_!=-\lten_{eAe} eA,\\
i_*=\RHom_{B}(B,-), &
j^!=\RHom_{A}(eA,-),\\
i_!=-\lten_{B}B, & j^*=-\lten_A Ae,\\
i^!=\RHom_A(B,-),& j_*=\RHom_{eAe}(Ae,-),
\end{array}\]
where $B$ is considered as an $A$-$A$-bimodule via the morphism $f$;
 \item[(b)] $B^i=0$ for $i>0$;
 \item[(c)] $H^0(B)$ is isomorphic to $A/AeA$.
\end{itemize}
\end{prop}

\begin{rem}
This result is known to hold in greater generality, see~\cite[Section 2 and 3]{Dwyer02} (which uses different terminologies). For convenience, we include a proof.
\end{rem}

\begin{proof}
The recollement \eqref{E:RecollFirstStep} yields a TTF triple on $\cd(A)$, by Lemma \ref{L:PropRecoll} (d). Since $k$ is a field, $A$ is a flat $k$-algebra. Therefore \cite[Theorem 4]{NicolasSaorin09} in conjunction with the first paragraph after Lemma 4 in loc.~cit. yields a dg $k$-algebra $B'$ together with a morphism of dg $k$-algebras $f'\colon A \ra B'$ such that there is a recollement of derived categories

\[
\begin{xy}
\SelectTips{cm}{}
\xymatrix{\cd(B')\ar[rr]|{\, \, i_*=i_! \, \, }&&\cd(A)\ar[rr]|{\, \, j^!=j^* \, \,}\ar@/^15pt/[ll]|{\, \, i^! \, \,}\ar@/_15pt/[ll]|{\, \, i^* \, \,}&&\cd(eAe)\ar@/^15pt/[ll]|{\, \, j_* \, \,}\ar@/_15pt/[ll]|{\, \, j_! \, \,}}
\end{xy}
\]
and the adjoint triples $(j_!,j^!=j^*,j_*)$ and  $(i^*,i_*=i_!,i^!)$
are given as in (a) (with $B$ replaced by $B'$). We claim that $H^i(B')=0$ for $i>0$ and that $H^0(f')$ induces an isomorphism of algebras
$A/AeA\cong H^0(B')$. Then taking $B=\sigma^{\leq 0}B'$
and $f=\sigma^{\leq 0}f'$ finishes the proof for (a), (b) and (c).

In order to prove the claim, we take the distinguished triangle associated to $A$.
\begin{eqnarray}\label{e:canonical-triangle-projective}
\begin{array}{cc}
\begin{tikzpicture}[description/.style={fill=white,inner sep=2pt}]
    \matrix (n) [matrix of math nodes, row sep=1em,
                 column sep=2.5em, text height=1.5ex, text depth=0.25ex,
                 inner sep=0pt, nodes={inner xsep=0.3333em, inner
ysep=0.3333em}]
    {  
       Ae\lten_{eAe}eA & A & B' & Ae\lten_{eAe}eA[1] \\
       j_!j^!(A) && i_*i^*(A) \\
    };
\draw[->] (n-1-1.east) -- node[scale=0.75, yshift=2.5mm] [midway] {$\varphi$} (n-1-2.west);    
\draw[->] (n-1-2.east) -- node[scale=0.75, yshift=2.5mm] [midway] {$f'$} (n-1-3.west);  
\draw[->] (n-1-3.east) -- (n-1-4.west); 
\draw[-] ($(n-1-1.south)+(-0.5mm,0)$) -- ($(n-2-1.north)+(-0.5mm,0)$);
\draw[-] ($(n-1-1.south)+(.5mm,0)$) -- ($(n-2-1.north)+(0.5mm,0)$);
\draw[-] ($(n-1-3.south)+(-1mm,1mm)$) -- ($(n-2-3.north)+(-1mm,0)$);
\draw[-] ($(n-1-3.south)+(0mm,1mm)$) -- ($(n-2-3.north)+(0mm,0)$);
\end{tikzpicture}
\end{array}\end{eqnarray}
By applying $H^0$ to the triangle (\ref{e:canonical-triangle-projective}), we obtain a long exact cohomology sequence
\[
\begin{tikzpicture}[description/.style={fill=white,inner sep=2pt}]
    \matrix (n) [matrix of math nodes, row sep=1em,
                 column sep=2.25em, text height=1.5ex, text depth=0.25ex,
                 inner sep=0pt, nodes={inner xsep=0.3333em, inner
ysep=0.3333em}]
    {  
       \, & H^i(Ae\lten_{eAe}eA) & H^i(A) & H^i(B') & H^{i+1}(Ae\lten_{eAe}eA) & \,  \\
    };
\draw[dotted, -] (n-1-1.east) -- (n-1-2.west);    
\draw[->] (n-1-2.east) -- node[scale=0.75, yshift=3mm] [midway] {$H^i(\varphi)$} (n-1-3.west);  
\draw[->] (n-1-3.east) -- node[scale=0.75, yshift=3mm] [midway] {$H^i(f')$} (n-1-4.west); 
\draw[->] (n-1-4.east) -- (n-1-5.west); 
\draw[dotted, -] (n-1-5.east) -- (n-1-6.west); 
\end{tikzpicture}
\]
If $i>0$, both $H^i(A)$ and $H^{i+1}(Ae\lten_{eAe}eA)$ are trivial, and hence $H^i(B')$ is trivial.
If $i=0$, then $H^0(B')\cong H^0(A)/\im(H^0(\varphi))$.
But $H^0(Ae\lten_{eAe}eA)\cong Ae\ten_{eAe}eA$ and the image of $H^0(\varphi)$ is precisely $AeA$.
Therefore, $H^0(f')\colon A\rightarrow H^0(B')$ induces an isomorphism $H^0(B')\cong A/AeA$, which is clearly a homomorphism of algebras.
\end{proof}

\begin{cor}\label{c:restriction-and-induction} Keep the assumptions and notations as in Proposition~\ref{p:recollement-from-projective-general-case}.
\begin{itemize}
 \item[(a)]
The functor
$i^*$ induces an equivalence of triangulated categories 
\begin{align} \label{E:Neeman} \left(K^b(\proj-A)/\thick(eA)\right)^\omega\stackrel{\sim}{\longrightarrow}\per(B), \end{align}
where $(-)^\omega$ denotes the idempotent completion (see \cite{BalmerSchlichting} and Subsection \ref{Sub:Idemp}).

 \item[(b)]  Let $\cd_{fd,A/AeA}(A)$ be the full subcategory of $\cd_{fd}(A)$ consisting of
complexes with cohomologies supported on $A/AeA$. The functor $i_*$
induces a triangle equivalence
$\cd_{fd}(B)\stackrel{\sim}{\longrightarrow}\cd_{fd,A/AeA}(A)$.
Moreover, the latter category coincides with
$\thick_{\cd(A)}(\fdmod-A/AeA)$.
\end{itemize}
\end{cor}
\begin{proof}
(a) Since $j_{!}(eAe)=eAe \lten_{eAe} eA \cong eA$, $eAe$ generates $\cd(eAe)$ and $j_{!}$ commutes with direct sums, we obtain $\im \, j_{!}=\Tria(eA)$. Hence, Lemma \ref{L:PropRecoll} (c) shows that
$i^*$ induces a triangle equivalence \begin{align} \label{E:BigQuot} \cd(A)/\Tria(eA)\cong\cd(B). \end{align} As a projective $A$-module $eA$ is compact in $\cd(A)$. By definition, $\Tria(eA)$ is the smallest localizing subcategory containing $eA$. Since $\cd(A)$ is compactly generated, Neeman's interpretation (and generalization) \cite[Theorem 2.1]{Neeman92a} of Thomason \& Trobaugh's and Yao's Localization Theorem shows that  restricting \eqref{E:BigQuot} to the subcategories of compact objects yields a triangle equivalence $K^b(\proj-A)/\thick(eA) \ra \per(B)$ up to direct summands\footnote{Without the compactness assumptions, Neeman's Theorem may fail in general. For example, let $A$ be the Auslander algebra of $R=k[x]/(x^2)$  and let $e=e_{2} \in A$ be the idempotent corresponding to the identity endomorphism of the projective injective object $k[x]/(x^2)$. Then the simple $A$-module $S_{2}$ is a compact object in $\cd(A)$, since $A$ has finite global dimension. But the image $j^*(S_{2})=\RHom_{A}(eA, S_{2}) \cong k[x]/(x)$ is not perfect over $R$. We see that the quotient functor $j^*$ does not respect compact objects. In particular, Neeman's Theorem fails for $\cd_{A/AeA}(A) \subseteq \cd(A)$. Note that the former category is compactly generated, since $\cd_{A/AeA}(A) \cong \cd(B)$. However, it cannot be expressed as the smallest localizing subcategory of a set of compact objects in $\cd(A)$. To see this, consider the unbounded complex
\begin{align}
P_{2}^*= \, \cdots \ra P_{2} \ra P_{2} \ra \cdots \ra P_{2} \ra \cdots,
\end{align} 
with maps defined by mapping the top of $P_{2}$ to the socle of $P_{2}$. All cohomologies of this complex are isomorphic to the simple $A$-module $S_{1}$. In particular, they are contained in $\Mod-A/AeA$. Since there are no non-zero $A$-module homomorphisms $S_{1} \ra P_{2}$, passing to the homotopy category of injective $A$-modules shows that $\Hom_{\cd(A)}(S_{1}[s], P_{2}^*)=0$ for all $s \in \Z$. Hence, $P_{2}^* \notin \Tria(S_{1})$, for otherwise the orthogonality would imply that $P_{2}^*=0$, contradicting the fact that $P_{2}^*$ has non-zero cohomologies. But all perfect (=compact) objects in $\cd(A)$ which are contained in $\cd_{A/AeA}(A)$ are already contained in $\Tria(S_{1})$. This proves our claim. 
}. Hence the equivalence \eqref{E:Neeman} follows. 

(b) By construction of the dg algebra $B$ in Proposition \ref{p:recollement-from-projective-general-case},  $i_*$ induces a triangle equivalence between $\cd(B)$ and $\cd_{A/AeA}(A)$, i.e.~the full subcategory of $\cd(A)$ consisting of
complexes of $A$-modules which have cohomologies supported on $A/AeA$.
Moreover, $i_*$ restricts to a triangle equivalence between $\cd_{fd}(B)$
and $i_*(\cd_{fd}(B))$. The latter category is contained in
$\cd_{fd}(A)$ because $i_*$ is the restriction along the
homomorphism $f\colon A\rightarrow B$. So $i_*(\cd_{fd}(B))$ is contained in
$\cd_{fd}(A)\cap\cd_{A/AeA}(A)=\cd_{fd,A/AeA}(A)$, which is equivalent to $\thick_{\cd(A)}(\fdmod-A/AeA)$. By
Proposition~\ref{p:standard-t-str} (b), $\fdmod-H^0(B)$ generates
$\cd_{fd}(B)$. But $i_*$ induces an equivalence from $\fdmod-H^0(B)$
to $\fdmod-A/AeA$. Therefore,
$i_*(\cd_{fd}(B))=\thick_{\cd(A)}(\fdmod-A/AeA)$, and hence
$i_*(\cd_{fd}(B))=\thick_{\cd(A)}(\fdmod-A/AeA)=\cd_{fd,A/AeA}(A)$.
\end{proof}

\begin{rem}\label{r:CompInterprOfRelSingCat}
The triangle equivalences \eqref{E:Neeman} and $\cd(B) \cong \cd_{A/AeA}(A)$ show that 
\begin{align}
\left(\cd_{A/AeA}(A)\right)^{c} \cong \left(K^b(\proj-A)/\thick(eA)\right)^\omega.
\end{align}
In particular, if $A$ is a non-commutative resolution of a complete Gorenstein singularity $R$, then the relative singularity category $\Delta_{R}(A) \cong \cd^b(\mod-A)/\thick(eA)$ is idempotent complete by Proposition \ref{P:IdempCompl}. Hence there is a triangle equivalence
\begin{align}
\left(\cd_{A/AeA}(A)\right)^{c} \cong \Delta_{R}(A).
\end{align}
\end{rem}

\newpage

\section{Global relative singularity categories}\label{S:Global}
In this section, we associate a triangulated quotient category - called the \emph{relative singularity category} - to a (non-commutative) resolution of a scheme $X$. This is similar in spirit and has relations to the
triangulated category of singularities of Buchweitz and Orlov. The main result of this section is a certain localization property, which reduces the description of the relative singularity category to the case of affine schemes. Our result generalizes a theorem obtained by Orlov in the commutative setting \cite{Orlov11}. This is joint work with Igor Burban \cite{BurbanKalck11}.
\subsection{Definition}
We briefly introduce the setup for this section.
Throughout this section, let $k$ be an algebraically closed field. Let $X$ be a separated excellent Noetherian  scheme over $\kk$ of finite Krull dimension such that any
coherent sheaf on $X$ is a quotient of a locally free sheaf. We are mainly interested in two special cases: $X$ is quasi-projective or the spectrum of a local \emph{complete} ring $(R, \mathfrak{m})$. Let  $Z=\Sing(X)$ be
 the singular locus of $X$. Let $\kF'$ be a coherent sheaf on $X$, $\kF = \kO \oplus \kF'$ and $\kA := {\mathcal End}_X(\kF)$. We consider the non-commutative ringed space $\XX = (X, \kA)$. Our main interest lies in the  category of coherent $\kA$-modules $\Coh(\XX)$ (or sometimes $\Coh(\kA)$), which consists of coherent sheaves of $\co_{X}$-modules, which have an $\kA$-module structure.
 Note that $\kF \cong \kA e$ is a \emph{locally projective} coherent left $\kA$--module, where $e \in \kA$ is the idempotent corresponding to the identity of $\kO$.
\begin{ex} \label{E:AuslanderSheaf}
Let $X$ be a reduced curve with only nodal and cuspidal singularities, i.e. the completion $\widehat{\co}_{x}$ of the local ring of a point $x \in \Sing(X)$ is either isomorphic to $k\llbracket x, y \rrbracket/(xy)$ (node) or $k\llbracket x, y \rrbracket/(x^2-y^3)$ (cusp). For example, the irreducible cubic curves in $\P^2$ defined by the homogeneous polynomials $Y^2Z-X^2(X+Z)$ (nodal) and $X^2Z-Y^3$ (cuspidal) are of this form. Let $\cf'=\ci$ be the ideal sheaf of the singular locus (with respect to its reduced scheme structure). Then $\ca={\mathcal End}_{X}(\cf)$ is called the \emph{Auslander sheaf} of $X$. The completions of the local rings $\widehat{\ca}_{x}$ are Morita equivalent to the Auslander algebra of the category of maximal Cohen--Macaulay $\co_{x}$-modules $\MCM(\co_{x})$. In particular, $\widehat{\ca}_{x}$ has global dimension at most $2$, by work of Auslander \& Roggenkamp \cite{AuslanderRoggenkamp} (see also Auslander \cite[Theorem A.1]{Auslander84}). This implies that $\Coh(\kA)$ has global dimension at most $2$, see for example \cite[Theorem 1 (4)]{Tilting}. We will come back to this example and give an explicit  description of the corresponding relative singularity category in the special case where $X$ has only nodal singularities in Subsection \ref{ss:NodalBlock}.
\end{ex} 
 
\begin{defn}
A complex $C$ of coherent $\co_{X}$-modules is called \emph{perfect} if it is isomorphic in  $\cd^b(\Coh(X))$ to a bounded complex of locally free sheaves of finite rank. The full subcategory of perfect complexes is denoted by $\Perf(X)$. If $X=\Spec(R)$ is affine, then $\Perf(X)\cong \thick(R)$.
\end{defn}

\begin{prop}\label{Prop:Embedding}
 There is a full embedding given by the derived functor
\[\FF:=\kF \lten_{X}- \colon \Perf(X) \rightarrow \cd^b(\Coh(\XX)).\]
\end{prop}
\begin{proof}
For the full subcategory $\mathsf{Vect}(X) \subseteq \Perf(X)$ of locally free sheaves of finite rank, fully faithfulness may be checked on an open affine cover, where it follows from the algebra isomorphism $eAe \cong \End_{A}(Ae)$. Here $A=\ca(U)$ for an open affine subset $U \subseteq X$ and $e \in A$ is the idempotent satisfying $\kF(U)=Ae$. Since $\mathsf{Vect}(X)$ generates $\Perf(X)$ as a triangulated category, the claim follows from Beilinson's Lemma.
See also \cite[Theorem 2]{Tilting}.
\end{proof}

Denote by $\cP(X) \subseteq \cd^b(\Coh(\XX))$ the essential image of $\Perf(X)$ under the embedding $\FF$. $\cP(X)$ admits the following local characterization, see \cite[Prop. 2]{Tilting}.
\begin{align}
\cP(X) = \Bigl\{\kH^\bullet \in
\cd^b\bigl(\Coh(\XX)\bigr) \, \Big| \, \kH_x^\bullet \in
\mathrm{Im}\Bigl(K^b\bigl(\add(\kF_x)\bigr) \lar  \cd^b\bigl(\kA_x-\mod\bigr)\Bigr)\Bigr\}
\end{align}

We are now able to give the main definition of this section.
\begin{defn}
In the notations of the setup above, the idempotent completion $(-)^\omega$ (see Subsection \ref{Sub:Idemp}) of the triangulated quotient category
\begin{align}
\Delta_{X}(\XX):=\left(\frac{\cd^b(\Coh(\XX))}{\Perf(X)}\right)^\omega \cong \left(\frac{\cd^b(\Coh(\XX))}{\cP(X)}\right)^\omega
\end{align}
is called \emph{(global) relative singularity category} of $\XX$ relative to $X$.
\end{defn}

\begin{rem}
If the abelian category $\Coh(\XX)$ has finite global dimension, then one may view $\cd^b(\Coh(\XX))$ as a categorical resolution of $X$ in the spirit of works of Van den Bergh \cite{VandenBergh04}, Kuznetsov \cite{Kuznetsov} and Lunts \cite{Lunts}. Moreover, it is common to view $\Perf(X)$ as the smooth part of the category $\cd^b(\Coh(X))$.
Hence, $\Delta_{X}(\XX)$ is a measure for the size of the resolution relative to the smooth part $\Perf(X)$. This may be summarized as follows
\begin{align}
\cd^b(\Coh(\XX)) \overset{\text{resolution}}{\supseteq} \Perf(X) \overset{\text{smooth part}}{\subseteq} \cd^b(\Coh(X)).
\end{align}
Recall that the quotient category corresponding to the embedding on the left is the \emph{triangulated category of singularities} $\cd_{sg}(X)$ of Buchweitz and Orlov. It is natural to study the other quotient category $\Delta_{X}(\XX)$ on the left as well as their mutual relations.
 \end{rem}

\subsection{The localization property}
The aim of this subsection is to prove the following localization result.
\begin{thm}\label{t:main-global}
In the notations of the setup above, assume additionally that $\cf$ is locally free on $X\setminus \Sing(X)$. Then there is an equivalence of triangulated categories
\begin{align}
\Delta_{X}(\XX) \cong \bigoplus_{i=1}^n \Delta_{\widehat{\co}_{x_{i}}}\left(\widehat{\ca}_{x_{i}}\right):=  \bigoplus_{i=1}^n \left(\frac{\cd^b(\widehat{\ca}_{x_{i}}-\mod)}{\Perf(\widehat{\co}_{x_{i}})}\right)^\omega.
\end{align}
\end{thm}
\begin{rem}
If $\co_{x_{i}}$ is a complete Gorenstein ring and $\ca_{x_{i}}$ has finite global dimension, then the quotient category $\cd^b(\widehat{\ca}_{x_{i}}-\mod)/\Perf(\widehat{\co}_{x_{i}})$ is idempotent complete by Proposition \ref{P:IdempCompl}.  
\end{rem}

The proof takes several steps, which are enclosed in the lemmas and propositions below.

\begin{lem}[{\cite[Proposition 1.7.11]{KashiwaraSchapira2}}]\label{Lem:KashiwaraSchapira}
Let $\kB$ be an abelian category and $\kC \subseteq \kB$ be a full abelian subcategory such that the following properties hold.
\begin{itemize}
\item[i)] If $V \rightarrow W \rightarrow X \rightarrow Y \rightarrow Z$ is exact in $\kB$ and $V, W, Y, Z$ are in $\kC$ then $X$ is in $\kC$.
\item[ii)] For every monomorphism $\phi\colon M \rightarrow N$ with $M$ in $\kC$ and $N$ in $\kB$, there exists $K$ in $\kC$ and a morphism $\psi\colon N \rightarrow K$ such that the composition $\psi\phi$ is a monomorphism.
\end{itemize}
Then the canonical functor $\cd^b(\kC) \rightarrow \cd^b_{\kC}(\kB)$ is an equivalence of triangulated categories. Here $\cd^b_{\kC}(\kB)$ denotes the full subcategory consisting of complexes with cohomology in $\kC$.
\end{lem}

We apply this lemma to prove the following.

\begin{lem}\label{P:DercCatandSupp}
Let $X$ and $\XX=(X, \kA)$ be as in the setup above and $Z \subseteq X$ be a closed subscheme.
Denote by $\Coh_{Z}(\XX) \subseteq \Coh(\XX)$ the full subcategory of coherent left $\kA$-modules with support contained in $Z$ and by $\cd^b_{Z}(\Coh(\XX)) \subseteq \cd^b(\Coh(\XX))$ the full subcategory of complexes whose cohomology is supported in $Z$. Then the canonical functor
\[\cd^b\bigl(\Coh_{Z}(\XX)\bigr) \rightarrow \cd^b_{Z}\bigl(\Coh(\XX)\bigr)\] is an equivalence of triangulated categories.
\end{lem}

\begin{proof}
We want to apply Lemma \ref{Lem:KashiwaraSchapira} to $\Coh_{Z}(\XX) \subseteq \Coh(\XX)$. The first property is satisfied since localization is an exact functor. To prove the second property let $\phi\colon \kM \rightarrow \kN$ be a monomorphism in $\Coh(\XX)$ and let $\kM$ be supported on $Z$. Let $\{U_{i}=\Spec(O_{i})\}$ be an open affine cover of $X$. We consider the restrictions $M_{i}=\kM\big|_{U_{i}}$ and $N_{i}=\kN\big|_{U_{i}}$ as finitely generated $A_{i}$-modules, where $A_{i}:=\kA\big|_{U_{i}}$. In particular, we have monomorphisms $\phi_{i}\colon M_{i} \ra N_{i}$. Let  $\theta_{i}\colon M_{i} \rightarrow I_{i}$ be an injective envelope of $M_{i}$. Then there exists a morphism $\alpha_{i}\colon N_{i} \rightarrow I_{i}$ such that
$\alpha_{i} \phi_{i} = \theta_{i}$. Note that $K_{i} := \mathsf{Im}(\alpha_{i})$ is a left Noetherian $A_{i}$-module.
As in \cite[Lemma 3.2.5]{BrunsHerzog} one can  show that for any $\mathfrak{p} \in \Spec(O_{i})$ and any $M \in A_{i}-\mod$ we
have: $E(M)_\mathfrak{p} \cong E(M_\mathfrak{p})$, where $E(-)$ denotes the injective envelope. Hence, $(K_{i})_{\mathfrak{p}}=0$ for all $\mathfrak{p} \notin \Supp(M_{i}) \subseteq Z$. Let $j(i)\colon (U_{i}, \kA\big|_{U_{i}}) \ra (X, \kA)$. Then $\kK':= \bigoplus_{i } j(i)_{*} \widetilde{K_{i}}$ is an quasi-coherent $\kA$-module and $\Supp(\kK') \subseteq Z$. Moreover, the $\alpha_{i}$ yield a morphism $\alpha\colon \kN \ra \kK'$ such that $\alpha\phi$ is a monomorphism. Then $\kK=\mathsf{Im}(\alpha)$ is coherent, since $\kN$ is coherent and $\kK'$ is quasi-coherent.  This shows the second property and completes the proof.
\end{proof}

\begin{rem}
One has to be careful in 
 Lemma \ref{P:DercCatandSupp}. It is important that $\kA$ is a coherent module over its center $Z(\kA)$. This is satisfied in our situation since $\kA$ is a coherent $\kO$-module and $\kO \subseteq Z(\kA)$. Let $(O,\mathfrak{m},k)$ be a $k$-algebra, $X=\Spec(O)$ be the corresponding affine scheme and $Z=\{\mathfrak{m}\}$. Then Lemma \ref{P:DercCatandSupp} Êyields an equivalence
\begin{align}\label{F:Equiv}
\cd^b(A-\fdmod) \rightarrow \cd^b_{\rm fd}(A-\mod),
\end{align}
where $A=\End_{O}(F)$ is the endomorphism algebra of a finitely generated $O$-module $F$ and hence finitely generated over $O$. In particular, $A$ is finitely generated over its center $Z(A) \supseteq O$.
 
If  $A$ is assumed to be just a left Noetherian $k$-algebra, the canonical functor (\ref{F:Equiv}) need not be an equivalence in general. The following example was pointed out by Bernhard Keller. 

\begin{ex}
Let $\mathfrak{g}$
be a finite dimensional simple Lie algebra  over $\mathbb{C}$ and $U = U(\mathfrak{g})$ its universal enveloping algebra. Then $U$ is left Noetherian, see for instance \cite[Section I.7]{McConnellRobson}.
 By Weyl's complete reducibility theorem, the category $U-\fdmod$ of finite dimensional
left $U$--modules is semi-simple. However, higher extensions between finite dimensional modules
do not necessarily vanish in $U-\mod$, see for instance \cite[p. 122]{Humphreys}.
 In particular, the canonical functor \[\cd^b(U-\fdmod) \rightarrow \cd^b_{\rm fd}(U-\mod)\] is \emph{not full}.
 \end{ex}
 \end{rem}

We need the following lemma, see {\cite[Proposition 1.6.10]{KashiwaraSchapira2}}.

\begin{lem}\label{Lem:Fullyfaithful}
Let $\kB$ and $\kU$ be full triangulated subcategories of a triangulated category $\kC$ and set $\kU_{\kB}:= \kB \cap \kU$. Assume that any morphism $\phi\colon U \rightarrow B$ with $U \in \kU$ and $B \in \kB$ factors through an object in $\kU_{\kB}$. Then the canonical functor
\[\frac{\kB}{\kU_{\kB}} \rightarrow \frac{\kC}{\kU}\] is fully faithful.
\end{lem}

\begin{prop}\label{P:FullyFaithf}
Let $\cd^b_Z\bigl(\Coh(\XX)\bigr)$ be the full subcategory of $\cd^b\bigl(\Coh(\XX)\bigr)$
consisting of complexes whose cohomology is supported in $Z$ and
$\cP_Z(X) = \cP(X) \cap  \cd^b_Z\bigl(\Coh(\XX)\bigr)$. Then the canonical functor
\[
\HH\colon \quad
\frac{\cd^b_Z\bigl(\Coh(\XX)\bigr)}{\cP_Z(X)} \lar
\frac{\cd^b\bigl(\Coh(\XX)\bigr)}{\cP(X)}
\]
is fully faithful.
\end{prop}

\begin{proof} Our approach is inspired by a recent paper of Orlov \cite{Orlov11}.  By Lemma \ref{Lem:Fullyfaithful} it is sufficient to show that for any $\kP^\bullet \in \cP(X)$,
$\kC^\bullet \in \cd^b_Z\bigl(\Coh(\XX)\bigr)$ and $\varphi \colon \kP^\bullet \rightarrow \kC^\bullet$ there
exists $\kQ^\bullet \in \cP_Z(X)$ and a factorization
\[
\begin{xy}\SelectTips{cm}{}
\xymatrix{
\kP^\bullet \ar[rr]^{\varphi} \ar[dr]_{\varphi'} & & \kC^\bullet \\
& \kQ^\bullet \ar[ur]_{\varphi''}&
}
\end{xy}
\]
By Lemma \ref{P:DercCatandSupp}, we know that the functor $\cd^b\bigl(\Coh_{Z}(\XX)\bigr)
\rightarrow \cd^b_{Z}\bigl(\Coh(\XX)\bigr)$ is an equivalence of categories. So without loss of generality, we may
 assume that $\kC^\bullet$ is a bounded complex of objects of $\Coh_Z(\XX)$.
Let $\kI=\kI_{Z}$ be the ideal sheaf of  $Z$. Then there is a $t \ge 1$ such that
$\kI^t$ annihilates every term of  $\kC^\bullet$. Consider the ringed space
$\ZZ = \bigl(Z, \kA/\kI^t)$. Then we have a morphism of ringed spaces $\eta \colon \ZZ \rightarrow \XX$ and
an adjoint pair
\[
\left\{
\begin{array}{ll}
\eta_* = \mathsf{forget}\colon             & D^-\bigl(\Coh(\ZZ)\bigr) \rightarrow D^-\bigl(\Coh(\XX)\bigr) \\
\eta^* = \kA/\kI^t \otimes_\kA \,-\, \colon & D^-\bigl(\Coh(\XX)\bigr) \rightarrow D^-\bigl(\Coh(\ZZ)\bigr).
\end{array}
\right.
\]
By our choice of $t$, there exists $\kE^\bullet \in \cd^b\bigl(\Coh(\ZZ)\bigr)$ such that
$\kC^\bullet = \eta_*(\kE^\bullet)$. Moreover, we have an isomorphism
$
\gamma \colon \Hom_{\ZZ}\bigl(\eta^*\kP^\bullet, \kE^\bullet\bigr) \lar
\Hom_{\XX}\bigl(\kP^\bullet, \eta_*(\kE^\bullet)\bigr)
$
such that for $\psi \in \Hom_{\ZZ}\bigl(\eta^*\kP^\bullet, \kE^\bullet\bigr)$ the corresponding
morphism $\varphi = \gamma(\psi)$ fits into the commutative diagram
\[
\begin{xy}\SelectTips{cm}{}
\xymatrix{
\kP^\bullet \ar[rr]^{\xi_{\kP^\bullet}} \ar[rd]_{\varphi} & & \eta_* \eta^* \kP^\bullet  \ar[ld]^{\eta_*(\psi)} \\
& \eta_* \kE^\bullet
}
\end{xy}
\]
where $\xi \colon \mathbbm{1}_{D^-(\XX)} \rightarrow \eta_* \eta^*$ is the unit of adjunction.
Thus it is sufficient to find a factorization of the morphism $\xi_{\kP^\bullet}$ through
an object of $\cP_Z(X)$.

By Definition of $\cP(X)$, there exists a bounded complex of locally free
$\kO_X$--modules $\kR^\bullet$ such that $\kP^\bullet \cong \kF \otimes_X \kR^\bullet$ in
$\cd^b\bigl(\Coh(\XX)\bigr)$. Note
that we have the following commutative diagram in the category $\mathsf{Com}^b(\XX)$
of bounded complexes of coherent left $\kA$--modules:
\[
\begin{xy}\SelectTips{cm}{}
\xymatrix
{
\kF \otimes_X \kR^\bullet \ar[rr]^{\mathbbm{1} \otimes \theta_{\kR^\bullet}} \ar[rd]_{\zeta_{\kR^\bullet}} & & \kF \otimes_X \bigl(\kO/\kI^t \otimes_X \kR^\bullet\bigr) \ar[ld]^{\cong} \\
& \kA/\kI^t \otimes_\kA \bigl(\kF \otimes_X \kR^\bullet\bigr) &
}
\end{xy}
\]
where $\zeta_{\kR^\bullet} = \xi_{\kP^\bullet}$ in $D^-\bigl(\Coh(\XX)\bigr)$ and
$\theta_{\kR^\bullet} \colon \kR^\bullet \rightarrow \kO/\kI^t \otimes_X \kR^\bullet$ is the canonical map.
Since any coherent sheaf on $X$ is a quotient of a locally free sheaf. Let $\pi\colon \kK \rightarrow \kI^t$ be an epimorphism with $\kK$ locally free of rank $n$ and $\iota\colon \kI^t \rightarrow \kO$ be the embedding. Let $d_{1}=\iota\pi\colon \kK \rightarrow \kO$. The Koszul complex $\kK^{\bullet}$ associated to $d_{1}$ is defined as
\[\kK^\bullet=(0 \rightarrow \Lambda^n \kK \xrightarrow{d_{n}} \Lambda^{n-1}\kK \xrightarrow{d_{n-1}} \cdots \xrightarrow{d_{2}} \Lambda^1 \kK \xrightarrow{d_{1}} \kO \rightarrow 0)\] with differentials \[d_{p}(t_{1} \wedge \ldots \wedge t_{p})= \sum_{j=1}^p (-1)^{j-1} d_{1}(t_{j}) t_{1} \wedge \ldots \wedge \widehat{t_{j}} \wedge \ldots \wedge t_{p}.\] Let $\kO$ be concentrated in degree $0$ in $\kK^\bullet$. By definition $\kK^\bullet$ is a complex of locally free coherent $\kO$-modules with $\kH^0(\kK^\bullet) \cong \kO/\kI^t$. Moreover, using for example \cite[Proposition IV.2.1. c)]{FultonLang} $\kK^{\bullet}$ is exact outside $\Supp(\kO/\kI^t) \subseteq Z$. In particular the cohomologies of 
$\kK^\bullet$ are supported in $Z$.
Hence we have a factorization of the canonical morphism
$\kO \rightarrow \kO/\kI^t$ in the category of complexes  $\mathsf{Com}^b\bigl(\Coh(X)\bigr)$:
$\kO[0] \rightarrow \kK^\bullet \rightarrow \kO/\kI^t[0]$, which induces  a factorization
\[
\kR^\bullet \lar \kK^\bullet \otimes_X \kR^\bullet \lar \kO/\kI^t \otimes_X \kR^\bullet
\]
of the canonical map $\theta_{\kR^\bullet}$. The complex
$\kK^\bullet \otimes_X \kR^\bullet$ is perfect since both $\kK^{\bullet}$ and $\kR^\bullet$ are perfect. Moreover its cohomology is supported at $Z$. This may be checked on an open affine cover $\{U_{i}\}$ where $\kR^{\bullet}|_{U_{i}}$ consists of free $\kO_{U_{i}}$-modules. Hence we get
the factorization of the adjunction unit  $\xi_{\kP^\bullet}$
\[
\kP^\bullet \cong \kF  \otimes_X \kR^\bullet \lar \kQ^\bullet := \kF \otimes_X \bigl(\kK^\bullet \otimes_X \kR^\bullet\bigr)
 \lar \kA/\kI^t \otimes_\kA \bigl(\kF \otimes_X \kR^\bullet\bigr) \cong \kA/\kI^t \otimes_\kA \kP^\bullet
\]
 This concludes the proof.
\end{proof}

\begin{thm}\label{T:keylocaliz}
Using the notations and assumptions from Theorem \ref{t:main-global}, there is an equivalence of triangulated categories
\[
\HH^\omega\colon \quad
\left(
\frac{\cd^b_Z\bigl(\Coh(\XX)\bigr)}{\cP_Z(X)}
\right)^\omega
\lar
\left(
\frac{\cd^b\bigl(\Coh(\XX)\bigr)}{\cP(X)}
\right)^\omega
\]
\end{thm}

\begin{proof}
Proposition \ref{P:FullyFaithf} implies that the functor $\HH^\omega$ is fully faithful. Hence
we have to show it is essentially surjective, which we deduce from the following

\vspace{1mm}
\noindent
\underline{Statement}. For any
$ \kM^\bullet \in \cd^b\bigl(\Coh(\XX)\bigr)/{\cP(X)}
$ there exist  $ \widetilde\kM^\bullet \in
{\cd^b\bigl(\Coh(\XX)\bigr)}/{\cP(X)}
$  and $
\kN^\bullet \in \cd^b_Z\bigl(\Coh(\XX)\bigr)/{\cP_Z(X)}
$ such that  $\kM^\bullet \oplus \widetilde\kM^\bullet \cong \HH(\kN^\bullet).
$

\vspace{1mm}
\noindent
Indeed any object in the idempotent completion $(\cd^b(\Coh(\XX))/{\cP}(X))^{\omega}$ has the form $(\kM^\bullet, e_{\kM^\bullet})$, where $\kM^\bullet$ is an object in $(\cd^b(\Coh(\XX))/{\rm P}(X))$ and $ e_{\kM^\bullet}$ is an idempotent endomorphism. The statement above yields $\widetilde\kM^{\bullet}$ in $\cd^b(\Coh(\XX))/\cP(X)$ and $\kN^\bullet$ in $\cd^b_{Z}(\Coh(\XX))/\cP_{Z}(X)$ with $\HH(\kN^{\bullet}) \cong \kM^\bullet \oplus \widetilde\kM^\bullet$. The idempotent $e_{\kM^\bullet}$ induces an idempotent endomorphism
\[e_{\kM^\bullet \oplus \widetilde\kM^\bullet} \colon \kM^\bullet \oplus \widetilde\kM^\bullet \xrightarrow{\begin{pmatrix} e_{\kM} & 0 \\ 0 & 0 \end{pmatrix}} \kM^\bullet \oplus \widetilde\kM^\bullet\]
and one can check that $(\kM^\bullet \oplus \widetilde\kM^\bullet, e_{\kM^\bullet \oplus \widetilde\kM^\bullet}) \cong (\kM^\bullet, e_{\kM^\bullet})$ holds in the idempotent completion. Since $\HH$ is fully faithful there exists an idempotent endomorphism $e_{\kN^\bullet}$ of $\kN^\bullet$ with $\HH(e_{\kN^\bullet})=e_{\kM^\bullet \oplus \widetilde\kM^\bullet}$. Using the considerations above this yields $\HH^{\omega}((\kN^\bullet, e_{\kN^\bullet})) \cong (\kM^{\bullet}, e_{\kM^\bullet})$, which completes the argument.

We turn to the proof of the statement above. Let $U=X \setminus Z$ be the complement of $Z$ and $\UU=(U, \kA\big|_{U})$ the corresponding (non-commutative) ringed space. The restriction defines an exact functor $\iota^*\colon \Coh(\XX) \rightarrow \Coh(\UU)$ whose kernel is the Serre subcategory $\Coh_{Z}(\XX)$. The universal property of Serre quotient categories yields an exact functor  $\overline{\iota^*}\colon \Coh(\XX)/\Coh_{Z}(\XX) \rightarrow \Coh(\UU)$ such that the following diagram commutes
\[
\xymatrix{
\Coh(\XX) \ar[rd]^\pi \ar[rr]^{\iota^*} && \Coh(\UU) \\
& \frac{\displaystyle \Coh(\XX)}{\displaystyle \Coh_{Z}(\XX)} \ar[ru]^{\overline{\iota^*}} 
}
\]
\cite[Example III.5. a)]{Gabriel62} shows that $\overline{\iota^*}$ is an equivalence. Since all functors in the diagram above are exact we can pass directly to the derived categories and obtain a commutative diagram there. We use the same symbols for the functors in this diagram. The triangulated functor $\pi\colon \cd^b(\Coh(\XX)) \rightarrow \cd^b(\Coh(\XX)/\Coh_{Z}(\XX))$ has kernel $\cd^b_{Z}(\Coh(\XX))$. Thus the universal property of Verdier quotient categories yields a commutative diagram
\[
\xymatrix{
\cd^b(\Coh(\XX)) \ar[rr]^\pi \ar[rd]^{\PP} && \cd^b\left(\frac{\displaystyle \Coh(\XX)}{\displaystyle \Coh_{Z}(\XX)}\right) \\
&  \frac{\displaystyle \cd^b(\Coh(\XX))}{\displaystyle \cd^b_{Z}(\Coh(\XX))} \ar[ru]^{\overline{\pi}} 
}
\]
where $\PP$ is the canonical projection and $\overline{\pi}$ is an equivalence by Miyachi's theorem \cite[Theorem 3.2.]{Miyachi}.
Summing up, we have the following diagram of categories and functors
\[
\begin{xy}\SelectTips{cm}{}
\xymatrix{
\frac{\displaystyle \cd^b\bigl(\Coh(\XX)\bigr)}{\displaystyle \cd^b_Z\bigl(\Coh(\XX)\bigr)}
\ar[rrrr]^{\overline{\pi}} &&&&
\cd^b\left(\frac{\displaystyle \Coh(\XX)}{\displaystyle \Coh_Z(\XX)}\right) \ar[dd]^{\overline{\iota^*}} \\
\cd^b\bigl(\Coh(\XX)\bigr) \ar[u]^{\PP} & & \\
\Perf(X) \ar[u]^{\kF \lten_X \,-\,} \ar[rr]^{\iota^*} && \Perf(U)
\ar[rr]^{\kF\big|_{U} \lten_U \,-\,}  && \cd^b\bigl(\Coh(\mathbb{U})\bigr) \\
}
\end{xy}
\]
By the commutativity of the diagrams above the `upper composition' is just $\iota^* \circ ({\kF \lten_X \,-\,})$, whereas the `lower composition' is ${\iota^*(\kF) \lten_X \,\iota^*(-)\,}.$ 
Thus \cite[Formula (3.12)]{Huybrechts} both compositions $\Perf(X) \rightarrow \cd^b\bigl(\Coh(\mathbb{U})\bigr)$ are isomorphic.
The functor
\[\kF\big|_U \lten_U \,-\,\colon \quad
\Perf(U) = \cd^b\bigl(\Coh(U)\bigr) \lar  \cd^b\bigl(\Coh(\UU)\bigr)\]
has a left adjoint \[\mathbbm{R}{\mathcal Hom}_{\UU}(\kF\big|_{U}, -)\colon \cd^b(\Coh(\UU)) \lar \cd^b(\Coh(U)).\] 
We claim that the unit and counit of adjunction are equivalences. This may be checked locally, where it follows from Morita theory, since we assume the coherent sheaf $\kF\big|_U$ to be locally free.

Let $\kM^{\bullet}$ be an arbitrary object in $\cd^b(\Coh(\XX))/\cP(X)$ and recall that the functors $\overline{\pi}$, $\overline{\iota^*}$ and $\FF_{U}:=\kF\big|_{U} \lten_U \,-\,$ are equivalences of categories. Application of these functors to $\kM^\bullet$ yields a perfect complex $\kS^\bullet:=\FF_{U}^{-1} \overline{\iota^*} \overline{\pi} (\kM^\bullet)$ on $U$. By a result of Thomason and Trobaugh \cite[Lemma 5.5.1]{ThomasonTrobaugh}, there exist $\widetilde\kS^\bullet \in \Perf(U)$
and $\kR^\bullet \in \Perf(X)$ such that
$
\iota^* \kR^\bullet \cong \kS^\bullet \oplus \widetilde\kS^\bullet.
$
Now we can `go back' and define $\widetilde\kM^\bullet:=\overline{\pi}^{-1}(\overline{\iota^*})^{-1}\FF_{U}(\widetilde\kS^\bullet)$. By the commutativity of the big diagram above we obtain a chain of isomorphisms in $\cd^b(\Coh(\XX))/\cP(X)$
\[
\PP(\kF^\bullet \otimes \kR^{\bullet}) \cong \overline{\pi}^{-1}\bigl(\overline{\iota^*}\bigr)^{-1}\FF_{U} \iota^* (\kR^\bullet) \cong
\overline{\pi}^{-1}\bigl(\overline{\iota^*}\bigr)^{-1}\FF_{U}(\kS^\bullet \oplus \widetilde\kS^\bullet) \cong \kM^\bullet \oplus \widetilde\kM^\bullet
\]

The definition of (iso-)morphisms in a Verdier quotient category shows that the last statement is equivalent to the following fact. There exist  $\kT^\bullet \in  \cd^b\bigl(\Coh(\XX)\bigr)$
and a pair of distinguished triangles
\[
\kC_\xi^\bullet  \lar \kT^\bullet \stackrel{\xi}\lar  \kM^\bullet \oplus \widetilde\kM^\bullet \lar
\kC_\xi^\bullet[1] \quad
\mathrm{and}
\quad
\kC_\theta^\bullet  \lar \kT^\bullet \stackrel{\theta}\lar \kF^\bullet \otimes \kR^{\bullet}  \lar \kC_\theta^\bullet[1]
\]
in $\cd^b\bigl(\Coh(\XX)\bigr)$ such that
$\kC_\xi^\bullet$ and $\kC_\theta^\bullet$   belong to the category $\cd^b_Z\bigl(\Coh(\XX)\bigr)$.
Since $\kR^\bullet$ is perfect, $\kF^\bullet \otimes \kR^{\bullet}$ is an object in $\cP(X)$. The triangle on the right shows that $\kC_\theta^\bullet$ and $\kT^\bullet$ are isomorphic in the Verdier quotient
$\cd^b\bigl(\Coh(\XX)\bigr)/\cP(X)$. Now the left triangle above yields a distinguished triangle 
\[
\kC_\xi^\bullet \stackrel{\alpha}\lar \kC_\theta^\bullet \lar \kM^\bullet \oplus \widetilde\kM^\bullet \lar
\kC_\xi^\bullet[1]
\]
in  $\cd^b\bigl(\Coh(\XX)\bigr)/\cP(X)$. By Proposition \ref{P:FullyFaithf},
the functor $\HH\colon \, \cd^b_Z\bigl(\Coh(\XX)\bigr)/\cP_Z(X) \rightarrow
\cd^b\bigl(\Coh(\XX)\bigr)/\cP(X)$ is fully faithful, hence $\kM^\bullet \oplus \widetilde\kM^\bullet$
belongs to the essential image of $\HH$.
This concludes the proof of the statement.
\end{proof}
\begin{Setup}
Let us introduce some notations and state some assumptions which we use in the remainder of this section.
\begin{itemize} 
\item We assume that $X$ has only \emph{isolated singularities} and denote
$Z = \Sing(X) = \bigl\{x_1, \dots, x_p \bigr\}$. 

\item For any $1 \le i \le p$, let
$O_i := \kO_{x_i}$ with maximal ideal $\mathfrak{m}_i$, $A_i = \kA_{x_i}$
and $F_i = \kF_{x_i}$. 
\item Let $\widehat{O}_i = \varprojlim O_i/ \idm_i^t O_i$, $\widehat{A}_i:= \varprojlim A_i/ \idm_i^t A_i$ and $\widehat{F}_i:= \varprojlim F_i/ \idm_i^t F_i$ be the completions. 
In particular, $\widehat{A}_i \cong \End_{\widehat{O}_i}(\widehat{F}_i)$. 
\item For a local Noetherian ring $(R, \mathfrak{m})$, we denote by $R-\fdmod$ the category of $R$-modules of finite length\footnote{If $R$ is a $k$-algebra, with $k \cong R/\mathfrak{m}$, then $R-\fdmod$ is equivalent to the subcategory of those $R$-modules, which are finite dimensional over $k$. For example, this holds for local rings $\co_{x}$ of closed points $x \in X$, where $X$ is of finite type over $k$.}. Let $\Lambda$ be an $R$-algebra, which is finitely generated as an $R$-module. Then $\Lambda-\fdmod$ denotes the category of finitely generated left $\Lambda$-modules, which are contained in $R-\fdmod$ when considered as $R$-modules. 
\item Let $\Perf_{\rm fd}(O_{i})$ (respectively $\Perf_{\rm fd}(\widehat{O}_{i})$) denote the full subcategories of $\cd^b(O_{i}-\fdmod)$ (respectively $\cd^b(\widehat{O}_{i}-\fdmod)$) consisting of complexes which admit a bounded free resolution in $\cd^b(O_{i}-\mod)$ (respectively $\cd^b(\widehat{O}_{i}-\mod)$). 
\item Let $\cP_{i}$, $\widehat{\cP}_{i}$  respectively $\widetilde{\cP}_{i}$ be the essential images of the functors $F_{i}\otimes_{O_{i}} - \colon \Perf_{\rm fd}(O_{i}) \rightarrow \cd^b(A_{i}-\fdmod)$, $\widehat{F}_{i}\otimes_{\widehat{O}_{i}} - \colon \Perf_{\rm fd}(\widehat{O}_{i}) \rightarrow \cd^b(\widehat{A}_{i}-\fdmod)$ respectively $\widehat{F}_{i}\otimes_{\widehat{O}_{i}} - \colon \Perf(\widehat{O}_{i}) \rightarrow \cd^b(\widehat{A}_{i}-\mod)$.
\end{itemize}
\end{Setup}

Since $X$ has isolated singularities Theorem \ref{T:keylocaliz} yields the following block decomposition of the relative singularity category.

\begin{cor}\label{C:BlockDecomp}
There is an equivalence of triangulated categories \[\Delta_{X}(\XX)=\left(\frac{\cd^b(\Coh(\XX))}{\cP(X)}\right)^\omega \cong \prod_{i=1}^p \left(\frac{\cd^b(A_{i}-\fdmod)}{\cP_{i}} \right)^{\omega}.\]
\end{cor}
\begin{proof}
Consider the exact equivalence
\[\Coh_{Z}(\XX) \rightarrow \prod_{i=1}^p A_{i}-\fdmod\] given by $\kG \mapsto (\kG_{x_{1}}, \cdots, \kG_{x_{p}})$.   Thus we get an induced equivalence of the corresponding derived categories
$\cd^b(\Coh_{Z}(\XX)) \rightarrow \prod_{i=1}^p \cd^b(A_{i}-\fdmod)$ and using Lemma \ref{P:DercCatandSupp}  this gives an equivalence $\cd^b_{Z}(\Coh(\XX)) \rightarrow \prod_{i=1}^p \cd^b(A_{i}-\fdmod)$. In view of Theorem \ref{T:keylocaliz}, it remains to show that the full subcategory $\cP_{Z}(X)=\cP(X) \cap \cd^b_{Z}(\Coh(\XX))$ is identified with $\prod_{i=1}^p \cP_{i}$ under this equivalence. Up to isomorphism, $\cP_{Z}(X)$ consists of complexes with entries in $\add(\kF)$ and cohomologies supported on $Z$. In particular, all cohomologies are finite dimensional. Localizing in $x_{i}$ yields complexes with entries in $\add(F_{i})$ and finite dimensional cohomologies. But these complexes form precisely the subcategory $\cP_{i}$. The claim follows.  
\end{proof}

We show that the Verdier quotients $\cd^b(A_{i}-\fdmod)/\cP_{i}$ do not change when we pass to the completion.

\begin{lem}\label{L:Lemma2.9}
The completion functor induces an equivalence of triangulated categories
\[\frac{\cd^b(A_{i}-\fdmod)}{\cP_{i}} \rightarrow \frac{\cd^b(\widehat{A}_{i}-\fdmod)}{\widehat{\cP}_{i}}\]
\end{lem}
\begin{proof}
It suffices to show that the following diagram is commutative, the horizontal arrows define equivalences and the vertical arrows define full embeddings.
\[
\xymatrix{
\Perf_{fd}(O_{i}) \ar[d]_{F_{i} \otimes_{O_{i}} -} \ar[rr] && \Perf_{fd}(\widehat{O}_{i}) \ar[d]^{\widehat{F}_{i} \otimes_{\widehat{O}_{i}} -}\\
\cd^b(A_{i}-\fdmod) \ar[rr] && \cd^b(\widehat{A}_{i}-\fdmod)
}
\]
The horizontal functors are induced by the completion functors. The diagram is commutative since completion commutes with taking tensor products. The vertical arrows are full embeddings by Proposition \ref{Prop:Embedding}. 

The completion functors $O_{i}-\fdmod \ra \widehat{O}_{i}-\fdmod$ and $A_{i}-\fdmod \ra \widehat{A}_{i}-\fdmod$, are equivalences since we restrict to finite dimensional modules. This yields triangle equivalences 
$\cd^b(O_{i}-\fdmod) \ra \cd^b(\widehat{O}_{i}-\fdmod)$ and $\cd^b(A_{i}-\fdmod) \ra \cd^b(\widehat{A}_{i}-\fdmod)$
In particular, the restriction $\Perf_{fd}(O_{i}) \rightarrow \Perf_{fd}(\widehat{O}_{i})$ is fully faithful. In order to see that it is an equivalence, it remains to show that the completion of a non-perfect complexes cannot become perfect. This follows from the following characterization of perfect complexes. 

\underline{Claim.} Let $(R, \mathfrak{m})$ be a local ring. Then $X \in \cd^b(R-\mod)$ is a perfect complex if and only if $\Hom_{R}(X, R/\mathfrak{m}[n])=0$ for all large enough $n$.

The latter condition is indeed invariant under taking completions by the fully faithfulness of the completion functor.

Let us prove the claim. Assume that $X$ is not perfect. $X$ admits a free resolution $Y$, which is bounded below and has bounded cohomology. By \cite[Proposition 1.3.1]{BrunsHerzog}, we may replace $Y$ by a quasi-isomorphic free resolution $M$ which satisfies the following minimality condition $d_{M}^i(M^i) \subseteq \mathfrak{m}M^{i+1}$ for all small enough $i$. Then the following map of complexes
\[
\xymatrix
{
\cdots \ar[r] &  M^{i-1} \ar[d] \ar[r] & M^i \ar[d]\ar[r] & M^{i+1} \ar[d] \ar[r] & \cdots \\
\cdots \ar[r] & 0 \ar[r]& R/\mathfrak{m} \ar[r] &0 \ar[r]& \cdots
}
\]
is non-zero in the homotopy category for all small enough $i$. The other direction of the claim is clear.
\end{proof}

We need the following special case of Theorem \ref{T:keylocaliz}.
\begin{cor}\label{C:fdtofg}
The canonical functor
\[\left(\frac{\cd^b(\widehat{A}_{i}-\fdmod)}{\widehat{\cP}_{i}}\right)^{\omega} \rightarrow \left(\frac{\cd^b(\widehat{A}_{i}-\mod)}{\widetilde{\cP}_{i}}\right)^{\omega}\] is an equivalence of triangulated categories.
\end{cor}
\begin{proof}
Let $O=\widehat{O}_{i}$ and $A=\widehat{A}_{i}$.
Let $X=\Spec(O)$ and $Z=\{\mathfrak{m}_{i}\}\subseteq X$. $A$ is a finitely generated $O$-module. Thus the sheafification $\kA:=\widetilde{A}$ is a coherent $\kO_{X}$-algebra. Let $\XX=(X, \kA)$. The global section functor induces an equivalence of abelian categories 
\[\Gamma(-)\colon \Coh(\XX) \rightarrow A-\mod,\]
which identifies the full subcategory $\Coh_{Z}(\XX) \subseteq \Coh(\XX)$ with $A-\fdmod \subseteq A-\mod$. Now Theorem \ref{T:keylocaliz} proves the claim.  
\end{proof}

Using the results in this section, we are able to give a proof of Theorem \ref{t:main-global} above.

\begin{proof}
We have a chain of triangle equivalences
\begin{align}\label{E:KeyChain}
\begin{array}{cc}
\Delta_{X}(\XX)=\left(\frac{\displaystyle \cd^b(\Coh(\XX))}{\displaystyle \cP(X)}\right)^\omega &\cong {\displaystyle \prod_{i=1}^p } \left(\frac{\displaystyle \cd^b(A_{i}-\fdmod)}{\displaystyle \cP_{i}} \right)^{\omega} \\ &\cong {\displaystyle \prod_{i=1}^p }\left(\frac{\displaystyle \cd^b(\widehat{A}_{i}-\fdmod)}{\displaystyle \widehat{\cP}_{i}}\right)^{\omega} \\ &\cong {\displaystyle \prod_{i=1}^p }
\left(\frac{\displaystyle \cd^b(\widehat{A}_{i}-\mod)}{\displaystyle \widetilde{\cP}_{i}}\right)^{\omega}.
\end{array}
\end{align}
The first equivalence holds by Corollary \ref{C:BlockDecomp}. For the second equivalence, we pass to the idempotent completions in Lemma \ref{L:Lemma2.9} and the last equivalence is proved in Corollary \ref{C:fdtofg}. 
\end{proof}

\newpage

\section{Local relative singularity categories} \label{S:Local}

The first subsection is an extended version of a joint work with Igor Burban \cite{BurbanKalck11}. The remaining parts of this section do not depend on the first subsection and are based on a joint work with Dong Yang \cite{KalckYang12}. 

Subsection \ref{ss:NodalBlock} gives an explicit description of the relative singularity categories of the Auslander and cluster resolutions of simple hypersurface singularities of type $\mathbb{A}_{1}$ in Krull dimension one and zero. Using Theorem \ref{t:Classical-versus-generalized-singularity-categories} this yields descriptions of the relative singularity categories of Auslander resolutions of simple hypersurface singularities of type $\mathbb{A}_{1}$ in all Krull dimensions. In particular, we describe all the indecomposable objects and all morphism spaces between them. Moreover, using a tilting result of Burban \& Drozd \cite{Tilting} together with the localization property of the global relative singularity category (Theorem \ref{t:main-global}) yields a relation to gentle algebras.

In Subsections \ref{ss:fract-cy} and \ref{ss:Independance} the dg algebra techniques from Section \ref{s:dg} are applied to study Frobenius categories (see Section \ref{s:Frobenius}). The examples, which we have in mind are of finite representation type: namely, the category of maximal Cohen--Macaulay modules over ADE-singularities and the category of finite dimensional modules over finite dimensional \emph{representation-finite} selfinjective algebras.
The key statement is Theorem \ref{t:main-thm-2}, which describes the relative singularity category of an Auslander resolution as the perfect derived category of a certain dg algebra, which we call the \emph{dg Auslander algebra}. The proof relies on a fractional CY-property of certain simple modules over the Auslander algebra, which we explain in Subsection \ref{ss:fract-cy} and which is related to work of Keller \& Reiten \cite{KellerReiten07}. Thanhoffer de V\"olcsey and Van den Bergh \cite{ThanhofferdeVolcseyMichelVandenBergh10} obtained similar descriptions of relative singularity categories for cluster resolutions of certain Gorenstein quotient singularities. However, the only examples we have in common are the ADE-singularities in Krull dimension two. We refer to Remark \ref{R:Overlap} for more details.

We apply the techniques developed in Subsection \ref{ss:Independance}  and Section \ref{s:dg} to study  relative singularity categories over Gorenstein rings in Subsection \ref{S:MCM-over-Gor}. The main result shows that for Auslander resolutions of ADE-singularities the singularity category and the relative singularity category mutually determine each other.

Subsection \ref{ss:DGAuslander} contains a complete list of the dg Auslander algebras for ADE-singularities in all Krull dimensions.

In Subsection \ref{ss:Bridgeland}, we remark on the derived category $\cd_{fd}(B_{Q})$ of dg modules with finite dimensional total cohomology over the dg Auslander algebra $B_{Q}$ of an even dimensional singularity of Dynkin type $Q$. It is an intrinsically defined subcategory of the corresponding relative singularity category and Bridgeland has determined its stability manifold \cite{Bridgeland09}.

\subsection{An elementary description of the nodal block}\label{ss:NodalBlock}
The aim of this subsection is to give a description of the relative singularity category $\Delta_{X}(\XX)$, where $X$ is a reduced curve with only nodal singularities and $\XX=(X, \ca)$ is the non-commutative ringed space considered in Example \ref{E:AuslanderSheaf}, i.e. $\ca$ is the Auslander sheaf of $X$. By Theorem \ref{t:main-global}, it remains to describe the local block
\begin{align}\label{Eq:DeltaOnd}
\Delta_{O_{nd}}(A_{nd})=\frac{\cd^b(A_{nd}-\mod)}{K^b(\proj-O_{nd})} \cong \frac{\cd^b(A_{nd}-\mod)}{\thick(A_{nd}e)},
\end{align}
which we call the \emph{nodal block} and denote $\Delta_{nd}$.
Here, $O_{nd}=\widehat{\co}_{s}=k\llbracket x, y \rrbracket/(xy)$ and $A_{nd}=\widehat{\ca}_{s}=\End_{O_{nd}}(O_{nd} \oplus k\llbracket x \rrbracket \oplus k\llbracket y \rrbracket)$ is the Auslander algebra of $\MCM(O_{nd})$. Since $O_{nd}$ is Gorenstein and $A_{nd}$ has finite global dimension, the quotient category defined in \eqref{Eq:DeltaOnd} is idempotent complete by Proposition \ref{P:IdempCompl}.
It is convenient to write $A_{nd}$ as the arrow ideal completion of the path algebra of the following
quiver with relations $\vec{Q}_{nd}$, see \cite[Remark 1]{Tilting}
\begin{equation}\label{E:nodalquiver}
\begin{xy}\SelectTips{cm}{}
\xymatrix
{- \ar@/^/[rr]^{ \ro }  & &  \ast \ar@/^/[ll]^{ \lu }
 \ar@/_/[rr]_{ \ru }
 & &
\ar@/_/[ll]_{ \lo } +}\end{xy}  \qquad  \ru   \ro  = 0, \quad   \lu   \lo  = 0.
\end{equation}
This algebra belongs to the class of \emph{nodal} algebras. In particular, there is an explicit construction of all indecomposable objects in the homotopy category $K^b(\proj-A_{nd})$, by work of Burban \& Drozd \cite{Nodal}. Since $A_{nd}
$ has finite global dimension, we obtain a description of indecomposables in $\cd^b(A_{nd}-\mod)$. 

By definition, the objects in the quotient category $\Delta_{nd}$ are just the objects of the category $\cd^b(A_{nd}-\mod)$. However, some of the indecomposable objects in $\cd^b(A_{nd}-\mod)$ are isomorphic to zero or decomposable, when viewed as objects in the quotient. Discarding these indecomposables of $\cd^b(A_{nd}-\mod)$, we are able to describe the indecomposable objects of $\Delta_{nd}$.
 In order to describe morphisms in the quotient category, we either use a Lemma  of Verdier (see Lemma \ref{L:Verdier2}) to reduce to a computation in the homotopy category of projectives or a direct calculation with `roofs of morphisms' in the quotient category.
\subsubsection{Description of strings in $K^b(\proj-A_{nd})$}
 By work of Burban \& Drozd  \cite{Nodal}, the indecomposable objects in $K^b\bigl(\proj-A_{nd}\bigr)$ are explicitly known. They are either \emph{band} or \emph{string} objects.
We do not describe the band objects below, since they are contained in $K^b\bigl(\add(P_{*})\bigr)$ and thus are isomorphic to zero  in the quotient $\Delta_{nd}$. The string objects  can be described in the following way. Let $\ZZ\vec{A}_{\infty}^{\infty}$ be the oriented graph obtained by orienting the edges in a $\ZZ^2$--grid as indicated in Example \ref{E:Grid} below.
Let $\vec{\theta} \subseteq \ZZ \vec{A}_{\infty}^{\infty}$ be a finite oriented subgraph of type $A_{n}$ for a certain $n \in \mathbb{N}$.
 Let $\Sigma$ and $T$ be the terminal vertices of $\vec\theta$ and $\sigma, \tau \in \{\li, *, \re\}$. We insert the projective
 modules  $P_{\sigma}$ and  $P_{\tau}$ at the vertices $\Sigma$ and $T$ respectively. Next, we plug
 in $P_{*}$ at all intermediate vertices of $\vec\theta$.
  Finally, we put maps (given by multiplication with non-trivial paths in $\vec{Q}_{nd}$) on the arrows between the corresponding indecomposable projective modules. This has to be done
  in such a way that the composition of two subsequent arrows is always zero. Additionally, at the vertices
  where $\vec\theta$ changes orientation, the inserted paths  have to  be `alternating'. This means that if one adjacent path
  involves $\alpha$ or $\beta$ then the second should involve $\gamma$ or $\delta$.
    Taking a direct sum of modules and maps in every  column of the constructed diagram,
     we get a complex of projective $A_{nd}$--modules $\kS$, which is called \emph{string}.
\begin{ex}\label{E:Grid} Let $k, l, n, m$ be integers $\geq 1$. The following picture illustrates the construction of a string in $K^b(\proj-A_{nd})$.
\[
\begin{xy}\SelectTips{cm}{}
\xymatrix@!=1pc
{
& & \circ \ar[rd] & & P_{\li} \ar[rd]^{\cdot \lu(\ro\lu)^n} && \circ \\
\ar[ddd] &  \circ \ar[ru] \ar[rd]&&  \circ \ar[ru] \ar[rd] && P_{*} \ar[ru] \ar[rd]     \\
&&   \circ \ar[ru] \ar[rd] && P_{*} \ar[ru]^{\cdot(\lo\ru)^m} \ar[rd] && \circ  \\
&   \circ \ar[ru] \ar[rd]&&  P_{*} \ar[ru]^{\cdot(\ro\lu)^l} \ar[rd]^{\cdot(\lo\ru)^k} && \circ \ar[ru] \ar[rd]   \\
&& \circ \ar[ru]  & & P_{*} \ar[ru] &&\circ  \\
\kS = &\cdots \ar[r] &0\ar[r]&P_{*} \ar[r]^{d_{1}} & { \underset{\displaystyle P_{*}^{\oplus 2}}{\overset{\displaystyle P_{-}}{\oplus}}}\ar[r]^{d_{2}} &
P_{*} \ar[r] & 0 \ar[r] &\cdots }\end{xy}
\]
where $d_{1}=\begin{pmatrix} 0 & \cdot(\ro\lu)^l & \cdot(\lo\ru)^k\end{pmatrix}^{\mathsf{tr}}$ and $d_{2}=\begin{pmatrix} \cdot\lu(\ro\lu)^n & \cdot(\lo\ru)^m & 0 \end{pmatrix}.$
\end{ex}

In order to state our classification of indecomposable objects in $\Delta_{nd}$, we need to introduce the following family of strings in $K^b(\proj-A_{nd})$.

\begin{defn} \label{D:MinS}
Let $\sigma, \tau \in \{\li, \re\}$ and $l \in \mathbb{N}$. A \emph{minimal string} $\mathcal{S}_{\tau}(l)$ is a complex of indecomposable projective $A_{nd}$--modules
\[
\begin{xy}\SelectTips{cm}{}
\xymatrix{
\cdots\ar[r] & 0 \ar[r] & P_{\sigma} \ar[r] & P_{\ast} \ar[r] & \cdots \ar[r] &P_{\ast} \ar[r] &P_{\tau} \ar[r]& 0 \ar[r] &\cdots}
\end{xy}
\]
of length $l+2$ with differentials given by non-trivial paths of minimal possible
length and $P_{\tau}$ located in degree $0$. Note, that $\sigma$ is uniquely determined by $\tau$ and $l$:
\[\begin{cases} \sigma=\tau \quad \text{  if } l \text{ is even,}   \\  \sigma \ne \tau \quad \text{  if } l  \text{ is odd.} \end{cases}\]
\end{defn}

\begin{ex} The two complexes depicted below are minimal strings:
\begin{itemize}
\item $\mathcal{S}_{\re}(1)=\begin{xy}\SelectTips{cm}{}\xymatrix{
\cdots \ar[r]& 0  \ar[r] & \, 0 \, \,  \ar[r] & P_{\li} \ar[r]^{\cdot  \lu } & P_{\ast} \ar[r]^{\cdot  \lo }  &P_{\re} \ar[r]& 0 \ar[r] &\cdots}\end{xy}$
\item $\mathcal{S}_{\re}(2)=\begin{xy}\SelectTips{cm}{}\xymatrix{
\cdots\ar[r] & 0 \ar[r] & P_{\re} \ar[r]^{\cdot  \ru } & P_{\ast} \ar[r]^{\cdot  \ro  \lu } &P_{\ast} \ar[r]^{\cdot  \lo }  &P_{\re} \ar[r]& 0 \ar[r] &\cdots}\end{xy}$
\end{itemize}
\end{ex}

We show that minimal strings  remain indecomposable when viewed as objects in the quotient $\Delta_{nd}$. First, we need the following Lemma due to Verdier, which plays an important role in the sequel, see \cite[Proposition II.2.3.3]{Verdier}.

\begin{lem} \label{L:Verdier2}
Let $\kT$ be a triangulated category and let $\kU \subseteq \kT$ be a full triangulated subcategory. Let $Y$ be an object in $^{\perp}\kU = \bigl\{T \in \kT \big| \Hom_{\kT}(T, \kU)=0\bigr\}$ and let
\[
\PP\colon  \Hom_{\kT}(Y, X) \longrightarrow \Hom_{\kT/\kU}(Y, X)
\]
be the map induced by the localization functor.
Then $\PP$ is bijective for all $X$ in $\kT$.
\end{lem}

\noindent
There is a dual result for $Y$ in $\kU^{\perp}$.

\begin{lem}\label{L:MinS}
Let $\tau \in \{+, -\}$ and $l \in \mathbb{N}$. Then any minimal string
$\kS=\kS_{\tau}(l)$ belongs to  $^\perp\hspace{-1pt}K^b\bigl(\add(P_{*})\bigr)  \cap K^b\bigl(\add(P_{*})\bigr)^\perp$. Moreover,  $\kS$ is indecomposable in $\Delta_{nd}.$
\end{lem}
\begin{proof} 
Using the long exact Hom-sequence, it suffice to show that
\[\Hom_{K^b(\proj-A_{nd})}\bigl(P_{*}[m], \kS\bigr)=0=\Hom_{K^b(\proj-A_{nd})}\bigl(\kS, P_{*}[m]\bigr) \] holds for all $m \in \ZZ$. This is a direct computation.

Let us prove the second statement.
Since $\kS$ is a string object it is 
 indecomposable in $K^b(\proj-A_{nd})$. Since $A_{nd}$ is finitely generated as an $O_{nd}$-algebra and $O_{nd}
$ is complete, it follows that $K^b(\proj-A_{nd})$ is a Krull--Remak--Schmidt category, see \cite[Appendix A]{Nodal}. In particular, the endomorphism ring $\End_{K^b(\proj-A_{nd})}(\kS)$ is local.
 Lemma \ref{L:Verdier2} yields an algebra isomorphism
\begin{align*}
\End_{\Delta_{nd}}(\kS) \cong \End_{K^b(\proj-A_{nd})}(\kS).
\end{align*}
This shows that $\kS$ has a local endomorphism ring in $\Delta_{nd}$, which implies indecomposability.
\end{proof}

\begin{rem}\label{R:Tria}
The projective resolutions of the simple $A_{nd}$-modules $S_{+}$ and $S_{-}$ are
\[
0 \rightarrow  P_{-} \stackrel{\cdot \beta}\longrightarrow  P_{*} \stackrel{\cdot \gamma}\longrightarrow  P_{+} \rightarrow S_{+} \rightarrow  0 \quad \text{and} \quad  0 \rightarrow  P_{+} \stackrel{\cdot \delta}\longrightarrow  P_{*} \stackrel{\cdot \alpha}\longrightarrow  P_{-} \rightarrow  S_{-} \rightarrow 0.
\]
Thus in the derived category $S_{\pm} \cong \mathcal{S}_{\pm}(1)$ are isomorphic to minimal strings. Let $\rho, \sigma, \tau \in \{\li, \re\}$ and $l \in \mathbb{N}$. The cone of
\[
\begin{xy}\SelectTips{cm}{}
\xymatrix@C=1.5pc{
\kS_{\tau}(l)  &&&0 \ar[r]& P_{\sigma} \ar[r]^{\cdot d_{3}} \ar[d]^{\mathsf{id}}& P_{*} \ar[r]^{\cdot d_{4}} &\cdots \ar[r]^{\cdot d_{l+2}}& P_{*} \ar[r]^{\cdot d_{l+3}} & P_{\tau} \ar[r] & 0 \\
\kS_{\sigma}(1)[l+1] & 0 \ar[r] & P_{\rho} \ar[r]^{\cdot d_{1}} & P_{*} \ar[r]^{\cdot d_{2}} & P_{\sigma} \ar[r] & 0
} \end{xy}
\]
 is isomorphic to the following minimal string
\[
\begin{xy}\SelectTips{cm}{}
\xymatrix{
\kS_{\tau}(l+1)[1] &0 \ar[r]& P_{\rho} \ar[r]^{\cdot d_{1}} & P_{*} \ar[r]^{\cdot d_{2}d_{3}} &\cdots \ar[r]^{\cdot d_{l+2}}& P_{*} \ar[r]^{\cdot d_{l+3}} & P_{\tau} \ar[r] & 0. 
}\end{xy}
\]
Hence the minimal strings are generated by $S_{+}$ and $S_{-}$. In other words, $\kS_{\tau}(l)[n]$ is contained in  $\thick(S_{+}, S_{-}) \subseteq \cd^b(A_{nd}-\mod),$ for all $\tau \in \{\li, \re\}$, $l \in \mathbb{N}$ and $n \in \ZZ$.
\end{rem}

\subsubsection{Main result}

For each $n \in \mathbb{Z}$ define a bijection $\delta_{n}\colon \{-, +\} \ra \{-, +\}$ as follows: $\delta_{n}=\mathrm{id}$ if $n \equiv 0 \, \mod \, 2$ and $\delta_{n} \ne \mathrm{id}$ otherwise.
 
The following theorem is the main result of this subsection.

\begin{thm}\label{T:MainT} We use the notations from above.

\noindent
(a) Let $X$ be an indecomposable complex in $K^b\bigl(\proj-A_{nd}\bigr).$ Then the  image of $X$ in $\Delta_{nd}$ is either zero or isomorphic to one of the following objects
 \[ P_{\sigma}[n] \oplus P_{\tau}[m], \quad P_{\tau}[n] \quad {or} \quad \kS_{\tau}(l)[n], \quad
\text{ where } m, n \in \mathbb{Z},  l \in \mathbb{N} \text{ and } \sigma, \tau \in \{+, -\}.
 \]

\noindent
(b) Let $\sigma, \tau, \mu \in \{+, -\}$,  $n \in \ZZ$ and $l, l' \in \mathbb{Z}_{>0}$. We have the following isomorphisms:
\begin{align}\label{E:HomProjProj}
\Hom_{\Delta_{nd}}\bigl(P_{\mu}, P_{\tau}[n]\bigr) \cong \begin{cases} k \quad \text{  if }n \leq 0 \text{ and } \mu=\delta_{n}(\tau), \\ 0 \quad \text{  otherwise.}\end{cases}
\end{align}

\begin{align}\label{E:HomFromProj}
\Hom_{\Delta_{nd}}\bigl(P_{\mu}[n], \mathcal{S}_{\tau}(l)\bigr)\cong \begin{cases} k \quad \text{  if } 0 \leq n < l \text{ and } \mu=\delta_{n}(\tau),  \\ 0 \quad \text{  otherwise.}\end{cases}
\end{align}

\begin{align}\label{E:HomToProj}
\Hom_{\Delta_{nd}}\bigl( \mathcal{S}_{\tau}(l), P_{\mu}[n]\bigr)\cong \begin{cases} k \quad \text{  if } 2 \leq n \leq l+1 \text{ and } \mu \ne \delta_{n}(\tau),  \\ 0 \quad \text{  otherwise.}\end{cases}
\end{align}

\begin{align} \label{E:HomStringString}
\Hom_{\Delta_{nd}}\bigl( \mathcal{S}_{\tau}(l), \mathcal{S}_{\mu}(l')[n]\bigr)\cong \begin{cases} k \quad \text{  if }  n \leq 0; \, l \geq l'+n \geq 1 \text{ and }\mu = \delta_{n}(\tau),  \\ 
 k \quad \text{  if }  n \geq 2;\, l' \geq l+2-n \geq 1 \\ \quad \, \, \,  \text{ and }\mu \ne \delta_{n}(\tau),  \\
0 \quad \text{  otherwise.}\end{cases}
\end{align}

\noindent
(c) The set of indecomposable objects in $\Delta_{nd}$ is given by \[\bigl\{P_{\sigma}[n], \kS_{\tau}(l)[m] \, \big| \, \sigma, \tau \in \{+, -\}, \, n, m \in \ZZ, \, l \in \mathbb{N}\bigr\}.\] Moreover, the objects from this set  are pairwise non-ismorphic   in $\Delta_{nd}$.
\end{thm}
\begin{proof}
Note, that the strings with $P_{\sigma}=P_{\tau}=P_{*}$ vanish in $\Delta_{nd}$. Therefore,
in what follows we may and shall  assume that $\sigma$ or $\tau \in \{\li, \re\}$.

\medskip
\noindent
\textit{Proof of (a).}
Let $\kS \in K^b(\proj-A_{nd})$ be a string as defined above.

\noindent
1. If $P_{\tau}=P_{*}$  and $ \vec\theta=\Sigma \rightarrow \cdots$ hold, then there exists a distinguished triangle
\[
\kS \stackrel{f}\lar   P_{\sigma}[n] \lar \mathsf{cone}(f) \lar \kS[1]
\]
with $\mathsf{cone}(f) \in K^b\bigl(\add(P_{*})\bigr)$, yielding an isomorphism $\kS \cong P_{\sigma}[n]$ in $\Delta_{nd}$. Similarly, if $\vec\theta = \Sigma \leftarrow \cdots$ holds, then we obtain a triangle
$
P_{\sigma}[n] \stackrel{f}\lar    \kS \lar \mathsf{cone}(f) \lar  P_{\sigma}[n+1]
$
with $\mathsf{cone}(f) \in K^b\bigl(\add(P_{*})\bigr)$ and hence an isomorphism $\kS \cong P_{\sigma}[n]$ in $\Delta_{nd}.$

\medskip
\noindent
2. We may assume that $\sigma, \tau \in \{\li, \re\}$. If the graph $\vec{\theta}$ defining
$\kS$ is \emph{not} linearly oriented (i.e.~contains a subgraph ($\star$) $\begin{xy}\SelectTips{cm}{}\xymatrix{\circ \ar[r] & \circ & \circ \ar[l]}\end{xy}$ or ($\star\star$) $\begin{xy}\SelectTips{cm}{}\xymatrix{\circ & \circ \ar[r] \ar[l] &\circ}\end{xy}$),
then there exists a distinguished triangle  of the following form
\[
\begin{xy}\SelectTips{cm}{}
\xymatrix{
(\star) & P_{*}[s] \ar[r] &\kS \ar[r]& \kS' \oplus \kS'' \ar[r]& P_{*}[s+1]&  \\
(\star\star) & P_{*}[s-1] \ar[r] &\kS' \oplus \kS'' \ar[r] & \kS \ar[r] & P_{*}[s] }\end{xy}\]
and therefore $\kS \cong \kS' \oplus \kS'' \cong P_{\sigma}[n] \oplus P_{\tau}[m]$ is decomposable in $\Delta_{nd}$.

\medskip
\noindent
3. Hence without loss of generality, we may assume
$\sigma, \tau \in \{\li, \re\}$ and $\vec{\theta}$ to be  linearly oriented.
If $\kS$ has a `non-minimal' differential $d=\cdot p$ (i.e. the path $p$ in $\vec{Q}_{nd}$ contains $\ru \lo$ or $\lu \ro$ as a subpath), then we consider the following morphism of complexes
\[\begin{xy}\SelectTips{cm}{}\xymatrix{\kS' \ar[d]_f & P_{\sigma} \ar[r] &P_{*} \ar[r] & \cdots \ar[r] & P_{*}  \ar[d]^d\\
\kS'' &&&& P_{*} \ar[r] & P_{*} \ar[r] & \cdots \ar[r] &P_{*} \ar[r] & P_{\tau}}\end{xy}\]
which can be  completed to a distinguished triangle in $K^b\bigl(\proj-A_{nd}\bigr)$
\[
\kS' \stackrel{f}\lar  \kS'' \lar \kS \lar \kS'[1].
\] By our assumption on $d$,  the morphism $f$ factors through $P_{*}[s]$ for some $s \in \ZZ$ and therefore vanishes in $\Delta_{nd}$. Hence we have  a decomposition $\kS \cong \kS'[1] \oplus \kS'' \cong P_{\sigma}[n] \oplus P_{\tau}[m]$ in $\Delta_{nd}$.

\medskip
\noindent
4. If $\sigma, \tau \in \{\li, \re\}$, $\vec{\theta}$ is linearly oriented and $\kS$ has only minimal differentials, then $\kS$ is a minimal string. This  concludes the proof of part (a) of Theorem \ref{T:MainT}.

\noindent
\textit{Proof of (b).} 
We begin with the isomorphism \eqref{E:HomProjProj}.
Every morphism $P_{\sigma} \rightarrow P_{\tau}[n]$ in $\Delta_{nd}$ is given by a roof $P_{\sigma} \xleftarrow{f} Q \xrightarrow{g} P_{\tau}[n]$, where $f, g$ are morphisms in $K^b(\proj-A_{nd})$ and $\mathsf{cone}(f) \in K^b\bigl(\add (P_{*})\bigr)$. By a common abuse of terminology, we call $f$ a \emph{quasi-isomorphism}. Our aim is to find a convenient representative in each equivalence class of roofs. It turns out that $\sigma \in \{+, -\}$ and $n \in \ZZ$ determine $Q$ and $f$ of our representative and $g$ is either $0$ or determined by $\tau$ up to a  scalar. The proof is devided into several steps.

\medskip
\noindent
1. Without loss of generality, we may assume that $Q$ has no direct summands from
 $K^b\bigl(\add (P_{*})\bigr)$. Indeed, if
 $Q \cong Q' \oplus Q''$ with $Q'' \in K^b\bigl(\add (P_{*})\bigr)$, then the diagram
\[
\begin{xy}\SelectTips{cm}{}
\xymatrix{&&Q' \ar[rrd]^{g'} \ar[lld]_{f'} \ar[d]^{\left(\bsm \mathsf{id} \\ 0 \esm\right)} \\
P_{\sigma} && Q' \oplus Q'' \ar[ll]_{f=\left(\bsm f' & f'' \esm\right)} \ar[rr]^{g=\left(\bsm g' & g'' \esm\right)} && P_{\tau}[n]}
\end{xy}
\]
yields  an equivalence of roofs $P_{\sigma} \xleftarrow{f} Q \xrightarrow{g} P_{\tau}[n]$ and
$P_{\sigma} \xleftarrow{f'} Q' \xrightarrow{g'} P_{\tau}[n]$.

\medskip
\noindent
2. Using our assumptions on $f$ and $Q$ in conjunction with the description of indecomposable strings in $K^b(\proj-A_{nd})$, it is not difficult to see that  $Q$  can (without restriction) taken to be an indecomposable string with $\tau=*$ and $f$ to be of the following form:
\[
\begin{xy}\SelectTips{cm}{}
\xymatrix@R=0.3pc{
&\cdots \ar[r] & P_{*}  \ar[r]& P_{*} \ar[r] & P_{*} \ar[rdd] \\ Q \ar[ddd]_f \\
&P_{\sigma} \ar[dd]_{\mathsf{id}} \ar[r] & P_{*} \ar[r] & \cdots \ar[r] & P_{*} \ar[r] & P_{*} \\ \\
P_{\sigma }&P_{\sigma}}
\end{xy}
\]

\noindent
3.  Without loss of generality, we may assume that $Q$ is constructed from a linearly oriented graph $\vec{\theta}$. Indeed, otherwise we may consider the truncated complex $Q^{\leq}$ defined in the diagram below and replace our roof by an equivalent one.
\[
\begin{xy}\SelectTips{cm}{}
\xymatrix@R=0.3pc{
&\cdots \ar[r]& P_{*} \ar[r]& P_{*} \ar[r] & P_{*} \ar[rdd] \\
Q \\
&P_{\sigma} \ar[r]^{d_{1}} & P_{*} \ar[r]^{d_{2}} & \cdots \ar[r]^{d_{n-1}} & P_{*} \ar[r]^{d_{n}} & P_{*} \\ \\
Q^{\leq} \ar[uuu]^q&P_{\sigma} \ar[uu]^{\mathsf{id}}\ar[r]^{d_{1}} & P_{*} \ar[r]^{d_{2}} \ar[uu]^{\mathsf{id}} & \cdots \ar[r]^{d_{n-1}} & P_{*} \ar[r]^{d_{n}} \ar[uu]^{\mathsf{id}} & P_{*} \ar[uu]^{\mathsf{id}} }
\end{xy}
\]
\[
\begin{xy}\SelectTips{cm}{}
\xymatrix{&&Q^\leq \ar[rrd]^{gq} \ar[lld]_{fq} \ar[d]^{q} \\
P_{\sigma} &&Q \ar[ll]_{f} \ar[rr]^{g} && P_{\tau}[n]}
\end{xy}
\]
In particular,  $n >0 $ implies that $\Hom_{\Delta_{nd}}\bigl(P_{\sigma}, P_{\tau}[n]\bigr)=0$ holds.

\medskip
\noindent
4. By the above reductions, $g$ has the following form:
\[
\begin{xy}\SelectTips{cm}{}
\xymatrix{
Q \ar[d]^g &  P_{\sigma} \ar[r] & P_{*} \ar[r] & \cdots \ar[r] & P_{*} \ar[r] & P_{*} \ar[r]\ar[d]^{g} & P_{*} \ar[r] & \cdots \\
P_{\tau}[n] & &&&& P_{\tau}
}\end{xy}
\]
We may truncate again so that $Q$ ends at degree $-n$:
\[
\begin{xy}\SelectTips{cm}{}
\xymatrix{&&Q^{\leq -n} \ar[rrd] \ar[lld]  \\
P_{\sigma} &&Q \ar[ll] \ar[rr] \ar[u]&& P_{\tau}[n]}
\end{xy}
\]

\noindent
5. Next, we may  assume that $Q$ has minimal differentials (see  Definition \ref{D:MinS}) and thus is  uniquely determined by $\sigma$ and $n$. Indeed, otherwise there exists  a quasi-isomorphism:
\[
\begin{xy}\SelectTips{cm}{}
\xymatrix{
Q'  \ar[d]^q & P_{\sigma} \ar[r] \ar[d]^{\mathsf{id}}& P_{*} \ar[r] \ar[d]^{\mathsf{id}} & \cdots \ar[r] &P_{*}  \ar[r]^{d'} \ar[d]^{\mathsf{id}}& P_{*}  \ar[d]^{d''}\ar[r] &0 \ar[d]  \\
Q & P_{\sigma} \ar[r] & P_{*} \ar[r] & \cdots \ar[r] & P_{*} \ar[r]^d & P_* \ar[r] & P_* \ar[r] & \cdots
}\end{xy}
\]

\noindent
6. Summing up, our initial roof can be replaced by an equivalent one  of  the following form
\[
\begin{xy}\SelectTips{cm}{}
\xymatrix{
P_{\sigma} & P_{\sigma} \\
Q \ar[u]^f \ar[d]_g & P_{\sigma} \ar[u]^{\mathsf{id}} \ar[r] & P_{*} \ar[r] & \cdots \ar[r] &P_{*} \ar[d]_g \\
P_{\tau}[n] & &&& P_{\tau}
}
\end{xy}
\]
If $g$ is not minimal (i.e.~not given by multiplication with a single arrow), then it factors over $P_{*}[n]$ and therefore vanishes in $\Delta_{nd}$. Thus the morphism space $\Hom_{\Delta_{nd}}(P_{\sigma}, P_{\tau}[n])$ is at most one dimensional. Moreover, $g$ can be non-zero only if $n$ has the right parity.

\medskip
\noindent
7. Consider a roof $P_{\sigma} \xleftarrow{f} Q \xrightarrow{g} P_{\tau}[n]$
as in the previous step and assume that $g$ is non-zero and minimal. We want to show that the roof defines a non-zero homomorphism in $\Delta_{nd}$. We have a triangle
$ Q \stackrel{g}\rightarrow  P_{\tau}[n] \rightarrow  \String{\tau}{-n}{n} \rightarrow Q[1]$ in $K^b\bigl(A_{nd}-\mod\bigr)$ yielding a triangle
$P_{\sigma} \rightarrow  P_{\tau}[n] \rightarrow  \String{\tau}{-n}{n} \rightarrow P_{\sigma}[1]$ in $\Delta_{nd}$. Since $\String{\tau}{-n}{n}$ is indecomposable (see Lemma \ref{L:MinS}),  the map is non-zero. The claim follows.

\smallskip
Lemma \ref{L:Verdier2} and Lemma \ref{L:MinS} reduce the computation of morphism spaces in \eqref{E:HomFromProj}, \eqref{E:HomToProj} and  \eqref{E:HomStringString}  to a computation in the homotopy category $K^b(\proj-A_{nd})$. Using this \eqref{E:HomFromProj} and \eqref{E:HomToProj} can be checked directly.  Every minimal string may be presented as a cone of a morphism $P_{\sigma}[n] \rightarrow P_{\tau}[m]$ in $\Delta_{nd}$ (as in step 7). In conjunction with the isomorphisms \eqref{E:HomFromProj} and \eqref{E:HomToProj} and the long exact Hom-sequence, this can be used to verify  \eqref{E:HomStringString}.

\medskip
\noindent
\textit{Proof of (c).}
The first statement follows from part (a) together with Lemma \ref{L:MinS} and part (b).

For the second part, we use the description of the Grothendieck group of a local relative singularity category given in Proposition \ref{P:GrothendieckGroup}:
for  $X \in \thick(S_{+}, S_{-})$ we have $[X] =   n \cdot  \bigl([P_{+}]+[P_{-}]\bigr) \in K_{0}(\Delta_{nd})$ for a certain $n \in \mathbb{Z}$. Thus the images of indecomposable projective $A_{nd}$-modules $P_{+}$ and $P_{-}$ are not contained in $\thick(S_{+}, S_{-})$.
 By Lemma \ref{L:Verdier2} and the classification of indecomposable strings in $K^b(\proj-A_{nd})$, it remains to show that $P_{\sigma}[n] \cong P_{\tau}[m]$ implies $\sigma=\tau$ and $n=m$. Assume that $n>m$ holds. Then using Lemma \ref{L:Verdier2} again, we obtain
\[\Hom_{\Delta_{nd}}\bigl(P_{\sigma}[n], \String{\sigma}{1}{n}\bigr)\cong k \ne 0 =\Hom_{\Delta_{nd}}\bigl(P_{\tau}[m], \String{\sigma}{1}{n}\bigr).\] This is a contradiction. Similarly, the assumption $\sigma \ne \tau$ leads to a contradiction.
\end{proof}

\begin{rem}
Using Theorem \ref{t:Classical-versus-generalized-singularity-categories} together with Theorem \ref{T:MainT}, we obtain explicit descriptions of the relative singularity categories $\Delta_{R}(A)$, where $R$ is an odd dimensional $A_{1}$-singularity and $A$ is the Auslander algebra of $\MCM(R)$.
\end{rem}

\begin{cor}\label{C:indinTria}
The indecomposable objects of the triangulated subcategory \[\thick(S_-, S_+) \subset
\cd^b(A_{nd}-\mod)\] are precisely the shifts of the minimal strings $\kS_\tau(l)$.
\end{cor}

\begin{proof}
By Remark \ref{R:Tria}, we know that all minimal strings belong to $\thick(S_{-}, S_+)$.
Hence we just have to prove that there are no other indecomposable objects.
According to Lemma \ref{L:Verdier2} and  Lemma \ref{L:MinS},
the functor $\thick(S_{-}, S_+) \rightarrow \Delta_{nd}$ is fully faithful.
Therefore, the indecomposable objects of the category $\thick(S_{-}, S_+)$ and its essential 
image in $\Delta_{nd}$
are the same. By Theorem \ref{T:MainT}, all indecomposable objects of $\Delta_{nd}$ are known and
the shifts of the objects  $P_{+}$ and $P_{-}$ are not contained in $\thick(S_{-}, S_{+})$. Hence, 
the minimal strings are the only indecomposable objects of $\thick(S_{-}, S_+)$.
\end{proof}

\subsubsection{The cluster resolution and its relative singularity category}
In this paragraph, we describe another non-commutative resolution of the nodal curve singularity $O_{nd}=k\llbracket x, y \rrbracket/(xy)$. Consider the stable category of maximal Cohen--Macaulay modules $\ul{\MCM}(O_{nd})$. It is a $2$-Calabi--Yau category \cite{Auslander76} and  has only two indecomposable objects $k\llbracket x \rrbracket$ and $k\llbracket y \rrbracket$, which are interchanged by the shift functor $\Omega^{-1}$. Moreover, there are no non-trivial homomorphisms between these two indecomposable objects. In particular, each of $k\llbracket x \rrbracket$ and $k\llbracket y \rrbracket$ is a $2$-cluster tilting object in $\ul{\MCM}(O_{nd})$, see \cite{BIKR}. It follows from work of Iyama \cite{Iyama07} (see also Subsection \ref{ss:fract-cy}), that the algebra
$C_{nd}=\End_{O_{nd}}(O_{nd} \oplus k\llbracket x \rrbracket)$ has global dimension $3$. In particular, $C_{nd}$ is a non-commutative resolution of $O_{nd}$, which we call the \emph{cluster resolution}. Note, that there is an isomorphism of algebras $C_{nd}\cong eA_{nd} e$, where $A_{nd}$ is the Auslander algebra of $O_{nd}$ and $e \in A_{nd}$ is the idempotent endomorphism corresponding to the identity endomorphism of $O_{nd} \oplus k\llbracket x \rrbracket$. This yields a description of $C_{nd}$ as the completion of the following quiver with relations
\begin{equation}
\begin{xy}\SelectTips{cm}{}
\xymatrix
{- \ar@/^/[rr]^{ \ro }  & &  \ast \ar@/^/[ll]^{ \lu } \ar@(rd, ru)[]_{[\lo\ru]} 
}\end{xy}  \qquad  [\lo\ru]   \ro  = 0, \quad   \lu   [\lo\ru]  = 0.
\end{equation}
Moreover, there is a fully faithful functor
\begin{align*}
F\colon \cd^b(C_{nd}-\mod) \xrightarrow{ A_{\scriptsize nd} e \lten_{C_{\scriptsize nd}} -}  \cd^b(A_{nd}-\mod),
\end{align*}
which induces a fully faithful functor $\overline{F} \colon\Delta_{O_{nd}}(C_{nd}) \to \Delta_{O_{nd}}(A_{nd})$ between the corresponding relative singularity categories, see Proposition \ref{P:Embedding-of-relative} below for the general statement. The functor $F$ sends the indecomposable projective $C_{nd}$-modules to the corresponding indecomposable projective $A_{nd}$-modules. As a consequence, we can describe the image of the embedding $\overline{F}$ as follows:

\begin{prop}\label{P:ClusterAone}
The indecomposable objects in the image of \[\overline{F} \colon\Delta_{O_{nd}}(C_{nd}) \to \Delta_{O_{nd}}(A_{nd})\] are the shifts $P_{-}[n]$ and $\String{-}{2l}{n}$, with $l \in \mathbb{N}$ and $n \in \mathbb{Z}$.
\end{prop}

\begin{rem}
 Alternatively, one can obtain a description of the category $\Delta_{O_{nd}}(C_{nd})$ by adapting the method which we used to describe the category $\Delta_{nd}$ in the previous paragraph.
\end{rem}

\subsubsection{The relative singularity category of the zero dimensional $A_{1}$-singularity}
Let $R=k[x]/(x^2)$ be the zero dimensional $A_{1}$-singularity. Its Auslander algebra $A=\End_{R}(R \oplus k)$ may be written as a quiver with relations:
\begin{equation}
\begin{xy}\SelectTips{cm}{}
\xymatrix
{1 \ar@/^/[rr]^{ a }  & &  2 \ar@/^/[ll]^{ b } 
}\end{xy}  \qquad  ab=0,
\end{equation}
where the vertex $1$ corresponds to $R$ and the vertex $2$ corresponds to $k$. Since $R$ is a representation-finite ring, the Auslander algebra $A$ has global dimension $2$ by work of Auslander \cite[Sections III.2 and III.3]{Auslander71}. In particular, $A$ is a non-commutative resolution of $R$. We want to describe the corresponding  relative singularity category
\begin{align}
\Delta_{R}(A)=\frac{\cd^b(A-\mod)}{K^b(\proj-R)} \cong \frac{\cd^b(A-\mod)}{K^b(\add P_{1})}.
\end{align}
For every positive integer $l \in \mathbb{N}$, we define a complex of projective $A$-modules as follows
\begin{align}
\mathcal{S}(l)= \cdots \ra 0 \ra P_{2} \xrightarrow{\cdot a} P_{1} \xrightarrow{\cdot ba} P_{1} \xrightarrow{\cdot ba}  \cdots \xrightarrow{\cdot ba} P_{1} \xrightarrow{\cdot b} P_{2} \ra 0 \ra \cdots
\end{align}

The following proposition summarizes the properties of $\Delta_{R}(A)$. It can be proved along the lines of the proof of Theorem \ref{T:MainT}.

\begin{prop} \label{P:ZeroAone}
We use the notations from above.

\noindent
(a) The set of indecomposable objects in $\Delta_{R}(A)$ is given by \[\bigl\{P_{2}[n], \kS(l)[m] \, \big|  \, n, m \in \ZZ, \, l \in \mathbb{N}\bigr\}.\] Moreover, the objects from this set  are pairwise non-ismorphic   in $\Delta_{R}(A)$.

\noindent
(b) Let  $n \in \ZZ$ and $l, l' \in \mathbb{Z}_{>0}$. We have the following isomorphisms:
\begin{align}
\Hom_{\Delta_{R}(A)}\bigl(P_{2}, P_{2}[n]\bigr) \cong \begin{cases} k \quad \text{  if }n \leq 0, \\ 0 \quad \text{  otherwise.}\end{cases}
\end{align}

\begin{align}
\Hom_{\Delta_{R}(A)}\bigl(P_{2}[n], \mathcal{S}(l)\bigr)\cong \begin{cases} k \quad \text{  if } 0 \leq n < l,  \\ 0 \quad \text{  otherwise.}\end{cases}
\end{align}

\begin{align}
\Hom_{\Delta_{R}(A)}\bigl( \mathcal{S}(l), P_{2}[n]\bigr)\cong \begin{cases} k \quad \text{  if } 2 \leq n \leq l+1,  \\ 0 \quad \text{  otherwise.}\end{cases}
\end{align}

\begin{align}
\Hom_{\Delta_{R}(A)}\bigl( \mathcal{S}(l), \mathcal{S}(l')[n]\bigr)\cong \begin{cases} k \quad \text{  if }  n \leq 0 < l'+n \leq l,  \\ 
 k \quad \text{  if } 2 \leq n \leq l+1 < n+l',  \\
0 \quad \text{  otherwise.}\end{cases}
\end{align}
\end{prop}

\begin{rem}
Using Theorem \ref{t:Classical-versus-generalized-singularity-categories} together with Proposition \ref{P:ZeroAone}, we obtain explicit descriptions of the relative singularity categories $\Delta_{R}(A)$, where $R$ is an even dimensional $A_{1}$-singularity and $A$ is the Auslander algebra of $\MCM(R)$.
\end{rem}

\subsubsection{Connection with the category $\bigl(\cd^b(\Lambda-\mod)/\Band(\Lambda)\bigr)^\omega$}\label{S:Quivers}
\noindent
Let  $\Lambda$  be the path algebra of the following quiver with relations
\begin{equation}\label{E:gentlequiver}
\begin{xy}\SelectTips{cm}{}
\xymatrix
{
1 \ar@/^/[rr]^{a} \ar@/_/[rr]_{c} & & 2 \ar@/^/[rr]^{b} \ar@/_/[rr]_{d} & & 3
}
\end{xy}
\qquad ba = 0,  \quad dc =  0
\end{equation}
and  $\Band(\Lambda)$ be the full subcategory of $\cd^b(\Lambda-\mod)$ consisting of those objects, which are
invariant under the Auslander--Reiten translation $\tau$  in $\cd^b(\Lambda-\mod)$. Recall \cite{Happel88} that $\tau$ is the composition of the Serre functor with the negative shift $[-1]$ and that the Serre functor is given as the derived functor of the Nakayama functor $D\Lambda \otimes_{\Lambda} -$. By \cite[Corollary 6]{Tilting},
the subcategory $\Band(\Lambda)$ is triangulated.
Hence we can define the triangulated category
\[\widetilde{\Delta}_{nd} := \bigl(\cd^b(\Lambda-\mod)/\Band(\Lambda)\bigr)^\omega,\]
i.e.~the idempotent completion of the Verdier quotient $\cd^b(\Lambda-\mod)/\Band(\Lambda)$ (see Subsection \ref{Sub:Idemp}).
The main goal of this section is to show that $\widetilde{\Delta}_{nd}$ and $\Delta_{nd}$ are triangle equivalent.
\begin{lem}
The indecomposable projective $\Lambda$-modules are pairwise isomorphic in $\widetilde{\Delta}_{nd}.$
\end{lem}
\begin{proof} Complete the following exact sequences of $\Lambda$-modules
\begin{align*}
0 \longrightarrow P_{2} \longrightarrow P_{1} \longrightarrow \left(\begin{xy}\SelectTips{cm}{}\xymatrix{k \ar@/^/[r]^1 \ar@/_/[r]_1 & k \ar@/^/[r] \ar@/_/[r] &0}\end{xy}\right) \longrightarrow 0 \\
0 \longrightarrow P_{3} \longrightarrow P_{2} \longrightarrow \left(\begin{xy}\SelectTips{cm}{}\xymatrix{0 \ar@/^/[r] \ar@/_/[r] & k \ar@/^/[r]^1 \ar@/_/[r]_{1} & k}\end{xy}\right) \longrightarrow 0
\end{align*}
to triangles in $\cd^b(\Lambda-\mod)$ and note that the modules on the right-hand side  are bands.
\end{proof}
\noindent
Let $P \in \widetilde{\Delta}_{nd}$ be the common image of the indecomposable projective $\Lambda$--modules.
\begin{lem}\label{L:Non-Split}
The endomorphisms of $P$, which are given by the roofs
\begin{equation}\label{E:Idempotents}
e_{+}= P_{1} \xleftarrow{\cdot(a+c)} P_{2} \xrightarrow{\cdot a} P_{1} \quad  \text{     and     } \quad e_{-}=P_{1} \xleftarrow{\cdot(a+c)}
P_{2} \xrightarrow{\cdot c} P_{1}
\end{equation}
satisfy $e_{-}e_{+}=0=e_{+}e_{-}$ and $e_{-}+e_{+}=\mathsf{id}_{P}$ and thus are \emph{idempotent}.
In particular, we have a direct sum decomposition $P \cong P^+ \oplus P^-$, where $P^{+}=(P, e_{+})$ and $P^{-}=(P, e_{-})$.
\end{lem}
\begin{proof} It is clear that $e_{-}+e_{+}=\mathsf{id}_{P}$. The equality $e_{+}e_{-}=0$ follows
from the  diagram
\begin{align*}
e_{+}e_{-}=\begin{split}\begin{xy}\SelectTips{cm}{}\xymatrix{&&P_{3} \ar[rd]^{\cdot d} \ar[ld]_{\cdot(b+d)} \\
                              &P_{2} \ar[rd]^(.44){\cdot a} \ar[ld]_{\cdot(a+c)}&& P_{2} \ar[rd]^{\cdot c} \ar[ld]_{\cdot(a+c)} \\
                              P_{1}&&P_{1}&&P_{1}}\end{xy}\end{split}=\, \, P_{1} \xleftarrow{\cdot(da+bc)} P_{3} \xrightarrow{0} P_{1}.
\end{align*}
The second equality $e_{-}e_{+}=0$ follows from  a similar calculation. Hence,  $e_{\pm}^2 = e_{\pm}(\mathsf{id}_{P}-e_{\mp})=e_{\pm}$.
\end{proof}

\noindent
Next, note the following easy but useful result.
\begin{lem}\label{L:fractCY}
Let $\kA$ be an abelian category and let $\mathbb{S}\colon \cd^b(\kA) \rightarrow \cd^b(\kA)$ be a triangle equivalence. If $X_{1}, X_{2} \in \cd^b(\kA)$ and $n_{1}, n_{2}, m_{1}, m_{2} \in \ZZ$ satisfy
\[ \mathbb{S}^{m_{1}}X_{1} \cong X_{1}[n_{1}], \quad  \quad \mathbb{S}^{m_{2}}X_{2} \cong X_{2}[n_{2}]
 \quad \text{ and  } \quad  d=m_{1}n_{2}-m_{2}n_{1} \ne 0,\] then
$\Hom_{\cd^b(\kA)}(X_{1}, X_{2}) = 0 = \Hom_{\cd^b(\kA)}(X_{2}, X_{1}).$
\end{lem}
\begin{proof} By the symmetry of the claim, it suffices to show that $\Hom_{\cd^b(\kA)}(X_{1}, X_{2})$ vanishes. Since $\mathbb{S}$ is an equivalence, we have a chain of isomorphisms
\begin{align*}
\Hom_{\cd^b(\kA)}(X_{1}, X_{2}) &\cong \Hom_{\cd^b(\kA)}\bigl(\mathbb{S}^{\pm m_{1}m_{2}}X_{1},\mathbb{S}^{\pm m_{1}m_{2}} X_{2}\bigr) \\ &\cong   \Hom_{\cd^b(\kA)}\bigl(X_{1}[\pm m_{2}n_{1}], X_{2}[\pm m_{1}n_{2}]\bigr)
         \\ &\cong \Hom_{\cd^b(\kA)}\bigl(X_{1}, X_{2}[\pm d]\bigr)  \\ &\cong \Hom_{\cd^b(\kA)}\bigl(X_{1}, X_{2}[\pm kd]\bigr)
\end{align*}
for all $k \in \mathbb{N}$.  Hence the claim follows from the boundedness of $X_{1}$ and $X_{2}$ together with the fact that there are no non-trivial Ext--groups $\Ext^{-n}_{\kA}(A_{1}, A_{2}) \cong \Hom_{\cd^b(\kA)}(A_{1}, A_{2}[-n])$, where $A_{1}, A_{2} \in \kA$ and $n$ is a positive integer.
\end{proof}
\noindent
A direct calculation in $\cd^b(\Lambda-\mod)$ yields the following result.
\begin{lem}\label{L:Xplusminus}
Let $\mathbb{S}\colon \cd^b(\Lambda-\mod) \rightarrow \cd^b(\Lambda-\mod)$ be the Serre functor,
\[X_{+} =\begin{xy}\SelectTips{cm}{}\xymatrix{k \ar@/^/[r]^1 \ar@/_/[r]_0 & k \ar@/^/[r]^0 \ar@/_/[r]_{1} &k} \end{xy} \quad \text{ and } \quad
X_{-}=\begin{xy}\SelectTips{cm}{}\xymatrix{k \ar@/^/[r]^{0} \ar@/_/[r]_1 & k \ar@/^/[r]^1 \ar@/_/[r]_{0} &k}\end{xy}.\]
Then $\mathbb{S}(X_{\pm}) \cong X_{\mp}[2].$
In particular, $X_{\pm}$ are $\frac{4}{2}$-fractionally Calabi--Yau objects, i.e. $\mathbb{S}^2(X_{\pm}) \cong X_{\pm}[4]$
\end{lem}

\begin{cor}
The following composition of the inclusion and projection functors
\[\thick(X_{+}, X_{-}) \hookrightarrow
\cd^b(\Lambda-\mod) \lar
 \frac{\displaystyle \cd^b(\Lambda-\mod)}{\displaystyle \Band(\Lambda)}\]
 is fully faithful.
\end{cor}
\begin{proof}
Lemma \ref{L:Xplusminus} and Lemma \ref{L:fractCY} applied to the Serre functor $\mathbb{S}$ in $\cd^b(\Lambda-\mod)$ imply that
$X_{\pm} \in  ^{\perp}\!\!\Band(\Lambda) \cap \Band(\Lambda)^{\perp}$. Hence the claim
 follows from Lemma \ref{L:Verdier2}.
\end{proof}
\begin{thm} \label{T:Maintilde}
There exists an equivalence of triangulated categories
\[
\GG \colon \, \frac{\displaystyle \cd^b(A_{nd}-\mod)}{\displaystyle K^b\bigl(\add(P_{*})\bigr)} \longrightarrow
\left(\frac{\displaystyle \cd^b(\Lambda-\mod)}{\displaystyle \Band(\Lambda)}\right)^{\omega}.
\]
\end{thm}
\begin{proof}
Let $E = V(zy^2 - x^3 - x^2 z) \subset \mathbb{P}^2$ be a nodal cubic curve and
$\kF' = \kI$ be the ideal sheaf of the singular point of $E$. Let
$\kF = \kO \oplus \kI$, $\kA = {\mathcal End}_E(\kF)$  and $\mathbb{E} = (E, \kA)$.
By a result of Burban and Drozd \cite[Section 7]{Tilting}, there exists a triangle  equivalence
\[\mathbb{T}\colon \cd^b\bigl(\Coh (\mathbb{E})\bigr) \lar  \cd^b(\Lambda-\mod)\]
identifying
the image of the category $\Perf(E)$ with the category $\Band(\Lambda)$. Moreover, by
\cite[Proposition 12]{Tilting}, the functor   $\TT$ restricts to an equivalence
$\thick(S_{+}, S_{-}) \rightarrow \thick(X_{+}, X_{-}).$ In combination with \eqref{E:KeyChain}, we obtain
the following commutative diagram of categories and functors
\begin{equation}\label{E:importantdiagram}
\begin{array}{c}
\begin{xy}\SelectTips{cm}{}
\xymatrix{
\frac{\displaystyle \cd^b(A_{nd}-\mod)}{\displaystyle K^b\bigl(\add(P_{*})\bigr)} \ar@/^{10pt}/[rrd]^-\GG & & \\
\left(\frac{\displaystyle \cd^b(A_{nd}-\fdmod)}{\displaystyle K^b_{\mathsf{fd}}\bigl(\add(P_{*})\bigr)}\right)^{\omega}  \ar[r]^-{\sim} \ar[u]^-\sim &\left(\frac{\displaystyle \cd^b\bigl(\Coh(\mathbb{E})\bigr)}{\displaystyle \cP(E)}\right)^{\omega} \ar[r]^-\sim & \left(\frac{\displaystyle \cd^b(\Lambda-\mod)}{\displaystyle \Band(\Lambda)}\right)^{\omega} \\
 \cd^b(A_{nd}-\fdmod) \ar[u]^-{\mathsf{can}} \ar[r] & \cd^b\bigl(\Coh(\mathbb{E})\bigr) \ar[u]^-{\mathsf{can}} \ar[r]^-{\TT}  &  \cd^b(\Lambda-\mod) \ar[u]_-{\mathsf{can}} \\
 & \thick(S_{+}, S_{-})  \ar@{^{(}->}[u]\ar[r]^{\sim} &
 \thick(X_{+}, X_{-})\ar@{^{(}->}[u]
}\end{xy}
\end{array}
\end{equation}
where $\GG\colon \Delta_{nd} \rightarrow \widetilde{\Delta}_{nd}$ is the induced  equivalence of triangulated categories.\end{proof}

\begin{lem}\label{L:IndtildeDn}
The indecomposable objects of the triangulated category $\widetilde{\Delta}_{nd}$ are
\begin{itemize}
\item $P^{\pm}[n] \cong \GG\bigl(P_{\pm}[n]\bigr), \quad n \in \ZZ;$
\item the indecomposables of the full subcategory  $\thick(X_{+}, X_{-})\cong \GG\bigl(\thick(S_{+}, S_{-})\bigr).$
\end{itemize}
\end{lem}
\begin{proof}
Consider the projective resolution of the simple $A_{nd}$-module $S_{*}$
\[\begin{xy}\SelectTips{cm}{}\xymatrix{0 \ar[r] & P_{\li} \oplus P_{\re} \ar[r]^(0.6){\left(\bsm \cdot \beta & \cdot \delta \esm\right)} & P_{*} \ar[r] & S_{*} \ar[r] & 0}\end{xy}.\]
Completing it to a distinguished triangle yields an isomorphism
$S_{*}[-1] \cong P_{+} \oplus P_{-}$ in $\Delta_{nd}$. In the notations
 of the diagrams (\ref{E:importantdiagram}) and (\ref{E:gentlequiver}) we have $\TT(S_{*}) \cong P_{3}[1]$ and therefore
\[\GG(P_{+} \oplus P_{-}) \cong \GG(S_{*}[-1]) \cong P \cong (P^+ \oplus P^-).\]
Recall that $X_{+}\cong \bigl(P_{3} \stackrel{\cdot d}\longrightarrow P_{2} \stackrel{\cdot c}\longrightarrow P_{1}\bigr)$, where $P_{1}$ is located in degree $0$ and $P^\pm := (P, e_{\pm})
\in \widetilde{\Delta}_{nd}$, with $e_{\pm}$ as defined in (\ref{E:Idempotents}). A direct calculation shows that the obvious morphism from $P_{1}$ to $X_{+}$ induces a non-zero morphism $P^+=(P, e_{+}) \rightarrow X_{+}$ in $\widetilde{\Delta}_{nd}$, whereas $\Hom_{\widetilde{\Delta}_{nd}}(P^{+}, X_{-})=0$. Moreover, it was shown in \cite{Tilting} that $\TT(S_{\pm})\cong X_{\pm}$. This implies $\GG(S_{\pm})\cong X_{\pm}$ and thus $\GG(P_{\pm})\cong P^{\pm}$.
Theorem \ref{T:MainT} and Corollary \ref{C:indinTria} yield  the stated classification of indecomposables   in $\widetilde{\Delta}_{nd}.$
\end{proof}

\subsubsection{Concluding remarks on $\Delta_{nd}$}
\begin{prop}\label{P:ARtriangles}
The  category $\thick(S_{+}, S_{-}) \subset \Delta_{nd}$ has Auslander--Reiten triangles.
\end{prop}
\begin{proof}
 As mentioned  above, we have an exact equivalence of triangulated categories
 \[\thick(S_{+}, S_{-}) \cong \thick(X_{+}, X_{-}) \subset \cd^b(\Lambda-\mod).\]
 The category $\cd^b(\Lambda-\mod)$ has a Serre functor $\mathbb{S}$ and therefore has Auslander--Reiten triangles, see
  \cite{Happel}. Let $\tau=\mathbb{S}\circ[-1]$ be the Auslander--Reiten translation. Using that $\tau$ is an equivalence and Lemma \ref{L:Xplusminus} we obtain $\tau\bigl(\thick(X_{+}, X_{-})\bigr) \cong \thick\bigl(\tau(X_{+}), \tau(X_{-})\bigr) \cong \thick(X_{+}, X_{-}).$ Now, the restriction of $\tau$ to $\thick(X_{+}, X_{-})$ is the Auslander--Reiten translation of this subcategory.
\end{proof}
\begin{rem}\label{r:AR-quiver}
Using an explicit description of the morphisms in 
$\thick(S_{+}, S_{-})$ ( by Lemma \ref{L:Verdier2}, all these morphism may be computed in the homotopy category
$K^b(\proj-A_{nd})$), one can show that the Auslander--Reiten quiver of $\thick(S_{+}, S_{-})$ consists of two $\ZZ A_{\infty}$--components. 

Alternatively, using the Happel functor, one may view $X_{+}$ and $X_{-}$ as objects in the stable category $\widehat{\Lambda}-\ul{\mod}$, where $\widehat{\Lambda}$ is the repetitive algebra of $\Lambda$, see \cite{Happel88}. Since $\Lambda$ is a gentle algebra, $\widehat{\Lambda}$ is special biserial by \cite{Ringel}. The Auslander--Reiten sequences over special biserial algebras are known by \cite{WW}. This may be used to determine the Auslander--Reiten quiver of $\thick(X_{+}, X_{-})$.

We draw one of the $\ZZ A_{\infty}$-components below, indicating the action of the Auslander--Reiten translation by $\SelectTips{cm}{} \xymatrix{&\ar@{-->}[l]}$. The other component is obtained from this one by changing the roles of $+$ and $-$.
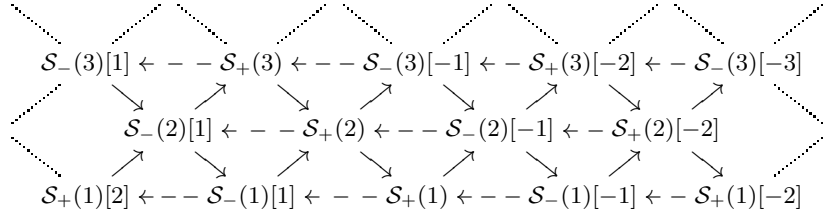
\[
{\scriptsize
\SelectTips{cm}{}
\begin{xy} 0;<0.62pt,0pt>:<0pt,-0.5pt>::
(0,100) *+{} ="0",
(0,0) *+{} ="1",
(50,150) *+{\kS_{+}(1)[2]} ="2",
(50,50) *+{\kS_{-}(3)[1]} ="3",
(100,100) *+{\kS_{-}(2)[1]} ="5",
(100,0) *+{} ="6",
(150,150) *+{\kS_{-}(1)[1]} ="7",
(150,50) *+{\kS_{+}(3)} ="8",
(200,100) *+{\kS_{+}(2)} ="10",
(200,0) *+{} ="11",
(250,150) *+{\kS_{+}(1)} ="12",
(250,50) *+{\kS_{-}(3)[-1]} ="13",
(300,100) *+{\kS_{-}(2)[-1]} ="15",
(300,0) *+{} ="16",
(350,150) *+{\kS_{-}(1)[-1]} ="17",
(350,50) *+{\kS_{+}(3)[-2]} ="18",
(400,100) *+{\kS_{+}(2)[-2]} ="20",
(400,0) *+{} ="21",
(450,150) *+{\kS_{+}(1)[-2]}="22",
(450,50) *+{\kS_{-}(3)[-3]} ="23",
(500,0) *+{} ="24",
(500,100) *+{} ="25",
"0", {\ar@{.}"2"},
"0", {\ar@{.}"3"},
"1", {\ar@{.}"3"},
"2", {\ar"5"},
"3", {\ar"5"},
"3", {\ar@{.}"6"},
"5", {\ar"7"},
"5", {\ar"8"},
"6", {\ar@{.}"8"},
"7", {\ar@{-->}"2"},
"7", {\ar"10"},
"8", {\ar@{-->}"3"},
"8", {\ar"10"},
"8", {\ar@{.}"11"},
"10", {\ar@{-->}"5"},
"10", {\ar"12"},
"10", {\ar"13"},
"11", {\ar@{.}"13"},
"12", {\ar@{-->}"7"},
"12", {\ar"15"},
"13", {\ar@{-->}"8"},
"13", {\ar"15"},
"13", {\ar@{.}"16"},
"15", {\ar@{-->}"10"},
"15", {\ar"17"},
"15", {\ar"18"},
"16", {\ar@{.}"18"},
"17", {\ar@{-->}"12"},
"17", {\ar"20"},
"18", {\ar@{-->}"13"},
"18", {\ar"20"},
"18", {\ar@{.}"21"},
"20", {\ar@{-->}"15"},
"20", {\ar"22"},
"20", {\ar"23"},
"21", {\ar@{.}"23"},
"22", {\ar@{-->}"17"},
"22", {\ar@{.}"25"},
"23", {\ar@{.}"24"},
"23", {\ar@{.}"25"},
"23", {\ar@{-->}"18"},
\end{xy}
}
\]
The category $\Delta_{nd}$ does not have Auslander--Reiten triangles, but we may still consider the quiver of irreducible morphisms in $\Delta_{nd}$, which has two \emph{additional} $A^{\infty}_{\infty}$--components.
\[
\begin{xy}\SelectTips{cm}{} \xymatrix{\ar@{.>}[r] & P_{\pm}[2] \ar[r] & P_{\mp}[1] \ar[r] &P_{\pm} \ar[r] & P_{\mp}[-1] \ar[r] & P_{\pm}[-2]\ar@{.>}[r] &}  \end{xy}\]
\end{rem}
\begin{prop}
The triangulated categories  $\thick(X_{+}, X_{-})$ and $\Delta_{nd}$  are not triangle equivalent to the bounded derived category of a finite dimensional algebra.
\end{prop}
\begin{proof}
 Assume that there exists a triangle equivalence to the derived category of a finite dimensional algebra $A$. Then $\cd^b(A-\mod)$ is of discrete representation type. Such algebras $A$ have been classified by Vossieck \cite{Vossieck}. All of them are gentle (see Definition \ref{D:Gentle}) and therefore Iwanaga--Gorenstein, by \cite{GeissReiten}.
 Therefore, the Nakayama functor defines a Serre functor $\mathbb{S}$ on $K^b(\proj-A)$ \cite{Happel}, whose action on objects is described in \cite[Theorem B]{BobinskiGeissSkowronski}. On the other hand, $\mathbb{S}^2(X) \cong X[4]$ holds for all objects $X$ in $\thick(X_{+}, X_{-})$, by Lemma \ref{L:Xplusminus} and Proposition \ref{P:ARtriangles}. This yields a contradiction.\end{proof}

\noindent
The following proposition generalizes Theorem \ref{T:Maintilde}.
\begin{prop}
Let $n\geq 1$ and $\Lambda_{n}$ be the path algebra of the following quiver
 
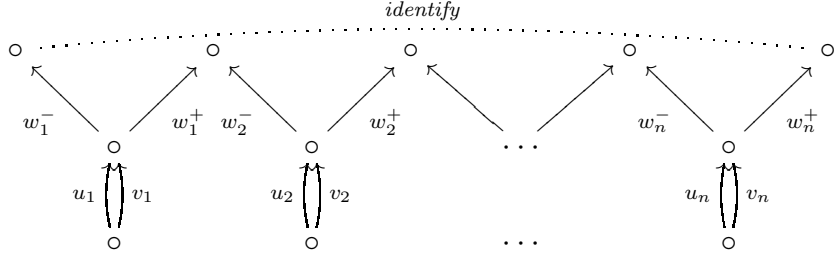
\[\begin{xy}\SelectTips{cm}{}\xymatrix{\circ  && \circ  && \circ  && \circ  && \circ \ar@{..}@/_8pt/[llllllll]_{\textit{identify}}
\\
& \circ  \ar[lu]^-{w_{1}^-} \ar[ru]_-{w_{1}^+} && \circ  \ar[lu]^-{w_{2}^-} \ar[ru]_-{w_{2}^+} && \cdots \ar[lu] \ar[ru] && \circ \ar[ru]_-{w_{n}^+} \ar[lu]^-{w_{n}^-}
\\
& \circ \ar@/^3pt/[u]^-{u_{1}} \ar@/_3pt/[u]_-{v_{1}} && \circ \ar@/^3pt/[u]^-{u_{2}} \ar@/_3pt/[u]_-{v_{2}} &&\cdots&& \circ
\ar@/^3pt/[u]^-{u_{n}} \ar@/_3pt/[u]_-{v_{n}}
}\end{xy}
\]
subject to the relations $w_i^- u_i = 0$ and $w_i^+ v_i=0$ for all $1 \le i \le n$. Then
\[\Delta_{n} := \left(\frac{\cd^b(\Lambda_{n}-\mod)}{\Band(\Lambda_{n})}\right)^{\omega} \cong \bigvee_{i=1}^n \Delta_{nd}.\]
In particular, the  category $\Delta_{n}$ is representation discrete, $\Hom$-finite and $K_{0}(\Delta_{n}) \cong (\ZZ^2)^{\oplus n}$.
\end{prop}
\begin{proof}
Let $E = E_{n}$ be a Kodaira cycle of $n$ projective lines and $\mathbb{E} = \bigl(E,
 {\mathcal End}_{E}(\kO \oplus \kI_{Z})\bigr)$, where $\kI_{Z}$ is the ideal sheaf of the singular locus $Z$. By  \cite[Proposition 10]{Tilting}  there exists an equivalence of triangulated categories
$\cd^b\bigl(\Coh(\mathbb{E})\bigr) \xrightarrow{\sim} \cd^b(\Lambda_{n}-\mod)$
identifying   $\Perf(E) \cong \cP(E) \subset \cd^b\bigl(\Coh(\mathbb{E})\bigr)$
with $\Band(\Lambda_{n})\subset \cd^b(\Lambda_{n}-\mod)$, see \cite[Corollary 6]{Tilting}. Thus
 Theorem \ref{t:main-global} yields  the proof.
\end{proof}

\subsection{The fractional Calabi--Yau property}\label{ss:fract-cy}
The following notion was introduced by Iyama \cite[Definition 3.1]{Iyama07a}.
\begin{defn}
Let $\cc \subseteq \ce$ be a full additive subcategory of an exact Krull--Remak--Schmidt $k$-category $\ce$ over an algebraically closed field $k$ and set $(X, Y)=\Hom_{\cc}(X, Y)$. An exact sequence in $\ce$ with objects contained in $\cc$
\begin{align}
0 \ra Y \xrightarrow{f_{d}} C_{d-1} \xrightarrow{f_{d-1}} C_{d-2} \xrightarrow{f_{d-2}} \ldots \xrightarrow{f_{1}} C_{0} \xrightarrow{f_{0}} X \ra 0
\end{align}
is called \emph{$d$-almost split} if the following three conditions hold. 
\begin{itemize}
\item[(AS1)] All morphisms $f_{i}$ are contained in the Jacobson radical $J_{\cc}$ of $\cc$, i.e.~the ideal of the category $\cc$, which is uniquely determined by $J_{\cc}(X, X) =\rad \End_{\cc}(X)$ for all $X$ in $\cc$, see \cite[Lemma 5]{GMKelly}.
\item[(AS2)] $0 \ra (-,Y) \xrightarrow{(-, f_{d})} (-, C_{d-1}) \xrightarrow{(-, f_{d-1})} \ldots \xrightarrow{(-, f_{1})} (-, C_{0}) \xrightarrow{(-, f_{0})} J_{\cc}(-, X) \ra 0$ is an exact sequence of functors.
\item[(AS3)] $0 \ra (X, -) \xrightarrow{(f_{0}, -)} (C_{0}, -) \xrightarrow{(f_{1}, -)} \ldots \xrightarrow{(f_{d-1}, -)} (C_{d-1}, -) \xrightarrow{(f_{d}, -)} J_{\cc}(Y, -) \ra 0$ is an exact sequence of functors.
\end{itemize}
Since $Y$ is determined by $X$ and vice versa it makes sense to write $\tau_{d}(X)$ for $Y$ respectively $\tau_{d}^{-1}(Y)=X$. We say that $\cc$ has $d$-almost split sequences if there is an $d$-almost split sequence ending (respectively starting) in every non-projective (respectively non-injective) object of $\cc$.
\end{defn}

\begin{rem} \label{R:AlmostSplit}
(a) \,  Assume that the stable category $\ul{\cc}$ is Hom-finite and let $X \in \cc$ be indecomposable. Denote by $S_{X}(-)\colon \cc \ra \Mod-k$ the contravariant functor defined by the short 
exact sequence $0 \ra J_{\cc}(-, X) \ra (-, X) \ra S_{X}(-) \ra 0$. By definition, this functor takes $X$ to $\End_{\cc}(X)/J_{\cc}(X, X) \cong k$ and all other indecomposable objects to zero.
Combining the exact sequence in (AS2) with this short exact sequence yields an exact sequence 
\begin{align*}
0 \ra (-,Y) \xrightarrow{(-, f_{d})} (-, C_{d-1}) \xrightarrow{(-, f_{d-1})} \ldots \xrightarrow{(-, f_{1})} (-, C_{0}) \xrightarrow{(-, f_{0})} (-, X) \ra S_{X}(-) \ra 0,
\end{align*}
which is a projective resolution of the simple module $S_{X}(-)$ in the category of right $\cc$-modules (i.e. contravariant functors $\cc \ra \Mod-k$).

One can define the simple module $S_{Y}(-)$ in another way. Namely,  by the short exact sequence $0 \ra S_{Y}(-) \ra D(Y, -) \ra DJ_{\cc}(Y, -) \ra 0 $. Combining this sequence with the $k$-dual of the sequence in (AS3) yields an exact sequence 
\begin{align*}
0 \ra S_{Y}(-) \ra D(Y, -) \ra D(C_{d-1}, -) &\xrightarrow{D(f_{d-1}, -)} \ldots \\ &\ldots \xrightarrow{D(f_{1}, -)} D(C_{0}, -) \xrightarrow{D(f_{0}, -)} D(X, -) \ra 0.
\end{align*}

\noindent (b) \, For indecomposable objects $X$ and $Y$, $1$-almost split sequences 
\begin{align}
0 \ra Y \xrightarrow{f_{1}} C_{0} \xrightarrow{f_{0}} X \ra 0
\end{align}
were introduced by Auslander \& Reiten, who called them almost split sequences. Nowadays, they are usually called Auslander--Reiten sequences. Let us briefly discuss this particular case. Condition (AS1) guarantees that the sequence does not split. Since $X$ is indecomposable, a morphism $g\colon T \ra X$ is in the radical $J_{\cc}(T, X)$ if and only $f$ is not a split epimorphism. Hence, (AS2) is equivalent to the statement that every morphism into $X$, which is not a split epimorphism factors over $f_{0}$. Similarly, (AS3) is equivalent to the fact that every non-split monomorphism $h\colon Y \ra S$ factors over $f_{1}$.

\end{rem}

\begin{ex}\label{Ex:AlmostSplit} Let $R$ be the quiver algebra of the following quiver with relations
\begin{align}
\begin{xy}
\SelectTips{cm}{}
\xymatrix{
+ \ar@/^7pt/[rr]^a && - \ar@/^7pt/[ll]^b 
} 
\end{xy} \qquad  \qquad  \text{ba=0=ab}.
\end{align}
(a) \, Let $\cc=\mod-R$. Then $\cc$ has Auslander--Reiten (or $1$-almost split) sequences. Indeed, the Auslander--Reiten quiver of $\mod-R$ is given by
\begin{align}
\begin{array}{c}
\begin{xy}
\SelectTips{cm}{}
\xymatrix{
P_{+}={\begin{smallmatrix} + \\ -\end{smallmatrix}} 
\ar[rr]^{\alpha} && S_{+} \ar[d]^{\beta} \ar@{..>}@/_8pt/[lld] \\
S_{-} \ar[u]^\delta \ar@{..>}@/_8pt/[rru] && P_{-}={\begin{smallmatrix} - \\ +\end{smallmatrix}} 
\ar[ll]^{{\gamma}}
} 
\end{xy} 
\end{array}
\end{align}
where the dotted arrows indicate the action of the Auslander--Reiten translation. For later reference, we label the vertices of this quiver clockwise from $1 \, \hat{=} \,   P_{+}$ to $4 \,  \hat{=} \,  S_{-}$. Let $G$ be the additive generator $P_{+} \oplus P_{-} \oplus S_{+} \oplus S_{-}$ of $\cc$ and $A=\End_{R}(G)$ be the Auslander algebra of $R$. It is given as a quiver algebra. Indeed, the quiver is just the Auslander--Reiten quiver and the relations are $\alpha\delta=0$ and $\gamma\beta =0$. Then there is an equivalence of additive categories 
\begin{align*}
\cc & \longrightarrow \proj-A \\
C & \mapsto \Hom_{\cc}(G, C)
\end{align*}
Under this equivalence, indecomposable objects in $\cc$ correspond to indecomposable projective $A$-modules. For example, $S_{+}$ corresponds to $e_{2}A$.
By Remark \ref{R:AlmostSplit} (a), the Auslander--Reiten sequence
$0 \ra S_{-} \stackrel{\delta}\ra P_{+} \stackrel{\alpha}\ra S_{+} \ra 0$ yields a projective resolution of the simple $A$-module $S_{2}$
\begin{align}
0 \ra P_{4} \xrightarrow{\delta\cdot} P_{1} \xrightarrow{\alpha\cdot} P_{2} \ra S_{2} \ra 0
\end{align}

\medskip

\noindent (b) \, In the notation from (a) above, let $H=P_{+} \oplus P_{-} \oplus S_{-}$, $\cb=\add (H)$ and $e \in A$ be the idempotent corresponding to the identity of $H$. The endomorphism algebra $B=\End_{\cb}(H)\cong eAe$ is given by the quiver
\begin{align}
\begin{array}{c}
\begin{xy}
\SelectTips{cm}{}
\xymatrix{
1
\ar[rrd]^{[\beta \alpha]}\\
4 \ar[u]^\delta  && 3 
\ar[ll]^{{\gamma}}
} 
\end{xy} 
\end{array}
\end{align}
with relations $[\beta \alpha]\delta=0$ and $\gamma [\beta \alpha]=0$. Moreover, as above there is an additive equivalence $\cb \rightarrow \proj-B$. We claim that the exact sequence
\begin{align}
0 \ra S_{-} \xrightarrow{\delta} P_{+} \xrightarrow{\beta\circ \alpha} P_{-} \xrightarrow{\gamma} S_{-} \ra 0  
\end{align}
is $2$-almost split. (AS1) is satisfied, since we consider non-isomorphisms between indecomposable  objects. For (AS2), note that the following sequence of right $B$-modules is exact
\begin{align}\label{E:2almostsplittoprojres}
0 \ra e_{4}B \xrightarrow{\delta\cdot} e_{1}B \xrightarrow{[\beta\alpha]\cdot} e_{3}B \xrightarrow{\gamma\cdot} \rad e_{4}B \ra 0
\end{align}
Similarly, (AS3) follows from the exactness of the following exact sequence of  left $B$-modules
\begin{align}\label{E:2almostsplittoinjres}
0 \ra Be_{4} \xrightarrow{\cdot\gamma} Be_{3} \xrightarrow{\cdot[\beta\alpha]} Be_{1} \xrightarrow{\cdot\delta} \rad B e_{4} \ra 0
\end{align}
In particular, as noted in Remark \ref{R:AlmostSplit} (a) above, \eqref{E:2almostsplittoprojres} 
yields a projective resolution of the simple right module $S_{4}$ 
\begin{align}
0 \ra e_{4}B \xrightarrow{\delta\cdot} e_{1}B \xrightarrow{[\beta\alpha]\cdot} e_{3}B \xrightarrow{\gamma\cdot}  e_{4}B \ra S_{4} \ra 0.
\end{align}
Moreover, dualizing \eqref{E:2almostsplittoinjres} yields an injective resolution 
\begin{align}
0 \la D\left(Be_{4}\right) \xleftarrow{D(\cdot\gamma)} D\left(Be_{3}\right) \xleftarrow{D(\cdot[\beta\alpha])} D\left(Be_{1}\right) \xleftarrow{D(\cdot\delta)} D\left(B e_{4}\right) \la S_{4} \la 0.
\end{align}
\end{ex}

\bigskip

Let $\ce$ be a Krull--Remak--Schmidt Frobenius $k$-category over an algebraically closed field $k$. 
Assume that $\proj\ce$ has an additive generator $P$.
Let $F=P\oplus F'$ be an object of $\ce$ such that $F'$ has no projective direct summands.
Let $A=\End_{\ce}(F)$, $R=\End_{\ce}(P)$ and $e=id_P\in A$. Then $A/AeA$ is the stable endomorphism algebra of $F$.
Recall from Corollary~\ref{c:restriction-and-induction} (b) that $\cd_{fd,A/AeA}(A)$ 
denotes the full subcategory of $\cd_{fd}(A)$ consisting of complexes of $A$-modules which have
cohomologies supported on $A/AeA$ and that $\cd_{fd,A/AeA}(A)$ is generated by $\fdmod-A/AeA$. 
In particular, if $A/AeA$ is finite-dimensional over $k$,
then $\cd_{fd,A/AeA}(A)$ is
generated by the (finitely many) simple $A/AeA$-modules.

\begin{thm}\label{t:fractionally-cy-property}
If $\add(F)$ has $d$-almost split
sequences and $A/AeA$ is finite dimensional over $k$, then the following statements hold.
\begin{itemize}
\item[(a)] Any finite dimensional $A/AeA$-module has finite projective dimension over~$A$. 
\item[(b)] The triangulated category $\cd_{fd,A/AeA}(A)$ admits a Serre functor. 
\item[(c)] Each simple $A/AeA$-module $S$ is fractionally $\frac{(d+1)n}{n}$--CY in $\cd_{fd,A/AeA}(A)$, where $n=n(S) \in\mathbb{N}$.
\item[(d)] There exists a permutation $\pi$ on the isomorphism classes of simple $A/AeA$-modules such that
$D\Ext_A^l(S,S')\cong \Ext^{d+1-l}_A(S',\pi(S)),$ holds  Êfor all  $l \in \Z.$
\end{itemize}
\end{thm}

\begin{proof}
 Assume that $F'$ is multiplicity-free and write $F'=F_1\oplus\ldots\oplus F_r$ such that $F_1,\ldots,F_r$ are indecomposable.
For $i=1,\ldots,r$, let $e_i=\id_{F_i}$ and consider it as an element in $A$. Then $1_A=e+e_1+\ldots+e_r$. Let $S_i$ be the simple $A$-module
corresponding to $e_i$. By assumption, there is an $d$-almost split sequence
\begin{align}
\eta\colon 0\longrightarrow F_j \longrightarrow C_{d-1} \longrightarrow \ldots \longrightarrow C_0 \longrightarrow F_i \longrightarrow 0
\end{align}
for some $j=1,\ldots,r$, where $C_{d-1},\ldots,C_0\in\add(F)$. The assignment $i\mapsto j$ defines a permutation
on the set $\{1,\ldots,r\}$, which we denote by $\pi$.  Set $(X, Y)=\Hom_{\ce}(X, Y)$.
Applying $\Hom_{\ce}(F, -)$ to $\eta$, we obtain  a projective resolution of $S_i$ as an $A$--module, by the definition of a $d$-almost split sequence (see Remark \ref{R:AlmostSplit} (a)).
\begin{align*} 
0 \longrightarrow (F,F_{\pi(i)}) \longrightarrow (F,C_{d-1}) \longrightarrow \ldots \longrightarrow (F,C_0) \longrightarrow (F,F_i) \longrightarrow S_i \longrightarrow & 0,
\end{align*}
In particular, this shows (a). Dually, we acquire an $\add D(A)$--resolution of $S_{\pi(i)}$
\begin{align*} 0 \rightarrow S_{\pi(i)} \rightarrow D(F_{\pi(i)},F)
\rightarrow D(C_{d-1},F) \rightarrow
\ldots \rightarrow D(C_0,F) \rightarrow D(F_i,F) \rightarrow
0,
\end{align*}
 by applying $D\Hom_{\ce}(-, F)$ to $\eta$. Recall from Subsection~\ref{ss:nakayama-functor} that there is a
triangle functor $\nu \colon \per(A)\rightarrow \thick(DA)$. We deduce from
the above long exact sequences that
$\cd_{fd,A/AeA}(A)=\thick(S_1,\ldots,S_r)\subseteq
\per(A)\cap\thick(DA)$ and that $\nu(S_i)=S_{\pi(i)}[d+1]$. It follows from the Auslander--Reiten
formula (\ref{E:AR-formula}) that the restriction of
$\nu$ on $\cd_{fd,A/AeA}(A)$ is a right Serre functor and hence fully faithful~\cite[Corollary I.1.2]{ReitenVandenBergh02}. 
Since, as shown above, $\nu$ takes a set of generators of $\cd_{fd,A/AeA}(A)$ to itself up to shift, it follows that 
$\nu$ restricts to an auto-equivalence of
$\cd_{fd,A/AeA}(A)$. In particular, $\nu$ is a Serre functor of $\cd_{fd,A/AeA}(A)$. 
Moreover, if
$n$ denotes the number of elements in the $\pi$-orbit of $i$, then
$\nu^n(S_i)\cong S_i[(d+1)n]$, i.e.~ $S_i$ is fractionally Calabi--Yau
of Calabi--Yau dimension $\frac{(d+1)n}{n}$. Finally, we have isomorphisms
\begin{align*}
D\Ext_{A}^l(S_i,S_j)&\cong D\Hom_{A}(S_i,S_j[l])
\cong \Hom_{A}(S_j,\nu(S_i)[-l])\\
&\cong \Hom_{A}(S_j,S_{\pi(i)}[d+1-l])
\cong \Ext_{A}^{d+1-l}(S_j,S_{\pi(i)}),
\end{align*}
where $i,j=1,\ldots,r$
and $l$ denotes an integer. This proves part (d).
\end{proof}

\begin{rem}
(a) \, In Example \ref{Ex:AlmostSplit} (a) above, $\cc$ has $1$-almost split sequences. The theorem shows that the corresponding simple $A$-modules $S_{2}$ and $S_{4}$ are $\frac{(1+1)\cdot 2}{2}$ fractionally Calabi-Yau objects. 

Moreover, by the same argument the simple $A_{nd}$-modules $S_{+}$ and $S_{-}$ are $\frac{4}{2}$-fractionally Calabi--Yau objects when considered as objects in the relative singularity category $\Delta_{nd}$. This was already observed in Subsection \ref{ss:NodalBlock} by a direct computation.

\noindent (b) \, In Example \ref{Ex:AlmostSplit} (b), the category $\cb$ has $2$-almost split sequences. Namely, there is exactly one such sequence with indecomposable end terms. It starts and ends in $S_{-}$ (as shown in the Example)\footnote{This has the following conceptual explanation: $S_{-}$ is a $2$-cluster--tilting object in the $2$-Calabi--Yau category $\ul{\cc}$. Hence its preimage $\cb \subseteq \cc$ has $2$-almost split sequences, see Keller \& Reiten \cite{KellerReiten07} and Part (c) of this Remark.}. Hence the theorem shows that the corresponding $B$-module $S_{4}$ is a $3$-Calabi--Yau object in $\cd_{B/BeB, fd}(B)$. A direct calculation shows that its Ext-Algebra $\Ext_{B}^*(S_{4}, S_{4})$ is isomorphic to the cohomology of the $3$-sphere as graded vector spaces. Hence, $S_{4}$ is a $3$-spherical object in the sense of Seidel \& Thomas \cite{SeidelThomas} Êand $\cd_{B/BeB, fd}(B)$ is  the triangulated category generated by a $3$-spherical object as studied by Keller, Yang \& Zhou \cite{KellerYangZhou}. In particular, $\cd_{B/BeB, fd}(B)$ is a $3$--Calabi--Yau category. 

\noindent (c) \, Part (b) fits into the following general framework of Keller \& Reiten's `relative Calabi--Yau property' \cite{KellerReiten07}.
We use the notations of the theorem above and assume that $\ce$ is Hom-finite. If $\add(F)\subseteq \ce$ is the preimage of a $d$-cluster--tilting subcategory in a $d$--Calabi--Yau category $\ul{\ce}$, then $\cd_{A/AeA, fd}(A)$ is a $(d+1)$-Calabi--Yau category. 

Roughly speaking, our theorem shows that if one takes a cluster--tilting object of the `wrong' dimension, then the category $\cd_{A/AeA, fd}(A)$ will not be Calabi--Yau any more. But the generators are fractionally Calabi--Yau objects. This is sufficient for our purposes.
\end{rem}

\subsection{Independence of the Frobenius model}\label{ss:Independance}
Let $\ct$ be an idempotent complete Hom-finite algebraic triangulated category with only finitely many isomorphism classes of indecomposable objects,
say $M_1,\ldots,M_r$. Then
$\ct$ has a Serre functor \cite[Theorem 1.1]{Amiot07a} and thus has Auslander--Reiten triangles \cite[Theorem I.2.4]{ReitenVandenBergh02}.
Let $\tau$ be the Auslander--Reiten translation. By abuse of notation, $\tau$ also denotes the
induced permutation on $\{1,\ldots,r\}$ defined by $M_{\tau(i)}=\tau M_i$.

The quiver of the \emph{Auslander algebra} $\Lambda(\ct)=\End_{\ct}(\bigoplus_{i=1}^r M_i)$ of $\ct$ 
is the Gabriel quiver $\Gamma$ of $\ct$, in which we identify $i$ with $M_i$. We assume that there exists a sequence of elements $\gamma=\{\gamma_1,\ldots,\gamma_r\}$ in $\widehat{k\Gamma}$, satisfying the following conditions:
\begin{itemize}
 \item[(A1)]  for each vertex $i$ the element $\gamma_i$ is a (possibly infinite) linear combination of paths of $\Gamma$
from $i$ to $\tau^{-1}i$,
\item[(A2)] $\gamma_i$ is non-zero if and only if $\Gamma$ has at least one arrow starting in $i$,
\item[(A3)] the non-zero
$\gamma_i$'s form a set of minimal relations for $\Lambda(\ct)$ (see Subsection~\ref{ss:minimal-relation}).
\end{itemize}
\begin{defn}\label{d:dg-aus-alg}
The \emph{dg Auslander algebra} $\Lambda_{dg}(\ct,\gamma)$ of $\ct$ with respect to $\gamma$
is the dg algebra $(\widehat{kQ},d)$, where $Q$ is a graded quiver and 
$d\colon \widehat{kQ}\rightarrow\widehat{kQ}$ is a map such that
\begin{itemize}
\item[(dgA1)] $Q$ is concentrated in degrees $0$ and $-1$,
\item[(dgA2)] the degree $0$ part of $Q$ is the same as the Gabriel quiver $\Gamma$ of
$\ct$,
\item[(dgA3)]for each vertex $i$, there is precisely one arrow $\begin{xy} \SelectTips{cm}{} \xymatrix{\rho_i\colon i \ar@{-->}[r] & \tau^{-1}(i)} \end{xy}$ of degree
$-1$,
\item[(dgA4)] $d$ is the unique continuous $k$-linear map
on $\widehat{kQ}$ of degree $1$ satisfying the graded Leibniz rule
and taking $\rho_i$ ($i\in Q_0$) to the relation $\gamma_i$.
\end{itemize}
\end{defn}

\noindent In fact, the dg Auslander algebra does not depend on the choice of the sequence $\gamma$:

\begin{prop}\label{p:dg-aus-alg} Let $\ct$ be as above, and let $\gamma=\{\gamma_1,\ldots,\gamma_r\}$ and $\gamma'=\{\gamma'_1,\ldots,\gamma'_r\}$ be sequences of elements of  $\widehat{k\Gamma}$ satisfying the conditions (A1)--(A3). Then the dg Auslander algebras $\Lambda_{dg}(\ct,\gamma)$ and $\Lambda_{dg}(\ct,\gamma')$ are isomorphic as dg algebras.
\end{prop}
\begin{proof}
By the assumptions (A1)--(A3), there exist $c_{i}\in k \setminus \{0\}$ where $i=1,\ldots,r$, an index set $P$ and $c_{pi},c^{pi}\in
\widehat{k\Gamma}$ with $(p,i)\in P\times \{1,\ldots,r\}$, such that for each pair $(p,i)$, at least one of $c_{pi}$ and $c^{pi}$ belongs to the ideal of $\widehat{k\Gamma}$ generated by all arrows,
and
\begin{eqnarray}\label{eq:gamma}
 \gamma'_i &=& c_{i}\gamma_i +\sum_{j=1}^r\sum_{p\in P}c_{pj}\gamma_jc^{pj}.
\end{eqnarray}
We define a continuous graded $k$-algebra homomorphism $\varphi\colon\Lambda_{dg}(\ct,\gamma')\rightarrow \Lambda_{dg}(\ct,\gamma)$ as follows: it is the identity on the degree $0$ part and for arrows of degree $-1$ we set
\begin{eqnarray}\label{eq:rho}
 \varphi(\rho'_i) &=& c_{i}\rho_i +\sum_{j=1}^r\sum_{p\in P}c_{pj}\rho_jc^{pj}.
\end{eqnarray}
Since $\gamma_i=d(\rho_i)$ and $\gamma'_i=d(\rho'_i)$, it follows from (\ref{eq:gamma}) and (\ref{eq:rho}) that
$\varphi$ is a homomorphism of dg algebras. The equation
(\ref{eq:rho}), yields
\begin{eqnarray}\label{eq:rho'}
 \rho_i &=& c_i^{-1}\varphi(\rho'_i) -c_i^{-1}\sum_{j=1}^r\sum_{p\in P}c_{pj}\rho_jc^{pj}.
\end{eqnarray}
By iteratively substituting $c_j^{-1}\varphi(\rho'_j) -c_j^{-1}\sum_{k=1}^r\sum_{p\in P}c_{pk}\rho_jc^{pk}$ for $\rho_{j}$ on the right hand side of (\ref{eq:rho'}), we see
that there exists an index set $P'$ and elements $c'_{pi},c'^{pi}\in\widehat{k\Gamma}$ ($(p,i)\in P'\times \{1,\ldots,r\}$) such that
for each pair $(p,i)$ at least one of $c'_{pi}$ and $c'^{pi}$ belongs to the ideal of $\widehat{k\Gamma}$ generated by all arrows,
and the following equation holds
\begin{eqnarray}
\rho_i = c_i^{-1}\varphi(\rho'_i) -\sum_{j=1}^r\sum_{p\in P'}c'_{pj}\varphi(\rho'_j)c'^{pj}.
\end{eqnarray}
Define a continuous graded $k$-algebra homomorphism $\varphi'\colon\Lambda_{dg}(\ct,\gamma)\rightarrow \Lambda_{dg}(\ct,\gamma')$ as follows: $\varphi'$ is the identity on the degree $0$ part and for arrows of degree $-1$ we set
\begin{eqnarray}
\varphi'(\rho_i) &=& c_i^{-1}\rho'_i -\sum_{j=1}^r\sum_{p\in P'}c'_{pj}\rho'_j c'^{pj}.
\end{eqnarray}
It is clear that $\varphi\circ\varphi'=id$ holds.
Since $\varphi'$ and $\varphi$ have a similar form, the same argument as above shows that there exists a continuous graded $k$-algebra homomorphism $\varphi''\colon\Lambda_{dg}(\ct,\gamma')\rightarrow \Lambda_{dg}(\ct,\gamma)$ such that $\varphi'\circ\varphi''=id$ holds. Therefore we have $\varphi=\varphi''$. In particular, we see that $\varphi$ is an isomorphism.
\end{proof}

Henceforth, we denote by $\Lambda_{dg}(\ct)$ the dg Auslander algebra of $\ct$ with respect to any sequence $\gamma$ satisfying (A1)--(A3). By the definition of $\ct$, there is a triangle equivalence $\ct \cong \ul{\ce}$, for a Frobenius category $\ce$.
We assume that $\ce$ additionally satisfies:
\begin{itemize}
\item[(FM1)] $\ce$ is a Krull--Remak--Schmidt category and $\proj\ce$ has an additive generator $P$,
\item[(FM2)] $\ce$ has only finitely many isoclasses of indecomposable objects, $N_{1}, \ldots, N_{s}$, 
\item[(FM3)] $\ce$ has (1-) almost split sequences,
\item[(FM4)] the \emph{Auslander algebra} $A=\End_{\ce}(\bigoplus_{i=1}^s N_{i})$ of $\ce$ is right Noetherian.
\end{itemize}
Let $e \in A$ be the idempotent endomorphism corresponding to $\id_{P}$, where $P$ denotes the additive generator of $\proj \ce$. We define the \emph{relative Auslander singularity category} as follows
\begin{align}\label{def:relative-auslander-singularity-category} \Delta_{\ce}(A)=\frac{K^b(\proj-A)}{\thick(eA)}. \end{align}
If $\ce=\MCM(R)$ for a Gorenstein ring $R$, then $\Delta_\ce(A)$ is equivalent to the relative singularity category $\Delta_R(\Aus(R))$ as defined in the introduction (\ref{E:DefRelSingCat}).

The following theorem shows that the relative Auslander singularity category depends only on the stable category $\ul{\ce}$. In other words, two Frobenius categories with triangle equivalent stable categories have triangle equivalent relative Auslander singularity categories (up to direct summands).


\begin{thm}\label{t:main-thm-2}
Let $\ce$ be a Frobenius category satisfying conditions (FM1)--(FM4). If $\ct:=\underline{\ce}$ is Hom-finite and idempotent complete, 
then the following statements hold 
\begin{itemize}
 \item[(a)] there is a sequence $\gamma$ of minimal relations for the Auslander algebra of $\ct$ satisfying the above conditions (A1)--(A3),
 \item[(b)] $\Delta_{\ce}(A)$ is triangle equivalent to $\per(\Lambda_{dg}(\ct))$ (up to direct summands),
 \item[(c)] $\Delta_{\ce}(A)$ is Hom-finite.
\end{itemize}
\end{thm}
\begin{rem}
If $\Delta_{\ce}(A)$ is idempotent complete, then we can omit the supplement `up to direct summands' in the statement above. In particular, this holds in the case $\ce=\MCM(R)$, where $(R, \mathfrak{m})$ is a local complete Gorenstein $(R/\mathfrak{m})$-algebra, see Proposition \ref{P:IdempCompl}.
\end{rem}

\begin{proof}
By
Corollary~\ref{c:restriction-and-induction} (a),
there exists a non-positive dg algebra $B$ with $H^0(B)\cong A/AeA$, such that $\left(\Delta_{\ce}(A)\right)^\omega$ is triangle equivalent to $\per(B)$. 
Hence it suffices to show that (a) holds,  that $B$ is quasi-isomorphic to $\Lambda_{dg}(\ct)$ and that $\per(B)$ is Hom-finite.

Let $M$ be an object of $\cd_{fd}(B)$. By Lemma~\ref{L:PropRecoll}, $M$ is isomorphic to $i^*i_*(M)$. So it is contained in $i^*(\cd_{fd,A/AeA}(A))$, by Corollary~\ref{c:restriction-and-induction} (b). Theorem~\ref{t:fractionally-cy-property} (a) shows that all finite-dimensional
$A/AeA$-modules have finite projective dimension over $A$. Thus $\cd_{fd,A/AeA}(A)$ is contained in $K^b(\proj-A)$.
So $M$ is contained in $i^*(K^b(\proj-A))$, which is contained in $\per(B)$, by Corollary~\ref{c:restriction-and-induction} (a). Summing up, we have shown that $\cd_{fd}(B)\subseteq\per(B)$ holds. 

It follows from Proposition~\ref{p:hom-finiteness-of-per} that $H^i(B)$ is finite-dimensional over $k$ for any $i\in\mathbb{Z}$ and $\per(B)$ is Hom-finite.
So by Corollary~\ref{c:koszul-double-dual}, we have that $B$ is quasi-isomorphic to $E(B^*)$, where $B^*$ is the $A_\infty$-Koszul dual of $B$.
Let $S_1,\ldots,S_r$ be a complete set of non-isomorphic simple $A/AeA$-modules and let $S=\bigoplus_{i=1}^r S_i$.
Then $B^*$ is the minimal model of $\RHom_B(S,S)=\RHom_A(S,S)$, see Proposition \ref{Lemma 11.1}. In particular, as a graded algebra, $B^*$ is
isomorphic to $\Ext_A^*(S,S)$. It follows from Theorem~\ref{t:fractionally-cy-property} that $\Ext_A^*(S,S)$
is concentrated in degrees $0$, $1$ and $2$. Certainly,
\begin{align*}
\Ext_A^0(S_i,S_j)\cong \begin{cases} k \quad \text{ if }Êi = j \\ 0 \quad \text{ else. } \end{cases}
\end{align*}
 In the current situation, the permutation $\pi$ from Theorem~\ref{t:fractionally-cy-property} is induced by the Auslander--Reiten translation $\tau$. Abusing notation, we denote this permutation also by $\tau$.
 Then Theorem~\ref{t:fractionally-cy-property} shows
\begin{align*}  
\Ext_A^2(S_i,S_j)\cong \begin{cases} k \quad \text{ if }Êj = \tau(i) \\ 0 \quad \text{ else. } \end{cases}
\end{align*}
Hence, $E(B^*)=(\widehat{kQ},d)$ for a graded quiver $Q$ and a continuous $k$-linear
differential $d$ of degree $1$, where the graded quiver $Q$ is concentrated in degree $0$ and $-1$,
and starting
from any vertex $i$ there is precisely one arrow $\rho_i$ of degree $-1$  whose target is $\tau^{-1}(i)$, see the paragraph after Corollary \ref{c:koszul-double-dual}. 

Let $Q^0$ denote the
degree $0$ part of $Q$.
Then $H^0(E(B^*))=\widehat{kQ^0}/\overline{(d(\rho_i))}$. Indeed, by continuity $\im(d) \subseteq \overline{(d(\rho_i))}$. To see the reverse inclusion, let $\mathfrak{n} \subseteq \widehat{kQ^0}$ be the ideal generated by the arrows. It suffices to show that there exists a positive integer $l$ such that $\mathfrak{n}^l \subseteq \im(d)$. Let $q\colon \widehat{kQ^0} \ra \widehat{kQ^0}/\im(d)$ be the canonical projection. Then $q(\mathfrak{n}) \subseteq \rad(\widehat{kQ^0}/\im(d))$ is a nilpotent ideal, since 
\begin{align}\label{E:0cohoAusl}
H^0(E(B^*))\cong H^0(B)\cong A/AeA=\Lambda(\ct)
\end{align}
is the Auslander algebra of $\ct$, which is finite dimensional by assumption.

It follows from \eqref{E:0cohoAusl} that $Q^0$ is the same as the Gabriel quiver $\Gamma$ of $\ct$. Moreover, 
$\gamma=\{d(\rho_1),\ldots,d(\rho_r)\}$ is
a set of relations for $\Lambda(\ct)$. We claim that $\gamma$ is a sequence satisfying the conditions (A1)--(A3). Then (a) holds and
$E(B^*)=\Lambda_{dg}(\ct,\gamma)=\Lambda_{dg}(\ct)$, which implies that $B$ is quasi-isomorphic to $\Lambda_{dg}(\ct)$. 

Since we already know that $ 
\begin{xy}
\SelectTips{cm}{}
\xymatrix{\rho_{i}\colon i \ar@{-->}[r] & \tau^{-1}(i)}
\end{xy}$ holds, $d(\rho_i)$ is a combination of paths from $i$ to $\tau^{-1}(i)$, for all $i=1, \ldots, r$. Hence condition (A1) holds and $d(\rho_i)\neq 0$ implies that
$\Gamma$ has at least one arrow starting in $i$. This is one implication in (A2). In order to show the other implication, we assume that $\Gamma$ has an arrow starting in $i$. Then there is an Auslander--Reiten triangle 
\begin{align}\label{E:ARtria}
M_{i} \xrightarrow{f_{i}} X_{i} \xrightarrow{g_{i}} M_{\tau^{-1}(i)} \ra M_{i}[1]
\end{align}
in $\ct$, where $f_{i}$ and $g_{i}$ are non-zero and irreducible. We may view $f_{i}$ and $g_{i}$ as elements of $\widehat{k\Gamma}$. The arrows\footnote{When we write `arrow', we also mean the corresponding irreducible map in $\ct$.}  of $\Gamma$ which start in $i$ form a basis of the vectorspace of irreducible maps $\rad(M_{i}, X_{i})/\rad^2(M_{i}, X_{i})$, see Happel \cite[Section 4.8]{Happel88}. In particular, $f_{i}$ may be written as follows
\begin{align}\label{E:LinComb1}
f_{i}= \sum_{j=1}^m \lambda_{j}\alpha_{j} + r_{i},
\end{align}
where the $\alpha_{j}$ are the arrows starting in $i$, $r_{i} \in \rad^2(M_{i}, X_{i})$ and the $\lambda_{j} \in k$ are not all zero. Similarly, 
\begin{align}\label{E:LinComb2}
g_{i}= \sum_{j=1}^m \mu_{j}\beta_{j} + s_{i},
\end{align}
where the $\beta_{j}$ are the arrows ending in $\tau^{-1}(i)$, $s_{i} \in \rad^2(X_{i}, M_{\tau^{-1}(i)})$ and the $\mu_{j} \in k$ are not all zero.

Define $m_{i}=g_{i}f_{i}$ in $\widehat{k\Gamma}$, which is a relation for $\Lambda(\ct)$, since \eqref{E:ARtria} is a triangle.
Therefore, it is generated by $\{d(\rho_1),\ldots, d(\rho_r)\}$. In other words, 
there exists an index set $P$ and elements $c_{pj},c^{pj}\in\widehat{k\Gamma}$ ($(p,j)\in P
\times \{1,\ldots,r\}$) such that
\begin{eqnarray}\label{E:GenMeshRel}
 m_i &=& \sum_{j=1}^r\sum_{p\in P} c_{pj}d(\rho_j)c^{pj}.
\end{eqnarray}
Let $J$ be the ideal of $\widehat{k\Gamma}$ generated by all arrows. Since $B^*$ is a minimal $A_\infty$-algebra, it follows that
$d(\rho_j)\in J^2$ holds for any $j=1,\ldots,r$, see Subsection~\ref{ss:dual-bar-construction}. If $j\neq i$ and $c_{pj}d(\rho_j)c^{pj}\neq 0$,
then $c_{pj}d(\rho_j)c^{pj}$ is a combination of paths of length at least $4$, because $m_i$ is a linear combination of paths from $i$ to $\tau^{-1}(i)$,
whereas $d(\rho_j)$ is a linear combination of path from $j$ to $\tau^{-1}(j)$. Using \eqref{E:LinComb1} and \eqref{E:LinComb2}, we see that $m_i$ has a non-zero length $2$ component. Thus \eqref{E:GenMeshRel} implies that $\sum_{p \in P} c_{pi}d(\rho_i)c^{pi}$
is non-zero and its length $2$ component equals that of $m_i$. 
In particular, $d(\rho_i)$ is non-zero and cannot be generated by 
$\{d(\rho_j)\}_{j\neq i}$. To summarise, $d(\rho_i)\neq 0$ if and only if $\Gamma$ has arrows starting in $i$ (A2), and the non-zero $d(\rho_i)$'s
form a set of minimal relations for $\Lambda(\ct)$ (A3). 
\end{proof}

\begin{rem}
 Let $\ct$ be an idempotent complete Hom-finite algebraic triangulated category with only finitely 
many isomorphism classes of indecomposable objects. We say that $\ct$ is \emph{standard} if the Auslander algebra
$\Lambda(\ct)$ is given by the Auslander--Reiten quiver with mesh relations, see~\cite[Section 5]{Amiot07a}. Examples of
non-standard categories can be found in~\cite{Riedtmann83,Asashiba99}.

Assume that $\ct$ is standard and $\ct\cong
\underline{\ce}$ for some Frobenius category $\ce$ satisfying (FM1)--(FM4). Then Theorem~\ref{t:main-thm-2} and Proposition \ref{p:dg-aus-alg} show that 
$\Delta_{\ce}(A)$ is determined by the Auslander--Reiten quiver of $\ct$ (up to direct summands).
\end{rem}

\subsection{Classical vs.~relative singularity categories}\label{S:MCM-over-Gor}
Let $k$ be an algebraically closed field. Throughout this subsection $(R, \mathfrak{m})$ and $(R', \mathfrak{m}')$ denote commutative local complete Gorenstein $k$-algebras, such that their respective residue fields are isomorphic to $k$. 

\subsubsection{Classical singularity categories}\label{sss:Classical}

Let $\MCM(R)$ be the category of maximal Cohen--Macaulay $R$-modules. Since $R$ is Gorenstein,  $\MCM(R)=\GP(R)$ is a Frobenius category with $\proj \MCM(R) = \proj-R$, see Proposition \ref{P:IwanagaFrobenius}. Hence, $\ul{\MCM}(R)=\MCM(R)/\proj-R$ is a triangulated category \cite{Happel88}.

The following concrete examples of hypersurface rings are of particular interest: 
Let $R=\C\llbracket z_{0}, \ldots, z_{d}\rrbracket/(f)$, where $d \geq 1$ and $f$ is one of the following polynomials
\begin{itemize}
\item[$(A_{n})$] \quad  $z_0^2 + z_1^{n+1} + z^2_{2} + \ldots + z_{d}^2 \quad  \,\, \, \, \, ( n \geq 1 )$,
\item[$(D_{n})$] \quad $z_0^2z_1 + z_1^{n-1} + z^2_{2} + \ldots + z_{d}^2 \quad ( n \geq 4 )$,
\item[$(E_{6})$] \quad $z_0^3 + z_1^4 + z^2_{2} + \ldots + z_{d}^2$,
\item[$(E_{7})$] \quad $z_0^3 + z_0z_1^3 + z^2_{2} + \ldots + z_{d}^2$,
\item[$(E_{8})$] \quad $z_0^3 + z_1^5 + z^2_{2} + \ldots + z_{d}^2$.
\end{itemize}

Such a $\C$-algebra $R$ is called \emph{ADE--singularity} of dimension $d$. As hypersurface singularities they are known to be Gorenstein, see e.g.~\cite{BrunsHerzog}. The following result is known as Kn\"orrer's Periodicity Theorem, see \cite{Knoerrer}. It was the main motivation for Theorem \ref{t:main-thm-2} and Theorem \ref{t:Classical-versus-generalized-singularity-categories}.

\begin{thm} \label{t:knoerrer}
Let $d \geq 1$ and $k$ an algebraically closed field such that $\mathsf{char} k \neq 2$. Let $S=k\llbracket z_{0}, \ldots, z_{d} \rrbracket$ and $f \in (z_{0}, \ldots, z_{d}) \setminus \{0\}$. Set $R=S/(f)$ and $R'=S\llbracket x, y \rrbracket/(f+xy)$. Then there is a triangle equivalence 
\begin{align}
\ul{\MCM}(R') \rightarrow \ul{\MCM}(R).
\end{align}
\end{thm}

\begin{defn}
We say that $R$ is \emph{$\MCM$--finite} if there are only finitely many isomorphism classes of indecomposable maximal Cohen--Macaulay $R$--modules.
\end{defn}

\begin{rem}
Solberg \cite{Solberg89} showed that Theorem \ref{t:knoerrer} also holds in characteristic $2$ if $R$ is $\MCM$-finite.
\end{rem}

It follows from Theorem \ref{t:knoerrer} that $R$ is $\MCM$--finite if and only if $R'$ is $\MCM$-finite. The ADE--curve singularities are  $\MCM$-finite by work of Drozd \& Roiter \cite{DrozdRoiter} and Jacobinsky \cite{Jacobinsky}. Moreover, the ADE--surface singularities are $\MCM$-finite by work of Artin \& Verdier  \cite{ArtinVerdier85}. This has the following well-known consequence.

\begin{cor} Let $R$ be an ADE--singularity as above. Then $R$ is $\MCM$--finite.
\end{cor}

\begin{rem} 
If $k$ is an arbitrary algebraically closed field, then the ADE-polynomials listed above still describe $\MCM$--finite singularities. Yet there exist further $\MCM$--finite rings if $k$ has characteristic $2$, $3$ or $5$ (complete lists are contained in \cite{GreuelKroening90}).
\end{rem}

\subsubsection{Relative singularity categories}

Henceforth, let $F'$ be a finitely generated $R$-module and $F=R \oplus F'$. We call $A=\End_{R}(F)$ a \emph{partial resolution} of $R$. If $A$ has finite global dimension, then we call we call $A$ is a \emph{non-commutative resolution} of $R$. Denote by $e\in A$ the idempotent endomorphism corresponding to the identity morphism $\id_{R}$ of $R$.

The situation is particularly nice if $R$ is $\MCM$--finite. Let $M_{0}=R, M_{1}, \ldots, M_{t}$ be representatives of 
the indecomposable objects of $\MCM(R)$. Their endomorphism algebra $\Aus(\MCM(R))=\End_{R}(\bigoplus_{i=0}^t M_{i})$ is called 
the \emph{Auslander algebra}. Auslander \cite[Theorem A.1]{Auslander84} has shown that its global dimension is bounded above by the Krull dimension of $R$ (respectively by $2$ in Krull dimensions $0$ and $1$; for this case see also Auslander's treatment in \cite[Sections III.2 and III.3]{Auslander71}). Hence, $\Aus(\MCM(R))$ is a resolution of $R$.
 The next lemma motivates the definition  of the \emph{relative singularity categories}.

\begin{lem}\label{L:embeds} 
There is a fully faithful triangle functor $K^b(\proj-R) \ra \cd^b(\mod-A)$.
\end{lem}
\begin{proof}
By the definition of $A$, there exists an idempotent $e \in A$ such that $R \cong eAe$. Moreover, Proposition \ref{P:AdjTriple} yields a fully faithful functor
\begin{align}
-\lten_{eAe}eA \colon \cd(eAe) \longrightarrow \cd(A),
\end{align}
which descends to an embedding between the categories of compact objects
\begin{align}
-\lten_{eAe}eA\colon K^b(\proj-eAe) \longrightarrow K^b(\proj-A) \subseteq \cd^b(\mod-A).
\end{align}
This completes the proof.
\end{proof}

\begin{defn}
In the notations above and using Lemma \ref{L:embeds}, we can define the \emph{relative singularity category} of the pair $(R, A)$ as the triangulated quotient category
\begin{align}\label{E:Def-rel-sing-cat} \Delta_{R}(A)=\frac{\cd^b(\mod-A)}{K^b(\proj-R)}. \end{align}
\end{defn}

\begin{rem}
Note that the image of $R\cong eAe$ in $\cd^b(A)$ is $eA$. Hence,  \begin{align}\Delta_{R}(A) \cong \cd^b(\mod-A)/\thick(eA).\end{align}  We will use both presentations of $\Delta_{R}(A)$ in the sequel. Since $eA$ is a projective $A$-module, $\Delta_{R}(A)$ is a relative singularity category in the sense of Chen \cite{Chen11}. Different notions of relative singularity categories were introduced and studied by Positselski \cite{Positselski11} and also by Burke \& Walker \cite{BurkeWalker12}. We thank Greg Stevenson for bringing this unfortunate coincidence to our attention. 
\end{rem}

Let $G'$ be another finitely generated $R$-module, which contains $F'$ as a direct summand. As above, we define $G=R \oplus G'$, $A'=\End_{R}(G)$ and $e'=\id_{R} \in A'$. 

We compare the relative singularity categories of $A$ and $A'$ respectively.

\begin{prop}\label{P:Embedding-of-relative}
If $A$ has finite global dimension, then there is a fully faithful triangle functor 
\begin{align}
\Delta_{R}(A) \longrightarrow \Delta_{R}(A').
\end{align}
\end{prop}
\begin{proof}
There is an idempotent $f \in A'$ such that $A \cong fA'f$. This yields a commutative diagram
\begin{align}
\begin{array}{c}
\begin{xy}
\SelectTips{cm}{}
\xymatrix{
&&    \ar[lld]_{\displaystyle -\otimes_{e'A'e'}e'A'f \qquad \, \, \, }     K^b(\proj-e' A' e') \ar[rrd]^{\displaystyle \qquad -\otimes_{e'A'e'}e'A'} \\
K^b(\proj-fA'f) \ar[rrrr]^{\displaystyle - \otimes_{fA'f}fA'}   &&&&     K^b(\proj-A')
}
\end{xy}
\end{array}
\end{align} 
The two `diagonal' functors are the embeddings from Lemma \ref{L:embeds} and the horizontal functor is fully faithful by the same argument. Since $fA'f$ has finite global dimension, this functor yields an embedding $\cd^b(\mod-fA'f) \ra \cd^b(\mod-A')$. Passing to the triangulated quotient categories yields the claim.
\end{proof}

\begin{rem}
The assumption on the global dimension of $A$ is necessary. As an example consider the nodal curve singularity
$A=R=k\llbracket x, y \rrbracket/(xy)$ and its Auslander algebra $A'=\End_{R}(R \oplus k\llbracket x \rrbracket \oplus k\llbracket y \rrbracket)$. In this situation Proposition \ref{P:Embedding-of-relative} would yield an embedding $\underline{\MCM}(R)=\Delta_{R}(R) \rightarrow \Delta_{R}(A').$ But, $\underline{\MCM}(R)$ contains an indecomposable object $X$ with $X \cong X[2s] $ for all $s \in \mathbb{Z}$. Whereas, $\Delta_{R}(A')$ does not contain such objects by the explicit description obtained in \ref{ss:NodalBlock}. Contradiction. 
\end{rem}

\begin{prop}\label{C:From-gen-to-class}
There exists an equivalence of triangulated categories
\begin{align}
\frac{\Delta_{R}(A)}{\cd^b_{A/AeA}(\mod-A)} \longrightarrow \ul{\MCM}(R).
\end{align}
\end{prop}
\begin{proof}
Since $R \cong eAe$, this follows from Corollary \ref{P:OneIdempotent}. 
\end{proof}

We want to give an intrinsic description of the full subcategory $\cd^b_{A/AeA}(\mod-A)$ inside the relative singularity category $\Delta_{R}(A)$. We need some preparation. 

\begin{prop}\label{P:Intrinsic}
In the notations of Propositions \ref{p:recollement-from-projective-general-case} and  \ref{C:From-gen-to-class} assume additionally that $A$ has finite global dimension and $A/AeA$ is finite dimensional. 

Then there exists a non-positive dg algebra $B$ and a commutative diagram 
\begin{align}
\begin{array}{c}
\begin{xy}
\SelectTips{cm}{}
\xymatrix{ \thick_{\cd^b(\mod-A)}(\mod-A/AeA) \ar@{^{(}->}[r] & \Delta_{R}(A) \ar[d]_\cong^{\displaystyle i^{*}} \ar@{->>}[r]   & \ul{\MCM}(R)   \ar[d]_\cong^{\displaystyle\I} \\
\cd_{fd}(B) \ar[u]^\cong_{\displaystyle i_{*}}  \ar@{^{(}->}[r]  & \per(B) \ar@{->>}[r]  & \per(B)/\cd_{fd}(B) }
\end{xy}
\end{array}
\end{align}
where the horizontal arrows denote (functors induced by) the canonical inclusions and projections respectively. Finally, the triangle functor $\I$ is induced by $i^*$.
\end{prop}

\begin{rem}
Thanhoffer de V\"olcsey \& Van den Bergh \cite{ThanhofferdeVolcseyMichelVandenBergh10} obtained a similar result by different means.
\end{rem}

\begin{proof}
Firstly, $\Delta_{R}(A)$ is idempotent complete by Proposition \ref{P:IdempCompl}. Since $A$ has finite global dimension, Corollary \ref{c:restriction-and-induction} implies the existence of a dg $k$-algebra $B$ with $i^*\colon \cd^b(\mod-A)/\thick(eA) \cong \per(B)$. Moreover, since $\dim_{k}(A/AeA)$ is finite $i_{*}\colon \cd_{fd}(B) \cong \thick(\mod-A/AeA)$ by the same corollary and $\cd_{fd}(B) \subseteq \per(B)$ as seen in the proof of Theorem \ref{t:main-thm-2}. The inclusion 
$\thick_{\cd^b(\mod-A)}(\mod-A/AeA) \hookrightarrow \Delta_{R}(A)$ is induced by the inclusion $\mod-A/AeA \hookrightarrow \mod-A$ (see the proof of Corollary \ref{P:OneIdempotent}). Since $i_{*}$ and $i^*$ are part of a recollement (Proposition \ref{p:recollement-from-projective-general-case}) we obtain $i^*\circ i_{*}=\id_{\cd(B)}$. 
Hence the first square commutes. The second square commutes by definition of $\I$. 
\end{proof}
Note that under the assumptions of Proposition~\ref{P:Intrinsic}, we have equalities
\begin{align}\label{E:Dfd}
\cd_{fd,A/AeA}(A)=\thick_{\cd^b(\mod-A)}(\mod-A/AeA)=\cd^b_{A/AeA}(\mod-A).
\end{align}
Moreover, combining this Proposition with Proposition \ref{p:hom-finiteness-of-per} yields the following.
\begin{prop}\label{p:hom-finiteness-of-delta}
In the setup of Prop. \ref{P:Intrinsic}, the category $\Delta_{R}(A)$ is $\Hom$--finite.
\end{prop} 
\begin{rem}
For example, this holds for an isolated singularity $R$ and a non-commutative resolution $A=\End_{R}(F)$ of $R$, with $F \in \MCM(R)$, see \cite{Auslander84} and the proof of Theorem \ref{t:Classical-versus-generalized-singularity-categories}.
\end{rem}

In general, Hom-finiteness is not preserved under passage to quotient categories as the following example shows.
Therefore, Proposition \ref{p:hom-finiteness-of-delta} is quite surprising.
\begin{ex}\label{Ex:Hom-infinite}
Let $\kA=\mathsf{rep}_{k}(\xymatrix{1 \ar@/^/[r] \ar@/_/[r] & 2})$ be the category of representations of the Kronecker quiver. Let $\kR \subseteq \kA$ be the subcategory of regular modules. In this case, these are precisely the modules which are invariant under the Auslander-Reiten translation. In other words, the modules lying in \emph{tubes} of the AR-quiver.  The category $\kR$ is known to be abelian, but it is \emph{not} a Serre subcategory in $\kA$ as the exact sequence
\[0 \rightarrow S_{2} \rightarrow (\xymatrix{k \ar@/^/[r]^1 \ar@/_/[r]_1 & k})\rightarrow S_{1} \rightarrow 0\] shows. However, it is equivalent to the Serre subcategory $\Tor(\PP^1) \subseteq \Coh(\PP^1)$ of torsion sheaves. Moreover, this equivalences is induced by the well-known tilting equivalence $\cd^b(\kA) \rightarrow \cd^b(\Coh(\PP^1))$.  Using this and Miyachi's theorem \cite[Theorem 3.2]{Miyachi}, we obtain a chain of equivalences
\[\frac{\cd^b(\kA)}{\cd^b_{\kR}(\kA)} \cong \frac{\cd^b(\Coh(\PP^1))}{\cd^b_{\Tor(\PP^1)}(\Coh(\PP^1))} \cong \cd^b\left(\frac{\Coh(\PP^1)}{\Tor(\PP^1)}\right)\]
The Serre quotient $\Coh(\PP^1)/\Tor(\PP^1)$ is known to be equivalent to the category of finite dimensional modules over the function field $k(t)$, and in particular \emph{not} $\Hom$-finite over $k$.  
\end{ex}

\begin{defn}
For a triangulated $k$-category $\ct$ the full triangulated subcategory 
\begin{equation*}
\ct_{hf}= \left\{ X \in \ct \, \left| \,\,  \dim_{k} \bigoplus_{i \in \Z} \Hom_{\ct}(Y, X[i]) < \infty  \text{  for all  } Y \in \ct \right\}\right. \end{equation*} is called \emph{subcategory of right homologically finite objects}.  
\end{defn}

\begin{ex} \label{E:Homolog-fin}
If $B$ is a dg $k$-algebra satisfying $\cd_{fd}(B) \subseteq \per(B)$, then $\per(B)_{hf}=\cd_{fd}(B)$.
Indeed, this follows from $\Hom(B, X[i]) \cong H^i(X)$ for any dg $B$-module $X$.
\end{ex}

\begin{cor}\label{C:Intrinsic}
In the notations of  Proposition \ref{P:Intrinsic}  there is an equality 
\begin{align} 
\cd^b_{A/AeA}(\mod-A)=\Delta_{R}(A)_{hf}.
\end{align}
\end{cor}
\begin{proof} This follows from Proposition \ref{P:Intrinsic} in conjunction with Example \ref{E:Homolog-fin}. 
\end{proof}

\subsubsection{Main result}

Now, we are able to state and prove the main result of this section. In particular, it applies to the ADE--singularities, which are listed above. 

\begin{thm}\label{t:Classical-versus-generalized-singularity-categories}
 Let $R$ and $R'$ be $\MCM$--finite complete Gorenstein $k$-algebras with Auslander algebras $A=\Aus(\MCM(R))$ and $A'=\Aus(\MCM(R'))$, respectively. Then the following statements are equivalent.
\begin{itemize}
\item[(a)] There exists an additive equivalence $\ul{\MCM}(R) \cong \ul{\MCM}(R')$, which respects the action of the respective Auslander--Reiten translations on objects.
\item[(b)] There is an equivalence $\ul{\MCM}(R) \cong \ul{\MCM}(R')$ of triangulated categories.
\item[(c)] There exists a triangle equivalence $\Delta_{R}\!\left(A\right) \cong \Delta_{R'}\!\left(A'\right)$.
\end{itemize}
Moreover, the implication $[(c) \Rightarrow (b)]$ (and hence also $[(c) \Rightarrow (a)]$) holds under much weaker assumptions. Namely, if $A$ and $A'$ are non-commutative resolutions of isolated Gorenstein singularities $R$ and $R'$ respectively.
\end{thm}
\begin{proof} $[(b) \Rightarrow (a)]$ Since Serre functors are unique, this also holds for the Auslander--Reiten translation by \cite{ReitenVandenBergh02}. In particular, triangle equivalences commute with the Auslander--Reiten translations.

$[(a) \Rightarrow (c)]$ Let $R$ be $\MCM$--finite. It is sufficient to show that the Frobenius category $\MCM(R)$ satisfies the assumptions of Theorem \ref{t:main-thm-2}. Indeed, this implies 
\begin{align}\label{E:rel-sing-cat-as-per}
\Delta_{R}(A) \cong \per(\Lambda_{dg}(\ul{\MCM}(R)), 
\end{align} since $\Delta_{R}(A)$ is idempotent complete by Proposition \ref{P:IdempCompl}. Since the construction of the dg Auslander algebra $\Lambda_{dg}(\ul{\MCM}(R))$ only depends on the additive structure of $\ul{\MCM}(R)$ and the action of its Auslander--Reiten translation on objects, \eqref{E:rel-sing-cat-as-per} and (a) imply (c).

The assumptions, which we have to verify are: existence of almost split sequences in $\MCM(R)$; $\Hom$-finiteness and idempotent completeness of the stable category $\ul{\MCM}(R)$. The last statement follows from Lemma \ref{L:stable-complete}. The first two assertions were shown by Auslander \cite{Auslander84}: indeed, $R$ is an $S=k\llbracket x_{1},\ldots, x_{d}\rrbracket$-order, where $d=\krdim(R)$, see e.g.~\cite{Nagata}. Then an $R$-module is MCM if and only if it is MCM as an $S$-module. Since $S$ is regular, an $S$-module is MCM if and only if it is free.
Since $R$ has only finitely many isomorphism classes of indecomposable MCMs, \cite[Theorem 10]{Auslander84} implies that $R$ is regular or has an isolated singularity. Now, the main theorem in \emph{op.~cit.} completes the proof.

$[(c) \Rightarrow (b)]$ We claim that this is a consequence of Proposition \ref{C:From-gen-to-class} and Corollary \ref{C:Intrinsic}. Indeed, by Proposition \ref{C:From-gen-to-class} the stable category $\ul{\MCM}(R)$ is a quotient of $\Delta_{R}(A)$ and by Corollary \ref{C:Intrinsic} the kernel of the quotient functor $\cd^b_{A/AeA}(\mod-A) \subseteq \Delta_{R}(A)$ has an intrinsic characterization. Hence the triangle equivalence in (c) induces an equivalence between the respective quotient categories as in (b).

We verify the (stronger) assumptions of Corollary \ref{C:Intrinsic}. $\Hom$-finiteness of $\ul{\MCM}(R)$ follows as in the proof of $[(a) \Rightarrow (c)]$ and holds more generally for any (complete) isolated singularity $R$. In particular, the algebra $A/AeA$ is finite dimensional. Since $A$ is the Auslander algebra of $\MCM(R)$, it has finite global dimension by \cite[Theorem A.1]{Auslander84}.
\end{proof}

\begin{ex}\label{E:Conifold}
Let $R=\C\llbracket u, v  \rrbracket/(uv)$ and $R''=\C\llbracket u, v, w, x  \rrbracket/(uv+ wx)$ be the one and three dimensional $A_{1}$--singularities, respectively. The latter is also known as the `conifold'. The Auslander--Reiten quivers $\ca(R)$ and $\ca(R'')$ of $\MCM(R)$ respectively $\MCM(R'')$, are known, cf. \cite{Schreyer87} (in particular, \cite[Remark 6.3]{Schreyer87} in dimensions $\geq 3$):

 \begin{equation*}\begin{tikzpicture}[description/.style={fill=white,inner sep=2pt}]

    \matrix (n) [matrix of math nodes, row sep=2em,
                 column sep=1.5em, text height=1.5ex, text depth=0.25ex,
                 inner sep=0pt, nodes={inner xsep=0.3333em, inner
ysep=0.3333em}] at (0, 0)
    {+ && *  && - &&&&&             +  && \star &&- \\};
  
 \node at ($(n-1-1.west) + (-8mm, 0mm)$) {$\ca(R)=$};
 \node at ($(n-1-1.west) + (64mm, 0mm)$) {$\ca(R'')=$};
 
     \draw[->] ($(n-1-1.east) + (0mm,1mm)$) .. controls +(4.5mm,2mm) and
+(-4.5mm,+2mm) .. ($(n-1-3.west) + (0mm,1mm)$);

    \draw[->] ($(n-1-3.east) + (0mm,1mm)$) .. controls +(4.5mm,2mm) and
+(-4.5mm,+2mm) .. ($(n-1-5.west) + (0mm,1mm)$);

    \draw[->] ($(n-1-3.west) + (0mm,-1mm)$) .. controls +(-4.5mm,-2mm)
and +(+4.5mm,-2mm) .. ($(n-1-1.east) + (0mm,-1mm)$);
 
  \draw[->] ($(n-1-5.west) + (0mm,-1mm)$) .. controls +(-4.5mm,-2mm)
and +(+4.5mm,-2mm) .. ($(n-1-3.east) + (0mm,-1mm)$);

 \path[dash pattern = on 0.5mm off 0.3mm, ->] ($(n-1-1.north) + (0.5 mm,-0.5mm)$) edge [bend left=50] ($(n-1-5.north) + (-.5mm,-.5mm)$);

 \path[dash pattern = on 0.5mm off 0.3mm, <-] ($(n-1-1.south) + (.5 mm,.5mm)$) edge [bend right=50]  ($(n-1-5.south) + (-.5mm,.5mm)$);


    \draw[->] ($(n-1-10.east) + (0mm,1mm)$) .. controls +(4.5mm,2mm) and
+(-4.5mm,+2mm) .. ($(n-1-12.west) + (0mm,1mm)$);

    \draw[->] ($(n-1-12.east) + (0mm,1mm)$) .. controls +(4.5mm,2mm) and
+(-4.5mm,+2mm) .. ($(n-1-14.west) + (0mm,1mm)$);

    \draw[->] ($(n-1-12.west) + (0mm,-1mm)$) .. controls +(-4.5mm,-2mm)
and +(+4.5mm,-2mm) .. ($(n-1-10.east) + (0mm,-1mm)$);

    \draw[->] ($(n-1-14.west) + (1mm,-2mm)$) .. controls +(-3.7mm,-3mm)
and +(+3.7mm,-3mm) .. ($(n-1-12.east) + (-1mm,-2mm)$);

    \draw[->] ($(n-1-10.east) + (-1mm,2mm)$) .. controls +(3.7mm,3mm) and
+(-3.7mm,+3mm) .. ($(n-1-12.west) + (1mm,2mm)$);

    \draw[->] ($(n-1-12.east) + (-1mm,2mm)$) .. controls +(3.7mm,3mm) and
+(-3.7mm,+3mm) .. ($(n-1-14.west) + (1mm,2mm)$);

    \draw[->] ($(n-1-12.west) + (1mm,-2mm)$) .. controls +(-3.7mm,-3mm)
and +(+3.7mm,-3mm) .. ($(n-1-10.east) + (-1mm,-2mm)$);

 \draw[->] ($(n-1-14.west) + (0mm,-1mm)$) .. controls +(-4.5mm,-2mm)
and +(+4.5mm,-2mm) .. ($(n-1-12.east) + (0mm,-1mm)$);

 \path[dash pattern = on 0.5mm off 0.3mm, ->] ($(n-1-10.north) + (0.5 mm,-0.5mm)$) edge [bend left=50] ($(n-1-14.north) + (-.5mm,-.5mm)$);

 \path[dash pattern = on 0.5mm off 0.3mm, <-] ($(n-1-10.south) + (.5 mm,.5mm)$) edge [bend right=50]  ($(n-1-14.south) + (-.5mm,.5mm)$);

\end{tikzpicture}\end{equation*}
Let $A$ and $A''$ be the respective Auslander algebras of $\MCM(R)$ and $\MCM(R'')$. They are given as quivers as quivers with relations, where the quivers are just the `solid' subquivers of $\ca(R)$ and $\ca(R'')$, respectively. Now, Kn\"orrer's Periodicity Theorem \ref{t:knoerrer} and Theorem \ref{t:Classical-versus-generalized-singularity-categories} above show that there is an equivalence of triangulated categories
\begin{align}
\frac{\cd^b(\mod-A)}{K^b(\add P_{*})} \longrightarrow \frac{\cd^b(\mod-A'')}{K^b(\add P_\star)}, 
\end{align}
where $P_{*}$ is the indecomposable projective $A$-module corresponding to the vertex $*$ and similarly $P_{\star} \in \proj-A''$ corresponds to $\star$. 

 The relative singularity category $\Delta_{R}(A)=\cd^b(\mod-A)/K^b(\add P_{*})$ from above has an explicit description, as we have seen in Subsection \ref{ss:NodalBlock}.  

\end{ex}

\begin{rem}
Theorem \ref{t:Classical-versus-generalized-singularity-categories} also holds for selfinjective $k$-algebras $\Lambda$ of finite representation type. In this case, we have $\MCM(\Lambda):=\GP(\Lambda)=\mod-\Lambda$ and the proof of the theorem is very similar to the commutative case. Moreover, one can prove (the analogue of) implication [(b) $\Rightarrow$ (c)] in Theorem \ref{t:Classical-versus-generalized-singularity-categories} without relying on dg--techniques. Indeed, Asashiba \cite[Corollary 2.2.]{Asashiba99} has shown that in this context stable equivalence implies derived equivalence. Now, Rickard's \cite[Corollary 5.5.]{Rickard91} implies that the respective Auslander algebras are derived equivalent (a result, which was recently obtained by W. Hu and C.C. Xi in a much more general framework \cite[Corollary 3.13]{HuXi09}\footnote{I would like to thank Sefi Ladkani for pointing out this reference.}). One checks that this equivalence induces a triangle equivalence between the respective relative singularity categories. This result is stronger than the analogue of Theorem \ref{t:Classical-versus-generalized-singularity-categories} (c).
\end{rem}

\subsubsection{Grothendieck group}
This paragraph is contained in a joint work with Igor Burban \cite{BurbanKalck11}.
Let $(R, \mathfrak{m})$ be a local complete Gorenstein ring and let $A=\End_{R}(F)$ be a non-commutative resolution of $A$.
We compute the Grothendieck group of the relative singularity category $\Delta_{R}(A)$.

\begin{prop}\label{P:GrothendieckGroup}
 Let $F \cong R \oplus F_{1} \oplus \cdots \oplus F_{r}$ be a decomposition into indecomposable direct summands, which we may assume to be pairwise non-isomorphic. Then $K_{0}\bigl(\Delta_{R}(A)\bigr) \cong \ZZ^r$.
\end{prop}
\begin{proof}
Since $A$ is a finitely generated $R$-algebra and $R$ is complete, $A$ is \emph{semi-perfect} \cite{CurtisReiner}. Moreover, there are precisely $r+1$ pairwise nonisomorphic indecomposable projective $A$-modules.
\cite[Proposition 16.7]{CurtisReiner} implies that $K_{0}(\proj-A) \cong \ZZ^{r+1}$.
Since $\gldim(A) < \infty$, we get $K_{0}(\proj-A) \cong K_{0}(\mod-A)$  (\cite[Corollary 38.51]{CurtisReiner2}). This yields
$K_0\bigl(\cd^b(\mod-A)\bigr) \cong K_{0}(\mod-A) \cong \ZZ^{r+1}$. 
By \cite[Proposition VIII 3.1]{SGA5}, we have an exact sequence of abelian groups \[K_0\bigl(\Perf(R)) \xrightarrow{\mathsf{can}} K_0\bigl(\cd^b(\mod-A)\bigr) \rightarrow K_{0}(\Delta_{R}(A)) \rightarrow 0.\]
Moreover, the image of the canonical homomorphism $\mathsf{can}$
is the free abelian group generated by the class of the projective $A$-module $eA$, where $e$ is the idempotent endomorphism corresponding to the identity of $R$. Hence,
$K_0\bigl(\Delta_{R}(A)\bigr) \cong \coker(\mathsf{can}) \cong  \ZZ^{r}$.
\end{proof}

\subsection{DG-Auslander algebras for ADE--singularities} \label{ss:DGAuslander} 
\noindent The stable Auslander--Reiten quivers for the curve and surface singularities of Dynkin type ADE are known, see \cite{DieterichWiedemann86} and  \cite{Auslander86} respectively. 
Hence, the stable Auslander--Reiten quiver for any ADE--singularity $R$ is known by Kn\"orrer's periodicity (Theorem \ref{t:knoerrer}). The equivalence (\ref{E:rel-sing-cat-as-per}) in the proof of Theorem \ref{t:Classical-versus-generalized-singularity-categories} describes the triangulated category $\Delta_{R}(\Aus(R))$ as the perfect category for the dg-Auslander algebra associated to $\ul{\MCM}(R)$. We list the graded quivers\footnote{I would like to thank Hanno Becker for his help with the TikZ--package.} of these dg-algebras for the ADE--singularities in Paragraphs \ref{ss:OddTypeA} - \ref{ss:even}. For surfaces, this also follows from \cite{ThanhofferdeVolcseyMichelVandenBergh10, 
AIR}.
\begin{rem}
For ADE--singularities $R$,  it is well-known that the stable categories $\ul{\MCM}(R)$ are \emph{standard}, i.e.~the mesh relations form a set of minimal relations for the Auslander algebra $\Aus(\ul{\MCM}(R))$ of $\ul{\MCM}(R)$ (c.f.~\cite{Amiot07a, Riedtmann80}, respectively \cite{Iyama07a}). Hence the graded quivers completely determine the dg Auslander algebras in this case.
\end{rem}

The conventions are as follows. Solid arrows $\longrightarrow$ are in degree $0$, whereas broken arrows
$
\begin{xy}
\SelectTips{cm}{}
\xymatrix{ \ar@{-->}[r]&}
\end{xy}
$
are in degree $-1$ and correspond to the action of the Auslander--Reiten translation. The differential $d$ is uniquely determined by sending each broken arrow $\rho$ to the mesh relation starting in $s(\rho)$. If there are no irreducible maps (i.e.~ solid arrows) starting in the vertex $s(\rho)$, then we set $d(\rho)=0$ (see e.g. the case of type $(A_{1})$ in odd dimension in Paragraph \ref{ss:OddTypeA}). Let us illustrate this by means of two examples: in type $(A_{2m})$ in odd Krull dimension (see Paragraph \ref{ss:OddTypeA}) we have
\begin{align}Êd(\rho_{2})=\alpha_{1}\alpha_{1}^*+\alpha_{2}^*\alpha_{2}, \end{align}
whereas in odd dimensional type $(E_{8})$  (see Paragraph \ref{ss:OddTypeE8}) 
\begin{align} d(\rho_{10})=\alpha_{8}\alpha_{8}^*+ \alpha_{16}\alpha_{16}^*+\alpha_{9}^*\alpha_{10}. \end{align}

\subsubsection{DG-Auslander algebras for Type $A$--singularities in odd dimension}\label{ss:OddTypeA}
 \begin{equation*}
\end{equation*}

\subsubsection{DG--Auslander algebras of odd dim. $\left(D_{2m+1}\right)$-singularities, ${m\geq2}$}

 \begin{equation*}%
\end{equation*}

\subsubsection{DG--Auslander algebras of odd dimensional $(D_{2m})$-singularities, $m \geq 2$}

 \begin{equation*}%
\end{equation*}

\subsubsection{The DG--Auslander algebra of odd dimensional $\left(E_{6}\right)$-singularities.}
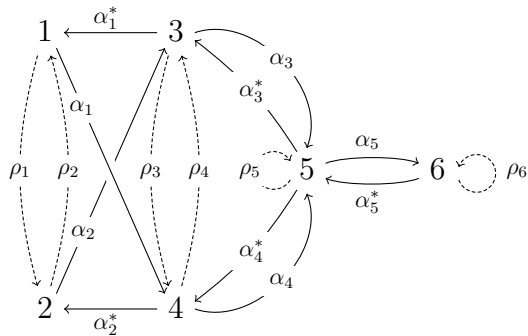
\begin{equation*}%
\end{equation*}

\subsubsection{The DG--Auslander algebra of odd dimensional $\left(E_{7}\right)$-singularities.}

 \begin{equation*}%
\end{equation*}

\subsubsection{The DG--Auslander algebra of odd dimensional $\left(E_{8}\right)$-singularities.}
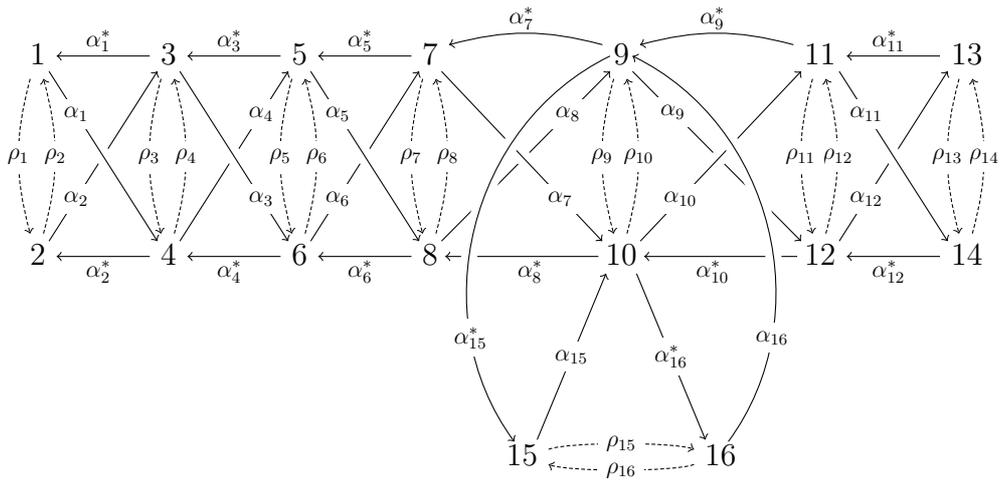
\label{ss:OddTypeE8}

\begin{equation*}%
\end{equation*}


%

%

\subsubsection{DG-Auslander algebras of even dimensional ADE--singularities }
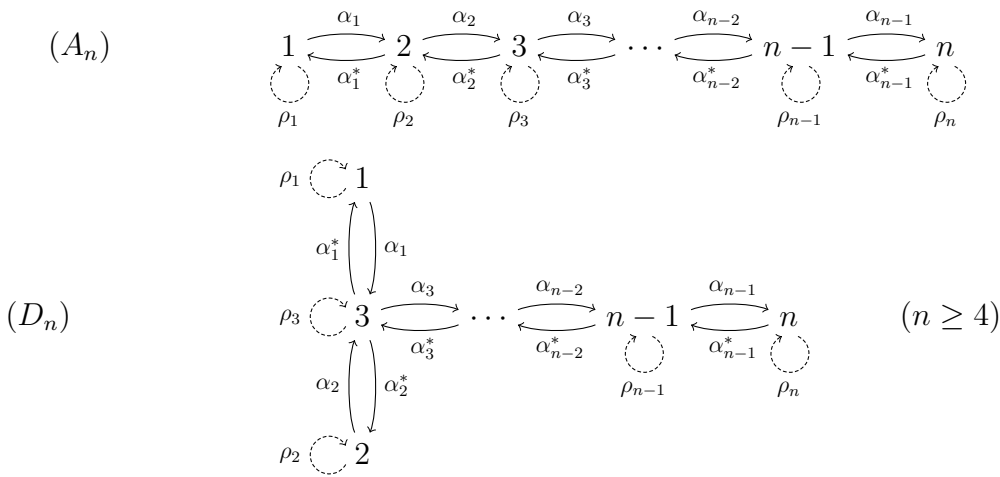
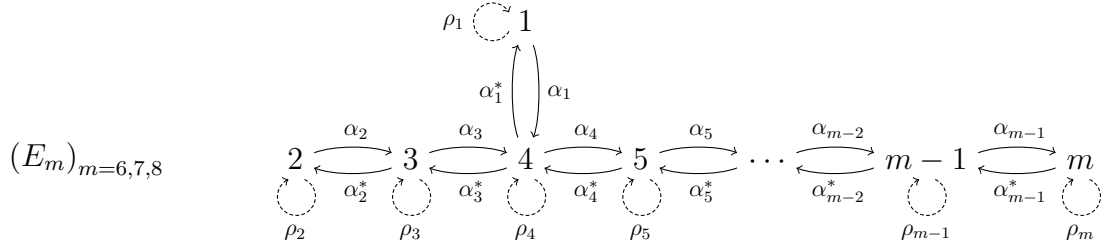
\label{ss:even}

\begin{equation*}%
\end{equation*}

\subsection{Relationship to Bridgeland's moduli space of stability conditions}\label{ss:Bridgeland} Let $X=\mathsf{Spec}(R_{Q})$ be a Kleinian singularity with \emph{minimal} resolution $f\colon Y \ra X$ and exceptional divisor $E=f^{-1}(0)$. Then $E$ is a tree of rational $(-2)$--curves, whose dual graph $Q$ is of ADE--type. Let us consider the following triangulated category
\begin{align}
\cd = \ker\left(\mathbb{R}f_{*}\colon \cd^b(\Coh(Y)) \longrightarrow \cd^b(\Coh(X))\right).
\end{align}
 Bridgeland determined a connected component $\mathsf{Stab}^\dagger(\cd)$ of the stability manifold of $\cd$ 
 \cite{Bridgeland09}. More precisely, he proves that $\mathsf{Stab}^\dagger(\cd)$ is a covering space of $\mathfrak{h}^{\rm reg}/W$, where $\mathfrak{h}^{ \rm reg}\subseteq \mathfrak{h}$ is the complement of the root hyperplanes in a fixed Cartan subalgebra $\mathfrak{h}$ of the complex semi--simple Lie algebra $\mathfrak{g}$ of type $Q$ and $W$ is the associated Weyl group. It turns out, that $\mathsf{Stab}^\dagger(\cd)$ is even a \emph{universal} covering of  $\mathfrak{h}^{\rm reg}/W$. This follows \cite{Bridgeland09} from a faithfulness result for the braid group actions generated by spherical twists  (see \cite{SeidelThomas} for type $A$ and \cite{BravThomas11} for general Dynkin types).
 
The category $\cd$ admits a different description. Namely, as category of dg modules with finite dimensional total cohomology $\cd_{fd}(B)$, where $B=B_{Q}$ is the dg-Auslander algebra $\Lambda_{dg}(\ul{\MCM}(R))$ of $R=\widehat{R}_{Q}$. Let $A=\Aus(\MCM(R))$ be the Auslander algebra of $\MCM(R)$ and denote by $e$ the identity endomorphism of $R$ considered as an idempotent in $A$. Then the derived McKay--Correspondence \cite{KapranovVasserot00, BKR} induces a commutative diagram of triangulated categories and functors, c.f.~\cite[Section 1.1]{Bridgeland09}.
\begin{align}
\begin{array}{c}
\begin{xy}
\SelectTips{cm}{}
\xymatrix@C=13pt{
\cd \ar@{=}[r] \ar[d]^{\cong} &\ker\big(\mathbb{R}f_{*}\colon \cd^b(\mathsf{Coh}(Y)) \ra \cd^b(\Coh(X))\big) \ar@{^{(}->}[r] \ar[d]_{\cong} & \cd^b_{E}(\Coh(Y)) \ar[d]^{\cong}\\
\cd_{fd}(B) \ar[r]^(0.4)\cong & \cd^b_{A/AeA}(\mod-A) \ar@{^{(}->}[r] & \cd^b_{fd}(\mod-A).
}
\end{xy}
\end{array}
\end{align}
For the equivalence $\cd_{fd}(B) \cong \cd^b_{A/AeA}(\mod-A)$, we refer to Proposition \ref{P:Intrinsic} and (\ref{E:Dfd}). Moreover, this category is triangle equivalent to the kernel of the quotient functor $\Delta_{R}(A) \ra \cd_{sg}(R)$, see Proposition \ref{C:From-gen-to-class}.

\begin{rem} 
It would be interesting to study Bridgeland's space of stability conditions for the categories $\cd_{fd}(B)$ in the case of  odd dimensional ADE--singularities $R$ as well! Note that the canonical $t$-structure on $\cd(B)$ restricts to a $t$-structure on $\cd_{fd}(B)$ by Proposition \ref{p:standard-t-str}. Its heart is the finite length category of finite dimensional modules over the stable Auslander algebra of $\MCM(R)$.
\end{rem}

\newpage

\section{Special Cohen--Macaulay modules over rational surface singularities}\label{s:Specials}

We give a conceptual description of Iyama \& Wemyss \cite{IWnewtria} stable category of special Cohen--Macaulay modules over rational surface singularities in terms of singularity categories of ADE-surface singularities (see Subsection \ref{ss:MainSpecial}). This is based on a joint work with Osamu Iyama, Michael Wemyss and Dong Yang \cite{IKWY12}. Moreover, Wemyss' Theorem \ref{main Db} and Subsection \ref{ss:Lift} are contained in that article. The results in Subsections \ref{ss:rationalsurface} -- \ref{ss:SCMFrobenius} are well-known. We include them in order to put our results into perspective.

Let us describe the content of this section in greater detail. In Subsection \ref{ss:rationalsurface}, we collect well-known notions and results from the theory of rational surface singularities. Special Cohen--Macaulay modules (SCMs) over rational surface singularities were introduced by Wunram \cite{Wunram88} in order to generalize the McKay Correspondence to all quotient surface singularities $\mathbb{C}\llbracket x, y \rrbracket^G$, where $G \subseteq \GL(2, \mathbb{C})$ is a finite subgroup. Therefore, we review the classical McKay Correspondence (i.e. the case $G \subseteq \SL(2, \mathbb{C})$) in Subsection \ref{ss:McKay}, before discussing SCMs in Subsection \ref{ss:SCMs}. We begin Subsection \ref{ss:DerEq} with a tilting theorem of Van den Bergh \cite{VandenBergh04}, which is inspired by Bridgeland's proof of the Bondal--Orlov Conjecture in dimension three \cite{Bridgeland02}. In particular, his notion of perverse sheaves (which has its origin in \cite{BeilinsonBernsteinDeligne82}) is important in Van den Bergh's article. Wemyss \cite{IKWY12} applied Van den Bergh's result to describe the bounded derived category of coherent sheaves on (partial) resolutions of  rational surface singularities in terms of endomorphism algebras of SCMs (see Theorem \ref{main Db}). This is one of the key ingredients in the proof of our main Theorem \ref{C:StandardStableSCM}.
Subsection \ref{ss:SCMFrobenius} follows Iyama \& Wemyss \cite{IWnewtria} to prove that the exact category $\SCM(R) \subseteq \mod-R$ of special Cohen--Macaulay modules over a rational surface singularity $R$ is a Frobenius category. Moreover, the projective-injective objects in this category are described in terms of the geometry of the exceptional divisor of the minimal resolution of $\Spec(R)$. Subsection \ref{ss:MainSpecial} contains the main result. Its proof combines the abstract Frobenius category results (see Theorem \ref{t:main-thm} or \ref{t:alternative-main}) with Wemyss' geometric Theorem \ref{main Db}. Moreover, Auslander \& Solberg's modification result Proposition \ref{new Frobenius structure} is applied to obtain new Frobenius structures on $\SCM(R)$. The corresponding stable categories are described in an analogous way, see Corollary \ref{C:ModifiedStableSCM}. In Subsection \ref{ss:ExampleSCM}, we illustrate our results by several examples of quotient singularities. Finally, Subsection \ref{ss:Lift} contains lifts of the equivalences between singularity categories from Subsection \ref{ss:MainSpecial} to equivalences between relative singularity categories. The main ingredient is an application of Theorem \ref{t:main-thm-2}.

\subsection{Rational surface singularities} \label{ss:rationalsurface}
Let throughout $k$ be an algebraically closed field of characteristic zero and let $(R, \mathfrak{m})$ be a commutative local complete normal domain over $k$ of Krull dimension two. In particular, $R$ is a Cohen--Macaulay ring by a result of Serre \cite{SerreAlgloc}, see also \cite[Theorem 2.2.22]{BrunsHerzog}. Moreover, a finitely generated $R$-module $M$ is \emph{maximal Cohen--Macaulay} (MCM) if and only if it is \emph{reflexive}, i.e. the canonical homomorphism $M \ra M^{**}$ is an isomorphism, where $(-)^*:=\Hom_{R}(-, R)$.
\begin{defn}
A \emph{resolution of singularities of $\Spec(R)$} is a proper, birational morphism $\pi\colon Y\ra \Spec(R)$, such that $X$ is smooth. The resolution $\pi$ is called \emph{minimal}, if every resolution of singularities $\pi'\colon Y'\ra \Spec(R)$ factors uniquely through $\pi$. In particular, if a minimal resolution exists, then it is unique up to a unique isomorphism.
\end{defn}

\begin{defn}
Let $\pi\colon Y\ra \Spec(R)$ be a resolution of singularities. Then $E:=\pi^{-1}(\{\mathfrak{m}\}) \subseteq Y$ is called the \emph{exceptional divisor}.
\end{defn}

The following proposition can be found in a work of Brieskorn \cite[Lemma 1.6]{Brieskorn66}, see also Lipman \cite[Corollary 27.3]{Lipman}.

\begin{prop}
Let $R$ be a normal surface singularity as above. Then there exists a minimal resolution of singularities $\pi\colon Y\ra \Spec(R)$. 

Let $E =\bigcup_{i=1}^t E_{i}$ be the exceptional divisor, where the $E_{i}$ are irreducible.
The minimal resolution can be characterized by the absence of exceptional curves $E_{i}$, which are smooth, rational and have selfintersection number $-1$.
\end{prop}

\begin{defn}
Let $\pi \colon Y \ra \Spec(R)$ be a resolution of singularities of $\Spec(R)$. Then $R$ is called \emph{rational singularity} if $ \mathsf{R}\pi_{*}\co_{Y} \cong \co_{\Spec(R)}$\footnote{This definition does \emph{not} depend on the choice of a resolution of singularities, see e.g. \cite{Brieskorn}.}.
\end{defn}

 Important and well studied examples of rational surface singularities are quotient singularities: 

\begin{ex}
Let $G \subseteq \GL(2, \C)$ be a finite subgroup. Then $G$ acts on the ring of formal power series in two variables $S=\C\llbracket x, y\rrbracket$ and the ring of invariants $R=S^G$ is a rational singularity, see \cite{Brieskorn}. 
\end{ex}

\begin{rem}
Using results of Prill \cite{Prill}, Brieskorn gave a complete classification of the quotient surface singularities \cite{Brieskorn}.
\end{rem}

\begin{rem}\label{R:Watanabe}
Watanabe \cite{Watanabe} has shown that quotient singularities are Gorenstein if and only if $G$ is a subgroup of the \emph{special} linear group $\SL(2, \mathbb{C})$. 
\end{rem}

The following notion is essential for the study of rational surface singularities:

\begin{defn}
Let $\pi\colon Y \ra \Spec(R)$ be a resolution of singularities with exceptional divisor $E =\bigcup_{i=1}^t E_{i}$, where the $E_{i}$ are irreducible. Define the \emph{dual intersection graph} (or \emph{dual graph} for short) as follows: the vertices are in bijection with the components $E_{i}$ and for $i\ne j$ the number of edges $a_{ij}=a_{ji}$ between $E_{i}$ and $E_{j}$ is given by the intersection number $E_{i}.E_{j}=E_{j}.E_{i}$. Moreover, the vertices $E_{i}$ are often labelled with the selfintersection number $E_{i}^2$. 
\end{defn}

The following theorem may be deduced from \cite{Artin66}, see also \cite{Brieskorn}.

\begin{thm} \label{T:ArtinTree}
Let $(R, \mathfrak{m})$ be a rational surface singularity with resolution of singularities $\pi\colon Y \ra \Spec(R)$. Then all the irreducible exceptional curves $E_{i}$ are smooth and rational and the dual intersection graph is a \emph{tree}, i.e.~a graph without doubled edges or cycles.
\end{thm}

We need to introduce another important definition due to Artin \cite{Artin66}.

\begin{defn}
In the notation of Theorem \ref{T:ArtinTree} above, a (formal) linear combination of exceptional curves $Z=\sum_{i=1}^t r_{i} E_{i}$ with non-negative integers $r_{i}$ and $r_{j}>0$ for some $j$, is called \emph{positive cycle}.
Artin has shown that there exists a minimal positive cycle $Z_{f}$ satisfying
\begin{align} \label{E:DefFundCycle}
Z_{f}. E_{i} \leq 0  \text{ for all $i$}. 
\end{align}
$Z_{f}$ is called the \emph{fundamental cycle}.
\end{defn}

There is the following algorithm to determine the fundamental cycle of a rational surface singularity, see Laufer \cite{Laufer}:

 Set $Z_{0}=E_{1}+ \ldots + E_{t}$. If $Z_{0}$ satisfies \eqref{E:DefFundCycle} set $Z_{f}=Z_{0}$. Otherwise take any $i$ such that \eqref{E:DefFundCycle} is violated and define $Z_{1}=Z_{0}+E_{i}$. Repeat this procedure until some $Z_{n}$ satisfying \eqref{E:DefFundCycle} for all $i$ is reached. Then $Z_{f}=Z_{n}$ is the fundamental cycle.

\subsection{The McKay--Correspondence}\label{ss:McKay} In 1979 John McKay \cite{McKay} observed that for finite subgroups $G$ of $\SL(2, \C)$ the structure of the exceptional divisor (of the minimal resolution) of the quotient singularity $\C^2/G$ is completely determined by the representation theory of the group $G$. 

However, the correspondence, which he observed, breaks down for finite subgroups of $\GL(2, \C)$, which are not contained in $\SL(2, \C)$. Of course, the representation theory of $G$ does not depend on the choice of an embedding into a larger group. However, the action of $G$ on $\C\llbracket x, y \rrbracket$ being inherited by the action of $\GL(2, \C)$ depends on the embedding. Thus the geometry of the corresponding quotient does.

Wunram \cite{Wunram88} recovered parts of McKay's correspondence by exhibiting a \emph{canonical} way to get rid of superfluous irreducible representations of $G$. We recall this story in order to motivate the notion of a \emph{special} maximal Cohen--Macaulay modules over rational surface singularities. In the case of quotient singularities these are precisely those maximal Cohen--Macaulay modules corresponding to $G$-representations, which are not discarded by Wunram.

We start by describing the classical McKay Correspondence. 

Let $G \subseteq \SL(2, \C)$ and $R=\C\llbracket x,y \rrbracket^G$ be the corresponding invariant ring.  Let $E=\bigcup_{i=1}^t E_{i}$ be the indecomposable components of the exceptional divisor of the minimal resolution of singularities $\pi\colon X \ra \Spec(R)$. The following observation is due to McKay \cite{McKay}.

\begin{obs}\label{O:McKay}  There exists a bijection of sets:
\begin{align}\label{E:BijouMcKay}
\{\text{non-trivial irreducible $G$-representations}\}/\!\sim \quad \stackrel{\sim}\longleftrightarrow \quad \{ E_{i} \}_{i=1}^t
\end{align}
\end{obs}

\begin{rem}
In fact, this bijection is only the shadow of much more sophisticated isomorphisms. On the right hand side there is the dual intersection graph. One may wonder whether there is also a natural graph associated with the left hand side. It turns out that this is indeed the case! For every finite subgroup $G \subseteq \SL(2, \C)$ McKay introduced a graph with vertices corresponding to the irreducible $G$-representations as follows: let $V=\C^2$ be the standard representation obtained by restricting the action of $\SL(2, \C)$ to $G$. Now, two irreducible representations $\rho_{i}$ and $\rho_{j}$ are connected by an edge in the McKay graph if $\rho_{i}$ occurs as a direct summand in the representation $\rho_{j} \otimes_{\C} V$ (since $V$ is selfdual this definition is in fact symmetric in $i$ and $j$). McKay showed that the bijection \eqref{E:BijouMcKay} induces an isomorphism of graphs between the dual intersection graph and the subgraph of non-trivial representations of the McKay graph.

Moreover, the McKay graph has a natural valuation on the vertices: namely, the dimension of the irreducible $G$-representation. McKay showed that under the isomorphism of graphs explained above, this corresponds to the multiplicity of the $E_{i}$ in Artin's \emph{fundamental cycle}. In other words, both graphs carry a natural valuation and these are compatible with the isomorphism!

Subsequently, Gonzalez-Sprinberg \& Verdier \cite{GonzalezVerdier} lifted this isomorphism to a ring isomorphism between the representation ring of $G$ and the Grothendieck ring of the minimal resolution $X$. Finally, Kapranov \& Vasserot \cite{KapranovVasserot00} deduced this isomorphism from a triangle equivalence between the derived category $\cd^b_{G}(\C^2)$ of $G$-equivariant coherent sheaves on $\C^2$ and the corresponding derived category of the minimal resolution $\cd^b(X)$  (see also Bridgeland, King \& Reid \cite{BKR} for generalisations to dimension three). 
\end{rem}

\begin{ex}
Let $\epsilon^3=1$ be a primitive third root of unity. Let $G$ be the finite group generated by
\begin{align}
g=\begin{pmatrix} \epsilon & 0 \\ 0 & \epsilon^2 \end{pmatrix} \subseteq \SL(2, \C). 
\end{align}
$G$ is isomorphic to the cyclic group $\ZZ/3 \ZZ$. Since $G$ is abelian the irreducible representations are known to be one dimensional. Moreover, the action of $G$ on such a representation $\C_{\lambda}$ is completely determined by the (scalar=$\lambda \in \C$) action of the generator $g \in G$. Since $g^3=1$, there are exactly three irreducible representations: the trivial representation $\C_{\epsilon^0}$ and two non-trivial representations $\C_{\epsilon^1}$ and $\C_{\epsilon^2}$. 

The action of $G$ on $\C\llbracket x, y \rrbracket$ is determined by the action of the generator $g \in G$ on the variables $x$ and  $y$: $x \mapsto \epsilon x$ and $y \mapsto \epsilon^2 x$. The following monomials are invariant with respect to this action: $A=x^3$, $B=xy$ and $C=y^3$. Moreover, every invariant formal power series may be written as a formal power series in these monomials. Hence, $R=\C\llbracket x, y \rrbracket^G \cong \C\llbracket x^3, xy, y^3 \rrbracket$. The monomials $A$, $B$ and $C$ satisfy the relation $B^3 +AC=0$. Since $R$ is a normal surface singularity these are all relations and we obtain an isomorphism $R \cong \C\llbracket A, B, C \rrbracket/(B^3 + AC)$. This is known as the $A_{2}$-singularity and the exceptional divisor of its minimal resolution consists of two rational ($-2$)-curves:
\begin{align}
\begin{array}{c}
\begin{tikzpicture} 
\draw (0,0,0) to [bend left=25] node[above] {$\scriptstyle E_1$}  node[below] {$\scriptstyle -2$} (4,0,0);
\draw (3.5,0,0) to [bend left=25] node[below] {$\scriptstyle -2$} node[above] {$\scriptstyle E_2$} (8,0,0);
\end{tikzpicture}
\end{array} 
\end{align}
Hence we see that there is indeed a bijection between the irreducible components of the exceptional divisor and the non-trivial irreducible representations.

\end{ex}

\begin{ex}\label{Ex:McKayFails}
Let $\epsilon$ be a primitive root of unity as above and let $G'$ be the group generated by 
\begin{align}
g'=\begin{pmatrix} \epsilon & 0 \\ 0 & \epsilon \end{pmatrix} \subseteq \GL(2, \C).
\end{align} 
Again there is a group isomorphism $G' \cong \ZZ/3 \ZZ$ (so the representation theory does not change) but this time $G'$ is not contained in $\SL(2, \C)$. The exceptional divisor of the minimal resolution consists of a single rational ($-3$)-curve 
\begin{align}
\begin{array}{c}
\begin{tikzpicture} 
\draw (0,0,0) to [bend left=25] node[above] {$\scriptstyle E_1$}  node[below] {$\scriptstyle -3$} (4,0,0);
\end{tikzpicture}
\end{array} 
\end{align}
In particular, Observation \ref{O:McKay} does not hold in this case.

In general, the dual intersection graph of a cyclic quotient singularity is a so called Jung--Hirzebruch string obtained from a continued fraction expansion of $\frac{r}{a}$, where the corresponding cyclic group $G=\frac{1}{r}(1, a)$ is generated by
\begin{align}
g=\begin{pmatrix} \epsilon & 0 \\ 0 & \epsilon^a \end{pmatrix} \subseteq \GL(2, \C),
\end{align}
where $\epsilon$ is a primitive $r$-th root of unity and $\gcd(a, r)=1$.
\end{ex}

\subsubsection{McKay's Correspondence as a composition of bijections}
Inspired by a construction of Gonzalez-Sprinberg \& Verdier \cite{GonzalezVerdier}, Esnault \cite{Esnault} proved the following theorem for arbitrary rational surface singularities.
\begin{thm}\label{T:Esnault}
Let $(R, \mathfrak{m})$ be a rational surface singularity with resolution $\pi\colon X \ra \Spec(R)$. There are quasi-inverse equivalences
\begin{align}\label{E:Esnault}
\begin{xy}
\SelectTips{cm}{}
\xymatrix{
\MCM(R) \ar@/^/[rrr]^{\displaystyle \pi^\#} &&& \VB^f(X) \ar@/^/[lll]^{\displaystyle \pi_{*}}
}
\end{xy},
\end{align} 
where $\VB^f(X):=\left\{\left. \cf \, \right| \cf \text{ is generated by global sections and } \Ext_{X}^1(\cf, \omega_{X})=0 \right\}\subseteq \VB(X)$ denotes the subcategory of \emph{full vector bundles} on $X$ and the functor $\pi^\#$ is defined as $\pi^\#:= \pi^*(-)/\displaystyle \tors(\pi^*(-))$.
\end{thm}

We need two further correspondences

\begin{thm}[Auslander \cite{Auslander86}]
Let $G \subseteq \GL(2, \C)$ be a finite group and $R=\C\llbracket x, y \rrbracket^G$ be the corresponding invariant ring. Then there is a bijection of sets
\begin{align}
\begin{array}{c}
\begin{xy}
\SelectTips{cm}{}
\xymatrix@R=3pt{
\{\text{irreducible $G$-representations}\}/\!\sim \ar[rrr] &&& \ind \MCM(R) \\
\rho_{i} \ar@{|->}[rrr] &&& (\C\llbracket x, y \rrbracket \otimes_{\C} \rho_{i})^G.
}
\end{xy}
\end{array}
\end{align} 
\end{thm}

\begin{thm}[Gonzalez-Sprinberg \& Verdier] Let $G \subseteq \SL(2, \C)$ be a finite subgroup, $R=\C\llbracket x, y \rrbracket^G$ and $\pi \colon X \ra \Spec(R)$ be the minimal resolution of singularities. Then there is a well-defined bijection of sets
\begin{align}
\begin{array}{c}
\begin{xy}
\SelectTips{cm}{}
\xymatrix@R=3pt{
\ind \VB^f(X) \setminusÊ\{\co_{X}\} \ar[rrr] &&& \{E_{i}\}_{i} \\
\cf \ar@{|->}[rrr] &&& E_{i} 
 }
\end{xy}
\end{array}
\end{align} 
such that $c_{1}(\cf).E_{i}=1$, where $c_{1}(\cv)$ denotes the first Chern class of a vector bundle $\cv$.
\end{thm}

For a finite subgroup $G \subseteq \SL(2, \C)$ and $R=\C\llbracket x, y \rrbracket^G$ we assemble all these correspondences into a diagram:
\begin{align*}
\begin{array}{c}
\begin{xy}
\SelectTips{cm}{}
\xymatrix@R=50pt{
\{\text{non-trivial irred. $G$-reps. }\}/\!\sim \ar@{<->}[d]_{ \text{ÊAuslander }}^\simÊ  \quad \ar@{<..>}[rrr]^(0.72){\text{McKay}}_(0.72){\sim} &&& \quad \{E_{i}\}_{i} \\
\ind \MCM(R)\setminus \{R\} \ar@{<->}[rrr]_{\text{Esnault}}^\sim &&& \ind \VB^f(X) \setminusÊ\{\co_{X}\} \ar@{<->}[u]_{\begin{smallmatrix} \text{ÊGonzalez-Sprinberg  \& } \\ \text{ÊVerdier } \end{smallmatrix}}^{\sim}
}
\end{xy}
\end{array}
\end{align*} 

The correspondences of Auslander respectively Esnault hold for finite subgroups of $\GL(2, \C)$. What breaks down in this generality is the correspondence of Gonzalez-Sprinberg \& Verdier. In Example \ref{Ex:McKayFails}, we have seen that there are `too many irreducible $G$-representations' if $G \nsubseteq \SL(2, \C)$. 

Now, Wunram \cite{Wunram88} gave a \emph{canonical} way to choose a subset of the irreducible $G$-representations (equivalently maximal Cohen--Macaulay $R$-modules) in bijection with the set $\{E_{i}\}_{i}$. His correspondence holds for arbitrary rational surface singularities, as soon as we replace $G$-representations by MCMs.

\begin{defn}
A full vector bundle $\cf$ is called \emph{special}, if $\Ext^1_{X}(\cf, \co_{X})=0$. Denote the corresponding full subcategory by $\mathsf{SVB}^f(X)$.

Accordingly, a maximal Cohen--Macaulay module $M$ is called \emph{special Cohen--Macaulay module}, if $\pi^\#(M) \in \mathsf{SVB}^f(X)$, i.e.~if it corresponds to a special vector bundle under Esnault's correspondence \eqref{E:Esnault}. The corresponding full subcategory of $\MCM(R)$ is denoted by $\SCM(R)$.
\end{defn}

Wunram \cite{Wunram88} showed that the corresponding indecomposable objects are indeed in bijection with the irreducible exceptional curves.

\begin{thm}\label{T:Wunram}
Let $(R, \mathfrak{m})$ be a rational surface singularity with minimal resolution $\pi\colon X \ra \Spec(R)$. Then there is a bijection of sets
\begin{align}
\begin{array}{cc}
\ind \mathsf{SVB}^f(X) \setminusÊ\{\co_{X}\}  \overset{\sim}\longrightarrow \{E_{i}\} \\
\cf_{i} \longmapsto E_{i}, 
\end{array}
\end{align}
such that $c_{1}(\cf_{i}).E_{j}=\delta_{ij}$ and $\rk(\cf_{i})=Z_{f}.c_{1}(\cf_{i})$.
\end{thm}

This motivates the study of the category of special Cohen--Macaulay modules in the next subsection.

\subsection{Special Cohen--Macaulay modules}\label{ss:SCMs}
Let throughout  $(R, \mathfrak{m})$ be a rational surface singularity as above. The aim of this subsection is to show that the category of special Cohen--Macaulay $R$-modules $\SCM(R)$ is a Frobenius category. We mainly follow Iyama \& Wemyss' article \cite{IWnewtria}. Using the modification result  of Auslander \& Solberg (Proposition \ref{new Frobenius structure}), we get new Frobenius structures from Iyama \& Wemyss' structure, see Corollary \ref{C:NewFrobSCM}.
\begin{rem}
Since rational surface singularities are \emph{not} Gorenstein in general (see for example Remark \ref{R:Watanabe}) the category of all maximal Cohen--Macaulay modules is \emph{not} Frobenius in general. In particular, Buchweitz equivalence $\cd_{sg}(R) \cong \ul{\MCM}(R)$ is not always available. The stable category $\ul{\ul{\SCM}}(R)$ of Iyama \& Wemyss' Frobenius category may be viewed as a substitute for $\ul{\MCM}(R)$ in the non-Gorenstein case.  
\end{rem}

Let $(-)^*=\Hom_{R}(-, R)$ and $\omega_{R}$ be the canonical $R$-module.
Wunram \cite{Wunram88} gave the following alternative characterisation of special Cohen--Macaulay modules. 
 
\begin{prop}\label{P:AltSpecial} Let $\pi\colon X \ra \Spec(R)$ be the minimal resolution.

A non-free indecomposable maximal Cohen--Macaulay $R$-module $M$ is \emph{special} if and only if there is an isomorphism of sheaves on $X$
\begin{align}\label{E:SpecialWunram}
\pi^{\#}\bigl((M \otimes_{R} \omega_{R})^{**}\bigr) \cong \pi^{\#}(M) \otimes_{X} \omega_{X},
\end{align}
where $\pi^\#:=\pi^*(-)/\tors(\pi^*(-))$ and $\omega_{X}$ is the canonical bundle on $X$.
\end{prop} 

\begin{cor}
If $R$ is also Gorenstein, then every maximal Cohen--Macaulay module is special.
\end{cor}
\begin{proof}
For Gorenstein rings the canonical module $\omega_{R}$ is isomorphic to the free module $R$. Moreover, since Cohen--Macaulay modules are reflexive the left hand side of \eqref{E:SpecialWunram} is isomorphic to $\pi^{\#}(M)$. The same holds true for the right hand side because for ADE-singularities minimal resolutions are crepant, i.e.~$\omega_{X}\cong \co_{X}$. Note, that the rational Gorenstein singularities are known to be the ADE-surface singularities, see e.g. \cite{Durfee}.
\end{proof}

\begin{rem}
Let $R$ be Gorenstein. Using Esnault's equivalence \eqref{E:Esnault} and the corollary above, we see that every full vector bundle is special. Hence Wunram's Theorem \ref{T:Wunram} specializes to the results of Gonzalez-Sprinberg \& Verdier \cite{GonzalezVerdier} and Artin \& Verdier \cite{ArtinVerdier85}, respectively.
\end{rem}

Riemenschneider found a characterization of special Cohen--Macaulay modules, which does not refer to the minimal resolution \cite{Riemenschneider}.

\begin{prop}
$M \in \MCM(R)$ is special if and only if 
\begin{align}
\frac{ \displaystyle M \otimes_{R} \omega_{R}}{\displaystyle \tors(M \otimes_{R} \omega_{R})} \in \MCM(R).
\end{align}
\end{prop}

Using this, Iyama \& Wemyss  gave further characterizations of specials \cite{IWClassOfSpecials}.

\begin{prop} \label{P:IWSpecials}Let $M \in \MCM(R)$ then the following are equivalent
\begin{itemize}
\item[(a)] $M$ is special,
\item[(b)] $\Ext^1_{R}(M, R)=0$,
\item[(c)] $\Omega(M) \cong M^*$  \text{( up to free summands  )}.
\end{itemize}
\end{prop}

\begin{rem}
Conditions (b) and (c) can be applied to classify indecomposable special Cohen--Macaulay modules over quotient surface singularities $R$. Indeed, in this situation the AR quiver of the category $\MCM(R)$ is known by work of Auslander \cite{Auslander86}. For example, checking (b) on (the covering of) this \emph{finite} quiver reduces to a certain counting procedure, which is well-known in the theory of finite dimensional algebras. Using this, a complete classification has been achieved by Iyama \& Wemyss \cite{IWClassOfSpecials}.  
\end{rem}

Let $\Omega\mathsf{CM}(R)=\{ \Omega(M) | M \in \MCM(R)\} \subseteq \MCM(R)$ be the full  subcategory of first syzygies of maximal Cohen--Macaulay modules. The equivalence $[(a) \Leftrightarrow (c)]$ in Proposition \ref{P:IWSpecials} has the following consequence, see \cite{IWClassOfSpecials}.

\begin{cor}
The duality $(-)^*=\Hom_{R}(-, R)\colon \MCM(R) \ra \MCM(R)$ restricts to a duality 
\begin{align}
\Hom_{R}(-, R)\colon \SCM(R) \longrightarrow \Omega\mathsf{CM}(R).
\end{align}
In particular, $\Omega\mathsf{CM}(R)$ has only finitely many indecomposable objects, by Wunram's Theorem \ref{T:Wunram}.
\end{cor}

The next corollary (see \cite{IWClassOfSpecials}) is a key ingredient in Iyama \& Wemyss' proof that $\SCM(R)$ is a Frobenius category.

\begin{cor}\label{C:ExtSymmetry}
Let $X, Y \in \SCM(R)$. Then there is an isomorphism $\Ext^1_{R}(X, Y) \cong \Ext_{R}^1(Y, X)$. In particular, an object $P \in \SCM(R)$ is injective in this exact category if and only if it is  projective.
\end{cor}
\begin{proof}
Let $P^\bullet(X)$ be a projective resolution of $X$. The map sending a morphism in $\ul{\Hom}_{R}(\Omega(X), Y)$ to a morphism of complexes $P^\bullet(X) \ra Y[1] \in \Ext^1_{R}(X, Y)$ is well-defined: indeed, $\Ext^1_{R}(X, F)=0$ for every finitely generated free $R$-module $F$, by Proposition \ref{P:IWSpecials} (b). Using the long exact sequence obtained from $0 \ra \Omega(X) \ra P^0 \ra X \ra 0$, we see that every morphism $\Omega(X) \ra F$ factors over $P^0$, showing that the map is well-defined. One checks that this map defines an isomorphism $\ul{\Hom}_{R}(\Omega(X), Y)\cong \Ext^1_{R}(X, Y)$. The duality $(-)^*$ descends to the stable category $\ul{\MCM}(R)$. Now, by Proposition \ref{P:IWSpecials} (c), we get $\Omega(X) \cong X^*$ in the stable category. This yields a chain of isomorphisms
\begin{align*}
\Ext^1_{R}(X, Y) \cong \ul{\Hom}_{R}(X^*, Y) \cong \ul{\Hom}_{R}(Y^*, X^{**}) \cong \ul{\Hom}_{R}(Y^*, X) \cong
\Ext^1_{R}(Y, X).  
\end{align*}
\end{proof}

\subsection{A derived equivalence} \label{ss:DerEq}
\subsubsection{Perverse sheaves and tilting}
It is well-known that tilting is a special case of changing the $t$-structure in a given triangulated category, see e.g.~\cite{HappelReitenSmalo}. Beilinson, Bernstein \& Deligne \cite{BeilinsonBernsteinDeligne82} showed that $t$-structures may be glued along recollements, which led them to the notion of a \emph{perverse} $t$-structure. Objects in the corresponding heart are called \emph{perverse}. This inspired Bridgeland's notion \cite{Bridgeland02} of perverse sheaves in the following context: let $f \colon Y \ra X$ be a projective morphism of quasi-projective schemes over an affine scheme, such that the following conditions hold:
\begin{itemize}
 \item[(B1)] $\Rf\co_{Y} \cong \co_{X}$; 
 \item[(B2)] the fibers of $f$ have dimension at most $1$.
 \end{itemize}
 Denote by $\cd_{\mathsf{coh}}(X)$ and $\cd_{\mathsf{coh}}(Y)$ the unbounded derived categories of quasi-coherent sheaves with coherent cohomologies. It is well-known that there is an adjoint triple of triangle functors \cite{ResiduesDuality} 
\begin{align}\label{E:GeomAdjTriple}
\begin{xy}
\SelectTips{cm}{}
\xymatrix{\cd_{\mathsf{coh}}(Y)\ar[rrrr]|{ \Rf }&&&&\cd_{\mathsf{coh}}(X).\ar@/^25pt/[llll]^{f^!}\ar@/_25pt/[llll]_{ \Lf} }
\end{xy}
\end{align}
Combining the projection formula with the assumption (B1) shows $\Rf \Lf=1_{\cd_{\mathsf{coh}}(X)}$. Therefore Lemma \ref{L:Miyachi} implies that $\Rf$ is a quotient functor and using \eqref{E:GeomAdjTriple} there exist  Bousfield localisation and colocalisation functors for $\ker \Rf \subseteq  \cd_{\mathsf{coh}}(Y)$. In particular, the adjoint triple extends to a recollement (Proposition \ref{P:AbstractRecoll})
\begin{align}\label{E:GeomRecollement}
\begin{xy}
\SelectTips{cm}{}
\xymatrix{ \cc:=\ker \Rf  \ar[rr] && \cd_{\mathsf{coh}}(Y) \ar@/^25pt/[ll] \ar@/_25pt/[ll] \ar[rr]|{ \Rf }&&\cd_{\mathsf{coh}}(X).\ar@/^25pt/[ll]^{f^!}\ar@/_25pt/[ll]_{ \Lf} }
\end{xy}
\end{align}
Condition (B2) and a spectral sequence argument show that an object $\cf$ is contained in the subcategory $\cc \subseteq \cd_{\mathsf{coh}}(Y)$ if and only if all cohomologies $H^i(\cf)$ are in $\cc$, see \cite[Lemma 3.1.]{Bridgeland02}.
In particular, the standard $t$-structure $\cd_{Y}^{\leq 0}$Êon $\cd_{\mathsf{coh}}(Y)$ induces a $t$-structure $\cc^{\leq 0}:=\cd_{Y}^{\leq 0} \cap \cc$ on $\cc$. For any $p \in \mathbb{Z}$, there is a shifted $t$-structure $(\cc^{\leq 0})[p]$, which can be glued \cite{BeilinsonBernsteinDeligne82} with the standard $t$-structure $\cd_{X}^{\leq 0}$ on $\cd_{\mathsf{coh}}(X)$ to give a $t$-structure $^p\cd^{\leq 0}$ on 
 $\cd_{\mathsf{coh}}(Y)$, see \cite{Bridgeland02}:
 \begin{align}
 \begin{array}{cc}
 ^p\cd^{\leq 0}= \{ \cf \in \cd_{\mathsf{coh}}(Y) \mid \Rf(\cf) \in \cd_{X}^{\leq 0} \text{  and  } \Hom_{\cd_{\mathsf{coh}}(Y)}(\cf, \cc^{>p})=0 \}, \\ \\
 ^p\cd^{\geq 0}= \{ \cf \in \cd_{\mathsf{coh}}(Y) \mid \Rf(\cf) \in \cd_{X}^{\geq 0} \text{  and  } \Hom_{\cd_{\mathsf{coh}}(Y)}(\cc^{<p}, \cf)=0 \}.
 \end{array}
 \end{align}
 The category of perverse sheaves $^p\mathsf{Per}(Y/X)$ is defined as the heart $^p\cd^{\leq 0} \cap \,  ^p\cd^{\geq 0}$ of this $t$-structure. For $p=-1, 0$, Van den Bergh \cite{VandenBergh04} considers the restriction of these $t$-structures\footnote{In this situation, the hearts of the restricted $t$-stuctures are equal to the hearts $^p\mathsf{Per}(Y/X)$, see e.g. \cite[Lemma 3.2.]{Bridgeland02}.}  to $\cd^b(\Coh(Y))$ and shows that they arise from tilting theory. We collect some of his results in the following theorem.

 \begin{thm}\label{T:VdB}
 Let $Y$ be a quasi-projective scheme over an affine scheme and let $R=(R, \mathfrak{m})$ be a complete local $k$-algebra, where $k\cong R/\mathfrak{m}$ is algebraically closed. Let $f\colon Y \ra X:=\Spec(R)$ be a projective morphism such that $\Rf \co_{Y}=\co_{X}$ and the fibers have dimensions $\leq 1$. Then the following statements hold:
 \begin{itemize}
 \item[(a)] There exist tilting bundles $\cp$ and $\cp^*$ on $Y$.
 \item[(b)] $\cp$ and $\cp^*$ are projective generators of the abelian categories of perverse sheaves
                $^{-1}\mathsf{Per}(Y/X)$ and $^{0}\mathsf{Per}(Y/X)$, respectively.
\item[(c)] In order to give a more precise description of the tilting bundles, we introduce some notation: let $\{E_{i}\}_{i=1}^n$ be the irreducible components of the exceptional fibre $E$ of $f$. It is well-known that the assignment $\cl \mapsto \deg(\cl \big| E_{i})_{i=1}^n$ defines an isomorphism $\Pic(Y) \cong \ZZ^n$. In particular, there are line bundles $\cl_{i} \in \Pic(Y)$ such that $\deg(\cl_{i}\big|E_{j})= \delta_{ij}$.  To each of these line bundles, one associates a vector bundle $\cm_{i}$ as follows: if $\Ext^1_{Y}(\cl_{i}, \co_{Y})=0$, then $\cm_{i}:=\cl_{i}$. Otherwise, $\cm_{i}$ is given as the maximal extension\footnote{More precisely, one first takes the direct sum $\eta$ of the generators $\eta_{1}, \ldots, \eta_{r_{i}-1} \in \Ext^1_{Y}(\cl_{i}, \co_{Y})$. Then the sequence \eqref{MaxExtension} is obtained from $\eta$ as a pull-back along the diagonal embedding $\cl_{i} \ra \cl_{i}^{r_{i}-1}$. It follows from the long exact sequence associated with \eqref{MaxExtension} and $\Rf \co_{Y}=\co_{X}$, that $\Ext^1_{Y}(\cm_{i}, \co_{Y})=0$.}
\begin{align}\label{MaxExtension}
0 \ra \co^{r_{i}-1}_{Y} \ra \cm_{i} \ra \cl_{i} \ra 0,
\end{align}
of a minimal set of the $r_{i}-1$ generators of $\Ext^1_{Y}(\cl_{i}, \co_{Y})$. Moreover, $\cm_{i}$ does \emph{not} depend on the choice of the set of generators. In other words, it is determined by $\cl_{i}$.

Then $\cp \cong \co_{Y} \oplus \bigoplus_{i=1}^n \cm_{i}$ and $\cp^*={\mathcal Hom}_{Y}(\cp, \co_{Y})$ is the dual bundle.
\item[(d)] In the notations of (c), the simple $\End_{Y}(\cp)$-modules correspond to the perverse sheaves $\co_{E}$ and $\co_{E_{i}}(-1)[1]$.  
 \end{itemize}
 \end{thm}
\begin{ex}
Let $\mathbb{P}^1$ the projective line over the complex numbers. In particular, we have a morphism $f\colon \mathbb{P}^1 \ra \Spec(\mathbb{C})$, satisfying the conditions of Van den Bergh's Theorem \ref{T:VdB}. In this case $E=\mathbb{P}^1$, so we get $\cm_{1}=\co(1)$ and $\cp=\co \oplus \co(1)$. It is well-known that the endomorphism algebra of $\cp$ is given by the path algebra of the Kronecker quiver. In other words, one recovers Beilinson's  derived equivalence \cite{Beilinson}. 
\end{ex}

Another example, which will be treated in the next subsection, is given by minimal resolutions of rational surface singularities. In this situation, it is essential to note that the vector bundles $\co_{Y}$ and $\{\cm_{i}\}$ occuring in part (c) of Theorem \ref{T:VdB} are precisely the indecomposable special full vector bundles in Wunram's Theorem \ref{T:Wunram}. Wunram's construction \cite{Wunram88} uses a technique of Artin \& Verdier \cite{ArtinVerdier85}. However, the resulting bundles $\cm_{i}$ coincide, which may be seen using the following commutative diagram with exact rows and columns, see also \cite{Esnault}:
\begin{align}\label{E:RelVdBArtinVerdier}
\begin{array}{c}
\begin{xy}
\SelectTips{cm}{}
\xymatrix{&& 0 & 0 \\
&& \co_{D_{i}} \ar[u] \ar@{=}[r] & \co_{D_{i}} \ar[u]  \\
0 \ar[r] & \co_{Y}^{r_{i}-1}   \ar[r]^f & \cm_{i} \ar[u]  \ar[r] & \cl_{i}  \ar[u]  \ar[r] & 0 \\
0 \ar[r] & \co_{Y}^{r_{i}-1} \ar@{=}[u]  \ar[r] & \co_{Y}^{r_{i}} \ar[u]^g  \ar[r] & \co_{Y}  \ar[u]  \ar[r] & 0 \\
&& 0 \ar[u] & 0 \ar[u]
}
\end{xy}
\end{array}
\end{align} 
Here, $f$ respectively $g$ are induced by $r_{i}-1$ respectively $r_{i}$ global sections of $\cm_{i}$.
\subsubsection{Application to rational surface singularities} \label{sss:Applicationtorational}

The following notion was introduced by Wemyss \cite{WemyssReconstructionTypeA}.

\begin{defn}
Let $M=R \oplus \bigoplus_{i \in I} M_{i}$ the direct sum of all indecomposable special $R$-modules. $\Lambda=\End_{R}(M)$ is called the \emph{reconstruction algebra} of $R$. 
\end{defn}

\begin{rem}
Wemyss shows that in many situations (e.g.~for quotient singularities) one can construct the minimal resolution of $\Spec(R)$ from $\Lambda$, by using quiver moduli spaces, see \cite{WemyssReconstructionTypeA, Wemyss10}. Moreover, the quiver of $\Lambda$ and the number of relations are encoded in the combinatorics of the dual intersection graph and the fundamental cycle. However, determining the precise form of the relations is usually a hard problem! 
\end{rem}

\begin{rem}\label{R:AlgebraicMcKay}
If $R$ is Gorenstein (hence an ADE-surface singularity), then $\Lambda$ is the Auslander algebra of the category $\MCM(R)=\SCM(R)$. This may also be described as the completion of the preprojective algebra $\Pi(\widehat{Q})$, where $\widehat{Q}$ is the affine Dynkin quiver corresponding to the dual graph of $R$, see Auslander \cite[Proof of Proposition 2.1.]{Auslander86} in conjunction with Reiten \& Van den Bergh \cite[Proof of Proposition 2.13]{ReitenVdB89}.
\end{rem}

Let $\pi \colon Y \ra \Spec(R)$ be the minimal resolution of singularities. 
Let $I=\cc\cup\cd$ be the index set of the irreducible components of the exceptional divisor $E=\pi^{-1}(\mathfrak{m})$, where $\cc$ denotes the set of $(-2)$-curves and $\cd$ the set of $(-n)$-curves with $n>2$.  We choose a subset $\cs\subseteq I$, and contract all curves in $\cs$. In this way, we obtain a normal scheme $X^{\cs}$ and a factorization of the minimal resolution of singularities $\pi$, see Artin \cite{Artin62} and also Lipman \cite[Theorem 27.1]{Lipman}.
\[
Y\xrightarrow{f^{\cs}} X^{\cs}\xrightarrow{g^{\cs}}\Spec R.
\]
The following result is due to Wemyss.

\begin{thm}\label{main Db}
Let $\cs\subseteq I$. We set $N^{\cs}:=R\oplus_{i\in I\setminus \cs}M_i$ and denote by $e \in \Lambda$ the idempotent endomorphism corresponding to the identity of $N^\cs$.  Then $e\Lambda e=\End_R(N^{\cs})$ is derived equivalent to $X^{\cs}$ via a tilting bundle $\cv_\cs$ in such a way that 
\[
\begin{array}{c}
{\SelectTips{cm}{10}
\xy0;/r.4pc/:
(-10,20)*+{\cd^b(\mod-\Lambda)}="A2",(20,20)*+{\cd^b(\Coh Y)}="A3",
(-10,10)*+{\cd^b(\mod-e\Lambda e)}="a2",(20,10)*+{\cd^b(\Coh X^{\cs})}="a3",
\ar"A3";"A2"_{\RHom_Y(\cv_\emptyset,-)}
\ar"A2";"a2"_{(-)e}
\ar"a3";"a2"_{\RHom_{X^\cs}(\cv_{\cs},-)}
\ar"A3";"a3"^{\Rfcs}
\endxy}
\end{array}
\]
commutes.

Moreover, $\RHom_{Y}(\cv_{\emptyset}, \co_{E_{i}}(-1)[1]) \cong S_{i}$ is the simple $\Lambda$-module with projective cover $\Hom_{R}(M, M_{i})$ and $\RHom_{Y}(\cv_{\emptyset}, \co_{Z_{f}}) \cong S_{\star}$ is the simple $\Lambda$-module with projective cover $\Hom_{R}(M, R)$. 
\end{thm}

\begin{proof}
Since $R$ has rational singularities, $\mathbf{R}\pi_{*}(\co_{Y}) \cong \co_{\Spec R }$. Moreover, all fibres of $\pi$ are at most one dimensional. Hence Theorem \ref{T:VdB} yields a tilting bundle $\cv_{\emptyset}=\co_{Y} \oplus \bigoplus_{i \in I} \cm_{i}^Y$ on $Y$.

For convenience, we denote $X:={X^{\cs}}$, and further $Y\xrightarrow{f^{\cs}} X^{\cs}\xrightarrow{g^{\cs}}\Spec R$ by 
\[
Y\xrightarrow{f} X\xrightarrow{g}\Spec R.
\]
We want to apply Theorem \ref{T:VdB} to $g\colon X \ra \Spec R$. Since $R$ is normal, $g_{*}(\co_{X}) \cong \co_{\Spec R}$ by Zariski's Main Theorem, see \cite[Corollary III.11.4]{Hartshorne}. It remains to show, that $\mathsf{R}^ig_{*}\co_{X}$ vanishes for $i>0$. This follows from the Grothendieck spectral sequence $\mathsf{R}^ig_{*} \mathsf{R}^jf_{*} \Rightarrow \mathsf{R}^{i+j}\pi_{*}$ and the vanishing of  $\mathsf{R}^{n}\pi_{*} \co_{Y}$ for all $n>0$.

Applying Theorem \ref{T:VdB} yields a tilting bundle $\cv_\cs=\co_{X}\oplus_{i\in I\setminus \cs}\cm^{X}_i$  on $X$.

We claim that its pullback along $f$ is a direct summand of the tilting bundle $\cv_{\emptyset}$ on $Y$. More precisely, $f^*(\cv_\cs)\cong\co_Y\oplus_{i\in I\setminus \cs}\cm^Y_i$. 

Recall that there are isomorphisms $\cl^X_i \cong \co_{X}(D^X_{i})$ respectively $\cl^Y_i \cong \co_{Y}(D^Y_{i})$, where the $D_{i}$ are divisors which intersect the exceptional divisor $E_{i}$ transversally and none of the other exceptional divisors, see e.g. the diagram \eqref{E:RelVdBArtinVerdier}. Since $f$ restricts to an isomorphism
\[f\colon Y\setminus \bigcup_{i \in I \setminus \cs} E_{i} \ra X \setminus \Sing(X)\] and $D^X_{i} \subseteq X \setminus \Sing(X)$, we obtain  $f^*\cl^X_i \cong \cl^Y_i$ for all $i\in I\setminus \cs$. Taking the pull back of the maximal extension
\begin{align}\label{E:MaxDown}
0\to\co_X^{\oplus(r_i-1)}\to\cm^X_i\to\cl^X_i\to 0
\end{align}
gives an exact sequence 
\begin{eqnarray}
0\to\co_Y^{\oplus(r_i-1)}\to f^*\cm^X_i\to\cl^Y_i\to 0,\label{On top}
\end{eqnarray}
since \eqref{E:MaxDown} is a sequence of vector bundles. We want to show that this sequence is again a maximal extension. This condition is equivalent to $\Ext^1_Y(f^*\cm^X_i,\co_Y)=0$, which follows from
\begin{align*}
\Ext^1_Y(f^*\cm^X_i,\co_Y) \cong \Ext^1_Y(\Lf\cm^X_i,\co_Y)\cong\\ \Ext^1_X(\cm^X_i,\Rf\co_Y)\cong\Ext^1_X(\cm^X_i,\co_X)=0.
\end{align*}
where the last equality holds since $\cv_\cs$ is a tilting bundle.  Hence \eqref{On top} is a maximal extension, so it follows (by construction, see Theorem \ref{T:VdB}) that $\cm^Y_i\cong f^*\cm^X_i$ for all $i\in I\setminus \cs$, so $f^*(\cv_\cs)\cong \co_Y\oplus_{i\in I\setminus \cs}\cm^Y_i$ as claimed.

Now by the projection formula 
\[
\Rf(f^*\cv_\cs)\cong \Rf(\co_Y\otimes f^*\cv_\cs)\cong \Rf(\co_Y)\otimes \cv_\cs\cong\co_X\otimes\cv_\cs \cong\cv_\cs
\]
and so it follows that
\[
\End_X(\cv_\cs)\cong\Hom_X(\cv_\cs,\Rf(f^*\cv_\cs))\cong\Hom_Y(\Lf\cv_\cs,f^*\cv_\cs)\cong\End_Y(f^*\cv_\cs),
\]
i.e. $\End_X(\cv_\cs)\cong\End_Y(\co_Y\oplus_{i\in I\setminus \cs}\cm^Y_i)$.  By Theorem \ref{T:Esnault}, $\pi_{*}$ induces an isomorphism  $\End_Y(\co_Y\oplus_{i\in I\setminus \cs}\cm^Y_i)\cong \End_R(R\oplus_{i\in I\setminus\cs} M_i)\cong\End_R(N^{\cs})$.

Hence we have shown that $\cv_\cs$ is a tilting bundle on $X^\cs$ with endomorphism ring isomorphic to $\End_R(N^{\cs})$, so the first statement follows.  For the commutativity statement, simply observe that we have functorial isomorphisms
\begin{eqnarray*}
\RHom_{X^\cs}(\cv_\cs,\Rf(-))& \cong&\RHom_{Y}(\Lf\cv_\cs,-)\\
&\cong& \RHom_{Y}(\co_Y\oplus_{i\in I\setminus \cs}\cm^Y_i,-)\\
&\cong& \RHom_{Y}(\co_Y\oplus_{i\in I}\cm^Y_i,-)e\\
&\cong& \RHom_{Y}(\cv_\emptyset,-)e.
\end{eqnarray*}

The last statement follows from \cite[3.5.7]{VandenBergh04}.
\end{proof} 

\begin{rem}
It is well-known that the derived McKay Correspondence \cite{BKR, KapranovVasserot00} for ADE-surface singularities $R$ is a special case of the tilting equivalence \[\cd^b(\Coh Y) \xrightarrow{\RHom_Y(\cv_\emptyset,-)} \cd^b(\mod-\Lambda)\] above. Indeed, in this case $\Lambda$ is the Auslander algebra of $\MCM(R)$, which is isomorphic to the preprojective algebra of the corresponding affine Dynkin quiver $\widehat{Q}$, see Remark \ref{R:AlgebraicMcKay}. This was one of the motivations for Van den Bergh's work \cite{VandenBergh04}.
\end{rem}

The following result is due to Wemyss \cite{Wemyss10}. It shows that $\SCM(R)$ admits a non-commutative resolution in the sense of Definition \ref{defNCR}. In particular, our Morita type Theorems (Theorem \ref{t:main-thm} and Theorem \ref{t:alternative-main}) may be applied to $\SCM(R)$.

\begin{cor}\label{C:FinGlobalDim}
The reconstruction algebra has finite global dimension.
\end{cor}
\begin{proof}
Since $Y$ is smooth, there is a triangle equivalence $\Perf(Y) \cong \cd^b(\Coh(Y))$. Using the derived equivalence from Theorem \ref{main Db}, we obtain an equivalence $\Perf(\Lambda) \cong \cd^b(\mod-\Lambda)$. Hence every finitely generated $\Lambda$-module has finite projective dimension. Since $\Lambda$ is a finitely generated $R$-module and $R$ is complete, $\Lambda$ has only finitely many simple modules and the global dimension is bounded by the maximum of their projective dimensions, see Auslander \cite[\S 3]{Auslander55}. 
\end{proof}

\subsection{$\SCM(R)$ is a Frobenius category}\label{ss:SCMFrobenius}
In this subsection, we follow Iyama \& Wemyss article \cite{IWnewtria}.
By Proposition \ref{P:IWSpecials} (b), $\SCM(R)$ is a full extension closed subcategory in $\MCM(R)$. Hence, $\SCM(R)$ is an exact category. Moreover, the classes of projective and injective objects coincide by Corollary \ref{C:ExtSymmetry}. To show that $\SCM(R)$ is a Frobenius category, it suffices to show that it has enough projective and injective objects. 

\begin{prop}\label{P:ExactRelProjandInj}
Let $\cb$ be a Krull--Remak--Schmidt exact category with enough  injective objects. Let $\cc \subseteq \cb$ be a contravariantly finite extension closed subcategory. Then $\cc$ has enough injective objects with respect to the induced exact structure.

A dual result holds for  projective objects and covariantly finite subcategories. 
\end{prop}
\begin{proof}
Let $X$ be an object in $\cc$. By our assumptions on  $\cb$, there is a conflation $X \ra I \ra X'$ in $\cb$, with $I$  injective in $\cb$. This yields the following epimorphism of functors $\Hom_{\cb}(-, X') \ra \Ext^1_{\cb}(-, X) \ra 0$. Since $\cc$ is contravariantly finite there exists $Y \in \cc$ and an epimorphism of functors $\Hom_{\cc}(-, Y) \ra \Hom_{\cb}(-,X')\big|_{\cc} \ra 0$. Combining these, we obtain an epimorphism $\phi\colon \Hom_{\cc}(-, Y) \ra \Ext^1_{\cb}(-, X)\big|_{\cc} \ra 0$.  Since $\cc$ is a Krull--Remak--Schmidt category, there exists a $\phi$, which is a projective cover.
Using the naturality of $\phi$, one can check that it is induced by the following conflation in $\cc$ 
\begin{align} \label{E:Confi}
\phi_{Y}(\id_{Y})\colon X \ra Z \ra Y.
\end{align}
In particular, $\Ext^1_{\cb}(-, X)\big|_{\cc}$ is a finitely presented $\cc$-module. 

It remains to show that $Z$ is injective in $\cc$. In other words, every conflation $Z\ra Z' \ra Z''$ in $\cc$ splits. Together with the conflation \eqref{E:Confi}, we obtain a commutative diagram, where the columns and rows are conflations in $\cc$
 
 \begin{align}\label{E:4Confi}
\begin{array}{c}
\begin{xy}
\SelectTips{cm}{}
\xymatrix{
X \ar[rr] \ar@{=}[d] && Z \ar[d] \ar[rr] && Y \ar[d]^\alpha \\
X  \ar[rr] && Z' \ar[rr] \ar[d] && Y' \ar[d]\\
&& Z'' \ar@{=}[rr] && Z''.
}
\end{xy}
\end{array}
\end{align} 
Since $\cc$ is closed under extensions $Y' \in \cc$. The two horizontal conflations yield a commutative diagram with exact rows
\begin{align}\label{E:twoConfl}
\begin{array}{c}
\begin{xy}
\SelectTips{cm}{}
\xymatrix@C=11pt{
0 \ar[r] & \Hom_{\cc}(-, X) \ar[r] \ar@{=}[d] & \Hom_{\cc}(-, Z) \ar[r] \ar[d] & \Hom_{\cc}(-, Y) \ar[r]^{\phi}  \ar[d]^{\alpha \circ -}& \Ext^1_{\cb}(-,X)\big|_{\cc} \ar@{=}[d] \ar[r] & 0 \\
0 \ar[r] & \Hom_{\cc}(-, X) \ar[r] & \Hom_{\cc}(-, Z') \ar[r] & \Hom_{\cc}(-, Y') \ar[r]& \Ext^1_{\cb}(-,X)\big|_{\cc} \ar@{..>}[r] & 0.
}
\end{xy}
\end{array}
\end{align} 
The monomorphism $\alpha \circ -$ is split since $\phi$ is a projective cover. By Yoneda's Lemma $\alpha$ is a split monomorphism. Now the upper sequence in \eqref{E:twoConfl} yields an injection
$0 \ra \Ext^1_{\cb}(Z'', Z)\big|_{\cc} \ra \Ext^1_{\cb}(Z'', Y)\big|_{\cc}$. The image of the left vertical conflation in \eqref{E:4Confi} under this map is just the vertical conflation on the right of \eqref{E:4Confi}. Since the latter is split and the map is injective the conflation $Z \ra Z' \ra Z''$ splits as claimed. This shows that $Z$ is  injective and completes the proof of the first statement. The second claim follows by a dual argument. 
\end{proof}

\begin{cor}
$\SCM(R)$ is a Frobenius category.
\end{cor}
\begin{proof}
We want to apply Proposition \ref{P:ExactRelProjandInj} to the pair $\SCM(R) \subseteq \MCM(R)$. $\MCM(R)$ is an extension closed subcategory in $\mod-R$. This gives an exact structure on $\MCM(R)$. Since $R$ is complete the Krull--Remak--Schmidt property holds. Since $\SCM(R)$ has only finitely many indecomposable objects (Theorem \ref{T:Wunram}), it is a functorially finite subcategory in $\MCM(R)$, by Example \ref{Ex:FunctFinite}. Since $\MCM(R)$ is closed under kernels of epimorphisms, it has enough  projective objects. As a complete Cohen--Macaulay ring, $R$ has a canonical module $K$ \cite[Corollary 3.3.8]{BrunsHerzog}. It is an  injective object in $\MCM(R)$ and induces a duality $\Hom_{R}(-, K)\colon \MCM(R) \ra \MCM(R)$ \cite[Theorem 3.3.10]{BrunsHerzog}. In particular, $\MCM(R)$ has enough  injective objects since it has enough  projective objects. Proposition \ref{P:ExactRelProjandInj} shows that $\SCM(R)$ has enough  projective and injective objects, which coincide by Corollary \ref{C:ExtSymmetry}. Hence, $\SCM(R)$ is a Frobenius category.
\end{proof}

\subsubsection{Description of  projective-injective objects} 
Iyama \& Wemyss \cite{IWnewtria} give an explicit description of the subcategory of  projective-injective objects $\proj \SCM(R)$. To explain their result, we need some notions and results from geometry. 

\begin{prop}[{Adjunction Formula \cite[Proposition V.1.5]{Hartshorne}}]\label{P:AdjForm} Let $C$ be a non-singular curve of genus $g=\dim_{k} H^1(C, \co_{C})$ on a smooth surface $X$, with canonical divisor $K$, then 
\begin{align}
2g-2=C.(C+K).
\end{align}
\end{prop}

This has the following well-known consequence.

\begin{cor}\label{C:minusTwoCurve} Let $C$ be a rational curve on a smooth surface $X$. If $C$ has selfintersection number $C^2=-2$, then
$\dim_{k} \Ext^2_{X}(\co_{C}(m), \co_{C}(m))=1$.    
\end{cor}
\begin{proof}
By the adjunction formula, $-2=2g-2=C.(C+K)=-2 + \deg( K\big|_{C})$. Hence, $\deg( K\big|_{C})=0$ and therefore $\omega_{X} \otimes_{X} \co_{C} \cong \omega_{X}\big|_{C} \cong \co_{C}$. Let $(-)^*=\Hom_{k}(-, k)$. Then Serre duality shows
\begin{align*}
\Hom_{X}(\co_{C}(m), \co_{C}(m)) \cong \Hom_{X}(\co_{C}(m), \co_{C}(m) \otimes_{X} \omega_{X}[\dim X])^* \\ \cong \Ext^2_{X}(\co_{C}(m), \co_{C}(m))^*,
\end{align*}
which proves the claim.
\end{proof}

The derived equivalence from Theorem \ref{main Db} and the discussion above yield the following statement.

\begin{prop}\label{P:notrelprojSCM}
Let $M_{i}$ be an indecomposable SCM corresponding to a $(-2)$-curve $E_{i}$. Then $\Ext^1_{R}(M_{i}, X) \neq 0$, for some special Cohen--Macaulay module $X$. In particular, $M_{i}$ is not  projective in $\SCM(R)$.
\end{prop}
\begin{proof}
Since $X$ is surface, the abelian category $\Coh(X)$ has global dimension two. In particular, $\Ext^3(\co_{E_{i}}(-1), \cf)=0$ for all $\cf \in \Coh(X)$. Using the derived equivalence from Theorem \ref{main Db}, this translates into $\Ext^3_{\Lambda}(S_{i}, S_{j})=0$ for all $j \in I$. A longer calculation in $\cd^b(\Coh(X))$ (involving Serre duality and the combinatorics of dual graph and fundamental cycle) shows that $\Ext^2_{X}(\co_{E_{i}}(-1), \co_{Z_{f}})=-E_{i}^2-2$ (see \cite[Proof of Thm.~3.2]{Wemyss10}), which is zero by our assumption on $E_{i}$. This translates into $\Ext^3_{\Lambda}(S_{i}, S_{\star})=0$. Hence, $\prdim_{\Lambda}(S_{i})\leq 2$. But Corollary \ref{C:minusTwoCurve} yields $\Ext^2_{\Lambda}(S_{i}, S_{i}) \neq 0$. Therefore, $\prdim_{\Lambda}(S_{i})=2$.

Let $M$ be an additive generator of $\SCM(R)$. Recall, that there is an additive equivalence
\begin{align}
\SCM(R)=\add(M) \xrightarrow{\Hom_{R}(M, -)} \proj-\End_{R}(M)=\proj-\Lambda.
\end{align}
Thus we may write a minimal projective resolution of $S_{i}$ as
\begin{align}
0 \ra \Hom_{R}(M, X) \xrightarrow{\alpha \circ -} \Hom_{R}(M, Y) \xrightarrow{\beta \circ -} \Hom_{R}(M, M_{i}) \ra S_{i} \ra 0
\end{align}
with $X$ and $Y$ in $\add M=\SCM(R)$.  Since $R$ is a direct summand of $M$, we obtain an exact sequence of special CMs
\begin{align}
0 \ra X \xrightarrow{\alpha} Y \xrightarrow{\beta} M_{i} 
\end{align}
Let us show that $\beta$ is surjective. Since $S_{i}$ is one dimensional, the only maps which are not in the image of $\beta \circ -$  are scalar multiples of the identity of $M_{i}$. In particular, for any $m \in M_{i}$ the map
$R \oplus \bigoplus_{j \in I} M_{j} \ra M_{i}$, defined by $(r, x_{1}, \ldots) \mapsto r\cdot m$ has a preimage $f$ in $\Hom_{R}(M, Y)$. Let $i\colon R \ra M$ be the canonical inclusion and set $y=f\circ i(1) \in Y$. Then $\beta(y)=m$.

Since $S_{i}$ has projective dimension two the exact sequence
\begin{align}
0 \ra X \xrightarrow{\alpha} Y \xrightarrow{\beta} M_{i} \ra 0
\end{align}   
yields a non-trivial element in $\Ext^1_{R}(M_{i}, X)$. This completes the proof.
\end{proof}

To show that the remaining specials are projective, we need the following result, see \cite[Theorem 3.3]{IWnewtria}.

\begin{thm}\label{T:Syzygy}
Let $M$ be a maximal Cohen--Macaulay module. Then \[\Omega(M) \cong \bigoplus_{i \in I} \bigl(\Omega(M_{i})\bigr)^{\oplus c_{1}(M).E_{i}},\] where by an abuse of notation $c_{1}(M)$ denotes the first Chern class of the full vector bundle $\pi^\#(M)$, see Esnault's Theorem \ref{T:Esnault}.
\end{thm}

Applying this result to the canonical module $\omega_{R}$, one obtains the following corollary.

\begin{cor}\label{C:Syzygy}
$\Omega(\omega_{R}) \cong \bigoplus_{i \in I} \bigl(\Omega(M_{i})\bigr)^{\oplus -2-E_{i}^2}$
\end{cor}
\begin{proof}
  By Proposition \ref{P:AltSpecial}, $\pi^\#(\omega_{R}) \cong \omega_{X}$. Now, Theorem \ref{T:Syzygy} shows that there is an isomorphism $\Omega(\omega_{R}) \cong \bigoplus_{i \in I} \bigl(\Omega(M_{i})\bigr)^{\oplus K.E_{i}}$, where $K$ denotes the canonical divisor on $X$ (in other words $\omega_{X} \cong \co_{X}(K)$). The adjunction formula (Proposition \ref{P:AdjForm}) completes the proof.
\end{proof}

\begin{rem}
Assume that $R$ is not Gorenstein (so $\omega_{R}\ncong R$). In conjunction with Proposition \ref{P:IWSpecials} (c), the corollary above shows the following: if $\omega_{R}$ is special, then it corresponds to a $(-3)$-curve and all other irreducible curves in the exceptional divisor are $(-2)$-curves.
\end{rem}

\begin{prop}\label{P:relprojSCM}
If $E_{i}$ is not a $(-2)$-curve, then the corresponding special Cohen--Macaulay module is  projective in $\SCM(R)$.
\end{prop}
\begin{proof}
We need the following preliminary statement: let $X,Y$ be maximal Cohen--Macaulay $R$-modules. If $\Ext^1_{R}(X, Y)=0$, then $\Ext^1_{R}(\tau^{-1} \Omega^{-1}(Y), X)=0$, where $\tau=\Hom_{R}(\Omega^2 \mathsf{tr}(-), \omega_{R})$, denotes the Auslander--Reiten translation. 

By definition of $\Omega^{-1}$ there exists a short exact sequence $0 \ra Y \ra I \ra \Omega^{-1}(Y) \ra 0$, with $I \in \add \omega_{R}$  injective. Applying $\Hom_{R}(X, -)$ to this sequence, yields a short exact sequence
\begin{align}
0 \ra \Hom_{R}(X, Y) \ra \Hom_{R}(X, I) \ra \Hom_{R}(X, \Omega^{-1}(Y)) \ra 0,
\end{align}
since $\Ext^1_{R}(X, Y)=0$, by assumption. In other words, every map from $X$ to $\Omega^{-1}(Y)$ factors through an injective object. Hence, $\ol{\Hom}_{R}(X, \Omega^{-1}(Y))=0$. Auslander--Reiten duality shows that $\Ext^1_{R}(\tau^{-1}\Omega^{-1}(Y), X)=0$ as claimed.

If $X$ is special, then $\Ext^1_{R}(X, R)=0$, by Proposition \ref{P:IWSpecials} (b). The statement above yields \begin{align} \label{E:ExtVanishing}\Ext^1_{R}(\tau^{-1}\Omega^{-1}(R), X)=0. \end{align}
 By definition of $\tau$, we have $\Omega^{-1}(R) \cong \tau(\Hom_{R}(\Omega^{-1}(R), \omega_{R})^*)$, up to free summands. Hence we obtain 
 \begin{align}
 \tau^{-1}\Omega^{-1}(R) \cong \Hom_{R}(\Omega^{-1}(R), \omega_{R})^*
 \end{align}
 up to free summands.
 
 We apply $\Hom_{R}(-, \omega_{R})$ to the short exact sequence $0 \ra R \ra I' \ra \Omega^{-1}(R) \ra 0$, with $I' \in \add \omega_{R}$. Since $\omega_{R}$ is  injective in $\SCM(R)$ and $\End_{R}(\omega_{R}) \cong R$ (see \cite[Theorem 3.3.10]{BrunsHerzog}) we obtain an exact sequence
\begin{align}
0 \ra \Hom_{R}(\Omega^{-1}(R), \omega_{R}) \ra \Hom_{R}(I', \omega_{R}) \ra \omega_{R} \ra 0,
\end{align}
with $\Hom_{R}(I', \omega_{R}) \in \add_{R}(R)$. Thus $\Hom_{R}(\Omega^{-1}(R), \omega_{R}) \cong \Omega(\omega_{R})$ up to free summands. In particular, \eqref{E:ExtVanishing} shows that $\Ext^1_{R}((\Omega(\omega_{R}))^*, X)=0$ for all specials $X$. In other words, $(\Omega(\omega_{R}))^*$ is  projective in $\SCM(R)$. Now, Corollary \ref{C:Syzygy} together with Proposition \ref{P:IWSpecials} (c) completes the proof.
\end{proof}

\subsection{Main result} \label{ss:MainSpecial}
Iyama \& Wemyss observed that the Auslander--Reiten quiver of the stable category of special Cohen--Macaulay modules is a disjoint union of double quivers $\ol{Q_{i}}$, with $Q_{i}$ of ADE-type \cite[Corollary 4.11.]{IWnewtria}. Moreover, by checking examples, they found that this phenomenon often stems from an equivalence $\ul{\ul{\SCM}}(R) \cong \ul{\MCM}(R')$\footnote{Iyama \& Wemyss use the notation $\ul{\ul{\SCM}}(R)$ for the stable category in order to avoid confusion with the factor category $\SCM(R)/\proj-R \subseteq \ul{\MCM}(R)$ obtained by factoring out $\proj-R$ only. We decided to use their notation here.}, where $R'$ is some \emph{Gorenstein} singularity \cite[Remark 4.14]{IWnewtria}. 

Applying our general Frobenius category results (Theorem \ref{t:main-thm} or Theorem \ref{t:alternative-main}), we are able to give an conceptual explanation for these observations. Moreover, using this approach, it is immediate that $\ul{\ul{\SCM}}(R) \cong \ul{\MCM}(R')$ is a \emph{triangle} equivalence, where the ring $R'$ has a precise meaning. Finally, we apply Auslander \& Solberg's `Modification of Frobenius structures'-result (Proposition \ref{new Frobenius structure}) to this setup. This yields new Frobenius structures on $\SCM(R)$ and we also describe the corresponding stable categories. The results in this subsection are based on a joint work with Osamu Iyama, Michael Wemyss and Dong Yang \cite{IKWY12}.

Let us recall some notations. Let $\pi \colon Y \ra \Spec(R)$ be the minimal resolution of singularities. 
Let $I=\cc\cup\cd$ be the index set of the irreducible components $E_{i}$ of the exceptional divisor $E=\pi^{-1}(\mathfrak{m})=\bigcup_{i \in I} E_{i}$, where $\cc$ denotes the set of $(-2)$-curves and $\cd$ the set of $(-n)$-curves with $n>2$. We choose a subset $\cs\subseteq I$, and contract all curves in $\cs$. In this way, we obtain a space which we denote by $X^{\cs}$. Note that $Y \ra X^{\cs}$ is a minimal resolution of singularities.
We need the following well-known statement about the singularities of $X^{\cs}$, see Lipman \cite{Lipman2}.

\begin{prop}\label{P:RationalDoublePoint}
If $\cs \subseteq \cc$, then the completion $\widehat{\co}_{x}$ of the local ring of a singular point $x \in X^\cs$ is a rational double point or equivalently an ADE-surface singularity. 
\end{prop}
\begin{proof}
By Theorem \ref{T:ArtinTree} and the construction of $X^\cs$, the exceptional fibre $E_{x}$ of the minimal resolution of $\Spec(\widehat{\co}_{x})$ is a tree of rational $(-2)$-curves. $\widehat{\co}_{x}$ is a rational singularity by work of Artin \cite{Artin66}, see also Lipman \cite[Theorem 27.1]{Lipman}. Artin also shows that a rational surface singularity is a rational double point if and only if the dual intersection graph of the exceptional fibre is of ADE-type. Mumford \cite{Mumford} has shown that the intersection form $(E_{i}, E_{j})=E_{i}.E_{j}$ of a normal surface singularity is negative definite. By our assumptions on $\cs$, we know that $(E_{i}, E_{i})=E_{i}.E_{i}=-2$ for all $i \in \cs$. Integral quadratic forms with these properties are well-known, for example from Lie theory. They are in one to one correspondence with the incidence matrices of the Dynkin diagrams of type $A$, $D$ and $E$, with valuation $-2$ on all vertices, see e.g. \cite[\S 1.2]{Brieskorn66}. It follows from work of Brieskorn \cite[Satz 1]{Brieskorn66} and Kirby \cite[2.6, 2.7]{Kirby} that the rational double points are precisely the hypersurface singularities defined by the ADE-equations listed in Paragraph \ref{sss:Classical}.
\end{proof}

We first describe the stable category $\ul{\ul{\SCM}}(R)$ of the standard Frobenius category structure on $\SCM(R)$.

\begin{thm}\label{C:StandardStableSCM}
Let $(R, \mathfrak{m})$ be a complete rational surface singularity and let $E= \End_{R}(R\oplus \bigoplus_{i \in \cd} M_{i})=\End_{R}(N^{\cc})$. Then the following statements hold:
\begin{itemize}
\item[(a)] $E$ is an Iwanaga--Gorenstein ring.
\item[(b)] There is a chain of triangle equivalences
\begin{align}
\ul{\ul{\SCM}}(R) \cong \cd_{sg}(E) \cong \cd_{sg}(X^{\cc}) \cong \bigoplus_{x \in \Sing(X^\cc)} \ul{\MCM}(\widehat{\co}_{x})
\end{align}
\end{itemize}
In particular, $\ul{\ul{\SCM}}(R)$ is a $1$-CY category and the shift functor satisfies $[2]\cong \id$. 
\end{thm}
\begin{proof}
The first statement follows from Iyama's Theorem \ref{t:main-thm}. Indeed, $\Lambda$ has finite global dimension, by Corollary \ref{C:FinGlobalDim}. Moreover, $\Lambda$ is Noetherian since it is an endomorphism algebra of a finitely generated module over a commutative Noetherian ring.

Let us prove part (b).
The first triangle equivalence follows from Theorem \ref{t:main-thm} (together with Buchweitz \ref{T:Buchweitz}) or the alternative Theorem \ref{t:alternative-main}. In both cases, we use that $\ul{\ul{\SCM}}(R)$ is idempotent complete, by an adaption of Lemma \ref{L:stable-complete}.  

The second equivalence follows from Wemyss' Theorem \ref{main Db} together with the fact that the respective perfect subcategories admit an intrinsic characterization.

Proposition \ref{P:RationalDoublePoint} shows that $X^\cc$ has only ADE-surface singularities. In particular, $X^\cc$ has isolated singularities and the last equivalence follows from Theorem \ref{t:main-global} and the fact that all categories in the chain of equivalences are idempotent complete, since the first one is.

Stable categories of ADE--surface singularities are $1$-Calabi--Yau by Auslander \cite{Auslander76}  and satisfy $[2]=\id$, by Eisenbud \cite{Eisenbud80}.
\end{proof}

We fix some notations. Let $I$ be the finite index set of the irreducible exceptional curves $E_{i}$. Let $\cs \subseteq \cc$ be a subset of the irreducible $(-2)$-curves in the exceptional divisor. Let $N^{\cs}=R \oplus \bigoplus_{i \in I \setminus \cs} M_{i}$.
For each such $\cs$, Auslander \& Solberg's result (Proposition \ref{new Frobenius structure}) yields a new exact Frobenius structure $\SCM_{N^\cs}(R)$ on $\SCM(R)$ as follows:
a conflation $X \ra Y \ra Z$ in $\SCM(R)$ is a conflation in $\SCM_{N^\cs}(R)$, if and only if the induced sequence
\begin{align}
0 \ra \Hom_{R}(N^\cs, X) \ra \Hom_{R}(N^\cs, Y) \ra \Hom_{R}(N^\cs, Z) \ra 0 
\end{align}
is exact.

\begin{cor}\label{C:NewFrobSCM}
$\SCM_{N^\cs}(R)$ is a Frobenius category with projective-injective objects $\add_{R} N^\cs$.
\end{cor}
\begin{proof}
We verify the conditions of Proposition \ref{new Frobenius structure}.

By Propositions \ref{P:relprojSCM} and \ref{P:notrelprojSCM} and our choice of $\cs$, $\add N^\cs$ contains all  projective-injective objects of $\SCM(R)$. Moreover, since $\add N^\cs$ contains only finitely many indecomposable objects, it is a functorially finite subcategory by Example \ref{Ex:FunctFinite}. We claim that we may take $\tau\colon \ul{\ul{\SCM}}(R) \ra \ul{\ul{\SCM}}(R)$ to be the identity functor. It suffices to show that we have a chain of functorial isomorphisms
\begin{align*}
\Ext^1_{R}(M, N) \cong \ul{\ul{\Hom}}_{R}(M, \Omega^{-1}N) \cong D\ul{\ul{\Hom}}_{R}(\Omega^{-1}N, \Omega^{-1}M) \cong D\ul{\ul{\Hom}}_{R}(N, M),
\end{align*}
for all $M$ and $N$ in $\SCM(R)$.
The first isomorphism holds in any Frobenius category. The second uses the $1$-Calabi--Yau property of $\ul{\ul{\SCM}}(R)$, see Theorem \ref{C:StandardStableSCM} and the last one follows since the shift functor $\Omega^{-1}$ is an autoequivalence.
\end{proof}

\begin{cor}\label{C:ModifiedStableSCM}
Let $(R, \mathfrak{m})$ be a complete rational surface singularity, let $\cs \subseteq \cc$ and let $E^\cs=\End_{R}(N^{\cs})$. Then the following statements hold:
\begin{itemize}
\item[(a)] $E^\cs$ is an Iwanaga--Gorenstein ring.
\item[(b)] There is a chain of triangle equivalences
\begin{align}
\ul{\SCM}_{N^\cs}(R) \cong \cd_{sg}(E^\cs) \cong \cd_{sg}(X^{\cs}) \cong \bigoplus_{x \in \Sing(X^\cs)} \ul{\MCM}(\widehat{\co}_{x})
\end{align}
\end{itemize}
In particular, $\ul{\SCM}_{N^\cs}(R)$ is a $1$-CY category and the shift functor satisfies $[2]\cong \id$. 
\end{cor}
\begin{proof}
Analogous to the proof of Theorem \ref{C:StandardStableSCM}.
\end{proof}

\begin{rem}
The Auslander--Reiten quivers of Hom-finite algebraic triangulated categories with finitely many indecomposable objects were studied by Amiot \cite{Amiot07a} and also by Xiao and Zhu \cite{XiaoZhu05}. However, there are subtle issues in the classification of $1$-CY categories, e.g.~it is not clear whether a $1$-CY category over $\mathbb{C}$ is \emph{standard}, i.e.~whether it is determined as a $k$-category by its Auslander--Reiten quiver. In characteristic $2$ there are examples of non-standard $1$-CY categories, see \cite{Amiot07a}. 
\end{rem}

\subsection{Examples}\label{ss:ExampleSCM}
\begin{defn}
A subgroup $G \subseteq \GL(2, \C)$ is \emph{small}, if none of the $g \in G$ is a \emph{pseudo-reflection}, i.e.~conjugated to a diagonal matrix $\mathsf{diag}(1, 1, \ldots, 1, a)$. 
\end{defn}

The following result of Prill shows that it suffices to study quotient singularities for finite small subgroups $G \subseteq \GL(2, \C)$ \cite{Prill}.
\begin{prop}
Every quotient surface singularity is isomorphic to a quotient singularity $R=\C\llbracket x, y\rrbracket^G$ for a finite small subgroup $G \subseteq \GL(2, \C)$. Moreover, two such subgroups yield isomorphic quotient singularities if and only if they are conjugated.
\end{prop}

\subsubsection{Cyclic quotient singularities}
Let $G \subseteq \GL(2, \C)$ be a small finite cyclic subgroup. Up to conjugation, such a group is generated by an element
\begin{align}
g=\begin{pmatrix} \epsilon_{n} & 0 \\ 0 & \epsilon_{n}^a \end{pmatrix} \in \GL(2, \C), 
\end{align}
where $\epsilon_{n}$ is a primitive $n$-th root of unity and $0<a<n$ is coprime to $n$ since $G$ is small. We denote this group by $\frac{1}{n}(1, a)$. 
Let $R_{n, a} = \C\llbracket x, y \rrbracket^{\frac{1}{n}(1, a)} $ be the corresponding quotient singularity.
\begin{rem}
By Watanabe's result (see Remark \ref{R:Watanabe}), $R_{n, a}$ is Gorenstein if and only if $a=n-1$.
\end{rem}
Then the geometry of the exceptional divisor of the minimal resolution $\pi \colon X \ra \Spec(R)$ is well-known by work of Jung and Hirzebruch, see e.g.~\cite{Brieskorn}: let
\begin{align}
\frac{r}{a}=\alpha_{1} - \cfrac{1}{\alpha_{2} - \cfrac{1}{\alpha_{3} - \ldots -\cfrac{1}{\alpha_{t-1}- \cfrac{1}{\alpha_{t}}}}}
\end{align}
be the Jung--Hirzebruch continued fraction expansion of $\frac{r}{a}$, where each $\alpha_{i} \geq 2$. Then the dual intersection graph of $R_{n, a}$ is given by the following string
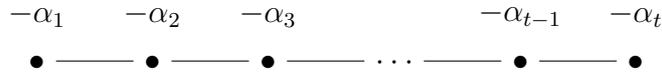
\begin{equation*}\begin{tikzpicture}[description/.style={fill=white,inner sep=2pt}]
    \matrix (n) [matrix of math nodes, row sep=3em,
                 column sep=2.5em, text height=1.5ex, text depth=0.25ex,
                 inner sep=0pt, nodes={inner xsep=0.3333em, inner
ysep=0.3333em}]
    {  
       \bullet & \bullet & \bullet &\cdots& \bullet & \bullet \\
    };
\path[-] (n-1-1) edge (n-1-2);
\path[-] (n-1-2) edge (n-1-3);
\path[-] (n-1-3) edge (n-1-4);
\path[-] (n-1-4) edge (n-1-5);
\path[-] (n-1-5) edge (n-1-6);

 \node at ($(n-1-1.north) + (0mm, 3mm)$) {$-\alpha_{1}$};
 \node at ($(n-1-2.north) + (0mm, 3mm)$) {$-\alpha_{2}$};
 \node at ($(n-1-3.north) + (0mm, 3mm)$) {$-\alpha_{3}$};
 \node at ($(n-1-5.north) + (0mm, 3mm)$) {$-\alpha_{t-1}$};
 \node at ($(n-1-6.north) + (0mm, 3mm)$) {$-\alpha_{t}$};
\end{tikzpicture}\end{equation*}

\begin{ex}
The fraction $\frac{27}{19}$ has the following continued fraction expansion
\begin{align}
\frac{27}{19}=2 - \cfrac{1}{2 - \cfrac{1}{ 5-\cfrac{1}{2-\cfrac{1}{2-\cfrac{1}{2}}}}}
\end{align}
Hence the dual graph associated with the cyclic group $\frac{1}{27}(1, 19)$ of order $27$ has the following form
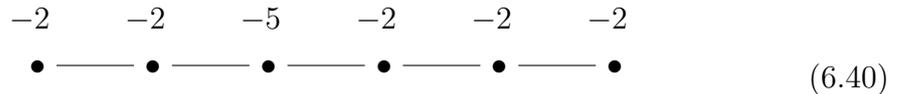
\begin{equation}\label{E:DualGraph2719}\begin{tikzpicture}[description/.style={fill=white,inner sep=2pt}]
    \matrix (n) [matrix of math nodes, row sep=3em,
                 column sep=2.5em, text height=1.5ex, text depth=0.25ex,
                 inner sep=0pt, nodes={inner xsep=0.3333em, inner
ysep=0.3333em}]
    {  
       \bullet & \bullet & \bullet &\bullet & \bullet & \bullet \\
    };
\path[-] (n-1-1) edge (n-1-2);
\path[-] (n-1-2) edge (n-1-3);
\path[-] (n-1-3) edge (n-1-4);
\path[-] (n-1-4) edge (n-1-5);
\path[-] (n-1-5) edge (n-1-6);

 \node at ($(n-1-1.north) + (-1mm, 3mm)$) {$-2$};
 \node at ($(n-1-2.north) + (-1mm, 3mm)$) {$-2$};
 \node at ($(n-1-3.north) + (-1mm, 3mm)$) {$-5$};
 \node at ($(n-1-4.north) + (-1mm, 3mm)$) {$-2$};
 \node at ($(n-1-5.north) + (-1mm, 3mm)$) {$-2$};
 \node at ($(n-1-6.north) + (-1mm, 3mm)$) {$-2$};
\end{tikzpicture}\end{equation}

By Theorem \ref{C:StandardStableSCM}, the stable category of the standard Frobenius structure on $\SCM(R_{27, 19})$ may be described as follows
\begin{align}
\ul{\ul{\SCM}}(R_{27, 19}) \cong \ul{\MCM}(R_{3, 2}) \oplus \ul{\MCM}(R_{4, 3}),
\end{align}
where $R_{3, 2}$ and $R_{4, 3}$ are the simple $A_{2}$ and $A_{3}$ singularities, respectively. Indeed, removing the $-5$-curve from the dual graph yields a disjoint union of an $A_{2}$ and $A_{3}$ graph with all selfintersection numbers equal to $-2$. 

In particular, the Auslander--Reiten quiver of $\ul{\ul{\SCM}}(R_{27, 19})$ is given by
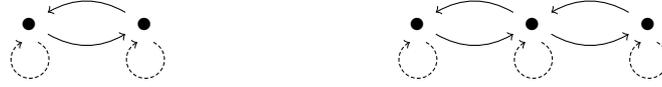
\begin{equation*}\begin{tikzpicture}[description/.style={fill=white,inner sep=2pt}]
    \matrix (n) [matrix of math nodes, row sep=3em,
                 column sep=2.5em, text height=1.5ex, text depth=0.25ex,
                 inner sep=0pt, nodes={inner xsep=0.3333em, inner
ysep=0.3333em}]
    {  
       \bullet & \bullet &&&\bullet & \bullet & \bullet \\
    };
\path[->] (n-1-1) edge [bend right=30](n-1-2);
\path[<-] (n-1-1) edge [bend left=30] (n-1-2);
\draw[dash pattern = on 0.5mm off 0.3mm,->] ($(n-1-1.south) +
    (1.2mm,0.5mm)$) arc (65:-245:2.5mm);
\draw[dash pattern = on 0.5mm off 0.3mm,->] ($(n-1-2.south) +
    (1.2mm,0.5mm)$) arc (65:-245:2.5mm);

\path[->] (n-1-5) edge [bend right=30](n-1-6);
\path[<-] (n-1-5) edge [bend left=30] (n-1-6);

\path[->] (n-1-6) edge [bend right=30](n-1-7);
\path[<-] (n-1-6) edge [bend left=30] (n-1-7);

\draw[dash pattern = on 0.5mm off 0.3mm,->] ($(n-1-5.south) +
    (1.2mm,0.5mm)$) arc (65:-245:2.5mm);
\draw[dash pattern = on 0.5mm off 0.3mm,->] ($(n-1-6.south) +
    (1.2mm,0.5mm)$) arc (65:-245:2.5mm);
\draw[dash pattern = on 0.5mm off 0.3mm,->] ($(n-1-7.south) +
    (1.2mm,0.5mm)$) arc (65:-245:2.5mm);
 \end{tikzpicture}\end{equation*}

We can also consider a modified Frobenius structure on $\SCM(R_{27, 19})$. For example, set $\cs=\{1\}$, where $E_{1}$ is the leftmost $(-2)$-curve in \eqref{E:DualGraph2719}. Then $N^{\cs}=R \oplus M_{1} \oplus M_{3}$ and the corresponding stable category decomposes as follows, see Corollary \ref{C:ModifiedStableSCM}
\begin{align}
\ul{\SCM}_{N^\cs}(R_{27, 19}) \cong \ul{\MCM}(R_{2, 1}) \oplus \ul{\MCM}(R_{4, 3}). 
\end{align}
\end{ex}

Generically, the triangulated category is trivial. Let us give a concrete example.

\begin{ex}
Consider the continued fraction expansion of the fraction $\frac{51}{11}$
\begin{align}
\frac{51}{11}=5-\cfrac{1}{3-\cfrac{1}{4}}
\end{align}
We obtain the following dual intersection graph for the singularity $R_{51, 11}$:
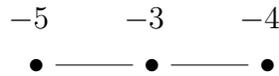
\begin{equation*}\begin{tikzpicture}[description/.style={fill=white,inner sep=2pt}]
    \matrix (n) [matrix of math nodes, row sep=3em,
                 column sep=2.5em, text height=1.5ex, text depth=0.25ex,
                 inner sep=0pt, nodes={inner xsep=0.3333em, inner
ysep=0.3333em}]
    {  
       \bullet & \bullet & \bullet  \\
    };
\path[-] (n-1-1) edge (n-1-2);
\path[-] (n-1-2) edge (n-1-3);

 \node at ($(n-1-1.north) + (-1mm, 3mm)$) {$-5$};
 \node at ($(n-1-2.north) + (-1mm, 3mm)$) {$-3$};
 \node at ($(n-1-3.north) + (-1mm, 3mm)$) {$-4$};
 \end{tikzpicture}\end{equation*}
 In particular, all objects in $\SCM(R)$ are  projective-injective. Hence,
 $\ul{\ul{\SCM}}(R) \cong 0$.
\end{ex}

Using cyclic quotient singularities, we see that the number of blocks in the decomposition of $\ul{\ul{\SCM}}(R)$ is not bounded above. Indeed, as an example consider the cyclic quotient singularities associated to Jung--Hirzebruch strings of the form

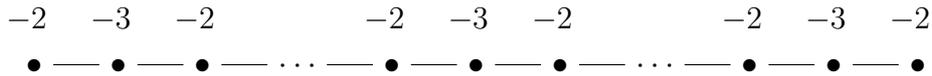
\begin{equation*}\begin{tikzpicture}[description/.style={fill=white,inner sep=2pt}]
    \matrix (n) [matrix of math nodes, row sep=3em,
                 column sep=1.5em, text height=1.5ex, text depth=0.25ex,
                 inner sep=0pt, nodes={inner xsep=0.3333em, inner
ysep=0.3333em}]
    {  
       \bullet & \bullet & \bullet & \cdots & \bullet & \bullet &\bullet & \cdots & \bullet & \bullet &\bullet  \\
    };
\path[-] (n-1-1) edge (n-1-2);
\path[-] (n-1-2) edge (n-1-3);
\path[-] (n-1-3) edge (n-1-4);
\path[-] (n-1-4) edge (n-1-5);
\path[-] (n-1-5) edge (n-1-6);
\path[-] (n-1-6) edge (n-1-7);
\path[-] (n-1-7) edge (n-1-8);
\path[-] (n-1-8) edge (n-1-9);
\path[-] (n-1-9) edge (n-1-10);
\path[-] (n-1-10) edge (n-1-11);

 \node at ($(n-1-1.north) + (-1mm, 3mm)$) {$-2$};
 \node at ($(n-1-2.north) + (-1mm, 3mm)$) {$-3$};
 \node at ($(n-1-3.north) + (-1mm, 3mm)$) {$-2$};
 \node at ($(n-1-5.north) + (-1mm, 3mm)$) {$-2$};
 \node at ($(n-1-6.north) + (-1mm, 3mm)$) {$-3$};
 \node at ($(n-1-7.north) + (-1mm, 3mm)$) {$-2$};
 \node at ($(n-1-9.north) + (-1mm, 3mm)$) {$-2$};
 \node at ($(n-1-10.north) + (-1mm, 3mm)$) {$-3$};
 \node at ($(n-1-11.north) + (-1mm, 3mm)$) {$-2$};
 \end{tikzpicture}\end{equation*}

\subsubsection{A tetrahedral group}
We follow Riemenschneider \cite{Riemenschneider77}.Ê  Let $\T_{13} \subseteq \GL(2, \C)$ be the subgroup generated by the following elements
\begin{align}
\begin{array}{c}
\psi_{4}=\begin{pmatrix} i & 0 \\ 0 & i^3 \end{pmatrix}, \quad
\phi_{26}=\begin{pmatrix} \epsilon_{26} & 0 \\ 0 & \epsilon_{26} \end{pmatrix}, \quad 
\tau=\begin{pmatrix} 0 & i \\ i & 0   \end{pmatrix}, \quad \eta=\frac{1}{\sqrt{2}}\begin{pmatrix} \epsilon_{8} & \epsilon_{8}^3 \\ \epsilon_{8} & \epsilon_{8}^7 \end{pmatrix}
\end{array},
\end{align}
where $i=\sqrt{-1}$ and $\epsilon_{n}$ denotes an $n$-th primitive root of unity. $\T_{13}$ is a group of order $13 \times 24=312$ and its dual intersection graph has the following form
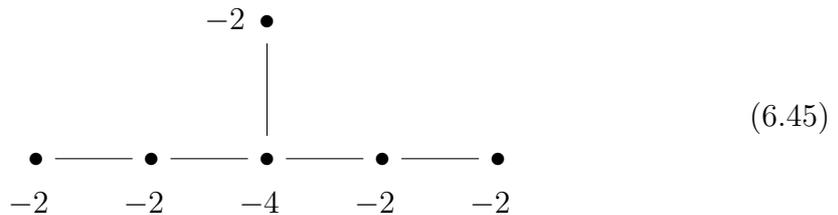
\begin{equation}\label{E:DualGraphT13}
\begin{array}{c}
\begin{tikzpicture}[description/.style={fill=white,inner sep=2pt}]
    \matrix (n) [matrix of math nodes, row sep=3em,
                 column sep=2.5em, text height=1.5ex, text depth=0.25ex,
                 inner sep=0pt, nodes={inner xsep=0.3333em, inner
ysep=0.3333em}]
    {  && \bullet \\
       \bullet & \bullet & \bullet &\bullet & \bullet  \\
    };
\path[-] (n-2-1) edge (n-2-2);
\path[-] (n-2-2) edge (n-2-3);
\path[-] (n-2-3) edge (n-2-4);
\path[-] (n-2-4) edge (n-2-5);
\path[-] (n-1-3) edge (n-2-3);

 \node at ($(n-2-1.south) + (-1mm, -3mm)$) {$-2$};
 \node at ($(n-2-2.south) + (-1mm, -3mm)$) {$-2$};
 \node at ($(n-2-3.south) + (-1mm, -3mm)$) {$-4$};
 \node at ($(n-2-4.south) + (-1mm, -3mm)$) {$-2$};
 \node at ($(n-2-5.south) + (-1mm, -3mm)$) {$-2$};
 \node at ($(n-1-3.west) + (-3mm, 0mm)$) {$-2$};
\end{tikzpicture}
\end{array}
\end{equation}

Removing the non-$(-2)$-curve, we obtain a disjoint union of diagrams of type $A_{1}$, $A_{2}$ and $A_{2}$. Hence Theorem \ref{C:StandardStableSCM} yields an equivalence of triangulated categories
\begin{align}
\ul{\ul{ \SCM}}(\C\llbracket x, y \rrbracket^{\T_{13}}) \cong \ul{\MCM}(R_{2, 1}) \oplus \ul{\MCM}(R_{3, 2}) \oplus \ul{\MCM}(R_{3, 2})
\end{align}

\subsection{Concluding remarks} \label{ss:Lift}

In the notations of Paragraph \ref{sss:Applicationtorational}, let $Y \stackrel{f^{\cs}} \longrightarrow X^\cs \stackrel{g^\cs} \longrightarrow \Spec(R)$ be a factorization of the minimal resolution of a rational surface singularity, with $\cs \subseteq I$. Let $\Lambda$ be the reconstruction algebra of $R$ and $e\in \Lambda$ be the idempotent corresponding to the identity endomorphism of the special Cohen--Macaulay $R$-module $N^\cs=R \oplus \bigoplus_{i \in I \setminus S} M_{i}$.

Recall that a sequence of triangulated categories and triangle functors $\cu \stackrel{F}\ra \ct \stackrel{G}\ra \cq$ is called \emph{exact}, if $G$ is a quotient functor with kernel $\cu$ and $F$ is the canonical inclusion. In this subsection, we extend triangle equivalences from Corollary \ref{C:ModifiedStableSCM} to exact sequences of triangulated categories. In particular, this yields triangle equivalences between the relative singularity categories studied in Sections \ref{S:Global} and \ref{S:Local}. 
\begin{prop} 
There exists  a commutative diagram of triangulated categories and functors such that the horizontal arrows are equivalences and the columns are exact.
\begin{align}\label{E:Commutative}
\begin{array}{c}
{\SelectTips{cm}{10}
\xy0;/r.4pc/:
(-10,20)*+{ \thick\left( \bigoplus_{i \in \cs} \co_{E_{i}}(-1)\right)}="A2",
(25,20)*+{\thick(\mod-\Lambda/\Lambda e\Lambda) }="A3",
(-10,10)*+{\frac{ \displaystyle \cd^b(\Coh Y)}{\displaystyle \thick\bigl(\co_{Y} \oplus \oplus_{i \in I \setminus \cs} \cm_{i}\bigr)}}="a2",
(25,10)*+{\frac{\displaystyle \cd^b(\mod-\Lambda)}{\displaystyle \thick(e \Lambda)}}="a3",
(-10,0)*+{\cd_{sg}(X^\cs)}="b2",
(25,0)*+{ \cd_{sg}(e\Lambda e)}="b3",
\ar"A2";"A3"^{\sim}
\ar@{^{(}->}"A2";"a2"
\ar@{^{(}->}"A3";"a3"
\ar@{->>}"a2";"b2"_{\Rfcs}
\ar@{->>}"a3";"b3"^{(-)e}
\ar"b2";"b3"^{\RHom_{X^\cs}(\cv_{\cs},-)}_{\sim}
\ar"a2";"a3"^(.55){\RHom_Y(\cv_\emptyset,-)}_{\sim}
\endxy}
\end{array}
\end{align}
By an abuse of notation, the induced triangle functors in the lower square are labelled by the inducing triangle functors from the diagram in Theorem \ref{main Db}.
\end{prop}
\begin{proof}
We start with the lower square. Since the corresponding diagram in Theorem \ref{main Db} commutes, it suffices to show that the induced functors above are well-defined. The equivalence $\RHom_{Y}(\cv_{\emptyset}, -)$ from Theorem \ref{main Db} maps $ \co_{Y} \oplus \bigoplus_{i \in I \setminus \cs} \cm_{i}$ to $e \Lambda$. Hence it induces an equivalence on the triangulated quotient categories. Since $\RHom_{X^\cs}(\cv_{\cs}, -)$ is an equivalence by Theorem \ref{main Db} and the subcategories $\Perf(X^\cs)$ respectively $\Perf(e\Lambda e)$ can be defined intrinsically, we get a well-defined equivalence on the bottom of diagram (\ref{E:Commutative}). The functor $(-)e$ on the right is a well-defined quotient functor by Proposition \ref{P:OneIdempotent}. Now, the functor on the left is a well-defined quotient functor by the commutativity of the diagram in Theorem \ref{main Db} and the considerations above.  

The category $\thick(\mod-\Lambda/\Lambda e\Lambda)$ is the kernel of the quotient functor $(-)e$, by~\ref{P:OneIdempotent}.  Since $R$ has isolated singularities, the algebra $\Lambda/\Lambda e \Lambda$ is finite dimensional (see \cite{Auslander84}) and so $\thick(\mod-\Lambda/\Lambda e \Lambda) = \thick\left(\bigoplus_{i \in \cs}S_{i} \right)$,
where $S_{i}$ denotes the simple $\Lambda$-module corresponding to the vertex $i$ in the quiver of $\Lambda$.  But under the derived equivalence $\RHom_{Y}(\cv_{\emptyset}, -)$, $S_i$ corresponds to $\co_{E_i}(-1)[1]$ (Theorem \ref{main Db}), so it follows that we can identify the subcategory $\thick(\mod-\Lambda/\Lambda e \Lambda)=\thick\left(\bigoplus_{i \in \cs}S_{i} \right)$ with $\thick\left( \bigoplus_{i \in \cs} \co_{E_{i}}(-1)\right)$, inducing the top half of the diagram.
\end{proof}
\begin{rem}
The functor $\RHom_{X^\cs}(\cv_{\cs}, -)$ identifies $\Perf(X^\cs)$ with $\Perf(e\Lambda e) \cong \thick(e\Lambda ) \subseteq \cd^b(\mod-\Lambda)$. Hence applying $\RHom_{Y}(\cv_{\emptyset}, -)$ yields a triangle equivalence
$\Perf(X^\cs) \cong \thick\bigl(\co_{Y} \oplus \oplus_{i \in I \setminus \cs} \cm_{i}\bigr)$. In particular, there is an equivalence
\begin{align}\label{E:RemarkRelSingCat}
\frac{\displaystyle \cd^b(\Coh Y)}{\displaystyle \Perf(X^\cs)} \stackrel{\sim}\longrightarrow \frac{\displaystyle \cd^b(\mod-\Lambda)}{\displaystyle \thick(e \Lambda)}. 
\end{align}
Actually, a careful analysis of the commutative diagram in Theorem \ref{main Db} shows that $\Perf(X^\cs) \cong \thick\bigl(\co_{Y} \oplus \oplus_{i \in I \setminus \cs} \cm_{i}\bigr)$ is obtained as a restriction of $\mathbf{L}(f^\cs)^{*}$.
\end{rem}

\medskip

If we contract only ($-2$)-curves (i.e.~if $\cs \subseteq \cc$ holds), then we know that $\cd_{sg}(X^\cs)$ splits into a direct sum of singularity categories of ADE--surface singularities (Corollary \ref{C:ModifiedStableSCM}).  In this case, it turns out that the diagram above admits an extension to the right and that in fact all the triangulated categories in our (extended) diagram split into blocks indexed by the isolated singularities of the Gorenstein scheme $X^\cs$.

Let us fix some notations. For a singular point $x \in \Sing X^\cs$ let $R_{x}=\widehat{\co}_{X^\cs, x}$, let $f_{x}\colon Y_{x} \ra \Spec\left(R_{x}\right)$ be the minimal resolution of singularities.

\begin{prop} Assume $\cs\subseteq\cc$.
There exists  a commutative diagram of triangulated categories and functors such that the horizontal arrows are equivalences and the columns are exact.
\begin{align}\label{E:extension-to-right}
\begin{array}{c}
{\SelectTips{cm}{10}
\xy0;/r.4pc/:
(-10,20)*+{\thick(\mod-\Lambda/\Lambda e\Lambda)}="A2",
(20,20)*+{ \bigoplus\limits_{x \in \Sing X^\cs} \ker(\mathbf{R} (f_{x})_{*}) }="A3",
(-10,10)*+{\frac{\displaystyle \cd^b(\mod-\Lambda)}{\displaystyle \thick(e \Lambda)}}="a2",
(20,10)*+{ \bigoplus\limits_{x \in \Sing X^\cs} \frac{\displaystyle \cd^b(\Coh Y_{x})}{\displaystyle \Perf(R_{x})} }="a3",
(-10,0)*+{\cd_{sg}(e\Lambda e)}="b2",
(20,0)*+{  \bigoplus\limits_{x \in \Sing X^\cs} \cd_{sg}(R_{x})}="b3",
\ar"A2";"A3"^{\sim}
\ar@{^{(}->}"A2";"a2"
\ar@{^{(}->}"A3";"a3"
\ar@{->>}"a3";"b3"^{\bigoplus_{x \in \Sing X^\cs}\mathbf{R} (f_{x})_{*}}
\ar@{->>}"a2";"b2"_{(-)e}
\ar"b2";"b3"^{\sim}
\ar"a2";"a3"^{\sim}
\endxy}
\end{array}
\end{align}
\end{prop}
\begin{proof}
We need some preparation. Note that by the derived McKay correspondence \cite{KapranovVasserot00, BKR}, there are derived equivalences 
$\cd^b(\Coh Y_{x}) \ra \cd^b(\mod-\Pi_{x})$, where $\Pi_{x}$ is the Auslander algebra of the Frobenius category of maximal Cohen--Macaulay $R_{x}$-modules $\MCM(R_{x})$.   Now we have two Frobenius categories $\ce_{1}:=\SCM_{N^\cs}(R)$ and $\ce_{2}:=\bigoplus_{x \in \Sing X^\cs} \MCM(R_{x})$, which clearly satisfy the conditions (FM1)--(FM4) in Subsection \ref{ss:Independance} and whose stable categories are Hom-finite, idempotent complete and whose stable Auslander algebras satisfy (A1)--(A3) in Subsection \ref{ss:Independance}. 
Furthermore, $\ce_1$ and $\ce_2$ are stably equivalent by Corollary \ref{C:ModifiedStableSCM}.

Now, by Theorem \ref{t:main-thm-2} there are triangle equivalences  
\begin{align}
\cd^b(\mod-\Lambda)/\thick(e \Lambda) \cong \mathsf{per}\bigl(\Lambda_{dg}(\underline{\ce}_{1})\bigr)\label{E:Tria1}\\
\bigoplus_{x \in \Sing X^\cs} \cd^b(\mod-\Pi_{x})/\Perf(R_{x}) \cong \mathsf{per}\bigl(\Lambda_{dg}(\underline{\ce}_{2})\bigr)\label{E:Tria2}
\end{align}
where by definition $\Lambda_{dg}(\underline{\ce}_{1})$ and $\Lambda_{dg}(\underline{\ce}_{2})$ are dg algebras that depend only on (the triangulated structure of) the stable Frobenius categories $\underline{\ce}_{1}$ and $\underline{\ce}_{2}$. Since $\ce_1$ and $\ce_2$ are stably equivalent, these two dg algebras are isomorphic.  Thus the combination of the equivalences \eqref{E:Tria1} and \eqref{E:Tria2} yields a triangle equivalence
\begin{align}\label{E:RelSgCat}
\frac{\displaystyle \cd^b(\mod-\Lambda)}{\displaystyle \thick(e\Lambda)}  \longrightarrow \bigoplus_{x \in \Sing X^\cs} \frac{\displaystyle \cd^b\bigl(\mod- \Pi_{x} \bigr)}{\displaystyle \Perf(R_{x})}
\end{align}
which, in conjunction with the derived McKay Correspondence,  yields the equivalence of triangulated categories in the middle of \eqref{E:extension-to-right}.

Furthermore, the functors $(-)e$ and $\bigoplus_{x \in \Sing X^\cs} \mathbf{R} (f_{x})_{*}$ are quotient functors with kernels $\thick(\mod-\Lambda/\Lambda e\Lambda)$ and $\bigoplus_{x \in \Sing X^\cs} \ker(\mathbf{R} (f_{x})_{*})$, respectively. These subcategories admit intrinsic descriptions (see Corollary \ref{C:Intrinsic}). Hence there is an induced equivalence, which renders the upper square commutative. This in turn induces an equivalence on the bottom of (\ref{E:extension-to-right}), such that the lower square commutes.
\end{proof}

\begin{rem}
Using \eqref{E:RemarkRelSingCat} together with an appropriate adaption of the techniques developed in Section \ref{S:Global} may yield a more geometric explanation for the block decomposition in \eqref{E:extension-to-right}. 
\end{rem}

\newpage

\section{Singularity categories of gentle algebras}\label{S:Gentle}

This section is based on our preprint \cite{Kalck12}. We give an explicit description of the singularity category $\cd_{sg}(\Lambda) \cong \cd^b(\Lambda-\mod)/\Perf(\Lambda)$, where $\Lambda$ is a finite dimensional \emph{gentle} algebra (see Definition \ref{D:Gentle} below). More precisely, the category $\cd_{sg}(\Lambda)$ decomposes into a finite product of $m$-cluster categories of type $\mathbb{A}_{1}$, i.e.~orbit categories of the form $\cd^b(k-\mod)/[m]$, where $m$ is a natural number. As an application, we recover a special case of a derived invariant for gentle algebras introduced by Avella-Alaminos \& Gei{\ss} \cite{GeissAvellaAlaminos}. Moreover, examples arising from singularity theory \cite{Burban} and from triangulations of unpunctured marked Riemann surfaces are discussed \cite{ABCP}.  

\subsection{Main result}
\noindent Let $k$ be an algebraically closed field. In this section, all modules are left modules.

\begin{defn}\label{D:Gentle}
A \emph{gentle algebra} is a finite dimensional algebra $\Lambda=kQ/I$ such that:
\begin{itemize}
\item[(G1)]  At any vertex, there are at most two incoming and at most two outgoing arrows.
\item[(G2)] The admissable two-sided ideal $I$ is generated by paths of length two.
\item[(G3)] For each arrow $\beta \in Q_{1}$,  there is at most one arrow $\alpha \in Q_{1}$ such that 
$\alpha\beta \in I$ and at most one arrow $\gamma \in Q_{1}$ such that 
$ \beta \gamma \in I.$
\item[(G4)] For each arrow $\beta\in Q_{1}$, there is at most one arrow $\alpha \in Q_{1}$ such that 
$\alpha\beta \notin I$ and at most one arrow $\gamma \in Q_{1}$ such that 
$\beta \gamma \notin I.$
\end{itemize}
\end{defn}

For a gentle algebra $\Lambda=kQ/I$, we denote by $\cc(\Lambda)$ the set of equivalence classes of repetition free cyclic paths $\alpha_{1}\ldots \alpha_{n}$ in $Q$ (with respect to cyclic permutation) such that $\alpha_{i}\alpha_{i+1} \in I$ for all $i$, where we set $n+1=1$. 

\begin{ex} \label{Ex:Illustrative} An example of a gentle algebra $\Lambda=kQ/I$  is given by the quiver $Q$
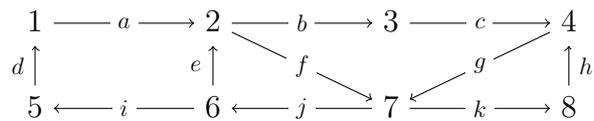
\[
\begin{tikzpicture}[description/.style={fill=white,inner sep=2pt}]
    \matrix (n) [matrix of math nodes, row sep=1.3em,
                 column sep=2.25em, text height=1.5ex, text depth=0.25ex,
                 inner sep=0pt, nodes={inner xsep=0.3333em, inner
ysep=0.3333em}]
    {  1&&2 && 3 && 4 \\
       5&& 6 &&7&& 8 \\
    };
\draw[->] (n-1-1) edge node[fill=white, scale=0.75, yshift=0mm] [midway] {$a$} (n-1-3); 
\draw[->] (n-1-3) edge node[fill=white, scale=0.75, yshift=0mm] [midway] {$b$} (n-1-5);
\draw[->] (n-1-5) edge node[fill=white, scale=0.75, yshift=0mm] [midway] {$c$} (n-1-7);

\draw[->] (n-2-1) edge node[fill=white, scale=0.75, xshift=-3mm] [midway] {$d$} (n-1-1);
\draw[->] (n-2-3) edge node[fill=white, scale=0.75, xshift=-3mm] [midway] {$e$} (n-1-3);

\draw[->] (n-1-3) edge node[fill=white, scale=0.75, yshift=0mm] [midway] {$f$} (n-2-5);
\draw[->] (n-1-7) edge node[fill=white, scale=0.75, yshift=0mm] [midway] {$g$} (n-2-5);

\draw[->] (n-2-7) edge node[fill=white, scale=0.75, xshift=3mm] [midway] {$h$} (n-1-7);

\draw[<-] (n-2-3) edge node[fill=white, scale=0.75, yshift=0mm] [midway] {$j$} (n-2-5);
\draw[<-] (n-2-7) edge node[fill=white, scale=0.75, yshift=0mm] [midway] {$k$} (n-2-5);

\draw[<-] (n-2-1) edge node[fill=white, scale=0.75, yshift=0mm] [midway] {$i$} (n-2-3);   

\end{tikzpicture}
\]
with two-sided ideal $I$ generated by the relations $ba$, $fe$, $jf$, $ej$, $kg$, $hk$ and $gh$.
Then $\cc(\Lambda)=\{jfe, kgh\}$.
\end{ex}

Let $i \in Q_{0}$ be a vertex lying on a cycle $c \in \cc(\Lambda)$ and $P_{i}=\Lambda e_{i}$ be the corresponding indecomposable projective $\Lambda$-module. The radical of $P_{i}$ has at most two direct summands $R_{1}$ and $R_{2}$. 
We consider the radical embedding, where one of $R_{1}$ and $R_{2}$ may be zero and exactly one of the arrows $\iota_{1}$ and  $\iota_{2}$ lies on the cycle $c$ 
\begin{align}\label{E:Radical}
R_{1} \oplus R_{2}  \xrightarrow{\begin{pmatrix} \cdot \iota_{1} & \cdot \iota_{2}  \end{pmatrix}} P_{i}.
\end{align}
 Let $R(c)_{i}:=R_{j}$ be the corresponding direct summand of the radical.
\begin{ex} \label{Ex:DefRc}
In Example \ref{Ex:Illustrative}, we consider the vertex $6$ lying on the cycle $c=jfe$. The indecomposable projective $\Lambda$-module $P_{6}$ and its radical summand $R(c)_{6}$ have the form
\[
\begin{tikzpicture}[description/.style={fill=white,inner sep=2pt}]
    \matrix (n) [matrix of math nodes, row sep=1em,
                 column sep=2.25em, text height=1.5ex, text depth=0.25ex,
                 inner sep=0pt, nodes={inner xsep=0.3333em, inner
ysep=0.3333em}]
    {  &2&3 &4& 7 &6& 5&1 &2&7&8 \\
        6&& & &   \\
       &5&1 &2&7&8 \\
    };
\draw[->] (n-2-1) edge node[fill=white, scale=0.75, yshift=0mm] [midway] {$e$} (n-1-2); 
\draw[->] (n-2-1) edge node[fill=white, scale=0.75, yshift=0mm] [midway] {$i$} (n-3-2); 
   
\draw[->] (n-1-2) edge node[fill=white, scale=0.75, yshift=0mm] [midway] {$b$} (n-1-3); 
\draw[->] (n-1-3) edge node[fill=white, scale=0.75, yshift=0mm] [midway] {$c$} (n-1-4); 
\draw[->] (n-1-4) edge node[fill=white, scale=0.75, yshift=0mm] [midway] {$g$} (n-1-5); 
\draw[->] (n-1-5) edge node[fill=white, scale=0.75, yshift=0mm] [midway] {$j$} (n-1-6);
\draw[->] (n-1-6) edge node[fill=white, scale=0.75, yshift=0mm] [midway] {$i$} (n-1-7); 

\draw[->] (n-1-7) edge node[fill=white, scale=0.75, yshift=0mm] [midway] {$d$} (n-1-8); 
\draw[->] (n-1-8) edge node[fill=white, scale=0.75, yshift=0mm] [midway] {$a$} (n-1-9); 
\draw[->] (n-1-9) edge node[fill=white, scale=0.75, yshift=0mm] [midway] {$f$} (n-1-10); 
\draw[->] (n-1-10) edge node[fill=white, scale=0.75, yshift=0mm] [midway] {$k$} (n-1-11);

\draw[->] (n-3-2) edge node[fill=white, scale=0.75, yshift=0mm] [midway] {$d$} (n-3-3); 
\draw[->] (n-3-3) edge node[fill=white, scale=0.75, yshift=0mm] [midway] {$a$} (n-3-4); 
\draw[->] (n-3-4) edge node[fill=white, scale=0.75, yshift=0mm] [midway] {$f$} (n-3-5); 
\draw[->] (n-3-5) edge node[fill=white, scale=0.75, yshift=0mm] [midway] {$k$} (n-3-6); 

\draw[decoration={brace, amplitude=2.5mm}, decorate] ($(n-1-2.west) +(0, 3mm)$) -- node[ scale=0.75, yshift=6.5mm] [midway] {$R(c)_{6}$} ($(n-1-11.east)+(0, 3mm)$);

\end{tikzpicture}
\]
Let us look at another projective $\Lambda$-module from Example \ref{Ex:Illustrative}. Let $c'=kgh \in \cc(\Lambda)$. Then the projective $\Lambda$-module $P_{7}$ and its radical summands $R(c)_{7}$ and $R(c')_{7}$ have the following form
\[
\begin{tikzpicture}[description/.style={fill=white,inner sep=2pt}]
    \matrix (n) [matrix of math nodes, row sep=1em,
                 column sep=2.25em, text height=1.5ex, text depth=0.25ex,
                 inner sep=0pt, nodes={inner xsep=0.3333em, inner
ysep=0.3333em}]
    {   &6& 5&1 &2&7&8 \\
        7&& & &   \\
       &8 \\
    };
\draw[->] (n-2-1) edge node[fill=white, scale=0.75, yshift=0mm] [midway] {$j$} (n-1-2); 
\draw[->] (n-2-1) edge node[fill=white, scale=0.75, yshift=0mm] [midway] {$k$} (n-3-2); 
   
\draw[->] (n-1-2) edge node[fill=white, scale=0.75, yshift=0mm] [midway] {$i$} (n-1-3);  
\draw[->] (n-1-3) edge node[fill=white, scale=0.75, yshift=0mm] [midway] {$d$} (n-1-4); 
\draw[->] (n-1-4) edge node[fill=white, scale=0.75, yshift=0mm] [midway] {$a$} (n-1-5); 
\draw[->] (n-1-5) edge node[fill=white, scale=0.75, yshift=0mm] [midway] {$f$} (n-1-6); 
\draw[->] (n-1-6) edge node[fill=white, scale=0.75, yshift=0mm] [midway] {$k$} (n-1-7); 

\draw[decoration={brace, amplitude=2.5mm}, decorate] ($(n-1-2.west) +(0, 3mm)$) -- node[ scale=0.75, yshift=6.5mm] [midway] {$R(c)_{7}$} ($(n-1-7.east)+(0, 3mm)$);

\draw[decoration={brace, amplitude=1.2mm}, decorate] ($(n-3-2.east) +(0, -2mm)$) -- node[ scale=0.75, yshift=-5mm] [midway] {$R(c')_{7}$} ($(n-3-2.west)+(0, -2mm)$);

\end{tikzpicture}
\]
Note, that there is a non-zero morphism $R(c')_{7} \ra R(c)_{7}$. However, this morphism factors over the projective $P_{7}$ and thus vanishes in the stable category. 
\end{ex}

Gei{\ss} \& Reiten \cite{GeissReiten} have shown that gentle algebras are \emph{Iwanaga--Gorenstein rings}, see Definition \ref{D:IwanagaGorensteinD}. Combining this result with Buchweitz Theorem \ref{T:Buchweitz}, reduces the  description of the singularity category to the description of the stable category of Gorenstein projective modules.

The following proposition is the main result of this section\footnote{We would like to thank Jan Schr\"oer for bringing this question to our attention}:

\begin{prop}\label{P:Main}
Let $\Lambda=kQ/I$ be a finite dimensional gentle algebra. 
\begin{itemize} 
\item[(a)] The indecomposable Gorenstein projective modules\footnote{ This notion was introduced in Definition \ref{D:IwanagaGorensteinD}.} are given by 
\begin{align}
\ind \GP(\Lambda) = \ind \proj-\Lambda \cup \{ R(c)_{s(\alpha_{1})}, \ldots, R(c)_{s(\alpha_{n})} \big| c= \alpha_{1}\ldots\alpha_{n} \in \cc(\Lambda) \}.
\end{align}
\item[(b)] There is an equivalence of triangulated categories 
\begin{align}\label{E:Decomposition}
\cd_{sg}(\Lambda) \cong \prod_{c \in \cc(\Lambda)} \frac{\displaystyle \cd^b(k-\mod)}{\displaystyle [l(c)]},
\end{align}
where $l(\alpha_{1}\ldots\alpha_{n})=n$ and $\cd^b(k)/[l(c)]$ denotes the triangulated orbit category, \cite{Orbit}. This category is also known as the $(l(c)-1)$-\emph{cluster category} of type $\mathbb{A}_{1}$, \cite{Thomas}.
\end{itemize}
\end{prop}

We prove this result in Subsection \ref{S:Proof} below.

\begin{rem}
As illustrated in Example \ref{Ex:Illustrative} there may be non-trivial morphisms between different indecomposable Gorenstein-projective $\Lambda$-modules. However, the proof of Proposition \ref{P:Main} shows that all of them factor over some projective module. Hence these morphisms vanish in the stable category.
\end{rem}

Since derived equivalences induce triangle equivalences between the corresponding singularity categories, Proposition \ref{P:Main} has the following consequence.

\begin{cor} Let $\Lambda$ and $\Lambda'$ be gentle algebras. If there is a triangle equivalence $\cd^b(\Lambda-\mod) \cong \cd^b(\Lambda'-\mod)$, then there is a bijection of sets
\begin{align}
f \colon \cc(\Lambda) \stackrel{\sim}\longrightarrow \cc(\Lambda'),
\end{align}
such that $l(c)=l(f(c))$ for all $c \in \cc(\Lambda)$.
\end{cor}

\begin{rem}
This is a special case of a derived invariant for gentle algebras introduced by Avella-Alaminos \& Gei{\ss} \cite{GeissAvellaAlaminos}.
\end{rem}

\begin{ex} The following two gentle algebras $\Lambda=kQ/I$ and $\Lambda'=kQ'/I'$ are \emph{not} derived equivalent:
\begin{align*}
\begin{xy}\SelectTips{cm}{}
\xymatrix
{Q=& 1 \ar[rr]^a  & &  2  \ar@(rd, ru)[]_{b} 
}\end{xy},  \qquad  I=(b^2) 
\end{align*}
and 
\begin{align*}
\begin{xy}\SelectTips{cm}{}
\xymatrix
{Q'=&1 \ar@/^/[rr]^{ a }  & &  2 \ar@/^/[ll]^{ b } 
}\end{xy},  \qquad  I'=(ab, ba).
\end{align*}
\end{ex}

\subsection{Examples} \label{S:Examples}
The following geometric example was pointed out by Igor Burban.
\begin{ex}
Let $\mathbb{X}_{n}$ be a chain of $n$ projective lines  
\begin{align}\label{E:ChainOfProjLines}
\begin{array}{c}
\begin{tikzpicture} 
\draw (0,0,0) to [bend left=25]  (2,0,0);
\node[scale=0.75] at (1, 0.5, 0) {$C_{1}$};
\node[scale=0.75] at (1.75, -0.25, 0) {$s_{1}$};
\draw (1.5,0,0) to [bend left=25]  (3.5,0,0);
\node[scale=0.75] at (2.5, 0.5, 0) {$C_{2}$};
\node at (4, 0, 0) {$\cdots$}; 
\draw (4.5,0,0) to [bend left=25]  (6.5,0,0);
\node[scale=0.75] at (5.5, 0.5, 0) {$C_{n-2}$};
\node[scale=0.75] at (6.25, -0.25, 0) {$s_{n-2}$};
\draw (6,0,0) to [bend left=25]  (8,0,0);
\node[scale=0.75] at (7, 0.5, 0) {$C_{n-1}$};
\node[scale=0.75] at (7.75, -0.25, 0) {$s_{n-1}$};
\draw (7.5,0,0) to [bend left=25]  (9.5,0,0);
\node[scale=0.75] at (8.5, 0.5, 0) {$C_{n}$};
\end{tikzpicture}
\end{array} 
\end{align}
Using Buchweitz' equivalence \eqref{E:Buchweitz}  and Orlov's localization theorem \cite{Orlov11}, it is well-known that there is an equivalence of triangulated categories 
\begin{align}\label{E:SingChain}
\left( \cd_{sg}(\mathbb{X}_{n}) \right)^{\omega}:= \left(\frac{\displaystyle \cd^b(\Coh \mathbb{X}_{n})}{\displaystyle \Perf(\XX_{n})}\right)^\omega \cong \bigoplus_{i=1}^{n-1} \ul{\MCM}(O_{nd}) \cong  \bigoplus_{i=1}^{n-1} \frac{\displaystyle \cd^b(k-\mod)}{\displaystyle [2]},
\end{align}
where $(-)^{\omega}$ denotes the idempotent completion (see Subsection \ref{Sub:Idemp}) and $\ul{\MCM}(O_{nd})$ denotes the stable category of maximal Cohen--Macaulay modules over the nodal singularity $O_{nd}=k\llbracket x, y \rrbracket/(xy)$. In particular, there is a fully faithful triangle functor
\begin{align}\label{E:WithoutCompl}
\cd_{sg}(\mathbb{X}_{n}) \ra \bigoplus_{i=1}^{n-1} \ul{\MCM}(O_{nd}),
\end{align}
 which is induced by
\begin{align}\label{E:LocalizationChain}
\cd^b(\Coh \mathbb{X}_{n}) \ni \cf  \mapsto (\widehat{\cf}_{s_{1}}, \ldots,\widehat{\cf}_{s_{n-1}}) \in \bigoplus _{i=1}^{n-1} O_{nd}-\mod,
\end{align}
where $s_{1}, \ldots, s_{n-1}$ denote the singular points of $\mathbb{X}_{n}$.
For $1\leq l \leq m \leq n$, let $\co_{[l,m]}$ be the structure sheaf of the subvariety $\bigcup_{k=l}^m C_{k} \subseteq \mathbb{X}_{n}$. Here, the $C_{i}$ denote the irreducible components of $\mathbb{X}_{n}$ as shown in \eqref{E:ChainOfProjLines}. Then \eqref{E:LocalizationChain} maps $\co_{[1, i]}$ to $(O_{nd}, \ldots, O_{nd}, k\llbracket x \rrbracket, 0, \ldots, 0)$ and $\co_{[j, n]}$ to $( 0, \ldots, 0, k\llbracket y \rrbracket, O_{nd}, \ldots, O_{nd})$, where $k\llbracket x \rrbracket$  and  $k\llbracket y \rrbracket$ are located in the $i$-th and $j$-th place, respectively. In particular, \eqref{E:WithoutCompl} is essentially surjective. Therefore, the singularity category $\cd_{sg}(\mathbb{X}_{n})$ is idempotent complete.

We explain an alternative approach to obtain the equivalence \eqref{E:WithoutCompl}, which uses and confirms Proposition \ref{P:Main}.
Burban \cite{Burban} showed  that $\cd^b(\Coh \XX_{n})$ has a tilting bundle with endomorphism algebra $\Lambda_{n}$
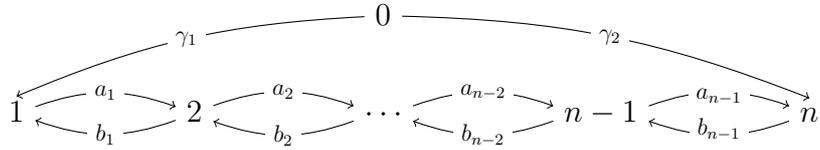
\[
\begin{tikzpicture}[description/.style={fill=white,inner sep=2pt}]
    \matrix (n) [matrix of math nodes, row sep=1.7em,
                 column sep=2.25em, text height=1.5ex, text depth=0.25ex,
                 inner sep=0pt, nodes={inner xsep=0.3333em, inner
ysep=0.3333em}]
    {  &&&& 0 \\
       1 && 2 && \cdots && n-1 && n \\
    };
\draw[->] (n-1-5) edge  [bend right=10] node[fill=white, scale=0.75, yshift=0mm] [midway] {$\gamma_{1}$} ($(n-2-1.north)+(0, -0.75mm)$); 
  
\draw[->] (n-1-5)  edge [bend left=10]  node[fill=white, scale=0.75, yshift=0mm] [midway] {$\gamma_{2}$} ($(n-2-9.north)+(0, -1mm)$); 

\draw[->] (n-2-1) edge [bend left=20] node[fill=white, scale=0.75, yshift=0mm] [midway] {$a_{1}$} (n-2-3); 
\draw[->] (n-2-3) edge [bend left=20] node[fill=white, scale=0.75, yshift=0mm] [midway] {$a_{2}$} ($(n-2-5.west)+(0, 1mm)$); 
\path[->]  ($(n-2-5.east)+(0, 1mm)$) edge [bend left=20] node[fill=white, scale=0.75, yshift=0mm] [midway] {$a_{n-2}$} ($(n-2-7.west)+(0, 1mm)$); 
\path[->] ($(n-2-7.east)+(0, 1mm)$) edge [bend left=20] node[fill=white, scale=0.75, yshift=-0.75mm] [midway] {$a_{n-1}$}  ($(n-2-9.west)+(0, 1mm)$); 

\draw[<-] (n-2-1) edge [bend right=20] node[fill=white, scale=0.75, yshift=0mm] [midway] {$b_{1}$} (n-2-3); 
\draw[<-] (n-2-3) edge [bend right=20] node[fill=white, scale=0.75, yshift=0mm] [midway] {$b_{2}$} ($(n-2-5.west)+(0, -1mm)$); 
\path[<-]  ($(n-2-5.east)+(0, -1mm)$) edge [bend right=20] node[fill=white, scale=0.75, yshift=0mm] [midway] {$b_{n-2}$} ($(n-2-7.west)+(0, -1mm)$); 
\path[<-] ($(n-2-7.east)+(0, -1mm)$) edge [bend right=20] node[fill=white, scale=0.75, yshift=1mm] [midway] {$b_{n-1}$} ($(n-2-9.west)+(0, -1mm)$); 

\end{tikzpicture}
\]
bounded by the relations $a_{i}b_{i}=0=b_{i}a_{i}$ for all $1 \leq i \leq n-1$. Hence we have a triangle equivalence $\cd^b(\Coh \XX_{n}) \ra \cd^b(\Lambda_{n}-\mod)$ inducing a triangle equivalence
\begin{align}
\cd_{sg}(\XX_{n}) \xrightarrow{\sim} \cd_{sg}(\Lambda_{n}).
\end{align}
Since $\Lambda_{n}$ is a gentle algebra, we can apply Proposition \ref{P:Main}. \begin{align*}\cc(\Lambda_{n})=\{c_{1}=a_{1}b_{1}, \ldots, c_{n-1}=a_{n-1}b_{n-1} \}\end{align*} consists of $n-1$ cycles of length two. Therefore $\cd_{sg}(\Lambda)$ is equivalent to the right hand side of \eqref{E:SingChain}.

We conclude this example by giving the indecomposable Gorenstein projective $\Lambda$-modules. For each $i\in \{1, \ldots, n-1\}$, we have two indecomposable GPs
\begin{align}
&R(c_{i})_{i}:=i+1 \xrightarrow{a_{i+1}} i+2 \xrightarrow{a_{i+2}} \cdots  \xrightarrow{a_{n-1}} n \\
&R(c_{i})_{i+1}:=i \xrightarrow{b_{i-1}} i-1 \xrightarrow{b_{i-2}} \cdots  \xrightarrow{b_{1}} 1
\end{align}
\end{ex}

Assem, Br\"ustle, Charbonneau-Jodoin \& Plamondon \cite{ABCP} studied a class of gentle algebras $A(S, \Gamma)$ arising from a triangulation $\Gamma$ of a marked Riemann surface $S=(S, M)$ without punctures. In particular, they show that the `inner triangles' of $\Gamma$ are in bijection with the elements of $\cc(A(S, \Gamma))$, which are all of length three. This has the following consequence.

\begin{cor}\label{C:Jacobian}
In the notation above, the number of direct factors in the decomposition \eqref{E:Decomposition} Êof the singularity category $\cd_{sg}(A(S, \Gamma))$ equals the number of inner triangles of $\Gamma$.  
\end{cor}

\begin{ex}
A prototypical case is the hexagon $S$ with six marked points on the boundary. We consider the following triangulation $\Gamma$ with exactly one inner triangle.
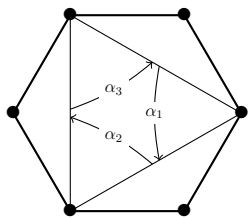
\[
\begin{tikzpicture}
\foreach \l in {1.5}
\foreach \ha in {0.52}
\foreach \hb in {0.48}
\foreach \h in {0.5}
{\foreach \a in {0, -120} \draw[thick=40mm] (0, 0) -- +(\a:\l); 
\draw[thick=40mm] (\l, 0) -- +(300:\l); 
\draw[thick=40mm] (\l+ \l*cos{60}, -\l*sin{60} ) -- +(-120:\l); 
\draw[thick=40mm] (\l, -2*\l*sin{60} ) -- +(180:\l); 
\draw[thick=40mm] (0, -2*\l*sin{60} ) -- +(120:\l); 

\draw (0,0) -- (\l + \l*cos{60}, -\l*sin{60});
\draw (0,-2*\l*sin{60}) -- (\l + \l*cos{60}, -\l*sin{60});
\draw (0,-2*\l*sin{60}) -- (0,0);

\path[->] (\ha*\l +\ha*\l*cos{60}, -\ha*\l*sin{60}) edge [bend right=10] node[fill=white, scale=0.55] [midway] {$\alpha_{1}$} (-\hb*\l - \hb*\l*cos{60} + \l +\l*cos{60} , -\hb*\l*sin{60} -\l*sin{60});
\path[->] (-\ha*\l - \ha*\l*cos{60} + \l +\l*cos{60} , -\ha*\l*sin{60} -\l*sin{60}) edge [bend right=10] node[fill=white, scale=0.55] [midway] {$\alpha_{2}$} (0, -\ha*2*\l*sin{60});
\path[<-] (\hb*\l +\hb*\l*cos{60}, -\hb*\l*sin{60})  edge [bend left=10] node[fill=white, scale=0.55] [midway] {$\alpha_{3}$}(0, -\hb*2*\l*sin{60});

\node at (\l, 0) {$\bullet$};
\node at (0, 0) {$\bullet$};
\node at (\l, -2*\l*sin{60}) {$\bullet$};
\node at (0, -2*\l*sin{60}) {$\bullet$};
\node at (\l + \l*cos{60}, -\l*sin{60}) {$\bullet$};
\node at ( -\l*cos{60}, -\l*sin{60}) {$\bullet$};
}
\end{tikzpicture}
\] 
The corresponding gentle algebra $A(S, \Gamma)$ is a $3$-cycle with relations $\alpha_{2}\alpha_{1}=0$, $\alpha_{3}\alpha_{2}=0$ and $\alpha_{1}\alpha_{3}=0$. This is a selfinjective algebra. Hence the singularity category $\cd_{sg}(A(S, \Gamma))$ is triangle equivalent to the stable module category $A(S, \Gamma)-\ul{\mod}$, by \eqref{E:Buchweitz}. This is in accordance with Proposition \ref{P:Main} and Corollary \ref{C:Jacobian}.
\end{ex}

\begin{rem}
More generally, the algebras arising as Jacobian algebras from ideal triangulations of Riemann surfaces \emph{with} punctures are usually of infinite global dimension. It would be interesting to study their singularity categories and relate them to the triangulation.
\end{rem}

We conclude this subsection with a more complicated example.
\begin{ex} Let $\Lambda$ be the algebra given by the following quiver
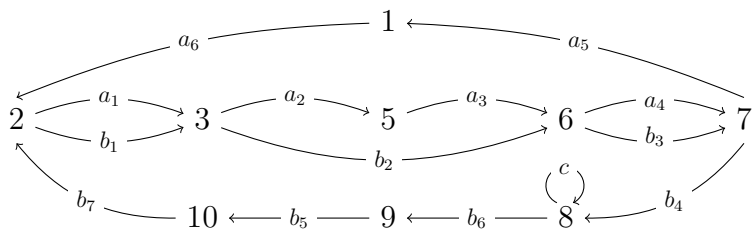
\[
\begin{tikzpicture}[description/.style={fill=white,inner sep=2pt}]
    \matrix (n) [matrix of math nodes, row sep=1.7em,
                 column sep=2.25em, text height=1.5ex, text depth=0.25ex,
                 inner sep=0pt, nodes={inner xsep=0.3333em, inner
ysep=0.3333em}]
    {  &&&& 1 \\
       2 && 3 && 5 && 6 && 7 \\
       &&10&&9&&8 \\
    };
\draw[->] (n-1-5) edge  [bend right=10] node[fill=white, scale=0.75, yshift=0mm] [midway] {$a_{6}$} ($(n-2-1.north)+(0, -0mm)$); 
  
\draw[<-] (n-1-5)  edge [bend left=10]  node[fill=white, scale=0.75, yshift=0mm] [midway] {$a_{5}$} ($(n-2-9.north)+(0, -0mm)$); 

\draw[->] (n-2-1) edge [bend left=20] node[fill=white, scale=0.75, yshift=0mm] [midway] {$a_{1}$} (n-2-3); 
\draw[->] (n-2-3) edge [bend left=20] node[fill=white, scale=0.75, yshift=0mm] [midway] {$a_{2}$} ($(n-2-5.west)+(0, 1mm)$); 
\path[->]  ($(n-2-5.east)+(0, 1mm)$) edge [bend left=20] node[fill=white, scale=0.75, yshift=0mm] [midway] {$a_{3}$} ($(n-2-7.west)+(0, 1mm)$); 
\path[->] ($(n-2-7.east)+(0, 1mm)$) edge [bend left=20] node[fill=white, scale=0.75, yshift=-0.75mm] [midway] {$a_{4}$}  ($(n-2-9.west)+(0, 1mm)$); 

\path[->] (n-2-1) edge [bend right=20] node[fill=white, scale=0.75, yshift=0mm] [midway] {$b_{1}$} (n-2-3); 
\path[->] (n-2-3) edge [bend right=20] node[fill=white, scale=0.75, yshift=0mm] [midway] {$b_{2}$} ($(n-2-7.west)+(0, -1mm)$); 
\path[->] ($(n-2-7.east)+(0, -1mm)$) edge [bend right=20] node[fill=white, scale=0.75, yshift=1mm] [midway] {$b_{3}$} ($(n-2-9.west)+(0, -1mm)$); 

\path[->] (n-2-9.south) edge [bend left=26] node[fill=white, scale=0.75, yshift=0mm] [midway] {$b_{4}$} (n-3-7.east);
\path[->] (n-3-3.west) edge [bend left=26] node[fill=white, scale=0.75, yshift=0mm] [midway] {$b_{7}$} (n-2-1.south); 

\path[<-] (n-3-3) edge node[fill=white, scale=0.75, yshift=0mm] [midway] {$b_{5}$} (n-3-5); 
\path[<-] (n-3-5) edge node[fill=white, scale=0.75, yshift=0mm] [midway] {$b_{6}$} (n-3-7);

\draw[->] ($(n-3-7.north) +
    (-1.2mm,-1.0mm)$) arc  (245:-65:2.5mm);
\node[fill=white, scale=0.75, yshift=0mm] at  ($(n-3-7.north) +
    (-0.3mm,3.5mm)$) {$c$};

\end{tikzpicture}
\]
bounded by the relations $a_{i+1}a_{i}=0$ for $i=1, \ldots, 6$, $b_{j+1}b_{j}=0$ for $j=1, \ldots, 7$ and $c^2=0$, where we set $a_{7}:=a_{1}$, $b_{8}:=b_{1}$ and so forth. This is a gentle algebra and $\cc(\Lambda)=\{a_{6} \cdot \ldots \cdot a_{1}, b_{7}\cdot \ldots \cdot b_{1}, c\}$. Proposition \ref{P:Main} yields a triangle equivalence
\begin{align}
\cd_{sg}(\Lambda) \cong \frac{\displaystyle \cd^b(k-\mod)}{\displaystyle [6]} \oplus \frac{\displaystyle \cd^b(k-\mod)}{\displaystyle [7]} \oplus \frac{\displaystyle \cd^b(k-\mod)}{\displaystyle [1]}
\end{align}
\end{ex}

\subsection{Proof}\label{S:Proof}
 We need an auxiliary result about submodules of projective modules over gentle algebras. For this, we use the classification of the indecomposable modules over gentle algebras: indecomposable modules are either \emph{string} or \emph{band} modules and are given by certain words in the alphabet $\{\alpha, \alpha^{-1} \big| \alpha \in Q \}$, we refer to Butler \& Ringel \cite{ButlerRingel} for a detailed account. 
 
Let $\Lambda=kQ/I$ be a gentle algebra. The indecomposable projective $\Lambda$-modules are either simple or of the following form:
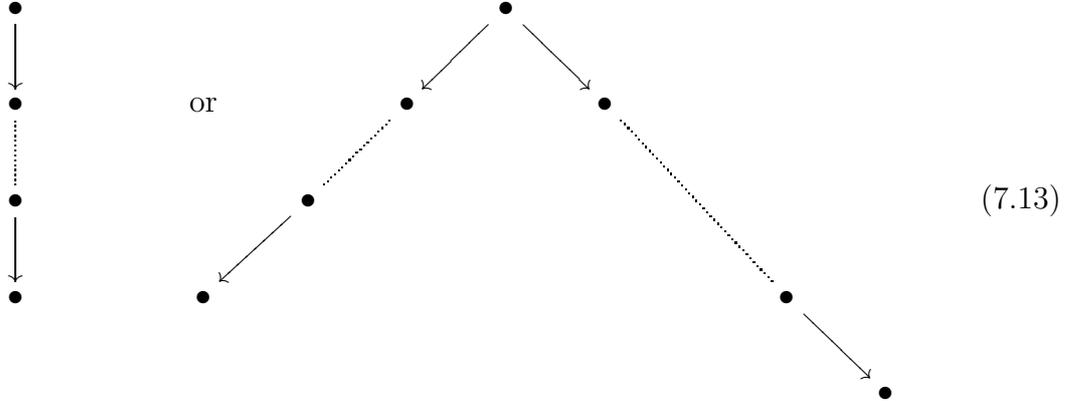
\begin{align}\label{E:IndProjGentle}
\begin{array}{c}
\begin{xy}
\SelectTips{cm}{}
\xymatrix{
\bullet \ar[d] &&&&& \bullet \ar[rd] \ar[ld]  \\
\bullet \ar@{..}[d] &&\text{or} && \bullet  \ar@{..}[ld]&& \bullet \ar@{..}[rrdd] \\
\bullet \ar[d] &&&\bullet \ar[ld] &&&&  \\
\bullet && \bullet &&&&&& \bullet \ar[rd] \\
&&&& &&&& & \bullet
}
\end{xy}
\end{array}
\end{align} 

\begin{lem}\label{L:TechnicalV} Let $M=M(w)$ be an indecomposable $\Lambda$-module, such that the corresponding word $w$ contains 
\begin{align}\label{E:PictureV}
\alpha^{-1} \beta=
\begin{array}{c}
\begin{xy}
\SelectTips{cm}{}
\xymatrix{
 x \ar[rd]_{\alpha}  && z \ar[ld]^{\beta} \\
& y
}
\end{xy}
\end{array}
\end{align} 
with $\alpha \neq \beta$ as a subword. Then there is no projective $\Lambda$-module $P$ such that $M \subseteq P$ is a submodule.
\end{lem}
\begin{rem}\label{R:TwoCases}
In the picture \eqref{E:PictureV}, the letters $x, y, z$ represent basis vectors of the module $M$. We do not  exclude the case $x=z$. For example, the indecomposable injective module $I_{2}$ over the Kronecker quiver $\begin{xy}\SelectTips{cm}{10}\xymatrix{1 \ar@/^/[r] \ar@/_/[r] &  2}\end{xy}$ is of the form \eqref{E:PictureV}, with pairwise different basis vectors $x, y, z$. On the other hand, the indecomposable modules
\[\begin{xy}\SelectTips{cm}{10}\xymatrix{k \ar@/^/[rr]^1 \ar@/_/[rr]_{\lambda} &&  k}\end{xy}\]
with $\lambda \in k^*$ contain a subword $\alpha^{-1} \beta$ but we have to identify $x$ and $z$ in \eqref{E:PictureV}.
 \end{rem}
 \begin{proof}
 Assume $M$ is a submodule of a projective module $P$. We distinguish the two cases discussed in Remark \ref{R:TwoCases}.  If $x \neq z$, then we must have basis vectors $x' \neq z'$ in $P$ such that $\alpha x' = \beta z' \neq 0$. In view of the form of the indecomposable projectives \eqref{E:IndProjGentle}, such basis vectors cannot exist. Contradiction. The second case $x=z$ is treated in a similar way. 
 \end{proof}

Throughout the proof, we use the properties (GP1)--(GP2) of Gorenstein projective modules over Iwanaga--Gorenstein rings, which were proved in Proposition \ref{P:IwanagaFrobenius}.

\subsubsection{Proof of part (a)}
Let $c \in \cc(\Lambda)$ be a cycle, which we label as follows $1 \xrightarrow{\alpha_{1}} 2 \xrightarrow{\alpha_{2}}  \cdots \xrightarrow{\alpha_{n-1}} n \xrightarrow{\alpha_{n}} 1$. Then there are short exact sequences 
\begin{align}\label{E:Shift}
0 \ra  R(c)_{i} \ra P_{i} \ra R(c)_{i-1} \ra 0,
\end{align}
for all $i=1, \ldots, n$. In particular, for every $n \geq 0$, $R(c)_{i}$ may be written as a $n$th-syzygy module $\Omega^n(X)$, for some $\Lambda$-module $X$. Thus,  $R(c)_{i} \in \GP(\Lambda)$ by (GP2). Since projective modules are GP by definition, this shows the inclusion $`\!\supseteq'$ in (a). 

 It remains to show that there are no further Gorenstein-projective modules. By property (GP2), we only have to consider submodules of projective modules. Using Lemma \ref{L:TechnicalV}, we can exclude all modules which correspond to a word containing $\alpha^{-1} \beta$. In particular, band modules are not GP. 
 
We claim that an indecomposable Gorenstein-projective $\Lambda$-module $M$ containing a subword of the form $\alpha \beta^{-1}$, with $\alpha \neq \beta$ is projective. We think of $\alpha \beta^{-1}$ as a `roof', where $s, t, u$ are basis vectors of $M$, such that $\alpha \cdot t=s$ and $\beta \cdot t =u$. 
\begin{align} \label{E:roof}
\begin{array}{c}
\begin{xy}
\SelectTips{cm}{}
\xymatrix{
& \ar[ld]_{\alpha} t \ar[rd]^{\beta} \\
s&&u
}
\end{xy}
\end{array}
\end{align}
By (GP2), $M$ is a submodule of some projective module $P$. Let $U(t) \subset P$ be the submodule generated by the image of $t$ in $P$. The properties (G1) and (G2) imply that $U(t)$ is projective. If $U(t) \subsetneq M$, then $M$ 
contains a subword of the form $\alpha^{-1} \beta$, with $\alpha \neq \beta$. By Lemma \ref{L:TechnicalV} this cannot happen. So we see that $M \cong U(t)$ is indeed projective.

We have reduced the set of possible indecomposable GP $\Lambda$-modules to projective modules or direct strings $S=\beta_{n} \ldots \beta_{1}$. We also allow $S$ to consists of a single lazy path $e_{i}$ (this corresponds to a simple module). Let $M(S)$ be the corresponding $\Lambda$-module. If $M(S)$ is properly contained in some projective module $P$, then there exists an arrow $\alpha$ such that $\beta_{n} \ldots \beta_{1}\alpha \notin I$ and  $\gamma \beta_{n} \ldots \beta_{1}\alpha \in I$ for every arrow $\gamma\in Q_{1}$. It follows that $M(S)$ is a direct summand of the radical of $P_{s(\alpha)}$.

We have already said that the radical of an indecomposable projective $\Lambda$-module $P$ has at most two indecomposable direct summands $R_{1}$ and $R_{2}$ \eqref{E:Radical}. We need the following claim: if $\iota_{i}$ does not lie on a cycle $c \in \cc(\Lambda)$ then $R_{i}$ has finite projective dimension. 

If $R_{i}$ is not projective the situation locally looks as follows (we allow $n$ to be zero)
\[
\begin{array}{c}
\begin{tikzpicture}[description/.style={fill=white,inner sep=2pt}]
    \matrix (n) [matrix of math nodes, row sep=1em,
                 column sep=2.25em, text height=1.5ex, text depth=0.25ex,
                 inner sep=0pt, nodes={inner xsep=0.3333em, inner
ysep=0.3333em}]
    {  
       \cdots & \sigma & \cdots  &  \bullet \\
       ^{\displaystyle \cdots} \\
    };
\draw[->] (n-1-1) edge  node[ scale=0.75, yshift=3mm] [midway] {$\iota_{i}$} (n-1-2); 
\draw[->] (n-1-2) edge  node[ scale=0.75, yshift=3mm] [midway] {$\beta_{1}$}  (n-1-3);
\draw[->] (n-1-3) edge  node[ scale=0.75, yshift=3mm] [midway] {$\beta_{n}$}  (n-1-4);
  
\draw[->] (n-1-2)  edge  node[scale=0.75, yshift=-1mm, xshift=3mm] [midway] {$\psi_{1}$} (n-2-1); 

\draw[decoration={brace, amplitude=2mm}, decorate] ($(n-1-4.east) +(0, -2.5mm)$) -- node[ scale=0.75, yshift=-5.5mm] [midway] {$R_{i}$} ($(n-1-2.west)+(0, -2.5mm)$);

\end{tikzpicture}
\end{array}
\] 
where $\psi_{1} \iota_{i} \in I$. Moreover, $\psi_{1}$ cannot lie on a cycle, since this would contradict our assumption on $\iota_{i}$. We have a short exact sequence
\begin{align}
0 \ra R' \xrightarrow{\psi_{1}} P_{\sigma} \ra R_{i} \ra 0
\end{align}
where $R'$ is a direct summand of the radical of $P_{\sigma}$. $R'$ has the same properties as $R_{i}$, so we may repeat our argument. After finitely many steps, one of the occuring radical summands will be projective and the procedure stops. Indeed, otherwise we get a path $\ldots\psi_{m} \ldots \psi_{1}\iota_{1}$, such that every subpath of length two is contained in $I$. Since there are only finitely many arrows in $Q$, this path is a cycle. Contradiction. Hence $R_{i}$ has finite projective dimension. Thus it is GP if and only if it is projective, by (GP1). 

We have shown that indecomposable GP modules are either projective or direct summands of the radical of some indecomposable projective, such that the radical embedding is defined by multiplication with an arrow on a cycle $c \in \cc(\Lambda)$. This proves part (a).

\subsubsection{Proof of part (b)}
By Buchweitz' equivalence \eqref{E:Buchweitz}, it suffices to describe the stable category $\ul{\GP}(\Lambda)$. By part (a), the indecomposable objects in this category are precisely the radical summands $R(c)_{i}$ for a cycle $c \in \cc$ and \eqref{E:Shift} shows that $R(c)_{i}[1] \cong R(c)_{i-1}$. In particular, $R(c)_{i}[l(c)] \cong R(c)_{i}$. It remains to prove that $\ul{\Hom}_{\Lambda}(R(c)_{i}, R(c')_{j})=\delta_{ij} \delta_{cc'} \cdot k$. $R(c)_{i}$ is given by a  string of the following form (it starts in $\sigma$ and we allow $n=0$)
\begin{align}
\begin{array}{c}
\begin{xy}
\SelectTips{cm}{}
\xymatrix{
\ar@{..>}[r]^{\iota_{i}} & \sigma \ar[r]^{\beta_{1}} & \ldots \ar[r]^{\beta_{n}} & \bullet,
}
\end{xy}
\end{array}
\end{align} 
Here, $\iota_{i}$ is on the cycle $c \in \cc(\Lambda)$ and $\beta_{1}\iota_{i} \notin I$, if $n \neq 0$. If there is a non-zero morphism of $\Lambda$-modules from $R(c)_{i}$ to $R(c')_{j}$, then the latter has to be a string of the following form
\begin{align}
\begin{array}{c}
\begin{xy}
\SelectTips{cm}{}
\xymatrix{
R(c')_{j}\colon & \sigma' \ar[r]^{\beta'_{1}} & \ldots \ar[r]^{\beta'_{m}} & \sigma \ar[r]^{\beta_{1}} & \cdots \ar[r]^{\beta_{k}} & \bullet,
}
\end{xy}
\end{array}
\end{align} 
where we allow $k=0$ or $m=0$. If both $k$ and $m$ are zero, then (G3) and our assumption that $R(c')_{j}$ is a submodule of an indecomposable projective $\Lambda$-module imply that there is only one arrow starting in $s$. Namely, the arrow on the cycle. Hence, $n=0$ and therefore $R(c)_{i}=R(c')_{j}$. 
If  $k \neq 0$ and $m=0$, then $R(c)_{i}=R(c')_{j}$ as well. 

In both cases, $\End_{\Lambda}(R(c)_{i}) \cong k$. Note, that the simple module $S_{\sigma}$ can appear (at most) twice as a composition factor of $R(c)_{i}$. However, in that case, $R(c)_{i}$ locally has the following form $\cdots \ra \sigma \xrightarrow{\alpha} \cdots \ra \bullet$, where $\alpha \neq \beta_{1}$ lies on the cycle $c$ (see also Example \ref{Ex:DefRc}). In particular, this does not yield additional endomorphisms.

If $k \neq 0$ and $m \neq 0$, then it follows from (G4) that $\beta'_{m}=\iota_{i}$. 
If $k =0$, $m \neq 0$ and $\beta'_{m} \neq\iota_{i}$ then there are two different arrows ending in $\sigma$. Since $\iota_{i}$ is on a cycle there is an arrow $\gamma\colon \sigma \ra \bullet$, such that $\gamma\iota_{i} \in I$. It follows from (G3) that $\gamma \beta'_{m} \notin I$. Since $R(c')_{j}$ is a submodule of a projective $\Lambda$-module the corresponding path starting in $\sigma'$ has to be maximal. In particular, it does not end in $\sigma$. Contradiction. So we have $\beta'_{m}=\iota_{i}$.

In both cases our morphism factors over a projective module
\begin{align}
R(c)_{i} \ra P_{s(\iota_{i})} \ra R(c')_{j}
\end{align}
and therefore $\ul{\Hom}_{\Lambda}(R(c)_{i}, R(c')_{j})=0$. This completes the proof.

\end{document}